\documentclass{amsart}

\usepackage{amssymb, amscd, amsthm, mathrsfs}

\newtheoremstyle{thmparen}
{}{}{\itshape}{}{\bfseries}{.}{5pt plus 1pt minus 1pt}
{\thmname{#1}\thmnumber{ (#2)}{\normalfont\thmnote{ (#3)}}}
\newtheoremstyle{rmkparen}
{}{}{}{}{\itshape}{.}{5pt plus 1pt minus 1pt}
{\thmname{#1}{\normalfont\thmnumber{ (#2)}\thmnote{ (#3)}}}

\newtheorem{MainThm}{Theorem}

\theoremstyle{thmparen}
\newtheorem*{GroConj}{Grothendieck's conjecture}
\newtheorem{Thm}[equation]{Theorem}
\newtheorem{Prop}[equation]{Proposition}

\newtheorem{Lem}[equation]{Lemma}

\theoremstyle{rmkparen}
\newtheorem{Rmk}[equation]{Remark}
\newtheorem*{Notation}{Notation}
\newtheorem*{Ack}{Acknowledgement}

\newcommand{\Order}{\mathcal{O}}
\newcommand{\into}{\hookrightarrow}
\newcommand{\onto}{\twoheadrightarrow}
\newcommand{\isomto}{\overset{\sim}{\to}}
\newcommand{\isomfrom}{\overset{\sim}{\leftarrow}}
\newcommand{\compose}{\mathrel{\circ}}
\newcommand{\tensor}{\otimes}

\newcommand{\closure}[1]{\overline{#1}}
\newcommand{\N}{\mathbb{N}}
\newcommand{\Z}{\mathbb{Z}}
\newcommand{\Q}{\mathbb{Q}}
\newcommand{\C}{\mathbb{C}}
\newcommand{\F}{\mathbb{F}}
\newcommand{\Affine}{\mathbb{A}}

\newcommand{\sep}{\mathrm{sep}}
\newcommand{\ur}{\mathrm{ur}}
\newcommand{\et}{\mathrm{et}}
\newcommand{\zar}{\mathrm{zar}}

\newcommand{\fppf}{\mathrm{fppf}}
\newcommand{\ffl}{\mathrm{ffl}}

\newcommand{\rat}{\mathrm{rat}}

\newcommand{\perf}{\mathrm{perf}}

\newcommand{\alge}{\mathrm{alg}}
\newcommand{\id}{\mathrm{id}}
\newcommand{\incl}{\mathrm{incl}}

\newcommand{\op}{\mathrm{op}}
\newcommand{\Alg}{\mathrm{Alg}}
\newcommand{\PDual}{\mathrm{PD}}
\newcommand{\CDual}{\mathrm{CD}}
\newcommand{\SDual}{\mathrm{SD}}
\newcommand{\LDual}{\mathrm{LD}}
\newcommand{\sAb}{\mathrm{sAb}}
\newcommand{\Gm}{\mathbf{G}_{m}}
\newcommand{\Ga}{\mathbf{G}_{a}}
\newcommand{\Et}{\mathrm{Et}}
\newcommand{\FEt}{\mathrm{FEt}}
\newcommand{\FGEt}{\mathrm{FGEt}}
\newcommand{\Fin}{\mathrm{Fin}}

\newcommand{\uc}{\mathrm{uc}}
\newcommand{\ind}{\mathrm{ind}}
\newcommand{\pro}{\mathrm{pro}}
\newcommand{\tor}{\mathrm{tor}}

\newcommand{\fp}{\mathrm{fp}}
\newcommand{\sm}{\mathrm{sm}}
\newcommand{\set}{\mathrm{set}}
\newcommand{\sh}{\mathrm{sh}}
\newcommand{\Ind}{\mathrm{I}}
\newcommand{\Pro}{\mathrm{P}}
\newcommand{\Loc}{\mathrm{L}}
\newcommand{\ideal}[1]{\mathfrak{#1}}
\newcommand{\alg}[1]{\mathbf{#1}}
\newcommand{\dirlim}{\varinjlim}
\newcommand{\invlim}{\varprojlim}
\newcommand{\fdirlim}[1][]{\mathop{\text{``$\dirlim$''}}_{#1}}
\newcommand{\finvlim}[1][]{\mathop{\text{``$\invlim$''}}_{#1}}
\newcommand{\fdsum}[1][]{\mathop{\text{``$\bigoplus$''}}_{#1}}
\newcommand{\fprod}[1][]{\mathop{\text{``$\prod$''}}_{#1}}
\mathchardef\mhyphen="2D
\newcommand{\Mod}[1]{#1\mhyphen\mathrm{Mod}}
\DeclareMathOperator{\Gal}{Gal}
\DeclareMathOperator{\Hom}{Hom}

\DeclareMathOperator{\Ker}{Ker}
\DeclareMathOperator{\Coker}{Coker}
\DeclareMathOperator{\Ext}{Ext}
\DeclareMathOperator{\Spec}{Spec}
\DeclareMathOperator{\Ab}{Ab}
\DeclareMathOperator{\Set}{Set}
\DeclareMathOperator{\Res}{Res}

\DeclareMathOperator{\sheafhom}{\alg{Hom}}
\DeclareMathOperator{\sheafext}{\alg{Ext}}

\DeclareMathOperator{\dlog}{dlog}

\newcommand{\BetweenThmAndList}{\leavevmode}

\title[Grothendieck's pairing on N\'eron component groups]
{Grothendieck's pairing on N\'eron component groups: Galois descent from the semistable case}
\author{Takashi Suzuki}
\address{
	Department of Mathematics, Tokyo Institute of Technology, 2-12-1
	Ookayama, Meguro, Tokyo 152-8551, Japan
}
\email{suzuki.t.aw@m.titech.ac.jp}
\date{March 22, 2018}

\begin{document}

\begin{abstract}
	In our previous study of duality for complete discrete valuation fields
	with perfect residue field,
	we treated coefficients in finite flat group schemes.
	In this paper, we treat abelian varieties.
	This in particular implies Grothendieck's conjecture on the perfectness of his pairing
	between the N\'eron component groups of an abelian variety and its dual.
	The point is that our formulation is well-suited with Galois descent.
	From the known case of semistable abelian varieties,
	we deduce the perfectness in full generality.
	We also treat coefficients in tori and, more generally, 1-motives.
\end{abstract}

\maketitle

\tableofcontents


\numberwithin{equation}{section}
\section{Introduction}
\label{sec: Introduction}

\numberwithin{equation}{subsection}
\subsection{Main results}
Let $K$ be a complete discrete valuation field with
ring of integers $\Order_{K}$ and perfect residue field $k$ of characteristic $p > 0$.
Let $A$ be an abelian variety over $K$,
$\mathcal{A}$ the N\'eron model of $A$ over $\Order_{K}$
and $\mathcal{A}_{x}$ the special fiber of $\mathcal{A}$ over $x = \Spec k$.
The component group $\pi_{0}(\mathcal{A}_{x})$ of $\mathcal{A}_{x}$
is a finite \'etale group scheme over $k$.
From the dual abelian variety $A^{\vee}$,
we have corresponding objects
$\mathcal{A}^{\vee}$, $\mathcal{A}_{x}^{\vee}$ and $\pi_{0}(\mathcal{A}_{x}^{\vee})$.
In \cite[IX, 1.2.1]{Gro72}, Grothendieck constructed a canonical pairing
	\[
			\pi_{0}(\mathcal{A}_{x}^{\vee}) \times \pi_{0}(\mathcal{A}_{x})
		\to
			\Q / \Z,
	\]
which appears as the obstruction to extending the Poincar\'e biextension on $A^{\vee} \times A$
to the N\'eron models $\mathcal{A}^{\vee} \times \mathcal{A}$.
In this paper, we prove the following conjecture of Grothendieck \cite[IX, Conj.\ 1.3]{Gro72}:

\begin{GroConj}
	The pairing
		$
				\pi_{0}(\mathcal{A}_{x}^{\vee}) \times \pi_{0}(\mathcal{A}_{x})
			\to
				\Q / \Z
		$
	above is perfect.
\end{GroConj}

Some previously known cases include:
$k$ is finite \cite{McC86};
$K$ has mixed characteristic \cite{Beg81};
$A$ has semistable reduction \cite{Wer97};
the prime-to-$p$ part \cite{Ber01};
$A$ is the Jacobian of a curve with a rational point \cite{BL02};
$A$ has potentially multiplicative reduction \cite{Bos97}.%
\footnote{
	The conjecture may fail when $k$ is imperfect \cite{BB00}.
	The results on the semistable case and the prime-to-$p$ part are valid
	for any (possibly imperfect or zero characteristic) residue field $k$.
	In this paper, we only consider the case that $k$ is perfect of positive characteristic.
}
McCallum used (the usual) local class field theory and local Tate duality.
B\'egueri instead used Serre's local class field theory \cite{Ser61},
which is applicable for general perfect residue fields $k$.
To deduce Grothendieck's conjecture,
B\'egueri used some dimension counting argument,
which does not work in equal characteristic
since the corresponding objects are infinite-dimensional.
Werner, on the other hand, used rigid analytic uniformization for semistable abelian varieties.
It does not seem possible to deduce the general case from her result
by a simple application of the semistable reduction theorem,
since N\'eron models are known to behave very badly under base change.

The method we take in this paper is first to formulate a local Tate duality type statement
for the general case.
This is a version for abelian varieties of the author's reformulation of a known duality theorem,
where coefficients were finite flat group schemes
(\cite{Suz13}).
Our formulation crucially relies on the techniques of sheaves on the category of fields
developed in \cite{Suz13}.
We prove that the duality we formulate here is equivalent to the conjunction of
Grothendieck's conjecture and \v{S}afarevi\v{c}'s conjecture,
the latter of which was posed in \cite{Saf61}
and solved by B\'egueri \cite{Beg81} (mixed characteristic case),
Bester \cite{Bes78} (equal characteristic, good reduction) and
Bertapelle \cite{Ber03} (equal characteristic, general reduction).
Our duality is very functorial in an appropriate derived category
that is strong enough to treat the infinite-dimensional groups involved,
and well-suited with Galois descent.
Therefore we are able to prove that if the duality is true over a finite Galois extension of $K$,
then it is true over $K$.
This and the semistable case cited above together imply Grothendieck's conjecture.

Here we sketch our formulation of the duality with coefficients in abelian varieties.
Recall from \cite{Suz13} that a $k$-algebra is said to be \emph{rational}
if it is a finite direct product of the perfections (direct limit along Frobenii) of
finitely generated fields over $k$,
and \emph{ind-rational} if it is a filtered union of rational $k$-subalgebras.
We denote the category of ind-rational $k$-algebras by $k^{\ind\rat}$.
We can endow this category with the pro-\'etale topology
(\cite{BS15}; see \S \ref{sec: Sites and algebraic groups: setup and first properties} below
for the details about the pro-\'etale topology in this ind-rational setting).
We denote the resulting site by $\Spec k^{\ind\rat}_{\pro\et}$.
The derived category of sheaves on $\Spec k^{\ind\rat}_{\pro\et}$ is denoted by $D(k^{\ind\rat}_{\pro\et})$
and the derived sheaf-Hom functor by $R \sheafhom_{k^{\ind\rat}_{\pro\et}}$.
For an object $C \in D(k^{\ind\rat}_{\pro\et})$, we define
	\[
			C^{\SDual}
		=
			R \sheafhom_{k^{\ind\rat}_{\pro\et}}(C, \Z)
	\]
and call it the \emph{Serre dual} of $C$
(cf.\ \cite[8.4, Remarque]{Ser60}, \cite[III, Thm.\ 0.14]{Mil06}).
The double dual $\SDual \SDual$ sends
the perfection (inverse limit along Frobenii) of a pro-unipotent group to itself.
It sends the perfection of a semi-abelian variety $G$ to its profinite Tate module placed in degree $-1$.
Hence the universal covering of $G$ placed in degree $-1$
gives a mapping cone of $G \to G^{\SDual \SDual}$,
which is uniquely divisible.

As in the beginning of the paper,
let $K$ be a complete discrete valuation field
with perfect residue field $k$ of characteristic $p > 0$
and $A$ an abelian variety over $K$.
We can regard the cohomology complex $R \Gamma(A) = R \Gamma(K, A)$ of $K$
as a complex of sheaves on $\Spec k^{\ind\rat}_{\pro\et}$ in a suitable way.
We denote the resulting complex by $R \alg{\Gamma}(A)$
and its $n$-th cohomology by $\alg{H}^{n}(A)$.
The sheaf $\alg{\Gamma}(A) = \alg{H}^{0}(A)$ is represented by
the perfection of the Greenberg transform of the N\'eron model of $A$.
This is a proalgebraic group over $k$,
and the reduction map induces an isomorphism
$\pi_{0}(\alg{\Gamma}(A)) \isomto \pi_{0}(\mathcal{A}_{x})$.
The sheaf $\alg{H}^{1}(A)$ is ind-algebraic,
and all higher cohomology vanishes.
We have
	\[
			R \alg{\Gamma}(A)^{\SDual}
		=
			R \sheafhom_{k^{\ind\rat}_{\pro\et}}(R \alg{\Gamma}(A), \Q / \Z)[-1].
	\]
The trace morphism $R \alg{\Gamma}(\Gm) \to \Z$ given in \cite[Prop.\ 2.4.4]{Suz13}
and the cup product formalism give a pairing
	\[
			R \alg{\Gamma}(A^{\vee})
		\times
			R \alg{\Gamma}(A)
		\to
			\Z[1]
	\]
in $D(k^{\ind\rat}_{\pro\et})$.
This induces morphisms
	\[
			R \alg{\Gamma}(A)
		\to
			R \alg{\Gamma}(A^{\vee})^{\SDual}[1]
		\quad \text{and} \quad
			R \alg{\Gamma}(A^{\vee})^{\SDual \SDual}
		\to
			R \alg{\Gamma}(A)^{\SDual}[1].
	\]
We will prove the following.

\begin{MainThm} \label{mainthm: duality for abelian varieties}
	The above defined morphism
		\[
				R \alg{\Gamma}(A^{\vee})^{\SDual \SDual}
			\to
				R \alg{\Gamma}(A)^{\SDual}[1]
		\]
	in $D(k^{\ind\rat}_{\pro\et})$ is an isomorphism.
\end{MainThm}

We recall \v{S}afarevi\v{c}'s conjecture,
or a theorem of B\'egueri and Bester-Bertapelle.
In our terminology, \v{S}afarevi\v{c}'s conjecture can be stated as the existence of a canonical isomorphism
	\[
			H^{1}(A^{\vee})
		\cong
			\Ext_{k^{\ind\rat}_{\pro\et}}^{1}(\alg{\Gamma}(A), \Q / \Z)
	\]
for $k$ algebraically closed.
The right-hand side is the Pontryagin dual
the fundamental group $\pi_{1}(\alg{\Gamma}(A))$
of the proalgebraic group $\alg{\Gamma}(A)$ in the sense of Serre \cite{Ser60}.
Hence \v{S}afarevi\v{c}'s conjecture claims the existence of a perfect pairing
	\[
			H^{1}(A^{\vee})
		\times
			\pi_{1}(\alg{\Gamma}(A))
		\to
			\Q / \Z
	\]
between the torsion group and the profinite group.
Note that $\pi_{1}$ is indifferent to component groups
and observe how this conjecture is complementary to Grothendieck's conjecture.

We will prove the following more precise form of the above theorem.

\begin{MainThm} \BetweenThmAndList \label{mainthm: relation with Gro and Sha conjectures}
	\begin{enumerate}
		\item \label{ass: Gro, Sha and main theorem}
			Theorem \ref{mainthm: duality for abelian varieties} is equivalent to the conjunction of
			Grothendieck's conjecture and \v{S}afarevi\v{c}'s conjecture.
		\item \label{ass: Galois descent}
			If Theorem \ref{mainthm: duality for abelian varieties} is true for
			$A \times_{K} L$ for a finite Galois extension $L$ of $K$
				\[
						R \alg{\Gamma}(L, A^{\vee})^{\SDual \SDual}
					\isomto
						R \alg{\Gamma}(L, A)^{\SDual}[1],
				\]
			then it is true for the original $A$.
	\end{enumerate}
\end{MainThm}

Since \v{S}afarevi\v{c}'s conjecture is already a theorem,
these results together with the semistable reduction theorem and Werner's result imply:

\begin{MainThm} \label{mainthm: Gro conj}
	Grothendieck's conjecture is true.
\end{MainThm}

In the next logically optional subsection,
we explain the general picture of our formulation and proof.
The organization of the paper will be explained in the subsection after next.


\numberwithin{equation}{subsection}
\subsection{General picture}
\label{sec: Motivation of our constructions}
We explain the general picture of our proof of Grothendieck's conjecture
without getting into the details or giving precise definitions,
in order to motivate our very peculiar constructions.
For the convenience of the reader,
we repeat the main ideas in our previous paper \cite{Suz13}.
Notation is as in the previous subsection.
Assume that $k$ is algebraically closed for simplicity.

Suppose for the moment that
we have a nice complex $R \alg{\Gamma}(K, A) = R \alg{\Gamma}(A)$
of ind- or proalgebraic groups over $k$
for an abelian variety $A$ over $K$,
so that its $k$-points is $R \Gamma(A)$
(hence $\alg{H}^{n}(A) = 0$ for $n \ge 2$ by a cohomological dimension reason).
The group $\alg{\Gamma}(A)$ is proalgebraic
and $\alg{H}^{1}(A)$ ind-algebraic.
What is nice about $R \alg{\Gamma}(A)$ is that
\emph{it can capture the N\'eron component group $\pi_{0}(\mathcal{A}_{x})$
and behaves well under base change}.

To explain this,
let $\ideal{p}_{K}$ be the maximal ideal of $\Order_{K}$.
We have a surjection
	\[
			\alg{\Gamma}(K, A)
		=
			\alg{\Gamma}(\Order_{K}, \mathcal{A})
		=
			\invlim_{n}
				\alg{\Gamma}(\Order_{K} / \ideal{p}_{K}^{n}, \mathcal{A})
		\stackrel{n = 1}{\onto}
			\mathcal{A}_{x}
	\]
of proalgebraic groups with connected kernel,
so $\pi_{0}(\alg{\Gamma}(A)) = \pi_{0}(\mathcal{A}_{x})$.
This explains how $\pi_{0}(\mathcal{A}_{x})$ can be functorially recovered from $\alg{\Gamma}(A)$.
Here it is essential to put a proalgebraic group structure on $\alg{\Gamma}(A)$
to make sense of its $\pi_{0}$.
On the other hand, the Hochschild-Serre spectral sequence
	\[
			R \Gamma \bigl(
				\Gal(L / K), R \alg{\Gamma}(L, A)
			\bigr)
		=
			R \alg{\Gamma}(K, A)
	\]
for a finite Galois extension $L / K$
shows that $R \alg{\Gamma}(K, A)$ can be recovered from $R \alg{\Gamma}(L, A)$ very cheaply,
while there is no such simple relation for N\'eron models.

We expect some duality between the proalgebraic group $\alg{\Gamma}(A)$
and the ind-algebraic group $\alg{H}^{1}(A^{\vee})$
that takes care of $\pi_{0}(\alg{\Gamma}(A))$.
This should be analogous to the usual local Tate duality in the finite residue field case.
A little more precisely,
let's expect that there should be an isomorphism
	\begin{equation} \label{eq: naive duality morphism}
			R \alg{\Gamma}(A^{\vee})
		\isomto
			R \sheafhom_{k}(R \alg{\Gamma}(A), \Q / \Z)
	\end{equation}
up to some completion of the left-hand side.
Here $R \sheafhom_{k}$ is some internal $R \Hom$ functor
for the category of ind- or proalgebraic groups over $k$.
The actual duality statement needs the completion or double-dual as in previous subsection,
since semi-abelian varieties (appearing in $\mathcal{A}_{x}$)
are not double-dual invariant (while unipotent groups are so).
Let's ignore the double-dual.
This duality statement is robust for Galois descent:
if it is true for $A$ over a finite Galois extension $L / K$,
then it is true for $A$ over $K$,
basically by the Hochschild-Serre spectral sequence above.

The statement is equivalent to the conjunction of
Grothendieck's and \v{S}afarevi\v{c}'s conjectures as follows.
The isomorphism (ignoring the completion) gives a hyperext spectral sequence
	\[
			E_{2}^{i j}
		=
			\sheafext_{k}^{i}(\alg{H}^{-j}(A), \Q / \Z)
		\Longrightarrow
			\alg{H}^{i + j}(A^{\vee}).
	\]
The functor $\sheafext_{k}^{i}(\;\cdot\;, \Q / \Z)$ is dual to the $i$-th homotopy group functor
by \cite[\S 5, Cor.\ to Prop.\ 7]{Ser60},
hence zero for $i \ge 2$ by \cite[\S 10, Thm.\ 2]{Ser60}.
Therefore this spectral sequence degenerates at $E_{2}$
and becomes (ignoring the term $\sheafhom_{k}(\alg{H}^{1}(A), \Q / \Z)$)
an exact sequence and an isomorphism
	\begin{gather*}
				0
			\to
				\sheafext_{k}^{1}(\alg{H}^{1}(A), \Q / \Z)
			\to
				\alg{\Gamma}(A^{\vee})
			\to
				\sheafhom_{k}(\alg{\Gamma}(A), \Q / \Z)
			\to
				0,
		\\
				\alg{H}^{1}(A^{\vee})
			\cong
				\sheafext_{k}^{1}(\alg{\Gamma}(A), \Q / \Z)
	\end{gather*}
(with a completion to $\alg{\Gamma}(A^{\vee})$).
The isomorphism gives \v{S}afarevi\v{c}'s conjecture.
As \cite[III Lem.\ 0.13 (c)]{Mil06} suggests,
the group $\sheafext_{k}^{1}(\alg{H}^{1}(A), \Q / \Z)$ should be connected.
Hence the exact sequence should be a connected-\'etale sequence.
The connected part gives the statement that
the identity component of $\alg{\Gamma}(A^{\vee})$ up to completion
is $\sheafext_{k}^{1}(\alg{H}^{1}(A), \Q / \Z)$.
This is dual to \v{S}afarevi\v{c}'s conjecture.
The \'etale part gives an isomorphism
	\[
			\pi_{0}(\alg{\Gamma}(A^{\vee}))
		\cong
			\sheafhom_{k}(\pi_{0}(\alg{\Gamma}(A)), \Q / \Z).
	\]
Writing the groups as N\'eron component groups,
we get the isomorphism predicted by Grothendieck's conjecture.

Here it is very important to have a \emph{complex}, $R \alg{\Gamma}(A)$,
of ind- or proalgebraic groups, and work in a derived category.
Looking at the terms individually seriously breaks the good behavior under base change.

How can we define such a complex?
A naive approach is the following.
For a perfect $k$-algebra $R$,
we define its canonical lift ``$R \tensor_{k} K$'' to $K$ to be
	\[
			\alg{K}(R)
		=
				(W(R) \hat{\tensor}_{W(k)} \Order_{K})
			\tensor_{\Order_{K}}
				K,
	\]
where $W$ is the affine scheme of Witt vectors of infinite length.
More explicitly, if $K = W(k)[1 / p][x] / (f(x))$ with some Eisenstein polynomial $f(x) \in W(k)[x]$,
then $\alg{K}(R) = W(R)[1 / p][x] / (f(x))$,
and if $K = k((T))$, then $\alg{K}(R) = R[[T]][1 / T]$.
If $R = k'$ is a perfect field, the ring $\alg{K}(k')$ is a complete discrete valuation field
obtained from $K$ by extending the residue field from $k$ to $k'$.
Take an injective resolution $I^{\cdot}$ of $A$ over the fppf site of $K$.
Consider the complex of presheaves
	\[
			R
		\mapsto
			\Gamma(\alg{K}(R), I^{\cdot})
	\]
on the category of perfect $k$-algebras.
Its (pro-)\'etale sheafification is the candidate of our complex.
Its $n$-th cohomology is the sheafification of the presheaf
	\[
			R
		\mapsto
			H^{n}(\alg{K}(R), A).
	\]
But this fppf cohomology group is very difficult to calculate.
It is a classical object only if we restrict the $R$'s to be perfect fields.
This means that we can only obtain the generic behavior of the sheaf from our classical knowledge.

But \emph{generic behavior is sufficient to describe algebraic groups}
in view of Weil's theory of \emph{birational groups}.
This means that there is no information lost by treating proalgebraic groups
as functors on the category of ind-rational $k$-algebras,
which is essentially the \emph{category of perfect fields} over $k$.
Now the sheafification of the complex of presheaves
	\[
			k'
		\mapsto
			\Gamma(\alg{K}(k'), I^{\cdot})
	\]
on our ind-rational pro-\'etale site $\Spec k^{\ind\rat}_{\pro\et}$ yields
the sought-for object $R \alg{\Gamma}(A)$.
This explain why this strange site $\Spec k^{\ind\rat}_{\pro\et}$ is needed in this paper.
Moreover, this type of complexes and cohomology groups looks
very similar to the well-known general description of higher pushforward sheaves.
Therefore we are tempted to define a \emph{morphism of sites}
	\[
			\pi
		\colon
			\Spec K_{\fppf}
		\to
			\Spec k^{\ind\rat}_{\pro\et}
	\]
corresponding to the functor $k' \mapsto \alg{K}(k')$ on the underlying categories,
so that we have a more systematic definition $R \alg{\Gamma}(A) = R \pi_{\ast} A$.
In fact, a little care and modifications are needed for continuity and exactness of pullback.

Now the picture is the following.
The duality we want should be a relative duality for the morphism $\pi$.
The cup product formalism in site theory will give us
the duality morphism \eqref{eq: naive duality morphism}.
The derived sheaf-Hom functor $R \sheafhom_{k^{\ind\rat}_{\pro\et}}$
on the site $\Spec k^{\ind\rat}_{\pro\et}$
should give the internal $R \Hom$ functor on the category of ind-proalgebraic groups.

All we do in this paper is to carry out these ideas rigorously
making necessary corrections to imprecise ideas.
There are three difficulties to overcome.
One is to show that
$R \sheafhom_{k^{\ind\rat}_{\pro\et}}(G, H)$ behaves well
for ind-proalgebraic groups (i.e., ind-objects of proalgebraic groups) $G, H$.
We already know this for proalgebraic $G$ and algebraic $H$
by \cite[Thm.\ 2.1.5]{Suz13}.
To extend this result for ind-proalgebraic groups,
we will need heavy derived limit arguments following Kashiwara-Schapira \cite[Chap.\ 15]{KS06}.
Second, cohomology groups of the form $H^{n}(\alg{K}(k'), A)$
are not completely classical when $k'$ has infinitely many direct factors
and $\Spec k'$ as a topological space is profinite.
The ring $\alg{K}(k')$ is a finite direct product of complete discrete valuation fields
if $k'$ has only finitely many direct factors,
but otherwise it is quite complicated.
We study some site-theoretic properties of the ring $\alg{K}(k')$
by approximation by complete discrete valuation subfields.
In \cite[\S 2.5]{Suz13}, we calculated $H^{n}(\alg{K}(k'), \Gm)$.
We need to further develop the techniques used there
in order to calculate $H^{n}(\alg{K}(k'), A)$ for an abelian variety $A$.
The third point is that we need to redo Bester's work \cite{Bes78}
within the style of this paper.
This is lengthy, but there is essentially no new idea needed in this part.


\numberwithin{equation}{subsection}
\subsection{Organization}
\label{sec: Organization}
This paper is organized as follows.
In \S \ref{sec: Site-theoretic preliminaries},
we study how to treat ind-proalgebraic groups over $k$
as sheaves on $\Spec k^{\ind\rat}_{\pro\et}$.
In \S \ref{sec: Local fields with ind-rational base},
we view cohomology of $K$ with various coefficients as sheaves on $\Spec k^{\ind\rat}_{\pro\et}$
and compute it.
In \S \ref{sec: Statement of the duality theorem},
we construct the duality morphism of Theorem \ref{mainthm: duality for abelian varieties}
and show that it induces two pairings
	\[
			\pi_{0}(\mathcal{A}_{x}^{\vee}) \times \pi_{0}(\mathcal{A}_{x})
		\to
			\Q / \Z,
		\quad
			H^{1}(A^{\vee})
		\times
			\pi_{1}(\alg{\Gamma}(A))
		\to
			\Q / \Z.
	\]
The perfectness of these pairings is equivalent to
Theorem \ref{mainthm: duality for abelian varieties}.
In \S \ref{sec: Comparison with other authors' work},
we show that the first pairing agrees with Grothendieck's pairing
and the second with B\'egueri-Bester-Bertapelle's pairing.
This proves Theorem \ref{mainthm: relation with Gro and Sha conjectures}
\eqref{ass: Gro, Sha and main theorem}
and the semistable case of Theorem \ref{mainthm: duality for abelian varieties}.
In \S \ref{sec: Galois descent},
we prove that Galois descent works for Theorem \ref{mainthm: duality for abelian varieties},
namely Theorem \ref{mainthm: relation with Gro and Sha conjectures} \eqref{ass: Galois descent}.
In \S \ref{sec: End of proof: Grothendieck's conjecture},
we summarize the results of the preceding sections
to conclude the proof of Theorems \ref{mainthm: duality for abelian varieties},
\ref{mainthm: relation with Gro and Sha conjectures}
and \ref{mainthm: Gro conj},
which finishes the proof of Grothendieck's conjecture.

In \S \ref{sec: Duality with coefficients in tori},
we formulate and prove an analogue of Theorem \ref{mainthm: duality for abelian varieties}
for tori.
Most of the results of this section has already been obtained by
B\'egueri (loc.cit.), Xarles \cite{Xar93} and Bertapelle-Gonz\'alez-Avil\'es \cite{BGA15}.
In \S \ref{sec: Duality with coefficients in 1-motives},
the dualities for abelian varieties and tori are combined into
a duality for $1$-motives.
This extends the finite residue field case of the duality shown by Harari-Szamuely \cite{HS05}.
In \S \ref{sec: Connection with classical statements for finite residue fields},
we show, in the finite residue field case,
how to pass from the above sheaf setting to the classical setting.
In Appendix \ref{sec: The pro-fppf topology for proalgebraic proschemes},
we explain how to treat proalgebraic groups that are proschemes but not schemes.
We do this by extending the pro-fppf topology introduced in \cite{Suz13} for affine schemes
to proalgebraic proschemes.

\begin{Ack}
	The author expresses his deep gratitude to Kazuya Kato for his support and helpful discussions.
	This paper is seriously based on his ideas on local duality
	\cite[Introduction]{Kat86}, \cite[\S 3.3]{Kat91}.
	The author is grateful to Clifton Cunningham and David Roe,
	who explained the author relations between their work \cite{CR15} and
	Serre's local class field theory \cite{Ser61}.
	Their invitation, hospitality and useful discussions on tori over local fields at Calgary
	were valuable.
	The author would like to thank Alessandra Bertapelle and Cristian D. Gonz\'alez-Avil\'es,
	who informed the author of their work \cite{BGA15}
	and encouraged to give details on the duality with coefficients in tori.
	Bertapelle also helped the author
	to understand several points in her work on \v{S}afarevi\v{c}'s conjecture.
	Special thanks to C\'edric P\'epin for sharing his interest in Grothendieck's conjecture
	and useful discussions,
	to Michal Bester for acknowledging our treatment of his work in this paper,
	and to the referee for their very careful reading and suggestions to make the paper more readable.
\end{Ack}

\begin{Notation}
	We fix a perfect field $k$ of characteristic $p > 0$.
	A perfect field over $k$ is said to be finitely generated
	if it is the perfection (direct limit along Frobenii) of a finitely generated field over $k$.
	The same convention is applied to morphisms of perfect $k$-algebras or $k$-schemes
	being finite type, finite presentation, etc.
	The categories of sets and abelian groups are denoted by $\Set$ and $\Ab$, respectively.
	Set theoretic issues are omitted for simplicity
	as the main results hold independent of the choice of universes.
	The opposite category of a category $\mathcal{C}$ is denoted by $\mathcal{C}^{\op}$.
	The procategory and indcategory of $\mathcal{C}$ are denoted by
	$\Pro \mathcal{C}$ and $\Ind \mathcal{C}$, respectively,
	so that $\Ind \Pro \mathcal{C} = \Ind (\Pro \mathcal{C})$
	is the ind-procategory.
	All group schemes (except for Galois groups) are assumed to be commutative.
	We say that a (commutative) \'etale group scheme over $k$ is finitely generated
	if its group of geometric points is finitely generated as an abelian group.
	A lattice is a finitely generated \'etale group with no torsion or,
	equivalently, a finite free abelian group with a Galois action.
	For an abelian category $\mathcal{C}$,
	we denote by $D^{b}(\mathcal{C})$, $D^{+}(\mathcal{C})$, 
	$D^{-}(\mathcal{C})$, $D(\mathcal{C})$
	its bounded, bounded below, bounded above and unbounded derived categories, respectively.
	If we say $A \to B \to C$ is a distinguished triangle in a triangulated category,
	we implicitly assume that a morphism $C \to A[1]$ to the shift of $A$ is given
	and $A \to B \to C \to A[1]$ is distinguished.
	For a Grothendieck site $S$ and a category $\mathcal{C}$,
	we denote by $\mathcal{C}(S)$ the category of sheaves on $S$ with values in $\mathcal{C}$.
	For an object $X$ of $S$,
	the category $S / X$ of objects of $S$ over $X$ is equipped with the induced topology
	(\cite[III, \S 3]{AGV72a}),
	which is the localization of $S$ at $X$ (\cite[III, \S 5]{AGV72a}).
	The category $D^{\ast}(\Ab(S))$ for $\ast = b, +, -$ or (blank) is also denoted by $D^{\ast}(S)$.
	By a continuous map $f \colon S' \to S$ between sites $S'$ and $S$,
	we mean a continuous functor from the underlying category of $S$ to that of $S'$,
	i.e., the right composition (or the pushforward $f_{\ast}$)
	sends sheaves on $S'$ to sheaves on $S$.
	By a morphism $f \colon S' \to S$ of sites
	we mean a continuous map whose pullback functor $\Set(S) \to \Set(S')$ is exact.
	For an abelian category $\mathcal{C}$,
	we denote by $\Ext_{\mathcal{C}}^{i}$ the $i$-th Ext functor for $\mathcal{C}$.
	If $\mathcal{C} = \Ab(S)$,
	we also write $\Ext_{S}^{i} = \Ext_{\mathcal{C}}^{i}$.
	The sheaf-Hom and sheaf-Ext functors are denoted by
	$\sheafhom_{S}$ and $\sheafext_{S}^{n}$, respectively.
	For sites such as $\Spec k^{\ind\rat}_{\pro\et}$,
	we also use $\Ext_{k^{\ind\rat}_{\pro\et}}$, $\Ab(k^{\ind\rat}_{\pro\et})$ etc.\
	omitting $\Spec$ from the notation.
	Similarly, cohomology of the site $\Spec k^{\ind\rat}_{\pro\et}$ for example
	is denoted by $R \Gamma(k^{\ind\rat}_{\pro\et}, A)$, where $A$ is a sheaf of abelian groups.
	
	Statements (theorems, propositions etc.), remarks and equations share the numbering system.
	Their numbers are preceded by (sub)(sub)section numbers and surrounded by parentheses.
	Therefore (6.2) for example means the second one of the statements, remarks and equations in \S 6,
	and (5.2.2.3) means the third one of such in \S 5.2.2.
	Listed items are labeled by letters.
	Therefore (b) means the second item in a list.
	As a consequence, (2.2.1) (c) means the third item in the list in (2.2.1).
\end{Notation}


\numberwithin{equation}{section}
\section{Site-theoretic preliminaries}
\label{sec: Site-theoretic preliminaries}
Let $k$ be a perfect field of characteristic $p > 0$.
In this section, we introduce the ind-rational pro-\'etale site $\Spec k^{\ind\rat}_{\pro\et}$.
We also recall the perfect pro-fppf site $\Spec k^{\perf}_{\pro\fppf}$ from \cite{Suz13}.
We call a help from $\Spec k^{\perf}_{\pro\fppf}$
to establish a basic method of treating ind-proalgebraic groups
as sheaves on $\Spec k^{\ind\rat}_{\pro\et}$.
This needs to extend the result \cite[Thm.\ 2.1.5]{Suz13} on Ext groups of proalgebraic groups
to ind-proalgebraic groups.
For this, we develop a general study on
the derived Hom functor $R \Hom$ on ind-procategories.
Then we can embed the derived category of ind-proalgebraic groups
into the derived category of sheaves on $\Spec k^{\ind\rat}_{\pro\et}$.
We also extend Serre duality
(\cite[8.4, Remarque]{Ser60}, \cite[III, Thm.\ 0.14]{Mil06})
on perfect unipotent groups ($=$ quasi-algebraic unipotent groups)
to proalgebraic and ind-algebraic groups.

At the end of this section,
we introduce the notion of P-acyclicity.
In the next section, we will view
the cohomology of complete discrete valuation fields with residue field $k$
as sheaves on the ind-rational \'etale site $\Spec k^{\ind\rat}_{\et}$ first
and then as sheaves on $\Spec k^{\ind\rat}_{\pro\et}$ by sheafification.
The cohomology of the pro-\'etale sheafification of an \'etale sheaf is
difficult to calculate in general.
In the situations we are interested in, however,
we will see that most of these \'etale sheaves are P-acyclic.
This means that the associated pro-\'etale sheaves still remember the original \'etale sheaves
and hence the cohomology of the original complete discrete valuation field.
This notion is also useful when we want to obtain some information
specific to a non-closed residue field $k$.
See \eqref{rmk: where we need P-acyclicity for abelian variety}
and \eqref{rmk: where we need P-acyclicity for tori}.


\numberwithin{equation}{subsection}
\subsection{Sites and algebraic groups: setup and first properties}
\label{sec: Sites and algebraic groups: setup and first properties}

As in \cite[Def.\ 2.1.1]{Suz13},
we say that a perfect $k$-algebra $k'$ is \emph{rational}
if it is a finite direct product of finitely generated perfect fields over $k$,
and \emph{ind-rational} if it is a filtered union of rational $k$-subalgebras.
Since any perfect field over $k$ is ind-rational,
we know that a $k$-algebra is ind-rational if and only if
it can be written as a filtered union of finite products of
(not necessarily finitely generated) perfect fields over $k$.
The rational (resp.\ ind-rational) $k$-algebras form a full subcategory of
the category of perfect $k$-algebras,
which we denote by $k^{\rat}$ (resp.\ $k^{\ind\rat}$).
Since a $k$-algebra homomorphism from $k' \in k^{\rat}$ to $k'' = \bigcup k''_{\lambda} \in k^{\ind\rat}$
factors through some $k''_{\lambda} \in k^{\rat}$,
we know that $k^{\ind\rat}$ is equivalent to the indcategory of $k^{\rat}$.%
\footnote{
	Do not confuse $k^{\ind\rat}$ with another indcategory $\Ind \mathcal{C}$,
	where $\mathcal{C}$ is the category of finite products of
	not necessarily finitely generated perfect fields over $k$
	with $k$-algebra homomorphisms.
	There are natural functors $k^{\ind\rat} \to \Ind \mathcal{C} \to k^{\ind\rat}$.
	The first one is fully faithful and the second one is essentially surjective,
	with composite the identity functor.
	The object of $\mathcal{C}$ given by the perfection of the field $k(x_{1}, x_{2}, \dots)$
	and the ind-object in $\mathcal{C}$ consisting of the increasing family of
	the perfections of the fields $k(x_{1}, \dots, x_{n})$
	are not isomorphic in $\Ind \mathcal{C}$,
	but become isomorphic in $k^{\ind\rat}$.
}
We define the \emph{rational (resp.\ ind-rational) \'etale site} of $k$ to be
the \'etale site on $k^{\rat}$ (resp.\ $k^{\ind\rat}$) (\cite[Def.\ 2.1.3]{Suz13}).
These sites are denoted by $\Spec k^{\rat}_{\et}$ and $\Spec k^{\ind\rat}_{\et}$, respectively.

We also introduce the pro-\'etale topology on $k^{\ind\rat}$.
The pro-\'etale topology for schemes is introduced in \cite{BS15}.
We throughout use the affine variant of the pro-\'etale site \cite[Def.\ 4.2.1, Rmk.\ 4.2.5]{BS15},
which behaves simpler for limit arguments.
Note that any $k$-algebra \'etale over an ind-rational $k$-algebra is
ind-rational \cite[Prop.\ 2.1.2]{Suz13}.
Hence if $k' \in k^{\ind\rat}$ and $R$ is a perfect $k'$-algebra,
then $R$ is ind-\'etale (\cite[Def.\ 2.2.1.5]{BS15}) over $k'$ if and only if
$R \cong \dirlim_{\lambda} k_{\lambda}'$
for some filtered direct system $\{k_{\lambda}'\}$ of \'etale $k'$-algebras ind-rational over $k$.
In particular, such an $R$ itself is ind-rational over $k$.
Therefore we can introduce the pro-\'etale topology on the category $k^{\ind\rat}$.
That is, a covering of an ind-rational $k$-algebra $k'$ is
a finite family $\{k'_{i}\}$ of ind-\'etale $k'$-algebras
such that $\prod k'_{i}$ is faithfully flat over $k'$.
We call the resulting site the \emph{ind-rational pro-\'etale site}
and denote it by $\Spec k^{\ind\rat}_{\pro\et}$.

Some care is needed for localizations (see Notation) of
$\Spec k^{\rat}_{\et}$, $\Spec k^{\ind\rat}_{\et}$ and $\Spec k^{\ind\rat}_{\pro\et}$,
which comes from subtleties of the underlying categories $k^{\rat}$ and $k^{\ind\rat}$.
See \cite[the paragraphs after Def.\ 2.1.3]{Suz13} for the details.
We quickly recall the facts and notation there.
As in Notation in this paper, for $k' \in k^{\ind\rat}$, the category of objects over $k'$ in $k^{\ind\rat}$
(i.e.\ the category of $k'$-algebras ind-rational over $k$) is denoted by $k^{\ind\rat} / k'$,
and the localization of $\Spec k^{\ind\rat}_{\pro\et}$ at $k'$ is denoted by $\Spec k^{\ind\rat}_{\pro\et} / k'$.
If $k'$ is a field that is in $k^{\rat}$ or algebraic over $k$, then
$k'^{\ind\rat} = k^{\ind\rat} / k'$ \cite[loc.\ cit.]{Suz13},
hence
	\[
			\Spec k'^{\ind\rat}_{\et}
		=
			\Spec k^{\ind\rat}_{\et} / k',
		\quad
			\Spec k'^{\ind\rat}_{\pro\et}
		=
			\Spec k^{\ind\rat}_{\pro\et} / k'.
	\]
For any $k' \in k^{\rat}$, we define
	\begin{gather*}
			k'^{\ind\rat} := k^{\ind\rat} / k',
		\\
				\Spec k'^{\ind\rat}_{\et}
			:=
				\Spec k^{\ind\rat}_{\et} / k',
		\quad
				\Spec k'^{\ind\rat}_{\pro\et}
			:=
				\Spec k^{\ind\rat}_{\pro\et} / k'.
	\end{gather*}

Let $k^{\perf}$ be the category of perfect $k$-algebras.
Some general results on perfect schemes in modern language can be found in
\cite[\S 3]{BS17} and \cite{BGA17}.
On $k^{\perf}$, we can introduce the \'etale,
(affine) pro-\'etale \cite[Def.\ 4.2.1, Rmk.\ 4.2.5]{BS15} and pro-fppf \cite[Def.\ 3.1.3]{Suz13} topologies.
We denote the resulting sites by
$\Spec k^{\perf}_{\et}$, $\Spec k^{\perf}_{\pro\et}$, $\Spec k^{\perf}_{\pro\fppf}$, respectively.
A covering in $\Spec k^{\perf}_{\pro\et}$ is a finite family $\{R \to S_{i}\}$
with each $S_{i}$ ind-\'etale over $R$ and $\prod S_{i}$ faithfully flat over $R$.
For $\Spec k^{\perf}_{\pro\fppf}$,
recall from \cite[\S 3.1]{Suz13} that
a homomorphism $R \to S$ in $k^{\perf}$ is said to be flat of finite presentation
if $S$ is the perfection of a flat $R$-algebra of finite presentation.
Also, a homomorphism $R \to S$ in $k^{\perf}$ is said to be flat of ind-finite presentation
if it can be written as a filtered direct limit of
flat homomorphisms $R \to S_{\lambda}$ of finite presentation (\cite[Def.\ 3.1.1]{Suz13}).
A covering in $\Spec k^{\perf}_{\pro\fppf}$ is a finite family $\{R \to S_{i}\}$
with each $S_{i}$ flat of ind-finite presentation
and $\prod S_{i}$ faithfully flat over $R$.
For a perfect $k$-algebra $R$, the category of objects over $R$ in $k^{\perf}$
is nothing but the category of perfect $R$-algebras,
in contrast to the case of the category of ind-rational $k$-algebras $k^{\ind\rat}$.
Hence we will write the localization of $\Spec k^{\perf}_{\pro\fppf}$ at $R$ by
$\Spec k^{\perf}_{\pro\fppf} / R = \Spec R^{\perf}_{\pro\fppf}$.
Similar notation applies to $\Spec k^{\perf}_{\et}$ and $\Spec k^{\perf}_{\pro\et}$.

For $R \in k^{\perf}$,
we define the small pro-\'etale site $\Spec R_{\pro\et}$ of $R$
to be the category $R_{\pro\et}$ of ind-\'etale $R$-algebras (\cite[Def.\ 2.2.1.5]{BS15})
with the topology induced from $\Spec k^{\perf}_{\pro\et}$.
That is, a covering $\{R'_{i}\}$ of an object $R' \in R_{\pro\et}$ is
a finite family of ind-\'etale $R'$-algebras
such that $\prod R'_{i}$ is faithfully flat over $R'$.

We have the following commutative diagram of continuous maps of sites:
	\begin{equation} \label{eq: diagram of sites we use}
		\begin{CD}
				\Spec k^{\perf}_{\pro\fppf}
			@>>>
				\Spec k^{\perf}_{\pro\et}
			@>>>
				\Spec k^{\perf}_{\et}
			\\
			@.
			@VVV
			@VVV
			\\
			@.
				\Spec k^{\ind\rat}_{\pro\et}
			@>>>
				\Spec k^{\ind\rat}_{\et}.
		\end{CD}
	\end{equation}
All the maps are defined by the identity.
The horizontal maps are morphisms of sites
(i.e., have exact pullback functors), but the vertical ones are not \cite[Prop.\ 3.2.3]{Suz13}.

Let $\Alg / k$ be the category of quasi-algebraic groups
(commutative, as assumed throughout the paper) over $k$
in the sense of Serre \cite{Ser60}.
Recall that a quasi-algebraic group is the perfection (inverse limit along Frobenii)
of an algebraic group \cite[\S 1.2 D\'ef.\ 2; \S 1.4, Prop.\ 10]{Ser60}, \cite[Prop.\ 1.2.10]{Pep14}.
The category $\Alg / k$ is an abelian category \cite[\S 1.2, Prop.\ 5]{Ser60}.
We simply call an object of the procategory $\Pro \Alg / k$ a proalgebraic group
following \cite[\S 2.1, D\'ef.\ 1; \S 2.6, Prop.\ 12]{Ser60}.
Similarly, we call an object of the indcategory $\Ind \Alg / k$ an ind-algebraic group
and an object of the ind-procategory
$\Ind \Pro \Alg / k = \Ind (\Pro \Alg / k)$ an ind-proalgebraic group.
(Therefore we will not say $\Z$ ind-algebraic or ind-proalgebraic in this paper.)
Note that the indcategory (and hence also the procategory) of an abelian category is
an abelian category \cite[Thm.\ 8.6.5 (i)]{KS06}.
Following \cite[\S 1.3]{Ser60},
we say that a quasi-algebraic group is a unipotent group, a torus or an abelian variety
if it is the perfection of a unipotent group, a torus or an abelian variety, respectively.
Let $\Alg_{\uc} / k \subset \Alg / k$ be the full subcategory
consisting of those whose identity component is unipotent.
This is the direct product of
the category of (not necessarily connected) unipotent quasi-algebraic groups
and the category of finite \'etale groups of order prime to $p$.
We have fully faithful exact embeddings of abelian categories
	\[
		\begin{CD}
				\Alg / k
			@>> \subset >
				\Pro \Alg / k
			\\
			@VV \cap V
			@V \cap VV
			\\
				\Ind \Alg / k
			@> \subset >>
				\Ind \Pro \Alg / k
		\end{CD}
	\]
by \cite[Thm.\ 8.6.5 (ii)]{KS06}.
Let $\Loc \Alg / k$ be the category of perfections of smooth group schemes over $k$
(which contains $\Z$ but is not abelian).
Let $\Et / k$ be the category of \'etale group schemes over $k$.
Let $\FGEt / k$ be its full subcategory consisting of finitely generated \'etale groups,
namely those whose group of geometric points is a finitely generated abelian group.
Let $\FEt / k$ be the full subcategory consisting of finite \'etale group schemes.
The categories $\FEt / k$, $\FGEt / k$ and $\Et / k$ are abelian.
We have fully faithful embeddings
$\FEt / k \subset \FGEt / k \subset \Et / k \subset \Loc \Alg / k$.
The identity component of $A \in \Loc \Alg / k$ is denoted by $A_{0}$.%
\footnote{
	More customary notation is $A^{0}$ or $A^{\circ}$.
	We prefer the subscript $0$ in this paper, however,
	since we extensively use derived categories and
	$A^{0}$ may be ambiguous
	to the zeroth term of a complex $\cdots \to A^{0} \to A^{1} \to \cdots$.
	A circle $\circ$ in the script size looks too similar to a zero $0$
	and $A^{\circ}$ may still be confusing in this paper.
	Moreover, we use below tons of duality operations
	such as the Serre dual $A^{\SDual}$, the Pontryagin dual $A^{\PDual}$,
	the dual abelian variety $A^{\vee}$ etc.
	We want to distinguish such contravariant operations from the covariant operation $A \mapsto A_{0}$.
}
The formation of $A_{0}$ is functorial,
so the notion and notation of identity component extend to
any $A \in \Ind \Alg / k$, $\Pro \Alg / k$ or $\Ind \Pro \Alg / k$.
We define $\pi_{0}(A) = A / A_{0}$,
which is in $\Ind \FEt / k$, $\Pro \FEt / k$ or $\Ind \Pro \FEt / k$ in each case.
We say that $A \in \Ind \Pro \Alg / k$ is connected if $\pi_{0}(A) = 0$.
We have natural additive functors
$\Ind \Pro \Alg / k, \Loc \Alg / k \to \Ab(k^{\ind\rat}_{\pro\et})$.

We say that a sheaf $A$ of abelian groups on a site $S$ is \emph{acyclic}
if $H^{n}(X, A) = 0$ for any object $X$ of $S$ and any $n \ge 1$.
If $S, S'$ are sites defined by pretopologies,
and if $u$ is a functor from the underlying category of $S$ to that of $S'$
that sends coverings to coverings and
$u(Y \times_{X} Z) = u(Y) \times_{u(X)} u(Z)$
whenever $Y \to X$ appears in a covering family,
then $u$ defines a continuous map $f \colon S' \to S$,
and $f_{\ast}$ sends acyclic sheaves to acyclic sheaves
and hence induces the Leray spectral sequence
$R \Gamma(X, R f_{\ast} A') = R \Gamma(u(X), A')$
for any $X \in S$ and $A' \in \Ab(S')$ \cite[\S II.4]{Art62}.

\begin{Prop} \label{prop: misc on ind-rational pro-etale topology} \BetweenThmAndList
	\begin{enumerate}
		\item \label{ass: quasi-compactness of the sites}
			On any of the sites in the diagram \eqref{eq: diagram of sites we use},
			cohomology of any object of the site commutes with filtered direct limits of coefficient sheaves.
			Products and filtered direct limits of acyclic sheaves are again acyclic.
		\item \label{ass: coherence of perfect sites}
			The sites $\Spec k^{\perf}_{\pro\fppf}$, $\Spec k^{\perf}_{\pro\et}$
			and $\Spec k^{\perf}_{\et}$ are coherent,
			i.e., the objects of their underlying categories are quasi-compact
			and stable under finite inverse limits.
		\item \label{ass: indrat proet and proet cohomology}
			Let $f \colon \Spec k^{\perf}_{\pro\et} \to \Spec k^{\ind\rat}_{\pro\et}$
			be the continuous map defined by the identity.
			Then $f$ induces isomorphisms on cohomology.
			More precisely, $f_{\ast}$ is exact,
			$f_{\ast} f^{\ast} = \id$ on $\Ab(k^{\ind\rat}_{\pro\et})$
			(hence $f^{\ast}$ is a fully faithful embedding), and
				\[
						R \Gamma(k^{\ind\rat}_{\pro\et} / k', f_{\ast} A)
					=
						R \Gamma(k'^{\perf}_{\pro\et}, A)
				\]
			for any $k' \in k^{\ind\rat}$ and $A \in \Ab(k'^{\perf}_{\pro\et})$,
			where the left-hand side is the cohomology of the site
			$\Spec k^{\ind\rat}_{\pro\et}$ at $k'$.
			The same is true for the \'etale version
			$\Spec k^{\perf}_{\et} \to \Spec k^{\ind\rat}_{\et}$.
		\item \label{ass: product is exact}
			In the categories $\Set(k^{\perf}_{\pro\fppf})$, $\Set(k^{\perf}_{\pro\et})$ and
			$\Set(k^{\ind\rat}_{\pro\et})$,
			the product of any family of surjections is a surjection.
			This implies that $\Ab(k^{\perf}_{\pro\fppf})$, $\Ab(k^{\perf}_{\pro\et})$ and
			$\Ab(k^{\ind\rat}_{\pro\et})$ are AB4* categories.
		\item \label{ass: viewing proalg as sheaves is exact}
			The natural functor from $\Pro \Alg / k$ to any of
			$\Ab(k^{\perf}_{\pro\fppf})$, $\Ab(k^{\perf}_{\pro\et})$, $\Ab(k^{\ind\rat}_{\pro\et})$
			is exact.
		\item \label{ass: Rlim vanishes on proalgebraic groups}
			Let $A = \invlim_{\lambda} A_{\lambda} \in \Pro \Alg / k$ with $A_{\lambda} \in \Alg / k$.
			Let $R \invlim_{\lambda}$ be the derived functor of
			$\invlim_{\lambda}$ considered in either
			$D(k^{\perf}_{\pro\fppf})$, $D(k^{\perf}_{\pro\et})$ or $D(k^{\ind\rat}_{\pro\et})$.
			Then we have $R \invlim_{\lambda} A_{\lambda} = A$ in each case.
		\item \label{ass: profppf, proet and et cohom of quasi-algebraic groups etc}
			Let $f \colon \Spec k^{\perf}_{\pro\fppf} \to \Spec k^{\perf}_{\pro\et}$
			be the morphism defined by the identity.
			Let $g$ be either the morphism
			$\Spec k^{\perf}_{\pro\et} \to \Spec k^{\perf}_{\et}$
			or $\Spec k^{\ind\rat}_{\pro\et} \to \Spec k^{\ind\rat}_{\et}$
			defined by the identity.
			If $A \in \Ind \Pro \Alg / k$, the pro-fppf cohomology with coefficients in $A$
			agrees with the pro-\'etale cohomology: $R f_{\ast} A = A$.
			If $A \in \Ind \Alg / k$ or $\Loc \Alg / k$, 
			the pro-fppf cohomology with coefficients in $A$
			agrees with the \'etale cohomology: 
			$R f_{\ast} A = A$ and $R g_{\ast} A = A$.
	\end{enumerate}
\end{Prop}

\begin{proof}
	\eqref{ass: quasi-compactness of the sites}
	This is true for any site defined by finite coverings.
	See \cite[III, Rmk.\ 3.6]{Mil80} for the statement that
	cohomology commutes with filtered direct limits.
	It follows that filtered direct limits of acyclics are acyclic.
	If a family of sheaves has vanishing higher \v{C}ech cohomology,
	then so is the product.
	We can deduce the corresponding statement for the derived functor cohomology
	by \cite[III, Prop.\ 2.12]{Mil80}.
	
	\eqref{ass: coherence of perfect sites}
	Obvious.
	
	\eqref{ass: indrat proet and proet cohomology}
	The exactness of $f_{\ast}$ is obvious.
	We know that $f_{\ast}$ commutes with sheafification
	and $f_{\ast} f^{-1} = \id$, where $f^{-1}$ is the pullback for presheaves.
	Hence $f_{\ast} f^{\ast} = \id$.
	The stated equality is the Leray spectral sequence for $f$,
	which is available by the remark in the paragraph before the proposition.
	The same proof works for the \'etale version
	$\Spec k^{\perf}_{\et} \to \Spec k^{\ind\rat}_{\et}$.
	
	\eqref{ass: product is exact}
	Let $\{G_{\lambda} \onto F_{\lambda}\}$ be a family of surjections in $\Set(k^{\perf}_{\pro\fppf})$.
	Let $R \in k^{\perf}$ and $\{s_{\lambda}\} \in \prod F_{\lambda}(R)$.
	For each $\lambda$, there are an object $S_{\lambda} \in k^{\perf}$
	faithfully flat of ind-finite presentation over $R$
	and a section $t_{\lambda} \in G_{\lambda}(S_{\lambda})$
	that maps to the natural image $s_{\lambda} \in F_{\lambda}(S_{\lambda})$.
	The tensor product $S$ of all the $S_{\lambda}$ over $R$
	(which is the filtered direct limit of the finite tensor products) is
	again faithfully flat of ind-finite presentation over $R$.
	Since $\{t_{\lambda}\} \in \prod G_{\lambda}(S)$ maps to $\{s_{\lambda}\} \in \prod F_{\lambda}(S)$,
	the morphism $\prod G_{\lambda} \to \prod F_{\lambda}$ is surjective.
	The same proof works for $\Set(k^{\perf}_{\pro\et})$ and $\Set(k^{\ind\rat}_{\pro\et})$.
	
	\eqref{ass: viewing proalg as sheaves is exact}
	In either cases, the left exactness is obvious.
	
	For the right exactness of $\Pro \Alg / k \to \Ab(k^{\perf}_{\pro\fppf})$,
	let $0 \to A \to B \to C \to 0$ be an exact sequence in $\Pro \Alg / k$.
	To show that $B \to C$ is a surjection in $\Ab(k^{\perf}_{\pro\fppf})$,
	it is enough to show that $B \to C$ is faithfully flat of profinite presentation.
	Suppose first that $A \in \Alg / k$.
	Let $\{B_{\mu}\}$ be a filtered decreasing family of proalgebraic subgroups of $B$
	such that $B \isomto \invlim B / B_{\mu}$
	(i.e., $\bigcap B_{\mu} = 0$; \cite[\S 2.5, Cor.\ 3 to Prop.\ 10]{Ser60})
	and $B / B_{\mu} \in \Alg / k$ for all $\mu$.
	Then $A \cap B_{\mu} = 0$ for some $\mu$
	since $\Alg / k$ is artinian \cite[\S 1.3, Prop.\ 6]{Ser60}.
	This implies $B = (B / B_{\mu}) \times_{C / B_{\mu}} C$.
	We have an exact sequence $0 \to A \to B / B_{\mu} \to C / B_{\mu} \to 0$ in $\Alg / k$.
	Hence $B / B_{\mu} \to C / B_{\mu}$ is faithfully flat of finite presentation.
	Its base change $B \to C$ is thus faithfully flat of finite presentation.
	Suppose next that $A \in \Pro \Alg / k$.
	Write $A = \invlim A / A_{\lambda}$ with $A / A_{\lambda} \in \Alg / k$ as above.
	Then $B \to C$ can be written as the filtered inverse limit of the morphisms
	$B / A_{\lambda} \to C$.
	The kernel of $B / A_{\lambda} \to C$ is $A / A_{\lambda} \in \Alg / k$.
	Hence the previous case shows that $B / A_{\lambda} \to C$ is faithfully flat of finite presentation.
	
	By \eqref{ass: indrat proet and proet cohomology} (exactness of pushforward),
	the exactness of $\Pro \Alg / k \to \Ab(k^{\ind\rat}_{\pro\et})$ is reduced to
	that of $\Pro \Alg / k \to \Ab(k^{\perf}_{\pro\et})$.
	To show this final statement,
	let $0 \to A \to B \to C \to 0$ be an exact sequence in $\Pro \Alg / k$.
	Suppose that we are given an $X$-valued point of $C$,
	where $X = \Spec R$ with $R \in k^{\perf}$.
	By the exactness of $\Pro \Alg / k \to \Ab(k^{\perf}_{\pro\fppf})$,
	the fiber $Y$ of the morphism $B \to C$ over the $X$-valued point $X \to C$
	is a pro-fppf torsor for $A$ over $X$.
	We want to show that $Y \to X$ admits a section pro-\'etale locally.
	By \cite[Cor.\ 2.2.14]{BS15},
	there exists a pro-\'etale cover $X' \onto X$
	such that $X'$ is w-strictly local \cite[Def.\ 2.2.1]{BS15}.
	In particular, every \'etale cover of $X'$ admits a section.
	Therefore we may assume that $X$ itself is w-strictly local.
	We show that $Y \to X$ then admits a section.
	Let $\{A_{\lambda}\}$ be a filtered decreasing family of proalgebraic subgroups of $A$
	such that $A \isomto \invlim A / A_{\lambda}$
	(i.e., $\bigcap A_{\lambda} = 0$) and $A / A_{\lambda} \in \Alg / k$ for all $\lambda$.
	Then we have $Y \isomto \invlim Y / A_{\lambda}$
	and $Y / A_{\lambda}$ is a pro-fppf torsor for $A / A_{\lambda}$ over $X$.
	Since $A / A_{\lambda}$ is quasi-algebraic,
	we know that $H^{1}(X_{\pro\fppf}, A / A_{\lambda}) = H^{1}(X_{\et}, A / A_{\lambda})$
	by \cite[Cor.\ 3.3.3]{Suz13}.
	Therefore $Y / A_{\lambda}$ is an \'etale torsor for $A / A_{\lambda}$ over $X$.
	For each $\lambda$, this torsor is trivial since $X$ is w-strictly local.
	In this situation, if $R$ is an algebraically closed field,
	then \cite[2.3, Prop.\ 2]{Ser60} or \cite[Lem.\ 3.7]{SY12} says that $Y \to X = \Spec R$ admits a section.
	The arguments there actually work for any w-strictly local $X$.
	
	\eqref{ass: Rlim vanishes on proalgebraic groups}
	The higher $R \invlim$ vanishes on $\Pro \Alg / k$.
	The functors from $\Pro \Alg / k$ to each of the categories preserves $R \invlim$
	by the previous two assertions and \cite[Lem.\ A.3.2]{Nee01}.
	
	\eqref{ass: profppf, proet and et cohom of quasi-algebraic groups etc}
	The statement for the morphism $g$ on the ind-rational sites
	follows from that for the morphism $g$ on the perfect sites
	by \eqref{ass: indrat proet and proet cohomology}.
	The cases $A \in \Ind \Pro \Alg / k, \Ind \Alg / k$ are reduced
	to the cases $A \in \Pro \Alg / k, \Alg / k$, respectively,
	by \eqref{ass: quasi-compactness of the sites}.
	Note that $R \invlim_{\lambda}$ commutes with $R f_{\ast}$.
	With this and \eqref{ass: Rlim vanishes on proalgebraic groups},
	the case $A \in \Pro \Alg / k$ is reduced to the case $A \in \Alg / k$.
	Summing all up, we are reduced to the case $A \in \Loc \Alg / k$.
	This case is \cite[Cor.\ 3.3.3]{Suz13}.
\end{proof}

For the rest of the paper,
we denote the object $R \Gamma(k^{\ind\rat}_{\pro\et} / k', f_{\ast} A)$
appearing in \eqref{ass: indrat proet and proet cohomology}
simply by
	\[
		R \Gamma(k'_{\pro\et}, A).
	\]
This is the same as the cohomology of the small pro-\'etale site $\Spec k'_{\pro\et}$
with coefficients given by the restriction of $A$ to $\Spec k'_{\pro\et}$,
since big sites ($\Spec k'^{\perf}_{\pro\et}$) and small sites ($\Spec k'_{\pro\et}$)
have the same cohomology theory (\cite[III, Rmk.\ 3.2]{Mil80}).
Its \'etale version $R \Gamma(k'_{\et}, A)$ is used in a similar sense.

Note that the sites $\Spec k^{\ind\rat}_{\et}$ and $\Spec k^{\ind\rat}_{\pro\et}$
are not coherent.
If $x$ and $y$ are the generic points of irreducible varieties over $k$,
then their fiber product as a sheaf on these sites is given by
the disjoint union of the points (identified with the $\Spec$ of their residue fields)
of the underlying set of the usual fiber product $x \times_{k} y$.
This is infinite unless $x$ or $y$ is finite over $k$.


\numberwithin{equation}{subsection}
\subsection{Generalities on derived categories of ind-procategories}
\label{sec: Generalities on derived categories of ind-procategories}

A general reference on derived categories of indcategories is \cite[Chap.\ 15]{KS06}.
We need to develop some more here.
For a certain abelian category $\mathcal{A}$,
we describe the derived Hom functor $R \Hom_{\Ind \Pro \mathcal{A}}$
on the ind-procategory $\Ind \Pro \mathcal{A} = \Ind (\Pro \mathcal{A})$
(i.e.\ the ind-category of the pro-category $\Pro \mathcal{A}$)
in terms of $R \Hom_{\mathcal{A}}$.
Recall that $\Hom_{\Ind \Pro \mathcal{A}}$ is defined as
the inverse limit of the direct limit of the inverse limit of the direct limit of $\Hom_{\mathcal{A}}$.
Roughly speaking, the two inverse limits and the one $\Hom$ functor should be derived
to get the required description of $R \Hom_{\Ind \Pro \mathcal{A}}$.
This turns out to be quite complicated both notationally and mathematically.
We organize it by introducing a relatively reasonable notation.
In the next subsection, we will apply these results and notation
to $\mathcal{A} = \Alg / k$ and $\Ab(k^{\ind\rat}_{\pro\et})$.
To clarify the argument, we generalize the situation and
treat small abelian categories and Grothendieck categories in this subsection.

We need notation.
Let $\mathcal{A}$ be an additive category.
A filtered direct system $\{A_{\lambda}\}_{\lambda \in \Lambda} \in \Ind \mathcal{A}$
will occasionally be written as
$\fdirlim[\lambda \in \Lambda] A_{\lambda}$.
This is the direct limit of the $A_{\lambda}$ in $\Ind \mathcal{A}$.
If $\{A_{\lambda}\}_{\lambda \in \Lambda}$ is a family of objects of $\mathcal{A}$,
then its direct sum in $\Ind \mathcal{A}$ (which is the filtered direct limit of finite sums) is denoted by
$\fdsum[\lambda \in \Lambda] A_{\lambda}$.
If $\mathcal{A'}$ is a full additive subcategory of $\mathcal{A}$,
then we denote by $\fdsum \mathcal{A}'$
the full additive subcategory of $\Ind \mathcal{A}$ consisting of objects of the form
$\fdsum[\lambda \in \Lambda] A_{\lambda}$ with $A_{\lambda} \in \mathcal{A}'$.
A similar notation is applied to $\Pro \mathcal{A}$,
for example, objects $\finvlim[\lambda] A_{\lambda}$, $\fprod[\lambda] A_{\lambda}$ of $\Pro \mathcal{A}$
and full additive subcategories $\fprod \mathcal{A}'$ of $\Pro \mathcal{A}$.

Now let $\mathcal{A}$ be an abelian category.
Assume that $\mathcal{A}$ is small.
Then $\Pro \mathcal{A}$ has exact filtered inverse limits
and a set of cogenerators given by objects of $\mathcal{A}$.
Hence $\Pro \mathcal{A}$ is a co-Grothendieck category (i.e.\ the opposite of a Grothendieck category) and,
in particular, has enough projectives.
Therefore the Hom functor on $\Pro \mathcal{A}$ admits a derived functor $R \Hom_{\Pro \mathcal{A}}$.
Its restriction to the bounded derived category of $\mathcal{A}$ is denoted by $R \Hom_{\mathcal{A}}$,
namely we define
	\[
			R \Hom_{\mathcal{A}}(A, B)
		:=
			R \Hom_{\Pro \mathcal{A}}(A, B)
		\quad \text{for} \quad
			A, B \in D^{b}(\mathcal{A}).
	\]
We do not claim that this $R \Hom_{\mathcal{A}}$ is the derived functor of $\Hom_{\mathcal{A}}$.
By \cite[Thm.\ 15.3.1, (i)]{KS06}, the natural functor
$D^{b}(\mathcal{A}) \to D^{b}(\Pro \mathcal{A})$ is fully faithful, hence
	\[
			H^{n} R \Hom_{\mathcal{A}}(A, B)
		=
			\Ext_{\mathcal{A}}^{n}(A, B)
		\quad \text{for} \quad
			A, B \in \mathcal{A},
	\]
where the right-hand side is the usual Ext functor for the abelian category $\mathcal{A}$.
Therefore if the derived functor of $\Hom_{\mathcal{A}}$ on the bounded derived category exists,
then the morphism of functors from this derived functor
to the above $R \Hom_{\mathcal{A}}$ defined by universality is an isomorphism.
In general, we do not assume this.
If $A = \finvlim[\lambda] A_{\lambda} \in \Pro \mathcal{A}$ and $B \in \mathcal{A}$, then
	\[
			\Ext_{\Pro \mathcal{A}}^{n}(A, B)
		=
			\dirlim_{\lambda}
			\Ext_{\mathcal{A}}^{n}(A_{\lambda}, B)
	\]
for all $n$ by \cite[Cor.\ 15.3.9]{KS06}.
We express this result by the equality
	\[
			R \Hom_{\Pro \mathcal{A}}(A, B)
		=
			\dirlim_{\lambda}
			R \Hom_{\mathcal{A}}(A_{\lambda}, B),
	\]
which is intuitive and convenient, but may be confusing
since the $\dirlim_{\lambda}$ in the right-hand side is not a direct limit
in the triangulated category $D(\Ab)$.

Next we want to treat the case that $B$ also is a pro-object and,
more generally, the case $A, B \in \Ind \Pro \mathcal{A}$.
An object $A \in \Ind \Pro \mathcal{A}$ can be written as
$A = \fdirlim[\lambda \in \Lambda] A_{\lambda}$,
where each $A_{\lambda} \in \Pro \mathcal{A}$ can be written as
$A_{\lambda} = \finvlim[\lambda' \in \Lambda'_{\lambda}] A_{\lambda \lambda'}$
with $A_{\lambda \lambda'} \in \mathcal{A}$.
Note that the morphisms $A_{\lambda_{1}} \to A_{\lambda_{2}} \in \Pro \mathcal{A}$
for $\lambda_{1} \le \lambda_{2} \in \Lambda$ are assumed as given,
but no maps $\Lambda'_{\lambda_{2}} \to \Lambda'_{\lambda_{1}}$
between the index sets are assumed.
The Hom functor for $\Ind \Pro \mathcal{A}$ is given by definition as
 	\[
 			\Hom_{\Ind \Pro \mathcal{A}}(A, B)
 		=
 			\invlim_{\lambda} \dirlim_{\mu} \invlim_{\mu'} \dirlim_{\lambda'}
 			\Hom_{\mathcal{A}}(A_{\lambda \lambda'}, B_{\mu \mu'})
 	\]
for $A = \fdirlim[\lambda] \finvlim[\lambda'] A_{\lambda \lambda'} \in \Ind \Pro \mathcal{A}$ and
$B = \fdirlim[\mu] \finvlim[\mu'] B_{\mu \mu'} \in \Ind \Pro \mathcal{A}$,
which is of the form
	\[
			\Hom_{\Ind \Pro \mathcal{A}}
		\colon
			\Pro \Ind (\mathcal{A}^{\op})
		\times
			\Ind \Pro \mathcal{A}
		\to
			\Ab.
	\]
A general method to derive this type of functors is the following.

\begin{Prop} \label{prop: derivation in ind-procategories}
	Let $\mathcal{A}, \mathcal{B}, \mathcal{C}$ be abelian categories
	and $F \colon \mathcal{A} \times \mathcal{B} \to \mathcal{C}$
	an additive bifunctor that is left exact in both variables.
	Let $\mathcal{A}', \mathcal{B}', \mathcal{C}'$ be full additive subcategories of
	$\mathcal{A}, \mathcal{B}, \mathcal{C}$, respectively, such that
	$F(\mathcal{A}' \times \mathcal{B}') \subset \mathcal{C}'$.
	Let
		$
				\Pro F
			\colon
				\Pro \mathcal{A} \times \Pro \mathcal{B}
			\to
				\Pro \mathcal{C}
		$,
		$
				\Ind \Pro F
			\colon
				\Ind \Pro \mathcal{A} \times \Ind \Pro \mathcal{B}
			\to
				\Ind \Pro \mathcal{C}
		$,
		$
				\Pro \Ind \Pro F
			\colon
					\Pro \Ind \Pro \mathcal{A}
				\times
					\Pro \Ind \Pro \mathcal{B}
			\to
				\Pro \Ind \Pro \mathcal{C}
		$
	etc.\ be the natural extensions of $F$.
	\begin{enumerate}
		\item \label{ass: deriving the pro and ind of a bifunctor}
			Assume that the pair $(\mathcal{A}', \mathcal{B}')$ is $F$-injective
			in the sense of \cite[Def.\ 13.4.2]{KS06}
			(which implies the existence of the derived functor of $F$ on the bounded below derived categories).
			Then $(\Ind \mathcal{A}', \Ind \mathcal{B}')$ is $\Ind F$-injective,
			and $(\Pro \mathcal{A}', \Pro \mathcal{B}')$ and $(\fprod \mathcal{A}', \fprod \mathcal{B}')$
			are both $\Pro F$-injective.
			In particular, $\Ind F$ and $\Pro F$ admit derived functors on the bounded below derived categories.
			The diagrams
				\[
					\begin{CD}
							D^{+}(\mathcal{A}) \times D^{+}(\mathcal{B})
						@>> RF >
							D^{+}(\mathcal{C})
						\\
						@VVV
						@VVV
						\\
							D^{+}(\Ind \mathcal{A}) \times D^{+}(\Ind \mathcal{B})
						@> R \Ind F >>
							D^{+}(\Ind \mathcal{C}),
					\end{CD} \qquad
					\begin{CD}
							D^{+}(\mathcal{A}) \times D^{+}(\mathcal{B})
						@>> RF >
							D^{+}(\mathcal{C})
						\\
						@VVV
						@VVV
						\\
							D^{+}(\Pro \mathcal{A}) \times D^{+}(\Pro \mathcal{B})
						@> R \Pro F >>
							D^{+}(\Pro \mathcal{C})
					\end{CD}
				\]
			commute.
			We have natural isomorphisms
			$R^{n} \Ind F \cong \Ind R^{n} F$ and $R^{n} \Pro F \cong \Pro R^{n} F$
			of functors for all $n$.
		\item \label{ass: derived functor of ind-pro limits}
			Assume that: $\mathcal{C}$ has products and exact filtered direct limits;
			$\mathcal{C}'$ is closed by products and filtered direct limits;
			the subcategory $\mathcal{C}' \subset \mathcal{C}$ is cogenerating, i.e.,
			any object of $\mathcal{C}$ has an injection into an object of $\mathcal{C}'$;
			and for any family of exact sequences
			$0 \to C_{\lambda}^{1} \to C_{\lambda}^{2} \to C_{\lambda}^{3} \to 0$ in $\mathcal{C}$
			with $C_{\lambda}^{1}, C_{\lambda}^{2} \in \mathcal{C}'$ for any $\lambda$,
			we have $C_{\lambda}^{3} \in \mathcal{C}'$ for any $\lambda$ and the sequence
			$0 \to \prod C_{\lambda}^{1} \to \prod C_{\lambda}^{2} \to \prod C_{\lambda}^{3} \to 0$ is exact.
			Then the sequence
				\[
						\Pro \Ind \Pro \Ind \mathcal{C}
					\xrightarrow{\Pro \Ind \Pro \dirlim}
						\Pro \Ind \Pro \mathcal{C}
					\xrightarrow{\Pro \Ind \invlim}
						\Pro \Ind \mathcal{C}
					\xrightarrow{\Pro \dirlim}
						\Pro \mathcal{C}
					\xrightarrow{\invlim}
						\mathcal{C}
				\]
			of functors restricts to the sequence
				\[
						\fprod \Ind \fprod \Ind \mathcal{C}'
					\to
						\fprod \Ind \fprod \mathcal{C}'
					\to
						\fprod \Ind \mathcal{C}'
					\to
						\fprod \mathcal{C}'
					\to
						\mathcal{C}'
				\]
			on the subcategories.
			Each category in the latter sequence is injective
			with respect to the functor that follows in the former sequence
			(i.e.\ $\fprod \Ind \fprod \mathcal{C}'$ is $\Pro \Ind \invlim$-injective etc.).
			In particular, the composite of the derived functors of the functors in the former sequence
			gives the derived functor of the ``take-all-the-limits'' functor
				\[
						\invlim \;
						\Pro \dirlim \;
						\Pro \Ind \invlim \;
						\Pro \Ind \Pro \dirlim
					\colon
						\Pro \Ind \Pro \Ind \mathcal{C} \to \mathcal{C}.
				\]
		\item \label{ass: derived functor of pro-ind extensions and limits}
			Under the assumptions of the previous two assertions,
			the composite
				\[
						\Pro \Ind \Pro \Ind \mathcal{A} \times \Pro \Ind \Pro \Ind \mathcal{B}
					\xrightarrow{\Pro \Ind \Pro \Ind F}
						\Pro \Ind \Pro \Ind \mathcal{C}
					\xrightarrow{
						\invlim \;
						\Pro \dirlim \;
						\Pro \Ind \invlim \;
						\Pro \Ind \Pro \dirlim
					}
						\mathcal{C}
				\]
			admits a derived functor on the bounded below derived categories,
			which is given by the composite of the derived functors.
	\end{enumerate}
\end{Prop}

Recall that $(\mathcal{A}', \mathcal{B}')$ being $F$-injective means that:
for any bounded below complex $A$ in $\mathcal{A}$,
there is a quasi-isomorphism $A \to A'$ to a bounded below complex $A'$ in $\mathcal{A}'$;
for any bounded below complex $B$ in $\mathcal{B}$,
there is a quasi-isomorphism $B \to B'$ to a bounded below complex $B'$ in $\mathcal{B}'$;
and $F(A', B')$ is an exact complex
if $A'$ (resp.\ $B'$) is a bounded below complex in $\mathcal{A}'$ (resp.\ $\mathcal{B}'$)
such that either $A'$ or $B'$ is exact.
In this situation, according to \cite[\S 13.4]{KS06}, the derived functor
	\[
			R F
		\colon
			D^{+}(\mathcal{A}) \times D^{+}(\mathcal{B})
		\to
			D^{+}(\mathcal{C})
	\]
of the two-variable functor $F$ is defined by $RF(A, B) = F(A', B')$,
where $A \isomto A'$ is a quasi-isomorphism to a bounded below complex in $A'$
and $B \isomto B'$ is a quasi-isomorphism to a bounded below complex in $B'$.
We need to replace the both variables at the same time.
In the proposition, we chose a pro-ind-pro-indcategory as an example.
There is nothing special about this choice.
There is a corresponding statement for any finite sequence of P's and I's.

\begin{proof}
	\eqref{ass: deriving the pro and ind of a bifunctor}
	This is nothing but a two-variable version of
	\cite[Prop.\ 15.3.2, 15.3.7]{KS06}.
	We merely indicate what should be modified from the original single-variable version.
	Let $\Tilde{\mathcal{A}}'$ be the full subcategory of $\Ind \mathcal{A}$
	consisting of objects $A \in \Ind \mathcal{A}$ such that
	$\Ind R^{n} F(A, B) = 0$ for any $B \in \Ind \mathcal{B}'$ and $n \ge 1$.
	Let $\Tilde{\Tilde{\mathcal{B}}}'$ be the full subcategory of $\Ind \mathcal{B}$
	consisting of objects $B \in \Ind \mathcal{B}$ such that
	$\Ind R^{n} F(A, B) = 0$ for any $A \in \Tilde{\mathcal{A}}'$ and $n \ge 1$.
	Then exactly in the same manner as \cite[loc.cit.]{KS06},
	we can show that $\Tilde{\mathcal{A}}' \times \Tilde{\Tilde{\mathcal{B}}}'$
	contains $\Ind \mathcal{A}' \times \Ind \mathcal{B}'$ and that
	$(\Tilde{\mathcal{A}}', \Tilde{\Tilde{\mathcal{B}}}')$ is $\Ind F$-injective.
	We can deduce from this that
	$(\Ind \mathcal{A}', \Ind \mathcal{B}')$ is $\Ind F$-injective.
	A similar argument for the pro version implies that
	$(\Pro \mathcal{A}', \Pro \mathcal{B}')$ is $\Pro F$-injective.
	Since $\fprod \mathcal{A}'$ is cogenerating in $\Pro \mathcal{A}'$,
	we know that $(\fprod \mathcal{A}', \fprod \mathcal{B}')$ is $\Pro F$-injective.
	We omit the details.
	
	\eqref{ass: derived functor of ind-pro limits}
	The assumptions imply that the subcategory $\fprod \mathcal{C}' \subset \Pro \mathcal{C}$ is
	injective with respect to the functor $\invlim \colon \Pro \mathcal{C} \to \mathcal{C}$
	by \cite[Prop.\ 13.3.15]{KS06}.
	The subcategory $\Ind \mathcal{C}' \subset \Ind \mathcal{C}$ is cogenerating by \cite[Thm.\ 15.2.5]{KS06},
	hence injective with respect to the exact functor
	$\dirlim \colon \Ind \mathcal{C} \to \mathcal{C}$.
	These imply the rest of the statement by an iterated usage of \cite[Prop.\ 15.3.2, 15.3.7]{KS06}
	and the theorem on derived functors of composition
	\cite[Prop.\ 13.3.13]{KS06}.
	
	\eqref{ass: derived functor of pro-ind extensions and limits}
	This follows from the previous two assertions and \cite[Prop.\ 13.3.13]{KS06}.
\end{proof}

Hence, in the situation of the proposition, we have
	\begin{align*}
		&
				R \bigl(
					\invlim \;
					\Pro \dirlim \;
					\Pro \Ind \invlim \;
					\Pro \Ind \Pro \dirlim \;
					\Pro \Ind \Pro \Ind F
				\bigr)(A, B)
		\\
		&	=
				R \invlim \;
				\Pro \dirlim \;
				R \Pro \Ind \invlim \;
				\Pro \Ind \Pro \dirlim \;
				R \Pro \Ind \Pro \Ind F(A, B)
			\in
				D^{+}(\mathcal{C})
	\end{align*}
for $A \in D^{+}(\Pro \Ind \Pro \Ind \mathcal{A})$, $B \in D^{+}(\Pro \Ind \Pro \Ind \mathcal{B})$.
Again, an intuitive and convenient but less rigorous way to denote the painful right-hand side
(when $A \in \Pro \Ind \Pro \Ind \mathcal{A}$, $B \in \Pro \Ind \Pro \Ind \mathcal{B}$) is
	\[
		R \invlim_{\lambda_{1}, \mu_{1}}
		\dirlim_{\lambda_{2}, \mu_{2}}
		R \invlim_{\lambda_{3}, \mu_{3}}
		\dirlim_{\lambda_{4}, \mu_{4}}
		RF(A_{\lambda_{1} \lambda_{2} \lambda_{3} \lambda_{4}},
		B_{\mu_{1} \mu_{2} \mu_{3} \mu_{4}}),
	\]
where
	\begin{gather*}
				A
			=
				\finvlim[\lambda_{1}]
				\fdirlim[\lambda_{2}]
				\finvlim[\lambda_{3}]
				\fdirlim[\lambda_{4}]
				A_{\lambda_{1} \lambda_{2} \lambda_{3} \lambda_{4}}
			\quad \text{with} \quad
				A_{\lambda_{1} \lambda_{2} \lambda_{3} \lambda_{4}} \in \mathcal{A},
		\\
				B
			=
				\finvlim[\mu_{1}]
				\fdirlim[\mu_{2}]
				\finvlim[\mu_{3}]
				\fdirlim[\mu_{4}]
				B_{\mu_{1} \mu_{2} \mu_{3} \mu_{4}}.
			\quad \text{with} \quad
				B_{\mu_{1} \mu_{2} \mu_{3} \mu_{4}} \in \mathcal{B}.
	\end{gather*}
What is rigorously true about this notation is that
there are two spectral sequences.
One (for fixed $\lambda_{1}, \lambda_{2}, \mu_{1}, \mu_{2}$) has $E_{2}^{i j}$-terms given by
	\[
			R^{i} \invlim_{\lambda_{3}, \mu_{3}}
			\dirlim_{\lambda_{4}, \mu_{4}}
			R^{j} F(
				A_{\lambda_{1} \lambda_{2} \lambda_{3} \lambda_{4}},
				B_{\mu_{1} \mu_{2} \mu_{3} \mu_{4}}
			),
	\]
converging to
$R^{i + j} \Bar{F}(A_{\lambda_{1} \lambda_{2}}, B_{\mu_{1} \mu_{2}})$, where
	\begin{gather*}
				A_{\lambda_{1} \lambda_{2}}
			=
				\finvlim[\lambda_{3}]
				\fdirlim[\lambda_{4}]
					A_{\lambda_{1} \lambda_{2} \lambda_{3} \lambda_{4}},
		\quad
				B_{\mu_{1} \mu_{2}}
			=
				\finvlim[\mu_{3}]
				\fdirlim[\mu_{4}]
					B_{\mu_{1} \mu_{2} \mu_{3} \mu_{4}},
		\\
				\Bar{F}
			=
				\invlim \;
				P \dirlim \;
				\Pro \Ind F,
			\quad \text{i.e.,} \quad
				\bar{F}(A_{\lambda_{1} \lambda_{2}}, B_{\mu_{1} \mu_{2}})
			=
				\invlim_{\lambda_{3}, \mu_{3}}
				\dirlim_{\lambda_{4}, \mu_{4}}
				F(
					A_{\lambda_{1} \lambda_{2} \lambda_{3} \lambda_{4}},
					B_{\mu_{1} \mu_{2} \mu_{3} \mu_{4}}
				).
	\end{gather*}
Varying $\lambda_{1}, \lambda_{2}, \mu_{1}, \mu_{2}$,
this spectral sequence takes values in $\Pro \Ind \mathcal{C}$.
The other has $E_{2}^{i j}$-terms given by
	\[
			R^{i} \invlim_{\lambda_{1}, \mu_{1}}
			\dirlim_{\lambda_{2}, \mu_{2}}
			R^{j} \bar{F}(A_{\lambda_{1} \lambda_{2}},
			B_{\mu_{1} \mu_{2}}),
	\]
converging to
$R^{i + j} \Bar{\Bar{F}}(A, B)$, where
	\begin{gather*}
				\Bar{\Bar{F}}
			=
				\invlim \;
				\Pro \dirlim \;
				\Pro \Ind \invlim \;
				\Pro \Ind \Pro \dirlim \;
				\Pro \Ind \Pro \Ind F,
			\quad \text{i.e.,}
		\\
				\Bar{\Bar{F}}(A, B)
			=
				\invlim_{\lambda_{1}, \mu_{1}}
				\dirlim_{\lambda_{2}, \mu_{2}}
				\invlim_{\lambda_{3}, \mu_{3}}
				\dirlim_{\lambda_{4}, \mu_{4}}
				F(
					A_{\lambda_{1} \lambda_{2} \lambda_{3} \lambda_{4}},
					B_{\mu_{1} \mu_{2} \mu_{3} \mu_{4}}
				).
	\end{gather*}

A special case we use below is where $A \in \Pro \Ind \mathcal{A}$ and $B \in \Ind \Pro \mathcal{B}$.
Here we embed $\Pro \Ind \mathcal{A}$ into $\Pro \Ind \Pro \Ind \mathcal{A}$
by adding $I$ and $P$ in the middle
and $\Ind \Pro \mathcal{B}$ into $\Pro \Ind \Pro \Ind \mathcal{B}$
by adding $\Pro$ from the left and $\Ind$ from the right.
These embeddings are exact functors.
The first one takes the subcategory $\fprod \Ind \mathcal{A}'$ into $\fprod \Ind \fprod \Ind \mathcal{A}'$.
The second takes the subcategory $\Ind \fprod \mathcal{B}'$ into $\fprod \Ind \fprod \Ind \mathcal{B}'$.
Hence the derived functor of the restriction of $\Bar{\Bar{F}}$ to
$\Pro \Ind \mathcal{A} \times \Ind \Pro \mathcal{A}$
is the restriction of $R \Bar{\Bar{F}}$ to $D^{+}(\Pro \Ind \mathcal{A}) \times D^{+}(\Ind \Pro \mathcal{A})$,
and we have
	\[
			R \Bar{\Bar{F}}(A, B)
		=
			R \invlim_{\lambda}
			\dirlim_{\mu}
			R \invlim_{\mu'}
			\dirlim_{\lambda'}
			RF(A_{\lambda \lambda'}, B_{\mu \mu'}),
	\]
where $A = \finvlim[\lambda] \fdirlim[\lambda'] A_{\lambda \lambda'} \in \Pro \Ind \mathcal{A}$,
$B = \fdirlim[\mu] \finvlim[\mu'] B_{\mu \mu'} \in \Ind \Pro \mathcal{B}$.
Similar observations apply to adding more or less $\Pro$'s and/or $\Ind$'s in different places.

\begin{Prop} \label{prop: fully faithful embeddings of ind-procateogries}
	Let $\mathcal{A}$ be a small abelian category.
	Then we have natural fully faithful embeddings
		\[
			\begin{CD}
					D^{b}(\mathcal{A})
				@>> \subset >
					D^{b}(\Pro \mathcal{A})
				\\
				@VV \cap V
				@V \cap VV
				\\
					D^{b}(\Ind \mathcal{A})
				@> \subset >>
					D^{b}(\Ind \Pro \mathcal{A})
			\end{CD}
		\]
	 of triangulated categories.
	 We have
	 	\[
	 			R \Hom_{\Pro \mathcal{A}}(A, B)
	 		=
	 			R \invlim_{\mu} \dirlim_{\lambda}
	 			R \Hom_{\mathcal{A}}(A_{\lambda}, B_{\mu})
	 	\]
	 for $A = \finvlim[\lambda] A_{\lambda} \in \Pro \mathcal{A}$ and
	 $B = \finvlim[\mu] B_{\mu} \in \Pro \mathcal{A}$,
	 	\[
	 			R \Hom_{\Ind \mathcal{A}}(A, B)
	 		=
	 			R \invlim_{\lambda} \dirlim_{\mu}
	 			R \Hom_{\mathcal{A}}(A_{\lambda}, B_{\mu})
	 	\]
	 for $A = \fdirlim[\lambda] A_{\lambda} \in \Ind \mathcal{A}$ and
	 $B = \fdirlim[\mu] B_{\mu} \in \Ind \mathcal{A}$, and
	 	\[
	 			R \Hom_{\Ind \Pro \mathcal{A}}(A, B)
	 		=
	 			R \invlim_{\lambda} \dirlim_{\mu} R \invlim_{\mu'} \dirlim_{\lambda'}
	 			R \Hom_{\mathcal{A}}(A_{\lambda \lambda'}, B_{\mu \mu'})
	 	\]
	 for $A = \fdirlim[\lambda] \finvlim[\lambda'] A_{\lambda \lambda'} \in \Ind \Pro \mathcal{A}$ and
	 $B = \fdirlim[\mu] \finvlim[\mu'] B_{\mu \mu'} \in \Ind \Pro \mathcal{A}$.
\end{Prop}

\begin{proof}
	The fully faithful embedding $D^{b}(\mathcal{A}) \into D^{b}(\Pro \mathcal{A})$
	has already been mentioned (\cite[Thm.\ 15.3.1 (i)]{KS06}).
	The same implies the fully faithfulness of $D^{b}(\mathcal{A}) \into D^{b}(\Ind \mathcal{A})$
	and $D^{b}(\Pro \mathcal{A}) \into D^{b}(\Ind \Pro \mathcal{A})$.
	The exactness of the embedding $\Ind \mathcal{A} \into \Ind \Pro \mathcal{A}$ yields
	a morphism of functors from $R \Hom_{\Ind \mathcal{A}}$ to
	$R \Hom_{\Ind \Pro \mathcal{A}}$ restricted to $D^{b}(\Ind \mathcal{A})$.
	This will turn out to be an isomorphism and hence
	the fully faithfulness of $D^{b}(\Ind \mathcal{A}) \to D^{b}(\Ind \Pro \mathcal{A})$ will follow
	once we compute these $R \Hom$ functors
	and verify the stated formulas.
	
	We compute $R \Hom_{\Ind \Pro \mathcal{A}}$.
	Let $\mathcal{B} = (\Pro \mathcal{A})^{\op}$, $\mathcal{B}'$ its full subcategory of injectives
	($=$ the opposite of projectives of $\Pro \mathcal{A}$) and
	$\mathcal{C} = \mathcal{C}' = \Ab$.
	We will apply the observations we made before the proposition for
		\[
				F
			:=
				\Hom_{\Pro \mathcal{A}}
			\colon
				(\Pro \mathcal{A})^{\op} \times \mathcal{A}
			=:
				\mathcal{B} \times \mathcal{A}
			\to
				\mathcal{C}
			=
				\Ab
		\]
	and
		\[
				\invlim \;
				\Pro \dirlim \;
				\Pro \Ind \invlim \;
				\Pro \Ind \Pro F
			=
				\Hom_{\Ind \Pro \mathcal{A}}
			\colon
				\Pro \mathcal{B} \times \Ind \Pro \mathcal{A}
			\to
				\Pro \Ind \Pro \mathcal{C}
			\to
				\mathcal{C}
		\]
	(the $\Pro$ in $\Pro \mathcal{B}$ corresponds to the left $\Pro$ in $\Pro \Ind \Pro \mathcal{C}$
	and the $\Ind \Pro$ in $\Ind \Pro \mathcal{A}$ the right $\Ind \Pro$ in $\Pro \Ind \Pro \mathcal{C}$).
	The pair $(\mathcal{B}', \mathcal{A})$ is $F$-injective.
	Obviously $\mathcal{C} = \mathcal{C}' = \Ab$ satisfies the conditions of
	\eqref{prop: derivation in ind-procategories} \eqref{ass: derived functor of ind-pro limits}.
	Hence we have
		\begin{align*}
					R \Hom_{\Ind \Pro \mathcal{A}}(A, B)
		 	&	=
		 			R \invlim_{\lambda} \dirlim_{\mu} R \invlim_{\mu'}
		 			R \Hom_{\Pro \mathcal{A}}(\finvlim[\lambda'] A_{\lambda \lambda'}, B_{\mu \mu'})
		 	\\
		 	&	=
		 			R \invlim_{\lambda} \dirlim_{\mu} R \invlim_{\mu'} \dirlim_{\lambda'}
		 			R \Hom_{\mathcal{A}}(A_{\lambda \lambda'}, B_{\mu \mu'})
		\end{align*}
	for $A = \fdirlim[\lambda] \finvlim[\lambda'] A_{\lambda \lambda'} \in \Ind \Pro \mathcal{A}$ and
	$B = \fdirlim[\mu] \finvlim[\mu'] B_{\mu \mu'} \in \Ind \Pro \mathcal{A}$.
	
	Since $D^{b}(\Pro \mathcal{A}) \into D^{b}(\Ind \Pro \mathcal{A})$ is fully faithful,
	this also verifies, by restriction, the stated formula for $R \Hom_{\Pro \mathcal{A}}$.
	Dualizing, this in turn verifies the stated formula for $R \Hom_{\Ind \mathcal{A}}$.
	This completes the proof.
\end{proof}

Note that this proposition in particular implies that
the restriction of $R \Hom_{\Ind \mathcal{A}}$ to $D^{b}(\mathcal{A})$
agrees with $R \Hom_{\mathcal{A}}$,
which was originally defined as the restriction of $R \Hom_{\Pro \mathcal{A}}$.
The proposition is also true for a Grothendieck category $\mathcal{A}$.
We omit the proof as we do not need this case.

What we need to know about Grothendieck categories is
when direct and derived inverse limits commute with $R \Hom$ in the both variables.
The following shows what is true in general
and what should be checked in specific cases.

\begin{Prop} \label{prop: RHom of ind-pro-limits}
	Let $\mathcal{A}, \mathcal{C}$ be Grothendieck categories
	and $F \colon \mathcal{A}^{\op} \times \mathcal{A} \to \mathcal{C}$
	an additive bifunctor that is left exact in both variables.
	Let $\mathcal{C}'$ be a full additive subcategory of $\mathcal{C}$.
	Assume that the functor $F(\,\cdot\,, I) \colon \mathcal{A}^{\op} \to \mathcal{C}$
	for any injective $I \in \mathcal{A}$ is exact
	with image contained in $\mathcal{C}'$.
	Assume also that $F$ commutes with filtered inverse limits.
	Assume finally that $\mathcal{C}' \subset \mathcal{C}$ satisfies
	the conditions of \eqref{prop: derivation in ind-procategories}
	\eqref{ass: derived functor of ind-pro limits}.
	Then we have canonical morphisms and isomorphisms
		\begin{align*}
			&
					R \invlim_{\lambda} \dirlim_{\mu} R \invlim_{\mu'} \dirlim_{\lambda'}
					RF(A_{\lambda \lambda'}, B_{\mu \mu'})
			\\
			&	\to
					R \invlim_{\lambda} \dirlim_{\mu} R \invlim_{\mu'}
					RF(\invlim_{\lambda'} A_{\lambda \lambda'}, B_{\mu \mu'})
			\\
			&	=
					R \invlim_{\lambda} \dirlim_{\mu}
					RF(\invlim_{\lambda'} A_{\lambda \lambda'},
						R \invlim_{\mu'} B_{\mu \mu'})
			\\
			&	\to
					R \invlim_{\lambda}
					RF(\invlim_{\lambda'} A_{\lambda \lambda'},
						\dirlim_{\mu} R \invlim_{\mu'} B_{\mu \mu'})
			\\
			&	=
					RF(\dirlim_{\lambda} \invlim_{\lambda'} A_{\lambda \lambda'},
						\dirlim_{\mu} R \invlim_{\mu'} B_{\mu \mu'})
		\end{align*}
	in $D^{+}(\mathcal{C})$ for any
	$A = \fdirlim[\lambda] \finvlim[\lambda'] A_{\lambda \lambda'} \in \Ind \Pro \mathcal{A}$ and
	$B = \fdirlim[\mu] \finvlim[\mu'] B_{\mu \mu'} \in \Ind \Pro \mathcal{A}$.
\end{Prop}

Note that the $R \invlim_{\mu'}$ in the displayed equations are the derived ones
but the $\invlim_{\lambda'}$ are not,
so that all the variables lie in $D^{-}(\mathcal{A})^{\op} \times D^{+}(\mathcal{A})$ where $RF$ is defined.

\begin{proof}
	In the rigorous terms,
	the morphisms and isomorphisms to be constructed are
		\begin{align*}
			&
					R \invlim \;
					\Pro \dirlim \;
					R \Pro \Ind \invlim \;
					\Pro \Ind \Pro \dirlim \;
					R \Pro \Ind \Pro \Ind F(A, B)
			\\
			&	\to
					R \invlim \;
					\Pro \dirlim \;
					R \Pro \Ind \invlim \;
					R \Pro \Ind \Pro F(\Ind \invlim A, B)
			\\
			&	=
					R \invlim \;
					\Pro \dirlim \;
					R \Pro \Ind F(\Ind \invlim A, R \Ind \invlim B)
			\\
			&	\to
					R \invlim \;
					R \Pro F(\Ind \invlim A, \dirlim \, R \Ind \invlim B)
			\\
			&	=
					RF(\dirlim \, \Ind \invlim A, \dirlim \, R \Ind \invlim B)
		\end{align*}
	for $A, B \in \Ind \Pro \mathcal{A}$.
	It suffices to construct a morphism
		\[
				\Pro \Ind \Pro \dirlim \;
				R \Pro \Ind \Pro \Ind F(A, B)
			\to
				R \Pro \Ind \Pro F(\Ind \invlim A, B)
			\quad \text{in} \quad
				D^{+}(\Pro \Ind \Pro \mathcal{C})
		\]
	for $A \in \Ind \Pro \mathcal{A}$ and $B \in D^{+}(\Ind \Pro \mathcal{A})$,
	an isomorphism
		\[
				R \Pro \Ind \invlim \;
				R \Pro \Ind \Pro F(A, B)
			=
				R \Pro \Ind F(A, R \Ind \invlim B)
			\quad \text{in} \quad
				D^{+}(\Pro \Ind \mathcal{C})
		\]
	for $A \in D^{-}(\Ind \mathcal{A})$, $B \in D^{+}(\Ind \Pro \mathcal{A})$,
	a morphism
		\[
				\Pro \dirlim \;
				R \Pro \Ind F(A, B)
			\to
				R \Pro F(A, \dirlim B)
			\quad \text{in} \quad
				D^{+}(\Pro \mathcal{C})
		\]
	for $A \in D^{-}(\Ind \mathcal{A})$, $B \in D^{+}(\Ind \mathcal{A})$,
	and an isomorphism
		\[
				R \invlim \;
				R \Pro F(A, B)
			=
				RF(\dirlim A, B)
			\quad \text{in} \quad
				D^{+}(\mathcal{C})
		\]
	for $A \in D^{-}(\Ind \mathcal{A})$, $B \in D^{+}(\mathcal{A})$.
	
	We construct the first morphism.
	We fix $A \in \Ind \Pro \mathcal{A}$.
	Let $\mathcal{A}' \subset \mathcal{A}$ be the full subcategory of injectives.
	Then $(\mathcal{A}^{\op}, \mathcal{A}')$ is $F$-injective by assumption.
	Hence $((\Ind \Pro \mathcal{A})^{\op}, \Ind \Pro \mathcal{A}')$ is $\Pro \Ind \Pro \Ind F$-injective
	by \eqref{prop: derivation in ind-procategories}
	\eqref{ass: deriving the pro and ind of a bifunctor}.
	This implies that $R \Pro \Ind \Pro \Ind F(A, B)$
	(which is a priori calculated by resolving $A$ and $B$ at the same time)
	is the value at $B$ of the derived functor of
	the single-variable functor $\Pro \Ind \Pro \Ind F(A, \,\cdot\,)$ (\cite[Cor.\ 13.4.5]{KS06}).
	Similarly $R \Pro \Ind \Pro F(\Ind \invlim A, B)$ is the value at $B$ of the derived functor of
	$\Pro \Ind \Pro F(\Ind \invlim A, \,\cdot\,)$.
	Hence we need to construct a morphism
		\[
				\Pro \Ind \Pro \dirlim \;
				R \Pro \Ind \Pro \Ind F(A, \,\cdot\,)
			\to
				R \Pro \Ind \Pro F(\Ind \invlim A, \,\cdot\,)
		\]
	of functors $D^{+}(\Ind \Pro \mathcal{A}) \to D^{+}(\Pro \Ind \Pro \mathcal{C})$.
	Since $\Pro \Ind \Pro \dirlim \colon \Pro \Ind \Pro \Ind \mathcal{C} \to \Pro \Ind \Pro \mathcal{C}$ is exact,
	we have
		\[
				\Pro \Ind \Pro \dirlim \;
				R \Pro \Ind \Pro \Ind F(A, \,\cdot\,)
			=
				R \bigl(
					\Pro \Ind \Pro \dirlim \;
					\Pro \Ind \Pro \Ind F
				\bigr)(A, \,\cdot\,).
		\]
	Deriving the natural morphism
		\[
				\Pro \Ind \Pro \dirlim \;
				\Pro \Ind \Pro \Ind F(A, \,\cdot\,)
			\to
				\Pro \Ind \Pro F(\Ind \invlim A, \,\cdot\,)
		\]
	of functors $\Ind \Pro \mathcal{A} \to \Pro \Ind \Pro \mathcal{C}$,
	we get the required morphism of derived functors.
	
	We construct the second isomorphism.
	Let $\mathcal{A}' \subset \mathcal{A}$ be the full subcategory of injectives.
	Since $F$ commute with filtered inverse limits,
	we have a commutative diagram
		\[
			\begin{CD}
					(\Ind \mathcal{A})^{\op} \times \Ind \Pro \mathcal{A}
				@>> \Pro \Ind \Pro F >
					\Pro \Ind \Pro \mathcal{C}
				\\
				@VV \id \times \Ind \invlim V
				@V \Pro \Ind \invlim VV
				\\
					(\Ind \mathcal{A})^{\op} \times \Ind \mathcal{A}
				@> \Pro \Ind F >>
					\Pro \Ind \mathcal{C}.
			\end{CD}
		\]
	This diagram restricts to the diagram
		\[
			\begin{CD}
					(\fdsum \mathcal{A})^{\op} \times \Ind \fprod \mathcal{A}'
				@>>>
					\fprod \Ind \fprod \mathcal{C}'
				\\
				@VVV
				@VVV
				\\
					(\fdsum \mathcal{A})^{\op} \times \Ind \mathcal{A}'
				@>>>
					\fprod \Ind \mathcal{C}'.
			\end{CD}
		\]
	of full subcategories
	since $F(\mathcal{A}^{\op} \times \mathcal{A}') \subset \mathcal{C}'$
	and products of injectives are injective.
	The category $(\fdsum \mathcal{A})^{\op} \times I \fprod \mathcal{A}'$
	is injective with respect to both $\Pro \Ind \Pro F$ and $\id \times \Ind \invlim$
	by \eqref{prop: derivation in ind-procategories}.
	Therefore the above commutative diagram derives to a commutative diagram
	of the derived functors on the derived categories.
	This gives the required isomorphism.
	
	We construct the third morphism.
	Consider the following (non-commutative) diagram:
		\[
			\begin{CD}
					(\Ind \mathcal{A})^{\op} \times \Ind \mathcal{A}
				@>> \Pro \Ind F >
					\Pro \Ind \mathcal{C}
				\\
				@VV \id \times \dirlim V
				@V \Pro \dirlim VV
				\\
					(\Ind \mathcal{A})^{\op} \times \mathcal{A}
				@> \Pro F >>
					\Pro \mathcal{C}.
			\end{CD}
		\]
	This gives two functors from the left upper term to the right lower term.
	There is a natural morphism of functors from the one factoring through the right upper term
	to the one factoring through the left lower term.
	Since the right vertical arrow is an exact functor,
	this morphism of functors induces a morphism of functors in the derived categories,
	which is the required morphism.
	
	The fourth isomorphism can be constructed in the same way as the second.
\end{proof}

One might expect that
the $\invlim_{\lambda'}$ in the proposition may be replaced by $R \invlim_{\lambda'}$
to make the statement more natural.
This would need to extend the results of \cite[Chap.\ 15]{KS06} on derivation of indcategories
cited here several times to unbounded derived categories.
It is not clear whether such an extension is possible or not.

The following should be a well-known fact
on the commutation of the derived pushforward and direct/derived inverse limits.
We note it here
since we need to treat continuous maps of sites without exact pullbacks
and non-coherent sites,
which do not frequently appear in the literature.

\begin{Prop} \BetweenThmAndList \label{prop: limits and cohomology commute}
	\begin{enumerate}
		\item \label{ass: acyclics are injective for inverse limits}
			Let $S$ be a site defined by finite coverings.
			Let $\mathcal{C}' \subset \Ab(S)$ be the full subcategory of acyclic sheaves.
			Then $\mathcal{C}'$ satisfies the conditions of
			\eqref{prop: derivation in ind-procategories}
			\eqref{ass: derived functor of ind-pro limits}.
		\item \label{ass: limits and pushforward}
			Let $f \colon S_{2} \to S_{1}$ be a continuous map of sites.
			Assume that both $S_{1}$ and $S_{2}$ are defined by finite coverings.
			Assume also that $f_{\ast}$ sends acyclic sheaves to acyclic sheaves.
			Then $R f_{\ast}$ commutes with $\dirlim$ and $R \invlim$ on bounded below derived categories.
			That is, the diagrams
				\[
					\begin{CD}
							D^{+}(\Ind \Ab(S_{2}))
						@>> R \Ind f_{\ast} >
							D^{+}(\Ind \Ab(S_{1}))
						\\
						@VV \dirlim V
						@V \dirlim VV
						\\
							D^{+}(\Ab(S_{2}))
						@> R f_{\ast} >>
							D^{+}(\Ab(S_{1})),
					\end{CD}
					\qquad
					\begin{CD}
							D^{+}(\Pro \Ab(S_{2}))
						@>> R \Pro f_{\ast} >
							D^{+}(\Pro \Ab(S_{1}))
						\\
						@VV R \invlim V
						@V R \invlim VV
						\\
							D^{+}(\Ab(S_{2}))
						@> R f_{\ast} >>
							D^{+}(\Ab(S_{1}))
					\end{CD}
				\]
			commute.
	\end{enumerate}
\end{Prop}

\begin{proof}
	\eqref{ass: acyclics are injective for inverse limits}
	The category $\mathcal{C}'$ is closed by products and filtered direct limits
	by the comments in the proof of 
	\eqref{prop: misc on ind-rational pro-etale topology}
	\eqref{ass: quasi-compactness of the sites}.
	All the other conditions are immediate.
	
	\eqref{ass: limits and pushforward}
	Let $\mathcal{C}_{i}' \subset \Ab(S_{i})$ be the full subcategory of acyclic sheaves, $i = 1, 2$.
	The commutative diagrams
		\[
			\begin{CD}
					\Ind \Ab(S_{2})
				@>> \Ind f_{\ast} >
					\Ind \Ab(S_{1})
				\\
				@VV \dirlim V
				@V \dirlim VV
				\\
					\Ab(S_{2})
				@> f_{\ast} >>
					\Ab(S_{1}),
			\end{CD}
			\qquad
			\begin{CD}
					\Pro \Ab(S_{2})
				@>> \Pro f_{\ast} >
					\Pro \Ab(S_{1})
				\\
				@VV \invlim V
				@V \invlim VV
				\\
					\Ab(S_{2})
				@> f_{\ast} >>
					\Ab(S_{1})
			\end{CD}
		\]
	restricts to their subcategories
		\[
			\begin{CD}
					\Ind \mathcal{C}_{2}'
				@>>>
					\Ind \mathcal{C}_{1}'
				\\
				@VVV
				@VVV
				\\
					\mathcal{C}_{2}'
				@>>>
					\mathcal{C}_{1}',
			\end{CD}
			\qquad
			\begin{CD}
					\fprod \mathcal{C}_{2}'
				@>>>
					\fprod \mathcal{C}_{1}'
				\\
				@VVV
				@VVV
				\\
					\mathcal{C}_{2}'
				@>>>
					\mathcal{C}_{1}'.
			\end{CD}
		\]
	Hence the diagrams derive to commutative diagrams of the derived categories
	by \eqref{prop: derivation in ind-procategories}.
\end{proof}


\numberwithin{equation}{subsection}
\subsection{The derived categories of ind-proalgebraic groups and of sheaves}
\label{sec: The derived categories of ind-proalgebraic groups and of sheaves}

We apply the results of the previous subsection to
$\Alg / k$ and $\Ab(k^{\ind\rat}_{\pro\et})$.
We continue using the notation in the previous subsection
about filtered direct/derived inverse limits in derived categories of ind/pro-objects.
We use, however, the usual $\dirlim$ and $\invlim$ instead of $\fdirlim$ and $\finvlim$ for simplicity
when they are considered in $\Ind \Alg / k$, $\Pro \Alg / k$ or $\Ind \Pro \Alg / k$.
The $\dirlim$ and $\invlim$ for sheaves mean limits as sheaves, not as ind-sheaves or pro-sheaves.

We will use the four operations formalism (about pushforward, pullback, sheaf-Hom and tensor products)
for derived categories of sites throughout the paper.
A good reference (for arbitrary morphisms of sites and unbounded derived categories) is \cite[Chap.\ 18]{KS06}.

The following theorem is a summary of the site-theoretic results proved in \cite{Suz13}.
We clarify below which results of \cite{Suz13} imply each statement.

\begin{Thm} \label{thm: Ext of proalgebraic groups as sheaves}
	Let
		\[
				\Spec k^{\perf}_{\pro\fppf}
			\stackrel{f}{\to}
				\Spec k^{\perf}_{\et}
			\stackrel{g}{\to}
				\Spec k^{\ind\rat}_{\et}
		\]
	be the continuous maps defined by the identity
	and $S$ either one of these sites.
	Let $A = \invlim_{\lambda} A_{\lambda} \in \Pro \Alg / k$ with $A_{\lambda} \in \Alg / k$.
	\begin{enumerate}
		\item \label{ass: RHom for profppf, perf etale and indrat etale}
			If $B \in \Ab(k^{\perf}_{\pro\fppf})$, then
				\begin{align*}
							g_{\ast} R f_{\ast}
							R \sheafhom_{k^{\perf}_{\pro\fppf}}(A, B)
					&	=
							g_{\ast}
							R \sheafhom_{k^{\perf}_{\et}}(A, R f_{\ast} B)
					\\
					&	=
							R \sheafhom_{k^{\ind\rat}_{\et}}(A, g_{\ast} R f_{\ast} B).
				\end{align*}
			(Note that $g_{\ast}$ is an exact functor by
			\eqref{prop: misc on ind-rational pro-etale topology}
			\eqref{ass: indrat proet and proet cohomology}.)
		\item \label{ass: Ext commutes with direct limits}
			If $B = \dirlim_{\mu} B_{\mu}$ with $B_{\mu} \in \Ab(S)$, then
				\[
						R \sheafhom_{S}(A, B)
					=
						\dirlim_{\mu}
						R \sheafhom_{S}(A, B_{\mu})
				\]
			for $S = \Spec k^{\perf}_{\pro\fppf}$ or $\Spec k^{\perf}_{\et}$.
		\item \label{ass: RHom between PAlg and LAlg}
			If $B \in \Loc \Alg / k$, then $R f_{\ast} B = B$ (and $g_{\ast} B = B$ as notation),
			and we have
				\[
						R \sheafhom_{S}(A, B)
					=
						\dirlim_{\lambda}
						R \sheafhom_{S}(A_{\lambda}, B)
				\]
			(with $S = \Spec k^{\ind\rat}_{\et}$ allowed).
			Its $R \Gamma(k', \,\cdot\,)$ for
			$k' = \bigcup k'_{\nu} \in k^{\ind\rat}$ with $k'_{\nu} \in k^{\rat}$ is
				\[
					\dirlim_{\lambda, \nu}
					R \Hom_{S / k'_{\nu}}(A_{\lambda}, B),
				\]
			where $S / k'_{\nu}$ is the localization of $S$ at $k'_{\nu}$
			(see the notation section of \S \ref{sec: Organization}).
			If $k'$ is a field and $S = \Spec k^{\ind\rat}_{\et}$, then this is further isomorphic to
				\[
					R \Hom_{k'^{\ind\rat}_{\et}}(A_{\lambda}, B).
				\]
		\item \label{ass: Ext by uniquely divisible is zero}
			If $B \in \Et / k$ ($\subset \Loc \Alg / k$) is uniquely divisible, then
				\[
						R \sheafhom_{k^{\ind\rat}_{\et}}(A, B)
					=
						0.
				\]
		\item \label{ass: RHom of proalgebraic groups as sheaves}
			If $B \in \Alg / k$, then
				\[
						R \Hom_{S}(A, B)
					=
						R \Hom_{\Pro \Alg / k}(A, B).
				\]
	\end{enumerate}
\end{Thm}

The corresponding results in \cite{Suz13} were proved for affine $A$.
(The symbol $\Alg / k$ in \cite{Suz13} denoted the category of affine quasi-algebraic groups.)
The same proof works for abelian varieties as remarked after \cite[Thm.\ 2.1.5]{Suz13}.
It even works for arbitrary proalgebraic groups;
see Appendix \ref{sec: The pro-fppf topology for proalgebraic proschemes}.

\begin{proof}
	\eqref{ass: RHom for profppf, perf etale and indrat etale}
	Note that $f^{\ast}$ is the pro-fppf sheafification functor and hence $f^{\ast} A = A$.
	($A$ is a sheaf even for the fpqc topology by the fpqc descent and the fact that
	inverses limits of sheaves are sheaves.)
	Therefore the sheafified adjunction between $f^{\ast}$ and $R f_{\ast}$ \cite[Thm.\ 18.6.9 (iii)]{KS06}
	gives an isomorphism
		\[
				R f_{\ast}
				R \sheafhom_{k^{\perf}_{\pro\fppf}}(A, B)
			=
				R \sheafhom_{k^{\perf}_{\et}}(A, R f_{\ast} B)
		\]
	in $D(k^{\perf}_{\et})$.
	This proves the first isomorphism.
	For the second,
	if $R f_{\ast} B \isomto J$ is an injective resolution over $\Spec k^{\perf}_{\et}$,
	then we have natural isomorphisms and the functoriality morphism of $g_{\ast}$:
		\begin{align*}
					g_{\ast}
					R \sheafhom_{k^{\perf}_{\et}}(A, R f_{\ast} B)
			&	=
					g_{\ast}
					\sheafhom_{k^{\perf}_{\et}}(A, J)
			\\
			&	\to
					\sheafhom_{k^{\ind\rat}_{\et}}(g_{\ast} A, g_{\ast} J)
			\\
			&	\to
					R \sheafhom_{k^{\ind\rat}_{\et}}(g_{\ast} A, g_{\ast} J)
			\\
			&	=
					R \sheafhom_{k^{\ind\rat}_{\et}}(A, g_{\ast} R f_{\ast} B)
		\end{align*}
	in $D(k^{\ind\rat}_{\et})$,
	where the middle $\sheafhom$'s are the total complexes of the sheaf-Hom double complexes
	and $g_{\ast} A = A$ as notation.
	Applying $R \Gamma(k'_{\et}, \;\cdot\;)$ for any $k' \in k^{\ind\rat}$
	and using the Leray spectral sequence for $g_{\ast}$,
	we have a morphism
		\[
				R \Hom_{k'^{\perf}_{\et}}(A, R f_{\ast} B)
			\to
				R \Hom_{k^{\ind\rat}_{\et} / k'}(A, g_{\ast} R f_{\ast} B)
		\]
	in $D(\Ab)$.
	This is an isomorphism by \cite[Prop.\ 3.7.4]{Suz13}.
	Therefore we obtain the second isomorphism.
	
	\eqref{ass: Ext commutes with direct limits}
	Let $S = \Spec k^{\perf}_{\et}$.
	We have a morphism
		\[
				\dirlim_{\mu}
				R \sheafhom_{k^{\perf}_{\et}}(A, B_{\mu})
			\to
				R \sheafhom_{k^{\perf}_{\et}}(A, B)
		\]
	by \eqref{prop: RHom of ind-pro-limits}.
	(The construction of this morphism given in the proof does not involve the subcategory $\mathcal{C}'$.)
	To see this is an isomorphism,
	we apply $R \Gamma(R_{\et}, \;\cdot\;)$ for any $R \in k^{\perf}$
	and take cohomology in degree $n \ge 0$:
		\[
				\dirlim_{\mu}
				\Ext_{R^{\perf}_{\et}}^{n}(A, B_{\mu})
			\to
				\Ext_{R^{\perf}_{\et}}^{n}(A, B).
		\]
	This is an isomorphism by \cite[Lem.\ 3.8.2 (2)]{Suz13}.
	Hence the result follows.
	
	The same proof (including the proof of \cite[loc.\ cit.]{Suz13}) works
	for $\Spec k^{\perf}_{\pro\fppf}$ in the same manner.
	
	\eqref{ass: RHom between PAlg and LAlg}
	We have a morphism
		\[
				\dirlim_{\lambda}
				R \sheafhom_{S}(A_{\lambda}, B)
			\to
				R \sheafhom_{S}(A, B)
		\]
	by \eqref{prop: RHom of ind-pro-limits}
	(with the same remark about $\mathcal{C}'$ as above).
	Again, by taking cohomology at each object of $S$ and cohomology at each degree,
	we are reduced to showing that the following four induced morphisms are isomorphisms:
		\begin{gather*}
					\dirlim_{\lambda, \nu}
					\Ext_{(R_{\nu})^{\perf}_{\et}}^{n}(A_{\lambda}, B)
				\to
					\Ext_{R^{\perf}_{\et}}^{n}(A, B),
			\\
					\dirlim_{\lambda, \nu}
					\Ext_{(R_{\nu})^{\perf}_{\pro\fppf}}^{n}(A_{\lambda}, B)
				\to
					\Ext_{R^{\perf}_{\pro\fppf}}^{n}(A, B)
		\end{gather*}
	for a filtered direct system $\{R_{\nu}\}$ in $k^{\perf}$ with limit $R$,
		\[
				\dirlim_{\lambda, \nu}
				\Ext_{(k'_{\nu})^{\ind\rat}_{\et}}^{n}(A_{\lambda}, B)
			\to
				\Ext_{k^{\ind\rat}_{\et} / k'}^{n}(A, B)
		\]
	for a filtered direct system $\{k'_{\nu}\}$ in $k^{\rat}$ with limit $k'$, and
		\[
				\Ext_{k^{\ind\rat}_{\et} / k'}^{n}(A, B)
			\to
				\Ext_{k'^{\ind\rat}_{\et}}^{n}(A, B)
		\]
	for a field $k' \in k^{\ind\rat}$.
	The first isomorphism is \cite[Lem.\ 3.8.2 (1)]{Suz13}.
	This lemma also works for the pro-fppf topology,
	so we get the second isomorphism.
	Since $R f_{\ast} B = B$ by
	\eqref{prop: misc on ind-rational pro-etale topology}
	\eqref{ass: profppf, proet and et cohom of quasi-algebraic groups etc},
	\cite[Prop.\ 3.7.4]{Suz13} implies that
		\[
				\Ext_{k'^{\perf}_{\pro\fppf}}^{n}(A, B)
			=
				\Ext_{k'^{\perf}_{\et}}^{n}(A, B)
			=
				\Ext_{k^{\ind\rat}_{\et} / k'}^{n}(A, B)
		\]
	for any $k' \in k^{\ind\rat}$.
	Hence the third isomorphism follows.
	The fourth isomorphism is \cite[Thm.\ 2.1.5]{Suz13}.
	
	\eqref{ass: Ext by uniquely divisible is zero}
	Such a group $B$ can be written as a filtered direct limit of \'etale groups over $k$
	whose groups of geometric points are finite-dimensional $\Q$-vector spaces.
	Hence by \eqref{ass: Ext commutes with direct limits},
	we may assume that $B$ is an \'etale group over $k$
	with $\dim_{\Q} B(\closure{k}) < \infty$.
	Let $k_{1}$ be a finite Galois extension of $k$ with Galois group $G$
	such that $B$ becomes a constant group over $k_{1}$.
	Let $k' \in k^{\ind\rat}$ and $k'_{1} = k' \tensor_{k} k_{1}$.
	The Hochschild-Serre spectral sequence
	for the coefficients in
		$
				R \sheafhom_{(k_{1})^{\ind\rat}_{\et}}(A, B)
			=
				R \sheafhom_{k^{\ind\rat}_{\et} / k_{1}}(A, B)
		$
	yields
		\[
				R \Gamma(G, R \Hom_{k^{\ind\rat}_{\et} / k'_{1}}(A, B))
			=
				R \Hom_{k^{\ind\rat}_{\et} / k'}(A, B).
		\]
	Therefore we may assume that $B$ is constant and, moreover, $B = \Q$.
	The statement to prove is thus
		\[
				\Ext_{k^{\ind\rat}_{\et} / k'}^{n}(A, \Q)
			=
				0
		\]
	for $k' \in k^{\ind\rat}$ and $n \ge 0$.
	This is \cite[Thm.\ 2.1.5]{Suz13}.
	
	\eqref{ass: RHom of proalgebraic groups as sheaves}
	By \cite[Thm.\ 2.1.5, Prop.\ 3.7.4 and Prop.\ 3.8.1]{Suz13}, we have
		\[
				R \Hom_{k^{\perf}_{\pro\fppf}}(A, B)
			=
				R \Hom_{k^{\perf}_{\et}}(A, B)
			=
				R \Hom_{k^{\ind\rat}_{\et}}(A, B)
			=
				R \Hom_{\Pro \Alg / k}(A, B).
		\]
\end{proof}

We generalize this theorem to the case
where both $A$ and $B$ are ind-proalgebraic groups.

\begin{Prop} \label{prop: comparison thm in the pro-etale topology}
	\eqref{thm: Ext of proalgebraic groups as sheaves} remains true
	when all the subscripts $\et$ are replaced by $\pro\et$.
\end{Prop}

\begin{proof}
	This is a straightforward consequence of the fact that
	$A \in \Pro \Alg / k$ is a sheaf for the pro-fppf topology
	(as commented in the proof of the theorem)
	and \eqref{prop: misc on ind-rational pro-etale topology}
	\eqref{ass: profppf, proet and et cohom of quasi-algebraic groups etc}.
	We only explain this for \eqref{ass: RHom for profppf, perf etale and indrat etale}
	and \eqref{ass: RHom between PAlg and LAlg}.
	
	\eqref{ass: RHom for profppf, perf etale and indrat etale}
	We need to show that the morphisms
		\[
				R \Hom_{k'^{\perf}_{\pro\fppf}}(A, B)
			\to
				R \Hom_{k'^{\perf}_{\pro\et}}(A, R \Tilde{f}_{\ast} B)
			\to
				R \Hom_{k^{\ind\rat}_{\pro\et} / k'}(A, \Tilde{g}_{\ast} R \Tilde{f}_{\ast} B)
		\]
	are isomorphisms for $k' \in k^{\ind\rat}$, where
		\[
			\begin{CD}
					\Spec k^{\perf}_{\pro\fppf}
				@>> \Tilde{f} >
					\Spec k^{\perf}_{\pro\et}
				@>> \Tilde{g} >
					\Spec k^{\ind\rat}_{\pro\et}
				\\
				@|
				@VV P V
				@V P VV
				\\
					\Spec k^{\perf}_{\pro\fppf}
				@> f >>
					\Spec k^{\perf}_{\et}
				@> g >>
					\Spec k^{\ind\rat}_{\et}.
			\end{CD}
		\]
	But $P^{\ast} A = A$.
	Hence the adjunction between $P^{\ast}$ and $R P_{\ast}$ reduces the statement to
	the original assertion \eqref{ass: RHom for profppf, perf etale and indrat etale}.
	
	\eqref{ass: RHom between PAlg and LAlg}
	Similarly, we need to show that the morphism
		\[
				\dirlim_{\lambda}
				R \Hom_{k^{\ind\rat}_{\pro\et} / k'}(A_{\lambda}, B)
			\to
				R \Hom_{k^{\ind\rat}_{\pro\et} / k'}(A, B)
		\]
	is an isomorphism for $k' \in k^{\ind\rat}$.
	But $R P_{\ast} B = B$ by
	\eqref{prop: misc on ind-rational pro-etale topology}
	\eqref{ass: profppf, proet and et cohom of quasi-algebraic groups etc}.
	Hence again the adjunction between $P^{\ast}$ and $R P_{\ast}$ reduces the statement to
	the original assertion \eqref{ass: RHom between PAlg and LAlg}.
	The rest can be proven similarly.
\end{proof}

\begin{Prop} \BetweenThmAndList \label{prop: comparison thm for ind-proalgebraic groups}
	Let $S$ be either $\Spec k^{\perf}_{\pro\fppf}$,
	$\Spec k^{\perf}_{\pro\et}$ or
	$\Spec k^{\ind\rat}_{\pro\et}$.
	Let $A = \dirlim A_{\lambda} \in \Ind \Pro \Alg / k$
	with $A_{\lambda} = \invlim A_{\lambda \lambda'} \in \Pro \Alg / k$
	with $A_{\lambda \lambda'} \in \Alg / k$.
	\begin{enumerate}
		\item \label{ass: RHom between IPAlg and IPAlg, sheaf version}
			Let $B = \dirlim B_{\mu} = \dirlim \invlim B_{\mu \mu'} \in \Ind \Pro \Alg / k$ similarly.
			Then
				\[
						R \sheafhom_{S}(A, B)
					=
						R \invlim_{\lambda} \dirlim_{\mu}
						R \invlim_{\mu'} \dirlim_{\lambda'}
						R \sheafhom_{S}(A_{\lambda \lambda'}, B_{\mu \mu'}).
				\]
		\item \label{ass: RHom between IPAlg and IPAlg}
			Its $R \Gamma(k, \,\cdot\,)$ is
				\begin{align*}
					&
							R \Hom_{\Ind \Pro \Alg / k}(A, B)
						=
							R \invlim_{\lambda} \dirlim_{\mu}
							R \Hom_{\Pro \Alg / k}(A_{\lambda}, B_{\mu})
					\\
					&	\quad
						=
							R \invlim_{\lambda} \dirlim_{\mu}
							R \invlim_{\mu'} \dirlim_{\lambda'}
							R \Hom_{\Alg / k}(A_{\lambda \lambda'}, B_{\mu \mu'}).
				\end{align*}
			When already $A_{\lambda}, B_{\mu} \in \Alg / k$ for all $\lambda, \mu$
			so that $A, B \in \Ind \Alg / k$,
			then this further equals to $R \Hom_{\Ind \Alg / k}(A, B)$.
		\item \label{ass: RHom between IPAlg and LAlg}
			If $B \in \Loc \Alg / k$, then
				\[
						R \sheafhom_{S}(A, B)
					=
						R \invlim_{\lambda} \dirlim_{\lambda'}
						R \sheafhom_{S}(A_{\lambda \lambda'}, B).
				\]
		\item \label{ass: RHom between IPAlg and Et}
			If $B \in \Et / k$ is uniquely divisible, then
				\[
						R \sheafhom_{k^{\ind\rat}_{\pro\et}}(A, B)
					=
						0.
				\]
	\end{enumerate}
\end{Prop}

\begin{proof}
	\eqref{ass: RHom between IPAlg and IPAlg, sheaf version}
	We first prove this for $S = \Spec k^{\perf}_{\pro\fppf}$ and $\Spec k^{\perf}_{\pro\et}$.
	We will apply \eqref{prop: RHom of ind-pro-limits} for
		\[
				F
			=
				\sheafhom_{S}
			\colon
					\Ab(S)^{\op}
				\times
					\Ab(S)
			\to
				\Ab(S).
		\]
	Let $\mathcal{C} = \mathcal{C}' = \Ab(S)$.
	This has exact products by
	\eqref{prop: misc on ind-rational pro-etale topology}
	\eqref{ass: product is exact}.
	This implies the conditions of 
	\eqref{prop: derivation in ind-procategories}
	\eqref{ass: derived functor of ind-pro limits}.
	Hence \eqref{prop: RHom of ind-pro-limits} yields morphisms and isomorphisms
		\begin{align*}
			&
					R \invlim_{\lambda} \dirlim_{\mu} R \invlim_{\mu'} \dirlim_{\lambda'}
					R \sheafhom_{S}(
						A_{\lambda \lambda'}, B_{\mu \mu'}
					)
			\\
			&	\to
					R \invlim_{\lambda} \dirlim_{\mu} R \invlim_{\mu'}
					R \sheafhom_{S}(
						\invlim_{\lambda'} A_{\lambda \lambda'}, B_{\mu \mu'}
					)
			\\
			&	=
					R \invlim_{\lambda} \dirlim_{\mu}
					R \sheafhom_{S}(
						\invlim_{\lambda'} A_{\lambda \lambda'},
						R \invlim_{\mu'} B_{\mu \mu'}
					)
			\\
			&	\to
					R \invlim_{\lambda}
					R \sheafhom_{S}(
						\invlim_{\lambda'} A_{\lambda \lambda'},
						\dirlim_{\mu} R \invlim_{\mu'} B_{\mu \mu'}
					)
			\\
			&	=
					R \sheafhom_{S}(
						\dirlim_{\lambda} \invlim_{\lambda'} A_{\lambda \lambda'},
						\dirlim_{\mu} R \invlim_{\mu'} B_{\mu \mu'}
					)
		\end{align*}
	in $D^{+}(S)$ for any
		\begin{gather*}
					A
				= 	\fdirlim[\lambda] \finvlim[\lambda']
						A_{\lambda \lambda'}
				\in
					\Ind \Pro \Ab(S),
			\\
					B
				=
					\fdirlim[\mu] \finvlim[\mu']
						B_{\mu \mu'}
				\in
					\Ind \Pro \Ab(S).
		\end{gather*}
	Assume $A_{\lambda \lambda'}, B_{\mu \mu'} \in \Alg / k$.
	Then $R \invlim_{\mu'} B_{\mu \mu'} = \invlim_{\mu'} B_{\mu \mu'}$
	by \eqref{prop: misc on ind-rational pro-etale topology}
	\eqref{ass: Rlim vanishes on proalgebraic groups}.
	The first and third morphisms are isomorphisms by
	\eqref{thm: Ext of proalgebraic groups as sheaves}
	\eqref{ass: RHom between PAlg and LAlg} and \eqref{ass: Ext commutes with direct limits}, respectively.
	
	For $S = \Spec k^{\ind\rat}_{\pro\et}$,
	let
		\[
				\Spec k^{\perf}_{\pro\fppf}
			\stackrel{f}{\to}
				\Spec k^{\perf}_{\pro\et}
			\stackrel{g}{\to}
				\Spec k^{\ind\rat}_{\pro\et}
		\]
	be the continuous maps defined by the identity.
	We will apply $g_{\ast}$ for the both sides of the isomorphism
	just proved for $S = \Spec k^{\perf}_{\pro\et}$.
	For the left-hand side, we have
		\[
				g_{\ast}
				R \sheafhom_{k^{\perf}_{\pro\et}}(A, B)
			=
				R \sheafhom_{k^{\ind\rat}_{\pro\et}}(A, B)
		\]
	by \eqref{thm: Ext of proalgebraic groups as sheaves}
	\eqref{ass: RHom for profppf, perf etale and indrat etale}
	and \eqref{prop: comparison thm in the pro-etale topology},
	since $B \in \Ind \Pro \Alg / k$ and hence $R f_{\ast} B = B$
	by \eqref{prop: misc on ind-rational pro-etale topology}
	\eqref{ass: profppf, proet and et cohom of quasi-algebraic groups etc}.
	For the right-hand side, we have
		\begin{align*}
			&
					g_{\ast}
					R \invlim_{\lambda} \dirlim_{\mu}
					R \invlim_{\mu'} \dirlim_{\lambda'}
					R \sheafhom_{k^{\perf}_{\pro\et}}(A_{\lambda \lambda'}, B_{\mu \mu'})
			\\
			&	=
					R \invlim_{\lambda} \dirlim_{\mu}
					R \invlim_{\mu'} \dirlim_{\lambda'}
					g_{\ast}
					R \sheafhom_{k^{\perf}_{\pro\et}}(A_{\lambda \lambda'}, B_{\mu \mu'})
			\\
			&	=
					R \invlim_{\lambda} \dirlim_{\mu}
					R \invlim_{\mu'} \dirlim_{\lambda'}
					R \sheafhom_{k^{\ind\rat}_{\pro\et}}(A_{\lambda \lambda'}, B_{\mu \mu'})
		\end{align*}
	by \eqref{prop: limits and cohomology commute}
	and again by \eqref{thm: Ext of proalgebraic groups as sheaves}
	\eqref{ass: RHom for profppf, perf etale and indrat etale}.
	
	\eqref{ass: RHom between IPAlg and IPAlg}.
	We have
		\begin{align*}
			&
					R \Gamma \bigl(
						k,
						R \invlim_{\lambda} \dirlim_{\mu}
						R \invlim_{\mu'} \dirlim_{\lambda'}
						R \sheafhom_{S}(A_{\lambda \lambda'}, B_{\mu \mu'})
					\bigr)
			\\
			&	=
					R \invlim_{\lambda} \dirlim_{\mu}
					R \invlim_{\mu'} \dirlim_{\lambda'}
					R \Hom_{S}(A_{\lambda \lambda'}, B_{\mu \mu'}).
			\\
			&	=
					R \invlim_{\lambda} \dirlim_{\mu}
					R \invlim_{\mu'} \dirlim_{\lambda'}
					R \Hom_{\Alg / k}(A_{\lambda \lambda'}, B_{\mu \mu'})
		\end{align*}
	by \eqref{thm: Ext of proalgebraic groups as sheaves}
	\eqref{ass: RHom of proalgebraic groups as sheaves}
	and \eqref{prop: comparison thm in the pro-etale topology}.
	
	\eqref{ass: RHom between IPAlg and LAlg}
	The proof of \eqref{ass: RHom between IPAlg and IPAlg, sheaf version} also applies.
	
	\eqref{ass: RHom between IPAlg and Et}
	This follows from \eqref{ass: RHom between IPAlg and LAlg} here,
	\eqref{thm: Ext of proalgebraic groups as sheaves}
	\eqref{ass: RHom between IPAlg and Et}
	and \eqref{prop: comparison thm in the pro-etale topology}.
\end{proof}

\begin{Prop} \label{prop: fully faithful embedding for ind-proalgebraic groups}
	We have fully faithful embeddings
		\[
			\begin{CD}
					D^{b}(\Alg / k)
				@>> \subset >
					D^{b}(\Pro \Alg / k)
				@.
				\\
				@VV \cap V
				@V \cap VV
				@.
				\\
					D^{b}(\Ind \Alg / k)
				@> \subset >>
					D^{b}(\Ind \Pro \Alg / k)
				@> \subset >>
					D^{b}(k^{\ind\rat}_{\pro\et}),
			\end{CD}
		\]
		\[
				D^{b}(\FGEt / k)
			\subset
				D^{b}(\Et / k)
			=
				D^{b}(k_{\et})
			\subset
				D^{b}(k^{\ind\rat}_{\pro\et})
		\]
	of triangulated categories.
	The same is true with $D^{b}(k^{\ind\rat}_{\pro\et})$ replaced by
	$D^{b}(k^{\perf}_{\pro\et})$ or $D^{b}(k^{\perf}_{\pro\fppf})$.
\end{Prop}

\begin{proof}
	This follows from \eqref{prop: fully faithful embeddings of ind-procateogries}
	and \eqref{prop: comparison thm for ind-proalgebraic groups}.
\end{proof}


\numberwithin{equation}{subsection}
\subsection{Serre duality and P-acyclicity}
\label{sec: Serre duality and P-acyclicity}

In this subsection, we mainly focus on the site $\Spec k^{\ind\rat}_{\pro\et}$.
From the next section,
it is important to use $\Spec k^{\ind\rat}_{\et}$ and $\Spec k^{\ind\rat}_{\pro\et}$
in order to view cohomology of local fields as sheaves over the residue field and compute it.
(Using $\Spec k^{\perf}_{\pro\et}$ instead might be slightly problematic
for the treatment of Serre duality for semi-abelian varieties.
The group $\Ext_{R^{\perf}_{\pro\et}}^{n}(\Gm, \Gm)$ for example might not be torsion for $n \ge 2$
if $R$ is not the perfection of a regular scheme and
the computation done in \cite{Bre70} does not literally apply.
This implies that the sheaf-Ext $\sheafext_{k^{\perf}_{\pro\et}}^{n}(\Gm, \Gm)$ might not be torsion.
The same problem exists for $\Spec k^{\perf}_{\pro\fppf}$.)

For a complex $A \in D(k^{\ind\rat}_{\pro\et})$,
we define
	\[
			A^{\SDual}
		=
			R \sheafhom_{k^{\ind\rat}_{\pro\et}}(A, \Z).
	\]
We call $A^{\SDual}$ the \emph{Serre dual} of $A$.%
\footnote{
	Milne \cite[III, 0]{Mil06} used the term ``Breen-Serre duality''.
	Breen's work \cite{Bre78} brought Serre's observation \cite[8.4, Remarque]{Ser60}
	to the setting of the perfect \'etale site
	and provided Artin-Milne \cite{AM76} a basis of their ``second context'' of global duality.
	Just to make it notationally parallel to and easily distinguishable with
	Cartier duality $\CDual$ and Pontryagin duality $\PDual$,
	we omit Breen's name and use the symbol $\SDual$ in this paper.
	Of course our treatment of the functor $\SDual$ heavily relies on the work of both Breen and Serre.
}
It defines a contravariant triangulated endofunctor $\SDual$ on $D(k^{\ind\rat}_{\pro\et})$.
If $A \in \Ind \Pro \Alg / k$, then
	\[
			A^{\SDual}
		=
			R \sheafhom_{k^{\ind\rat}_{\pro\et}}(A, \Q / \Z)[-1]
	\]
by \eqref{prop: comparison thm for ind-proalgebraic groups}
\eqref{ass: RHom between IPAlg and Et}.
Note that this is not true for $\Z$, whose Serre dual is itself.
For any $A \in D(k^{\ind\rat}_{\pro\et})$,
we have a natural morphism $A \to (A^{\SDual})^{\SDual} = A^{\SDual \SDual}$.
If this is an isomorphism,
we say that $A$ is \emph{Serre reflexive}.
Serre reflexive complexes form a full triangulated subcategory of $D(k^{\ind\rat}_{\pro\et})$,
where the functor $\SDual$ is a contravariant auto-equivalence with inverse itself.

\begin{Prop} \label{prop: Serre duality} \BetweenThmAndList
	\begin{enumerate}
	\item \label{ass: Ext of a proalg by Q mod Z is zero above degree 2}
		If $A \in \Pro \Alg / k$, then
			\[
					\sheafext_{k^{\ind\rat}_{\pro\et}}^{n}(A, \Q / \Z)
				=
					0
			\]
		for $n \ne 0, 1$.
		We have
			\[
					\sheafhom_{k^{\ind\rat}_{\pro\et}}(A, \Q / \Z)
				=
					\sheafhom_{k^{\ind\rat}_{\pro\et}}(\pi_{0}(A), \Q / \Z).
			\]
	\item \label{ass: Serre duality for unip and etale}
		The objects of the categories
			\[
					D^{b}(\Pro \Alg_{\uc} / k),
					D^{b}(\Ind \Alg_{\uc} / k),
					D^{b}(\FGEt / k)
			\]
		are Serre reflexive.
		The functor $\SDual$ gives an auto-equivalence on $D^{b}(\Alg_{\uc} / k)$,
		which agrees with the Breen-Serre duality on perfect unipotent groups
		(\cite[8.4, Remarque]{Ser60}, \cite[III, Thm.\ 0.14]{Mil06}) shifted by $-1$.%
		\footnote{
			Strictly speaking, the objects of $\Alg_{\uc} / k$ are not exactly unipotent
			since their component groups can be any objects of $\FEt / k$ not necessarily $p$-primary.
			The Breen-Serre duality for finite \'etale $p$-primary groups agrees with
			the Pontryagin duality \cite[III, Lem.\ 0.13 (b)]{Mil06}.
		}
		It also gives an auto-equivalence on $D^{b}(\FEt / k)$,
		which agrees with the Pontryagin duality $\PDual = \sheafhom_{k}(\;\cdot\;, \Q / \Z)$
		shifted by $-1$.
		It further gives an auto-equivalence on $D^{b}(\FGEt / k)$,
		and an equivalence between $D^{b}(\Pro \Alg_{\uc} / k)$ and $D^{b}(\Ind \Alg_{\uc} / k)$.
		For $A \in \Pro \Alg / k$ (resp.\ $A \in \Ind \Alg / k$),
		we have: $H^{1} A^{\SDual} = \pi_{0}(A)^{\PDual}$;
			\[
					H^{2} A^{\SDual}
				=
					\sheafext_{k^{\ind\rat}_{\pro\et}}^{1}(A, \Q / \Z)
			\]
		is a connected ind-unipotent (resp.\ pro-unipotent) group;
		and $H^{n} A^{\SDual} = 0$ for $n \ne 1, 2$.
	\item \label{ass: Serre duality for semi-abelian varieties}
		Let $A$ be (the perfection of) a semi-abelian variety over $k$.
		Let $T A$ be the Tate module of $A$,
		i.e.\ the pro-finite-\'etale group over $k$ given by
		the inverse limit of the kernel of multiplication by $n \ge 1$ on $A$.%
		\footnote{
			This is a reduced scheme due to the perfection.
			A perfect scheme is reduced.
			Hence $T A$ is the Pontryagin dual of the torsion part of the discrete Galois module $A(\closure{k})$.
		}
		Let $\bar{A}$ be the pro-\'etale universal covering group of $A$,
		i.e.\ the inverse limit of multiplication by $n \ge 1$ on $A$.
		Then the morphism $A \to TA[1]$ (where $[1]$ is the shift)
		coming from the natural exact sequence
			\[
				0 \to TA \to \bar{A} \to A \to 0
			\]
		in $\Ab(k^{\ind\rat}_{\pro\et})$ and \eqref{ass: Serre duality for unip and etale}
		gives isomorphisms
			\[
				A^{\SDual \SDual} \isomto TA[1]^{\SDual \SDual} = TA[1].
			\]
		We also have $A^{\SDual} = (TA)^{\PDual}[-2]$.
	\item \label{ass: Serre duality for general proalgebraic groups}
		For any $A \in \Pro \Alg / k$, there is a canonical morphism
		from $A^{\SDual \SDual}$ to the inverse limit $\invlim_{n} A$ of multiplication by $n \ge 1$
		placed in degree $-1$,
		so that we have
			\[
					A^{\SDual \SDual}
				=
					\bigl[
							\invlim_{n} A
						\to
							A
					\bigr],
			\]
		where the morphism $\invlim_{n} A \to A$ is the projection to the $n = 1$ term
		and $[\;\cdot\;]$ denotes the mapping cone.
		To give a more explicit description,
		let $A_{0}$ be the identity component of $A$,
		$A_{\mathrm{u}}$ the maximal pro-unipotent subgroup of $A_{0}$,
		$A_{\mathrm{sAb}} = A_{0} / A_{\mathrm{u}}$ the maximal semi-abelian quotient of $A_{0}$.
		Then we have $\invlim_{n} A = \closure{A_{\mathrm{sAb}}}$,
		whose Serre dual is zero.
		The commutative diagram with exact rows
			\[
				\begin{CD}
						0
					@>>>
						T A_{\mathrm{sAb}}
					@>>>
						\invlim_{n} A
					@>>>
						A_{\mathrm{sAb}}
					@>>>
						0
					\\
					@.
					@VVV
					@VVV
					@VVV
					@.
					\\
						0
					@>>>
						A_{\mathrm{u}}
					@>>>
						A
					@>>>
						A / A_{\mathrm{u}}
					@>>>
						0
				\end{CD}
			\]
		gives a distinguished triangle%
		\footnote{
			Perhaps a more helpful description of $A^{\SDual \SDual} = [\invlim A \to A]$ is that
			its $H^{0}$ is the quotient of $A$ by the maximal semiabelian
			\emph{sub}group $A_{\mathrm{sAb}}'$ of $A$
			(so $H^{0}$ is the maximal quotient of $A$ that belongs to $\Pro \Alg_{\uc}$)
			and $H^{-1}$ the Tate module of $A_{\mathrm{sAb}}'$.
			The composite $A_{\mathrm{sAb}}' \into A_{0} \onto A_{\mathrm{sAb}}$ is an isogeny.
		}
			\[
					\bigl[
							T A_{\mathrm{sAb}}
						\to
							A_{\mathrm{u}}
					\bigr]
				\to
					\bigl[
							\invlim_{n} A
						\to
							A
					\bigr]
				\to
					\pi_{0}(A).
			\]
		In particular, we have $A^{\SDual \SDual} \in D^{b}(\Pro \Alg_{\uc} / k)$.
		We also have $A^{\SDual} \in D^{b}(\Ind \Alg_{\uc} / k)$, which is Serre reflexive.
	\end{enumerate}
\end{Prop}

\begin{proof}
	\eqref{ass: Ext of a proalg by Q mod Z is zero above degree 2}
	For any $n$, the sheaf in question is the pro-\'etale sheafification of the presheaf
		\[
				k' \in k^{\ind\rat}
			\mapsto
				\Ext_{k^{\ind\rat}_{\pro\et} / k'}^{n}(A, \Q / \Z).
		\]
	Let $k' = \bigcup k'_{\nu} \in k^{\ind\rat}$ with $k'_{\nu} \in k^{\rat}$.
	Then the group on the right can also be written as
		\[
				\Ext_{k^{\perf}_{\pro\fppf} / k'}^{n}(A, \Q / \Z)
			=
				\Ext_{k^{\ind\rat}_{\et} / k'}^{n}(A, \Q / \Z)
			=
				\dirlim_{\nu}
				\Ext_{(k'_{\nu})^{\ind\rat}_{\et}}^{n}(A, \Q / \Z)
		\]
	by \eqref{thm: Ext of proalgebraic groups as sheaves}
	\eqref{ass: RHom for profppf, perf etale and indrat etale},
	\eqref{ass: RHom between PAlg and LAlg} and
	\eqref{prop: comparison thm in the pro-etale topology}.
	
	For the statement for $n \ne 0, 1$,
	it is enough, replacing $k'_{\nu}$ with $k$, to show that the \'etale sheafification of
		\[
				k' \in k_{\et}
			\mapsto
				\Ext_{k'^{\ind\rat}_{\et}}^{n}(A, \Q / \Z)
		\]
	is zero (for any perfect field $k$), where $k_{\et}$ is the category of \'etale $k$-algebras.
	Equivalently, it is enough to show that the group
		\[
				\Ext_{\closure{k}^{\ind\rat}_{\et}}^{n}(A, \Q / \Z)
			=
				\Ext_{k^{\ind\rat}_{\et} / \closure{k}}^{n}(A, \Q / \Z)
			=
				\dirlim_{k' / k \text{ finite}}
				\Ext_{k'^{\ind\rat}_{\et}}^{n}(A, \Q / \Z)
		\]
	is zero for $n \ne 0, 1$, where $\closure{k}$ is an algebraic closure of $k$.
	For this, we may assume that $k = \closure{k}$.
	We have
		\[
				\Ext_{k^{\ind\rat}_{\et}}^{n}(A, \Q / \Z)
			=
				\dirlim_{m}
				\Ext_{k^{\ind\rat}_{\et}}^{n}(A, \Z / m \Z)
			=
				\dirlim_{m}
				\Ext_{\Pro \Alg / k}^{n}(A, \Z / m \Z)
		\]
	by \eqref{thm: Ext of proalgebraic groups as sheaves}
	\eqref{ass: Ext commutes with direct limits},
	\eqref{ass: RHom of proalgebraic groups as sheaves}.
	This final group is the Pontryagin dual of the $n$-th homotopy group of $A$ over $k$
	in the sense of Serre \cite[\S 5, Cor.\ to Prop.\ 7]{Ser60}.
	It is zero for $n \ne 0, 1$ by \cite[\S 10, Thm.\ 2]{Ser60}.
	
	The statement for $n = 0$ is proven similarly.
	
	\eqref{ass: Serre duality for unip and etale}
	First assume that $A \in \Alg_{\uc} / k$.
	Let $g \colon \Spec k^{\perf}_{\pro\et} \to \Spec k^{\ind\rat}_{\pro\et}$ be
	the continuous map defined by the identity.
	Then
		\[
				A^{\SDual}
			=
				R \sheafhom_{k^{\ind\rat}_{\pro\et}}(A, \Q / \Z)[-1]
			=
				g_{\ast} R \sheafhom_{k^{\perf}_{\pro\et}}(A, \Q / \Z)[-1]
		\]
	by \eqref{prop: misc on ind-rational pro-etale topology}
	\eqref{ass: profppf, proet and et cohom of quasi-algebraic groups etc},
	\eqref{thm: Ext of proalgebraic groups as sheaves}
	\eqref{ass: RHom for profppf, perf etale and indrat etale}
	and \eqref{prop: comparison thm in the pro-etale topology}.
	Let $P \colon \Spec k^{\perf}_{\pro\et} \to \Spec k^{\perf}_{\et}$
	be the morphism defined by the identity.
	Then $R P_{\ast} \Q / \Z = \Q / \Z$ by
	\eqref{prop: misc on ind-rational pro-etale topology}
	\eqref{ass: profppf, proet and et cohom of quasi-algebraic groups etc}
	and $P^{\ast} A = A$.
	Hence
		\[
				R P_{\ast} R \sheafhom_{k^{\perf}_{\pro\et}}(A, \Q / \Z)
			=
				R \sheafhom_{k^{\perf}_{\et}}(A, \Q / \Z).
		\]
	The $p$-primary part of the right-hand side is the Breen-Serre dual of $A$ by definition
	\cite[III, Thm.\ 0.14]{Mil06},
	which is in $D^{b}(\Alg_{\uc} / k)$.
	The prime-to-$p$ part is the Pontryagin dual of the prime-to-$p$ part of $A$ (or $\pi_{0}(A)$).
	With $P^{\ast} R P_{\ast} = \id$, we have
		\[
				R \sheafhom_{k^{\perf}_{\pro\et}}(A, \Q / \Z)
			=
				P^{\ast} R \sheafhom_{k^{\perf}_{\et}}(A, \Q / \Z),
		\]
	hence
		\[
				A^{\SDual}
			=
				g_{\ast} P^{\ast}
				R \sheafhom_{k^{\perf}_{\et}}(A, \Q / \Z)[-1].
		\]
	Note that the objects in $\Alg / k$ are already pro-\'etale sheaves
	and hence invariant under the pro-\'etale sheafification $P^{\ast}$.
	Hence we know that $A^{\SDual}$ is the Breen-Serre dual of $A$ shifted by $-1$ in the $p$-primary part
	and the Pontryagin dual of $A$ in the prime-to-$p$ part shifted by $-1$.
	In particular, we have by \cite[III, Lem.\ 0.13 (b), (c), (d)]{Mil06}:
	$H^{1} A^{\SDual} = \pi_{0}(A)^{\PDual}$;
	$H^{2} A^{\SDual} \in \Alg / k$ is connected unipotent;
	and $H^{n} = 0$ for $n \ne 1, 2$.
	
	We show that $A^{\SDual \SDual} = A$.
	The Breen-Serre duality \cite[III, Thm.\ 0.14]{Mil06} and the Pontryagin duality show that
		\[
				A
			=
				R \sheafhom_{k^{\perf}_{\et}} \bigl(
					R \sheafhom_{k^{\perf}_{\et}}(A, \Q / \Z),
					\Q / \Z
				\bigr).
		\]
	By $R P_{\ast} \Q / \Z = \Q / \Z$, $P^{\ast} R P_{\ast} = \id$ and $P^{\ast} A = A$, we have
		\[
				A
			=
				R \sheafhom_{k^{\perf}_{\pro\et}} \bigl(
					P^{\ast}
					R \sheafhom_{k^{\perf}_{\et}}(A, \Q / \Z),
					\Q / \Z
				\bigr).
		\]
	Applying $g_{\ast}$, we have a morphism
		\begin{align*}
					A
			&	=
					g_{\ast}
					R \sheafhom_{k^{\perf}_{\pro\et}} \bigl(
						P^{\ast}
						R \sheafhom_{k^{\perf}_{\et}}(A, \Q / \Z),
						\Q / \Z
					\bigr)
			\\
			&	\to
					R \sheafhom_{k^{\ind\rat}_{\pro\et}} \bigl(
						g_{\ast} P^{\ast}
						R \sheafhom_{k^{\perf}_{\et}}(A, \Q / \Z),
						\Q / \Z
					\bigr)
			\\
			&	=
					A^{\SDual \SDual}.
		\end{align*}
	Since $R \sheafhom_{k^{\perf}_{\et}}(A, \Q / \Z) \in D^{b}(\Alg_{\uc} / k)$,
	this morphism is an isomorphism by
	\eqref{thm: Ext of proalgebraic groups as sheaves}
	\eqref{ass: RHom for profppf, perf etale and indrat etale}
	and \eqref{prop: comparison thm in the pro-etale topology}
	(applied to the cohomologies of the complex $R \sheafhom_{k^{\perf}_{\et}}(A, \Q / \Z)$).
	Thus $A^{\SDual \SDual} = A$.
	
	If $A = \invlim A_{\lambda} \in \Pro \Alg_{\uc} / k$,
	then 
		$
				A^{\SDual}
			=
				\dirlim A_{\lambda}^{\SDual}
		$
	by \eqref{prop: comparison thm for ind-proalgebraic groups}
	\eqref{ass: RHom between IPAlg and IPAlg, sheaf version}.
	Since the direct limit of sheaves is exact, we have
	$H^{n} A^{\SDual} = \dirlim H^{n} A_{\lambda}^{\SDual}$ for any $n$.
	Hence $A^{\SDual} \in D^{b}(\Ind \Alg_{\uc} / k)$.
	On the other hand, if $A = \dirlim A_{\lambda} \in \Ind \Alg_{\uc} / k$, then
		$
				A^{\SDual}
			=
				R \invlim A_{\lambda}^{\SDual}
		$
	by the same assertion.
	We have $H^{n} A^{\SDual} = \invlim H^{n} A_{\lambda}^{\SDual}$ for any $n$
	by \eqref{prop: misc on ind-rational pro-etale topology}
	\eqref{ass: Rlim vanishes on proalgebraic groups}.
	Hence $A^{\SDual} \in D^{b}(\Pro \Alg_{\uc} / k)$.
	Therefore the statements about $D^{b}(\Pro \Alg_{\uc} / k)$ and $D^{b}(\Ind \Alg_{\uc} / k)$
	are reduced to the statements about $D^{b}(\Alg_{\uc} / k)$,
	which have already been proved above.
	
	The statement for $D^{b}(\FGEt / k)$ follows from the easy computation $\Z^{\SDual} = \Z$.
	
	\eqref{ass: Serre duality for semi-abelian varieties}
	By \eqref{prop: comparison thm for ind-proalgebraic groups}
	\eqref{ass: RHom between IPAlg and IPAlg, sheaf version},
	we have
		\[
				R \sheafhom_{k^{\ind\rat}_{\pro\et}}(\closure{A}, \Q / \Z)
			=
				\dirlim_{n}
				R \sheafhom_{k^{\ind\rat}_{\pro\et}}(\closure{A}, \Z / n \Z)
			=
				0
		\]
	since $\closure{A} \in \Pro \Alg / k$ is uniquely divisible.
	With the vanishing result \eqref{ass: Ext of a proalg by Q mod Z is zero above degree 2} above,
	we have
		\[
				\sheafext_{k^{\ind\rat}_{\pro\et}}^{1}(A, \Q / \Z)
			=
				\sheafhom_{k^{\ind\rat}_{\pro\et}}(TA, \Q / \Z)
		\]
	by the universal covering exact sequence.
	Hence
		\[
				A^{\SDual}
			=
				R \sheafhom_{k^{\ind\rat}_{\pro\et}}(A, \Q / \Z)[-1]
			=
				\sheafhom_{k^{\ind\rat}_{\pro\et}}(TA, \Q / \Z)[-2]
			=
				(TA)^{\SDual}[-1].
		\]
	Therefore $A^{\SDual \SDual} = TA[1]$.
	
	\eqref{ass: Serre duality for general proalgebraic groups}
	We have
		\[
				A^{\SDual}
			=
				R \sheafhom_{k^{\ind\rat}_{\pro\et}}(A, \Q / \Z)[-1]
			=
				\dirlim_{m}
					R \sheafhom_{k^{\ind\rat}_{\pro\et}}(A, \Z / m \Z)[-1]
		\]
	by \eqref{prop: comparison thm for ind-proalgebraic groups}
	\eqref{ass: RHom between IPAlg and IPAlg, sheaf version},
	which is killed after tensoring with $\Q$.
	Hence the same assertion implies that
		\[
				\Bigl(
					\invlim_{n} A
				\Bigr)^{\SDual}
			=
				\dirlim_{n} (A^{\SDual})
			=
				0.
		\]
	Therefore we have a natural morphism and an isomorphism
		\[
				\bigl[
						\invlim_{n} A
					\to
						A
				\bigr]
			\to
				\bigl[
						\invlim_{n} A
					\to
						A
				\bigr]^{\SDual \SDual}
			=
				A^{\SDual \SDual}.
		\]
	We will know that the first morphism is an isomorphism
	once we show that $[\invlim_{n} A \to A] \in D^{b}(\Pro \Alg_{\uc} / k)$
	(whose objects are Serre reflexive
	by \eqref{ass: Serre duality for unip and etale}).
	
	Note that the exact endofunctor $\invlim_{n}$ on $\Pro \Alg / k$
	kills protorsion groups and hence profinite groups and pro-unipotent groups.
	This implies that
	$\invlim_{n} A = \invlim_{n} A_{\mathrm{sAb}} = \closure{A_{\mathrm{sAb}}}$.
	This yields the stated commutative diagram.
	In particular, the mapping cone $[A_{\mathrm{sAb}} \to A / A_{\mathrm{u}}]$
	is quasi-isomorphic to the mapping cone of the morphism of complexes
	$[T A_{\mathrm{sAb}} \to A_{\mathrm{u}}] \to [\invlim_{n} A \to A]$.
	This induces the mapping cone distinguished triangle
		\[
				[T A_{\mathrm{sAb}} \to A_{\mathrm{u}}]
			\to
				[\invlim_{n} A \to A]
			\to
				[A_{\mathrm{sAb}} \to A / A_{\mathrm{u}}].
		\]
	Since the kernel of $A_{\mathrm{sAb}} \to A / A_{\mathrm{u}}$ is zero
	and the cokernel is $\pi_{0}(A)$,
	we obtain the stated distinguished triangle.
	This implies that $[\invlim_{n} A \to A] \in D^{b}(\Pro \Alg_{\uc} / k)$
	since $[T A_{\mathrm{sAb}} \to A_{\mathrm{u}}], \pi_{0}(A) \in D^{b}(\Pro \Alg_{\uc} / k)$.
	Therefore $A^{\SDual \SDual} = [\invlim_{n} A \to A] \in D^{b}(\Pro \Alg_{\uc} / k)$.
	
	Applying the Serre dual to the stated distinguished triangle
	and using $(\invlim_{n} A)^{\SDual} = 0$, we have a distinguished triangle
			\[
					\pi_{0}(A)^{\SDual}
				\to
					A^{\SDual}
				\to
					\bigl[
							T A_{\mathrm{sAb}}
						\to
							A_{\mathrm{u}}
					\bigr]^{\SDual}.
			\]
	Hence $A^{\SDual} \in D^{b}(\Ind \Alg_{\uc} / k)$.
\end{proof}

As examples, we have
	\begin{gather*}
				W_{n}^{\SDual}
			=
				W_{n}[-2]
			\qquad
				\text{\cite[III, Lem.\ 0.13 (c)]{Mil06}},
		\\
				W^{\SDual}
			=
				\dirlim_{n} W_{n}^{\SDual}
			=
				\dirlim_{n} W_{n}[-2],
		\\
				(\Z / n \Z)^{\SDual}
			=
				\Z / n \Z[-1],
			\quad
				\Z^{\SDual}
			=
				\Z,
		\\
				\Gm^{\SDual}
			=
				\bigoplus_{l \ne p}
					(\Q_{l} / \Z_{l})(-1)[-2],
			\quad
				A^{\SDual}
			=
				(TA)^{\PDual}[-2],
	\end{gather*}
where $n \ge 1$, $W_{n}$ the group of Witt vectors of length $n$,
$W = \invlim W_{n}$,
$(\Q_{l} / \Z_{l})(-1)$ the negative Tate twist
and $A$ an abelian variety over $k$.
(Recall that all the groups here are the perfections of the corresponding group schemes.
In particular, the $p$-th power map on the quasi-algebraic group $\Gm$ is an isomorphism.)

Before introducing the notion of P-acyclicity,
we discuss sheaves locally of finite presentation on $\Spec k^{\ind\rat}_{\et}$.
A sheaf $F \in \Set(k^{\ind\rat}_{\et})$ is said to be locally of finite presentation
if it commutes with filtered direct limits as a functor on $k^{\ind\rat}$.
Such a sheaf is automatically a pro-\'etale sheaf: $F \in \Set(k^{\ind\rat}_{\pro\et})$.
For $k' = \bigcup k'_{\lambda} \in k^{\ind\rat}$ with $k_{\lambda} \in k^{\rat}$, we define
	\[
			F^{\fp}(k')
		=
			\dirlim_{\lambda}
				F(k'_{\lambda}).
	\]
Clearly $F^{\fp}$ is locally of finite presentation.
We see that $F^{\fp} \in \Set(k^{\ind\rat}_{\et})$
since an \'etale covering $k'' / k'$ in $k^{\ind\rat}$ is
a filtered direct limit of \'etale coverings in $k^{\rat}$.
We have a natural morphism $F^{\fp} \to F$,
which is an isomorphism if and only if
$F$ is locally of finite presentation.
If $\{F_{\lambda}\}$ is a filtered direct system of sheaves locally of finite presentation,
then $\dirlim F_{\lambda}$ is locally of finite presentation.
Since objects of $\Loc \Alg / k$ and $\Alg / k$ in particular are locally of finite presentation,
it follows that objects of $\Ind \Alg / k$ are also locally of finite presentation.
If $A \in \Pro \Alg / k$ and $B \in \Loc \Alg / k$,
then $\sheafext_{k^{\ind\rat}_{\et}}^{n}(A, B)$ for any $n \ge 0$ is locally of finite presentation
by \eqref{thm: Ext of proalgebraic groups as sheaves}
\eqref{ass: RHom between PAlg and LAlg}.
In particular, $\sheafext_{k^{\ind\rat}_{\et}}^{n}(A, B)$ and
$\sheafext_{k^{\ind\rat}_{\pro\et}}^{n}(A, B)$ are equal as functors on $k^{\ind\rat}$.
If $f \colon E \to F \in \Set(k^{\ind\rat}_{\et})$ is a morphism
between sheaves locally of finite presentation,
then the property of $f$ being injective, surjective or invertible
can be tested on geometric points
(i.e.\ if $E(k') \into F(k')$ for any algebraically closed field $k'$ over $k$,
then $E \into F$ etc.).
Similarly, if $A \in \Ab(k^{\ind\rat}_{\et})$ is locally of finite presentation
with vanishing geometric points
(i.e.\ $A(k') = 0$ for any algebraically closed field $k'$ over $k$),
then $A = 0$.
This argument has already been used in the proof of
\eqref{prop: Serre duality}
\eqref{ass: Ext of a proalg by Q mod Z is zero above degree 2}.
A sheaf $F \in \Set(k^{\ind\rat}_{\et})$ is representable by an \'etale scheme over $k$
if and only if it is locally of finite presentation
and $F(k')$ does not depend on algebraically closed $k'$ over $k$,
in which case $F = F(\closure{k})$ as $\Gal(\closure{k} / k)$-sets.

Now let $P \colon \Spec k^{\ind\rat}_{\pro\et} \to \Spec k^{\ind\rat}_{\et}$
be the morphism defined by the identity.
We say that a sheaf $A \in \Ab(k^{\ind\rat}_{\et})$ is \emph{P-acyclic}
if the natural morphism
	\[
		A \to R P_{\ast} P^{\ast} A
	\]
is an isomorphism.%
\footnote{
	The ``P'' in ``P-acyclicity'' is just the initial letter of ``pro''
	and not (necessarily) a variable.
	We just avoid naming it as ``pro-acyclicity'',
	since it is not a pro-object version of some acyclicity property.
}
This is equivalent that $A \in \Ab(k^{\ind\rat}_{\pro\et})$
and $H^{n}(k'_{\et}, A) \isomto H^{n}(k'_{\pro\et}, A)$ for all $k' \in k^{\ind\rat}$.
We say that a complex $A \in D^{+}(k^{\ind\rat}_{\et})$ is \emph{P-acyclic}
if its cohomology is P-acyclic in each degree.
This implies that
	\[
			A \isomto R P_{\ast} P^{\ast} A
		\quad \text{in} \quad
			D^{+}(k^{\ind\rat}_{\et})
	\]
by writing down the hypercohomology spectral sequence for $R P_{\ast} P^{\ast} A$.
If $0 \to A \to B \to C \to 0$ is an exact sequence in $\Ab(k^{\ind\rat}_{\et})$
and any two of the terms are P-acyclic, then so is the other.
P-acyclicity of a sheaf is preserved under filtered direct limits.
If $f \colon \Spec k' \to \Spec k$ is a finite \'etale morphism
and $A' \in \Ab(k'^{\ind\rat}_{\et})$ is P-acyclic,
then the Weil restriction $f_{\ast} A' \in \Ab(k^{\ind\rat}_{\et})$ is P-acyclic.
When $A \in D^{+}(k^{\ind\rat}_{\et})$ is P-acyclic,
we will use the same letter $A$ to mean its inverse image
$P^{\ast} A \in D^{+}(k^{\ind\rat}_{\pro\et})$.
Here is a criterion of P-acyclicity.

\begin{Prop} \label{prop: criterion of P-acyclicity} \BetweenThmAndList
	\begin{enumerate}
		\item \label{ass: local finite presentation implies P-acyclic}
			If $A \in \Loc \Alg / k$ or $\Ind \Alg / k$, then $A$ is P-acyclic.
			More generally, any sheaf $A \in \Ab(k^{\ind\rat}_{\et})$ locally of finite presentation
			is P-acyclic.
		\item \label{ass: connected unipotent system is P-acyclic}
			If $A \in \Pro \Alg / k$ can be written as
			$A = \invlim_{n \ge 1} A_{n}$ with $A_{n} \in \Alg / k$
			such that all the transition morphisms $\varphi_{n} \colon A_{n + 1} \to A_{n}$ are surjective
			with connected unipotent kernel,
			then $A$ is P-acyclic.
	\end{enumerate}
\end{Prop}

\begin{proof}
	\eqref{ass: local finite presentation implies P-acyclic}
	We already saw above that
	objects of $\Loc \Alg / k$ and $\Ind \Alg / k$ are locally of finite presentation
	and a sheaf locally of finite presentation is in $\Ab(k^{\ind\rat}_{\pro\et})$.
	Let $A \in \Ab(k^{\ind\rat}_{\et})$ be locally of finite presentation
	(so $A \in \Ab(k^{\ind\rat}_{\pro\et})$).
	We need to show that
		$
				H^{n}(k'_{\pro\et}, A)
			=
				H^{n}(k'_{\et}, P_{\ast} A)
		$
	for $n \ge 1$ and $k' \in k^{\ind\rat}$.
	In \cite[Prop.\ 3.3.1]{Suz13},
	a statement corresponding to pro-fppf and fppf cohomology
	instead of pro-\'etale and \'etale cohomology is proved.
	The same proof works.
	
	\eqref{ass: connected unipotent system is P-acyclic}
	Let $B_{n}$ be the maximal connected unipotent subgroup of $A_{n}$.
	Then $C := A_{n} / B_{n} \in \Alg / k$ does not depend on $n$ by assumption
	and $B = \invlim B_{n}$ satisfies the same assumption as $A$.
	If $B$ is P-acyclic,
	then the sequence $0 \to B \to A \to C \to 0$ is exact in $\Ab(k^{\ind\rat}_{\et})$,
	and $A$ is P-acyclic.
	Therefore we may assume that all $A_{n}$ are connected unipotent.
	Then $H^{m}(k'_{\et}, A_{n}) = 0$ for all $m \ge 1$.
	Therefore we have $R^{m} \prod_{n} A_{n} = 0$ for all $m \ge 1$
	by \cite[Prop.\ 1.6]{Roo06}.
	Therefore by \cite[Thm.\ 2.1, Lem.\ 2.2 and Rmk.\ 2.3]{Roo06},
	the object $R \invlim A_{n} \in D(k^{\ind\rat}_{\et})$ is represented by the complex
		\[
				\prod_{n} A_{n}
			\to
				\prod_{n} A_{n},
		\]
	concentrated in degrees $0$ and $1$,
	where the morphism is given by $(a_{n}) \mapsto (a_{n} - \varphi_{n}(a_{n + 1}))$.
	Note that the morphisms $\varphi_{n} \colon A_{n + 1}(k') \to A_{n}(k')$
	is surjective for all $k' \in k^{\ind\rat}$ and $n$
	since $\ker(\varphi_{n})$ is connected unipotent.
	Therefore the evaluated complex
		\[
				\prod_{n} A_{n}(k')
			\to
				\prod_{n} A_{n}(k'),
		\]
	which is $R \invlim A_{n}(k') \in D(\Ab)$, is concentrated in degree $0$.
	Hence $R \invlim A_{n} = A$ in $D(k^{\ind\rat}_{\et})$.
	Therefore
		\[
				R \Gamma(k'_{\et}, A)
			=
				R \Gamma(k'_{\et}, R \invlim A_{n})
			=
				R \invlim R \Gamma(k'_{\et}, A_{n})
			=
				R \invlim A_{n}(k')
			=
				A(k').
		\]
	Since $R \invlim A_{n} = A$ in $D(k^{\ind\rat}_{\pro\et})$
	follows from \eqref{prop: misc on ind-rational pro-etale topology}
	\eqref{ass: Rlim vanishes on proalgebraic groups},
	a similar calculation shows that $R \Gamma(k'_{\pro\et}, A) = A(k')$.
	Hence $A$ is P-acyclic.
\end{proof}

For example, the proalgebraic groups $\Ga^{\N}$, $W$ and $W^{\N}$ are P-acyclic.
The groups $\Z_{l}$ for a prime $l$ and $\Gm^{\N}$ are not.


\numberwithin{equation}{section}
\section{Local fields with ind-rational base}
\label{sec: Local fields with ind-rational base}

Let $K$ be a complete discrete valuation field
with perfect residue field $k$ of characteristic $p > 0$.
In this section, we view cohomology of $K$ as a complex of sheaves on $\Spec k^{\ind\rat}_{\pro\et}$.
As explained in \S \ref{sec: Motivation of our constructions},
the $n$-th cohomology of this complex is the pro-\'etale sheafification of the presheaf
	\[
			k'
		\mapsto
			H^{n}(\alg{K}(k'), A),
	\]
where $A$ is an fppf sheaf on $K$.
To study cohomology of this form,
we introduce a topological space structure
on the sets of rational points of schemes over the ring $\alg{K}(k')$ following \cite{Con12}.
We give some density and openness results
using some Artin type approximation methods and results,
which reduce the problems to classical results on complete discrete valuation fields.

Then we recall the relative fppf site $\Spec K_{\fppf} / k^{\ind\rat}_{\et}$
and the structure morphism
	\[
			\pi
		\colon
			\Spec K_{\fppf} / k^{\ind\rat}_{\et}
		\to
			\Spec k^{\ind\rat}_{\et}
	\]
from \cite{Suz13}.
The complex we want is defined to be
the pro-\'etale sheafification of the derived pushforward by $\pi$.


\numberwithin{equation}{subsection}
\subsection{Basic notions and properties}
\label{sec: Basic notions and properties}

We denote the ring of integers of $K$ by $\Order_{K}$.
The maximal ideal of $\Order_{K}$ is denoted by $\ideal{p}_{K}$.
We denote by $W$ the affine ring scheme of Witt vectors of infinite length.
Since $k$ is a perfect field of characteristic $p > 0$,
the ring $\Order_{K} = \invlim_{n \ge 0} \Order_{K} / \ideal{p}_{K}^{n}$ has
a canonical structure of a $W(k)$-algebra of pro-finite-length
\cite[II, \S 5, Thm.\ 4]{Ser79}
(which is actually finite free over $W(k)$ in the mixed characteristic case
and factors as $W(k) \onto k \into \Order_{K}$ in the equal characteristic case).
We define sheaves of rings on the site $\Spec k^{\ind\rat}_{\et}$
by assigning to each $k' \in k^{\ind\rat}$,
	\begin{gather*}
				\alg{O}_{K}(k')
			=
				W(k') \hat{\tensor}_{W(k)} \Order_{K}
			=
				\invlim_{n} \bigl(
					W_{n}(k') \tensor_{W_{n}(k)} \Order_{K} / \ideal{p}_{K}^{n}
				\bigr),
		\\
				\alg{K}(k')
			=
				\alg{O}_{K}(k') \tensor_{\Order_{K}} K.
	\end{gather*}
The Teichm\"uller lift gives a morphism $\omega \colon \Affine_{k}^{1} \into \alg{O}_{K}$.
Any element of $\alg{K}(k')$ can be written as
$\sum_{n \in \Z} \omega(a_{n}) \pi^{n}$
with a unique sequence of elements $a_{n}$ of $k'$
whose negative terms are zero for all but finitely many $n$,
where $\pi$ is a prime element of $\Order_{K}$.
It belongs to the subring $\alg{O}_{K}(k')$ if and only if $a_{n} = 0$ for all $n < 0$.
The sheaf $\alg{O}_{K}$ has a subsheaf $\alg{O}_{K} \tensor_{\Order_{K}} \ideal{p}_{K}^{n}$
of ideals for each $n \ge 0$,
which we denote by $\alg{p}_{K}^{n}$.
The ring $\alg{O}_{K}(k')$ for any $k' \in k^{\ind\rat}$ is $\alg{p}_{K}(k')$-adically complete Hausdorff.
The sheaves of invertible elements of $\alg{O}_{K}$ and $\alg{K}$ are denoted by
$\alg{U}_{K}$ and $\alg{K}^{\times}$, respectively.
For $X \in \Set(k^{\ind\rat}_{\et})$ and a prime (hence maximal) ideal $\ideal{m}$ of $k' \in k^{\ind\rat}$,
the reduction map $X(k') \to X(k' / \ideal{m})$ is denoted by $x \mapsto x(\ideal{m})$.
Its kernel, when $X \in \Ab(k^{\ind\rat}_{\et})$,
is denoted by $X(\ideal{m})$.
Here are some properties of the rings $\alg{O}_{K}(k')$ and $\alg{K}(k')$.

\begin{Prop} \label{prop: basic properties of local fields with ind-rational base}
	Let $k' \in k^{\ind\rat}$.
	\begin{enumerate}
		\item \label{ass: maximal ideals of local fields with ind-rational base}
			For any $\ideal{m} \in \Spec k'$,
			the ideal $\alg{K}(\ideal{m})$ (resp.\ $\alg{p}_{K}(k') + \alg{O}_{K}(\ideal{m})$) is
			a maximal ideal of $\alg{K}(k')$ (resp.\ $\alg{O}_{K}(k')$).
			The map $\ideal{m} \mapsto \alg{K}(\ideal{m})$
			(resp.\ $\ideal{m} \mapsto \alg{p}_{K}(k') + \alg{O}_{K}(\ideal{m})$) gives a homeomorphism
			from $\Spec k'$ onto the maximal spectrum of $\alg{K}(k')$ (resp.\ $\alg{O}_{K}(k')$).
		\item \label{ass: Zariski topology of local fields with ind-rational base}
			A neighborhood base of the maximal ideal
			$\alg{K}(\ideal{m})$ of $\alg{K}(k')$
			(resp.\ $\alg{p}_{K}(k') + \alg{O}_{K}(\ideal{m})$ of $\alg{O}_{K}(k')$) is
			given by the family of open and closed sets
				\[
						\Spec \alg{K}(k')[1 / \omega(e)]
					=
						\Spec \alg{K}(k'[1 / e])
				\]
			(resp.
				\[
						\Spec \alg{O}_{K}(k')[1 / \omega(e)]
					=
						\Spec \alg{O}_{K}(k'[1 / e])
					)
				\]
			for idempotents $e \in k' \setminus \ideal{m}$.
			In particular, any Zariski covering of
			$\Spec \alg{K}(k')$ (resp.\ $\Spec \alg{O}_{K}(k')$)
			can be refined by a disjoint Zariski covering.
		\item \label{ass: invertibility of series in ind-rational coefficients}
			Let $\pi$ be a prime element of $\Order_{K}$.
			An element $\sum_{n \ge 0} \omega(a_{n}) \pi^{n}$ of $\alg{O}_{K}(k')$
			(with $a_{n} \in k'$) is invertible if and only if $a_{0} \in k'^{\times}$.
			An element $\sum_{n \in \Z} \omega(a_{n}) \pi^{n}$ of $\alg{K}(k')$
			(with $a_{n} = 0$ for sufficiently small $n < 0$) is invertible if and only if
			some finitely many of the $a_{n}$ generate the unit ideal of $k'$.
	\end{enumerate}
\end{Prop}

\begin{proof}
	\eqref{ass: maximal ideals of local fields with ind-rational base}
	and \eqref{ass: Zariski topology of local fields with ind-rational base}
	for $\alg{K}(k')$ were proved in \cite[Lem.\ 2.5.1-2]{Suz13}.
	(To see that the map $\ideal{m} \mapsto \alg{K}(\ideal{m})$ is
	not only a bijection but also a homeomorphism,
	use \eqref{ass: Zariski topology of local fields with ind-rational base}.)
	The same proof works for $\alg{O}_{K}(k')$.
	
	\eqref{ass: invertibility of series in ind-rational coefficients}
	The statement for $\alg{O}_{K}(k')$ is obvious.
	If $\sum \omega(a_{n}) \pi^{n} \in \alg{K}(k')$ is invertible,
	then so is its image in the complete discrete valuation field $\alg{K}(k' / \ideal{m})$
	for any $\ideal{m} \in \Spec k'$,
	hence $a_{n} \notin \ideal{m}$ for some $n$.
	Therefore the $a_{n}$ generate the unit ideal of $k'$.
	Some finite linear combination of them gives $1 \in k'$.
	
	For the converse, denote by $D(\omega(a))$ and $V(\omega(a))$ for $a \in k'$
	the open set $\Spec \alg{K}(k')[1 / \omega(a)]$
	and its complementary closed set $\Spec \alg{K}(k' / (a))$ of $\Spec \alg{K}(k')$, respectively.
	They are both open and closed since $k' \in k^{\ind\rat}$ and hence
	$a$ is a unit times an idempotent.
	For any $m \in \Z$, on the open and closed set
	$U_{m} = \bigcap_{l < m} V(\omega(a_{l})) \cap D(\omega(a_{m}))$,
	the element
		\begin{equation} \label{eq: inverse formula in ind-rational coefficients}
				\sum \omega(a_{n}) \pi^{n}
			\equiv
				\sum_{n \ge m} \omega(a_{n}) \pi^{n}
			=
				 \pi^{m} \sum_{n \ge 0} \omega(a_{n + m}) \pi^{n}
		\end{equation}
	becomes a unit.
	Some finitely many of these disjoint open and closed sets $\{U_{m}\}$ cover $\Spec \alg{K}(k')$ by assumption.
	Hence $\sum \omega(a_{n}) \pi^{n}$ is a unit in $\alg{K}(k')$.
\end{proof}

Some properties of schemes over the rings $\alg{O}_{K}(k'), \alg{K}(k')$:

\begin{Prop} \label{prop: schemes over local fields with ind-rational base}
	Let $k' \in k^{\ind\rat}$.
	\begin{enumerate}
		\item	\label{ass: integral locus is closed, most likely open}
			If $X$ is an affine $\alg{O}_{K}(k')$-scheme
			and $x \in X(\alg{K}(k'))$,
			then the subset
				\[
						\{
							\ideal{m} \in \Spec k'
					\mid
							x(\ideal{m}) \in X(\alg{O}_{K}(k' / \ideal{m}))
						\}
				\]
			of $\Spec k'$ is closed.
			It is open if $X$ is of finite type.
		\item	\label{ass: point-wise integral implies integral}
			Let $X$ be a separated $\alg{O}_{K}(k')$-scheme locally of finite type
			and $x \in X(\alg{K}(k'))$.
			If $x(\ideal{m}) \in X(\alg{O}_{K}(k' / \ideal{m}))$ for all $\ideal{m} \in \Spec k'$,
			then $x \in X(\alg{O}_{K}(k'))$.
		\item	\label{ass: valuative criterion of properness}
			Let $X$ be either a proper $\alg{O}_{K}(k')$-scheme
			or written as $\mathcal{G} \times_{\Order_{K}} \alg{O}_{K}(k')$
			with $\mathcal{G}$ the N\'eron model of a smooth group scheme over $K$.
			Then we have $X(\alg{O}_{K}(k')) = X(\alg{K}(k'))$.
		\item	\label{ass: integral points are inverse limits}
			If $X$ is any $\alg{O}_{K}(k')$-scheme,
			then
				\[
						X(\alg{O}_{K}(k'))
					=
						\invlim_{n}
							X(\alg{O}_{K} / \alg{p}_{K}^{n}(k')).
				\]
	\end{enumerate}
\end{Prop}

\begin{proof}
	\eqref{ass: integral locus is closed, most likely open}
	Write $X = \Spec S$.
	Let $\varphi \colon S \to \alg{K}(k')$ be the $\alg{O}_{K}(k')$-algebra homomorphism
	corresponding to $x \in X(\alg{K}(k'))$.
	For $f \in S$, write $\varphi(f) = \sum \omega(a_{n, f}) \pi^{n}$ with $a_{n, f} \in k'$,
	where $\pi$ is a prime element of $\Order_{K}$.
	Then a maximal ideal $\ideal{m} \in \Spec k'$ is in the given subset if and only if
	the composite $S \to \alg{K}(k') \onto \alg{K}(k' / \ideal{m})$ factors through
	$\alg{O}_{K}(k' / \ideal{m})$.
	This is equivalent that $a_{n, f} \in \ideal{m}$ for all $n < 0$ and $f \in S$.
	Therefore the given subset is the closed subset defined by the ideal generated by the elements
	$a_{n, f} \in k'$ for all $n < 0$ and $f \in S$.
	We may consider only those $f$'s in any fixed set of generators of the $\alg{O}_{K}(k')$-algebra $S$.
	Hence the ideal is finitely generated if
	$S$ is finitely generated.
	A finitely generated ideal of $k' \in k^{\ind\rat}$ is generated by a single idempotent.
	Hence the given set is open in this case.
	
	\eqref{ass: point-wise integral implies integral}
	Note that the separatedness of $X$ and the injectivity of
	$\alg{O}_{K}(k') \into \alg{K}(k')$ implies that
	the natural map $X(\alg{O}_{K}(k')) \to X(\alg{K}(k'))$ is injective.
	We first treat the case where $X$ is affine.
	Using the notation in the proof of the previous assertion,
	the assumption implies that $a_{n, f} = 0$ for all $n < 0$ and $f \in S$.
	Hence $\varphi$ factors through $\alg{O}_{K}(k')$
	and therefore $x \in X(\alg{O}_{K}(k'))$ in this case.
	
	Next we treat the general case.
	Let $\{U_{\lambda}\}$ be an affine open cover of $X$.
	For any $\ideal{m} \in \Spec k'$,
	the sets $U_{\lambda}(\alg{O}_{K}(k' / \ideal{m}))$ indexed by $\lambda$
	cover $X(\alg{O}_{K}(k' / \ideal{m}))$
	since $\alg{O}_{K}(k' / \ideal{m})$ is local.
	Hence the assumption implies that the sets
		\[
				\{
					\ideal{m} \in \Spec k'
			\mid
					x(\ideal{m}) \in U_{\lambda}(\alg{O}_{K}(k' / \ideal{m}))
				\}
		\]
	indexed by $\lambda$ cover $\Spec k'$.
	For each $\lambda$, this set is open
	since the affine $\alg{O}_{K}(k')$-scheme $U_{\lambda}$ is of finite type
	and by \eqref{ass: integral locus is closed, most likely open}.
	Refine this open covering of $\Spec k'$ by a disjoint open covering
	$\{\Spec k'[1 / e_{\lambda}]\}$ with $e_{\lambda}$ idempotents.
	Note that all but finitely many $e_{\lambda}$ are zero.
	By considering $k'[1 / e_{\lambda}]$ for each $\lambda$ instead of $k'$,
	we are reduced to the affine case.
	
	\eqref{ass: valuative criterion of properness}
	By the previous assertion,
	it is reduced to showing
	$X(\alg{O}_{K}(k' / \ideal{m})) = X(\alg{K}(k' / \ideal{m}))$
	for all $\ideal{m} \in \Spec k'$.
	The proper case is the valuative criterion.
	For the N\'eron case,
	note that $k' / \ideal{m}$ is a separable (possibly transcendental) field extension of $k$
	in the sense of \cite[Chap.\ V, \S 15, no.\ 2--3]{Bou03}
	by [loc.\ cit., no.\ 5, Thm.\ 3 b)].
	Therefore \cite[10.1, Prop.\ 3]{BLR90} shows that
	$\mathcal{G} \times_{\Order_{K}} \alg{O}(k' / \ideal{m})$ is
	the N\'eron model, over the discrete valuation ring $\alg{O}(k' / \ideal{m})$,
	of its generic fiber.
	Hence the N\'eron mapping property implies that
	$X(\alg{O}_{K}(k' / \ideal{m})) = X(\alg{K}(k' / \ideal{m}))$.
	
	\eqref{ass: integral points are inverse limits}
	This is trivial if $X$ is affine.
	In general, let $\{U_{\lambda}\}$ be an affine open cover of $X$.
	Let $(x_{n})_{n}$ be an element of the limit in question.
	The pullback of $\{U_{\lambda}\}$ by the morphism $x_{1} \colon \Spec k' \to X$
	gives an open covering of $\Spec k'$.
	Refine it by a disjoint covering $\{\Spec k'[1 / e_{\lambda}]\}$
	with idempotents $e_{\lambda}$.
	Then $x_{1}$ restricted to $\Spec k'[1 / e_{\lambda}]$ factors through $U_{\lambda}$.
	Since the the surjection $\alg{O}_{K} / \alg{p}_{K}^{n}(k') \onto k'$ has a nilpotent kernel,
	we know that $x_{n}$ restricted to $\Spec \alg{O}_{K} / \alg{p}_{K}^{n}(k'[1 / e_{\lambda}])$
	factors through $U_{\lambda}$ for all $n$.
	The affine case then implies that $(x_{n})_{n}$ comes from $X(\alg{O}_{K}(k'))$.
\end{proof}


\numberwithin{equation}{subsection}
\subsection{Topology on rational points of varieties}
\label{sec: Topology on rational points of varieties}

We give a topology on the set $X(\alg{K}(k'))$
for any $\alg{K}(k')$-scheme $X$ locally of finite type with $k' \in k^{\ind\rat}$
and on the set $Y(\alg{O}_{K}(k'))$
for any $\alg{O}_{K}(k')$-scheme $Y$ locally of finite type.
We follow \cite[\S 2-3]{Con12}.

First for each $k' \in k^{\ind\rat}$,
the ring $\alg{O}_{K}(k')$ is a topological ring
by the ideals $\alg{p}_{K}^{n}(k')$, $n \ge 0$.
We give a topological ring structure on $\alg{K}(k')$ so that
the subring $\alg{O}_{K}(k')$ is open.
Recall from \cite[Prop.\ 2.1]{Con12} that
the set $X(\alg{K}(k'))$ for an affine $\alg{K}(k')$-scheme $X$ of finite type
has a canonical structure of a topological space.
Explicitly, choose an embedding $X \into \Affine^{n}$ for some $n$,
give the product topology on $\Affine^{n}(\alg{K}(k')) = \alg{K}(k')^{n}$
and give the subspace topology on $X(\alg{K}(k'))$.
This is independent of the choice,
and relies only on the fact that $\alg{K}(k')$ is a topological ring.
Similarly we have a canonical topological space structure on $Y(\alg{O}_{K}(k'))$
for an affine $\alg{O}_{K}(k')$-scheme $Y$ of finite type.

To proceed to the non-affine case, note that the subsets
$\alg{K}^{\times}(k') \subset \alg{K}(k')$ and $\alg{U}_{K}(k') \subset \alg{O}_{K}(k')$ are open,
and the inverse maps on them are continuous,
by \eqref{prop: basic properties of local fields with ind-rational base}
\eqref{ass: invertibility of series in ind-rational coefficients}
and \eqref{eq: inverse formula in ind-rational coefficients}.
Therefore if $X$ is an affine $\alg{K}(k')$-scheme and
$U$ is a basic open affine subset of $X$
(i.e., a localization by one element),
then $U(\alg{K}(k')) \subset X(\alg{K}(k'))$ is an open immersion of topological spaces
as explained in the proof of \cite[Prop.\ 3.1]{Con12}.
If we try to apply the whole cited proposition,
we will need the equality
	\[
			X(\alg{K}(k'))
		=
			\bigcup_{\lambda}
				U_{\lambda}(\alg{K}(k'))
	\]
for a $\alg{K}(k')$-scheme $X$ locally of finite type and
an arbitrary affine open cover $\{U_{\lambda}\}$ of $X$,
which requires that $\alg{K}(k')$ be local.
When $k' \in k^{\ind\rat}$ is a field,
then $\alg{K}(k')$ is a complete discrete valuation field,
so the equality above is true.
If $k' \in k^{\ind\rat}$ is a finite product of fields $\prod k'_{i}$,
then the above equality is not true in general,
but the equality $X(\alg{K}(k')) = \prod X(\alg{K}(k'_{i}))$
gives a product topology on $X(\alg{K}(k'))$.
For a general $k' = \bigcup k'_{\lambda} \in k^{\ind\rat}$ with $k'_{\lambda} \in k^{\rat}$,
it is not true that $\alg{K}(k') = \bigcup \alg{K}(k'_{\lambda})$,
and hence it is not immediately clear
how to use the topologies on the sets $U_{\lambda}(\alg{K}(k'))$
to topologize $X(\alg{K}(k'))$.
The situation is the same for $\alg{O}_{K}(k')$.
To topologize $X(\alg{K}(k'))$ for a general $k' \in k^{\ind\rat}$, we will use the following.

\begin{Prop} \label{prop: top space str from affine to general}
	Let $k' \in k^{\ind\rat}$.
	\begin{enumerate}
		\item \label{ass: K-points split into disjoint unions}
			Let $X$ be a $\alg{K}(k')$-scheme locally of finite type
			and $\{U_{\lambda}\}_{\lambda \in \Lambda}$ any affine open cover of $X$.
			Given a family $\{e_{\lambda}\}$ of disjoint idempotents of $k'$
			indexed by $\Lambda$ such that $\sum e_{\lambda} = 1$
			and a family of $K$-morphisms
			$\Spec \alg{K}(k'[1 / e_{\lambda}]) \to U_{\lambda}$ ($\into X$),
			a trivial patching gives a $K$-morphism $\Spec \alg{K}(k') \to X$.
			The map
				\[
						X(\alg{K}(k'))
					\gets
						\bigcup_{\{e_{\lambda}\}}
							\prod_{\lambda}
							U_{\lambda} \bigl(
								\alg{K}(k'[1 / e_{\lambda}])
							\bigr),
				\]
			thus obtained is bijective.
		\item
			Let $X$ be an affine $\alg{K}(k')$-scheme of finite type
			and $\{U_{i}\}$ any basic affine open finite cover of $X$.
			Let $\{e_{i}\}$ be a family of disjoint idempotents of $k'$
			such that $\sum e_{i} = 1$.
			Then the bijection and the natural map
				\[
						X(\alg{K}(k'))
					=
						\prod X \bigl(
							\alg{K}(k'[1 / e_{i}])
						\bigr)
					\gets
						\prod U_{i} \bigl(
							\alg{K}(k'[1 / e_{i}])
						\bigr)
				\]
			is a homeomorphism and an open immersion, respectively.
	\end{enumerate}
	There are statements for $\alg{O}_{K}(k')$ instead of $\alg{K}(k')$ in the obvious manner.
\end{Prop}

\begin{proof}
	The first assertion follows from
	\eqref{prop: basic properties of local fields with ind-rational base}
	\eqref{ass: Zariski topology of local fields with ind-rational base}.
	To check that the bijection in the second assertion is a homeomorphism,
	we can use the fact that
	$\alg{K}(k') = \prod \alg{K}(k'[1 / e_{i}])$
	is a homeomorphism.
	The final map is an open immersion as seen before.
	The statements for $\alg{O}_{K}(k')$ can be proven similarly.
\end{proof}

Given this, we introduce a topology on $X(\alg{K}(k'))$ by declaring that
a subset of $X(\alg{K}(k'))$ is open if
its intersection with the product topological space
	$
		\prod_{\lambda}
		U_{\lambda} \bigl(
			\alg{K}(k'[1 / e_{\lambda}])
		\bigr)
	$
is open for any $\{e_{\lambda}\}$.
This is independent of the choice of the affine open covering $\{U_{\lambda}\}$
as in \cite[Prop.\ 3.1]{Con12}.
Similarly we have a topology on $Y(\alg{O}_{K}(k'))$
for an $\alg{O}_{K}(k')$-scheme $Y$ locally of finite type.
Some properties of the topology given in \cite{Con12} also hold for our topology.
We list them.
A $\alg{K}(k')$-morphism $X_{1} \to X_{2}$ induces
a continuous map $X_{1}(\alg{K}(k')) \to X_{2}(\alg{K}(k'))$
for any $k' \in k^{\ind\rat}$.
We have homeomorphisms $\Affine^{n}(\alg{K}(k')) = \alg{K}(k')^{n}$
and
	$
			(X_{2} \times_{X_{1}} X_{3})(\alg{K}(k'))
		=
			X_{2}(\alg{K}(k')) \times_{X_{1}(\alg{K}(k'))} X_{3}(\alg{K}(k'))
	$.
A closed (resp.\ open) immersion $X_{1} \into X_{2}$ corresponds to
a closed (resp.\ open) immersion $X_{1}(\alg{K}(k')) \into X_{2}(\alg{K}(k'))$.
If $X$ is separated, then $X(\alg{K}(k'))$ is Hausdorff.
Similar for $\alg{O}_{K}(k')$.

We need the following three topological propositions.

\begin{Prop} \label{prop: integral points are open}
	Let $k' \in k^{\ind\rat}$.
	Let $X$ be a separated $\alg{O}_{K}(k')$-scheme locally of finite type.
	Then $X(\alg{O}_{K}(k'))$ is an open subset of $X(\alg{K}(k'))$.
\end{Prop}

\begin{proof}
	We may assume that $X$ is affine.
	Let $X \into \Affine^{n}$ be a closed immersion.
	Then $X(\alg{O}_{K}(k')) = X(\alg{K}(k')) \cap \Affine^{n}(\alg{O}_{K}(k'))$
	in $\Affine^{n}(\alg{K}(k'))$.
	Since $\alg{O}_{K}(k')$ is open in $\alg{K}(k')$ by definition,
	it follows that $X(\alg{O}_{K}(k'))$ is open in $X(\alg{K}(k'))$.
\end{proof}

Before the next proposition,
recall from \S \ref{sec: Serre duality and P-acyclicity} that
	\[
			\alg{K}^{\fp}(k')
		=
			\dirlim \alg{K}(k'_{\lambda}),
		\quad
			\alg{O}_{K}^{\fp}(k')
		=
			\dirlim \alg{O}_{K}(k'_{\lambda})
	\]
for $k' = \bigcup k'_{\lambda} \in k^{\ind\rat}$ with $k'_{\lambda} \in k^{\rat}$.
The first (resp.\ second) ring is a filtered union
of finite products of complete discrete valuation fields (resp.\ rings).
When $k' = \closure{k}$ (and hence the $k'_{\lambda}$ are finite extensions of $k$),
the $K$-algebra $\alg{K}^{\fp}(k')$ is the maximal unramified extension
$K^{\ur}$ of $K$, whose completion is $\Hat{K}^{\ur} = \alg{K}(\closure{k})$.%
\footnote{
	Do not confuse $\alg{O}_{K}^{\fp}(k')$ with
	the uncompleted tensor product $W(k') \tensor_{W(k)} \Order_{K}$.
	If $K$ has mixed characteristic (and hence $\Order_{K}$ is finite free over $W(k)$) and $k' = \closure{k}$,
	then the former is $\Order_{K}^{\ur}$, which is smaller than the latter $\Hat{\Order}_{K}^{\ur}$.
	If $K$ has equal characteristic and $k'$ is (the perfection of) $k(x)$,
	then the former is $k(x)[[T]]$, which is bigger than latter $k(x) \tensor_{k} (k[[T]])$.
}

\begin{Prop} \label{prop: Greenberg approximation}
	Let $X$ be a $K$-scheme locally of finite type
	and $k' \in k^{\ind\rat}$.
	Then $X(\alg{K}^{\fp}(k'))$ is a dense subset of $X(\alg{K}(k'))$.
	There is a similar statement for $\Order_{K}$ and $\alg{O}_{K}$
	in place of $K$ and $\alg{K}$, respectively.
\end{Prop}

\begin{proof}
	The ring $\alg{K}(k')$ is faithfully flat over $\alg{K}^{\fp}(k')$
	and the map $X(\alg{K}^{\fp}(k')) \to X(\alg{K}(k'))$ is injective.
	We may assume that $X$ is affine and defined over $\Order_{K}$.
	The map $X(\alg{O}_{K}^{\fp}(k')) \to X(\alg{O}_{K}(k'))$ is injective
	by the same reason.
	By scaling, it is enough to show that
	$X(\alg{O}_{K}^{\fp}(k'))$ is dense in $X(\alg{O}_{K}(k'))$.
	
	Recall Greenberg's approximation theorem \cite[Cor.\ 1 to Thm.\ 1]{Gre66}:
	there are integers $N \ge 1$, $c \ge 1$, $s \ge 0$ such that
	for any $n \ge N$ and any $x \in X(\Order_{K} / \ideal{p}_{K}^{n})$,
	the image of $x$ in $X(\Order_{K} / \ideal{p}_{K}^{\lceil n / c \rceil - s})$
	lifts to $X(\Order_{K})$.
	The proof of this theorem works for the following slightly stronger statement:
	there are integers $N \ge 1$, $c \ge 1$, $s \ge 0$ such that
	for any perfect field $k''$ over $k$,
	any $n \ge N$ and
	any $x \in X(\alg{O}_{K}(k'') / \alg{p}_{K}^{n}(k''))$,
	the image of $x$ in
		$
			X \bigl(
				\alg{O}_{K}(k'') / \alg{p}_{K}^{\lceil n / c \rceil - s}(k'')
			\bigr)
		$
	lifts to $X(\alg{O}_{K}(k''))$.
	(Note that $\alg{O}_{K}(k'')$ is a complete discrete valuation ring
	in which a prime element of $\Order_{K}$ remains prime.)
	
	Now let $n \ge N$ and $x \in X(\alg{O}_{K}(k'))$.
	Write $k' = \bigcup k'_{\lambda}$ with $k'_{\lambda} \in k^{\rat}$.
	Since
		$
				(\alg{O}_{K} / \alg{p}_{K}^{n})^{\fp}(k')
			=
				\alg{O}_{K} / \alg{p}_{K}^{n}(k')
			=
				\alg{O}_{K}(k') / \alg{p}_{K}^{n}(k')
		$,
	the image of $x$ in $X(\alg{O}_{K} / \alg{p}_{K}^{n}(k'))$ belongs to
	$X(\alg{O}_{K} / \alg{p}_{K}^{n}(k'_{\lambda}))$ for some $\lambda$.
	Since $k'_{\lambda}$ is a finite product of perfect fields over $k$,
	the image of $x$ in
		\[
				X \bigl(
					\alg{O}_{K} / \alg{p}_{K}^{\lceil n / c \rceil - s}(k'_{\lambda})
				\bigr)
			\subset
				X \bigl(
					\alg{O}_{K} / \alg{p}_{K}^{\lceil n / c \rceil - s}(k')
				\bigr)
		\]
	lifts to
		\[
				X(\alg{O}_{K}(k'_{\lambda}))
			\subset
				X(\alg{O}_{K}^{\fp}(k')).
		\]
	This shows that $X(\alg{O}_{K}^{\fp}(k'))$ is dense in $X(\alg{O}_{K}(k'))$.
\end{proof}

\begin{Prop} \label{prop: smooth morphism induces open map on points}
	Let $k' \in k^{\ind\rat}$ and $X, Y$ $\alg{K}(k')$-schemes locally of finite presentation.
	If $f \colon Y \to X$ is a smooth $\alg{K}(k')$-morphism,
	then the image of $Y(\alg{K}(k'))$ under $f$ is an open subset of $X(\alg{K}(k'))$.
\end{Prop}

\begin{proof}
	This is obvious if $Y$ is an affine space over $X$.
	Hence we may assume that $f$ is \'etale.
	We may further assume that $X$ and $Y$ are affine.
	Let $y \in Y(\alg{K}(k'))$ and $x = f(y)$.
	We want to show that any elements $x'$ of $X(\alg{K}(k'))$
	sufficiently close to $x$ come from $Y(\alg{K}(k'))$.
	For this, we are going to apply Toug\'eron's lemma \cite[Lem.\ 5.10]{Art69}
	to solve the equation $f(y') = x'$ for an unknown $y' \in Y(\alg{K}(k'))$.
	This lemma is true for any Henselian pair,
	and the pair $(\alg{O}_{K}(k'), \alg{p}_{K}(k'))$ is complete and hence Henselian.
	To apply the lemma, we need to write down all the conditions
	as polynomial equations and approximate solutions.
	
	Let $f' \colon Y' \to X'$ be a morphism between affine $\alg{O}_{K}(k')$-schemes of finite presentation
	whose base change to $\alg{K}(k')$ is $f$.
	Embed $X'$ and $Y'$ to affine spaces
	$\Affine_{\alg{O}_{K}(k')}^{m}$ and $\Affine_{\alg{O}_{K}(k')}^{n}$ over $\alg{O}_{K}(k')$, respectively,
	and extend $f'$ to a morphism $\Affine_{\alg{O}_{K}(k')}^{n} \to \Affine_{\alg{O}_{K}(k')}^{m}$.
	By scaling, we may assume that $y \in Y'(\alg{O}_{K}(k'))$.
	
	Let $P = (P_{1}, \dots, P_{l})$ be a polynomial system defining $Y' \subset \Affine_{\alg{O}_{K}(k')}^{n}$.
	Let $\pi$ be a prime element of $\Order_{K}$.
	Since $f \colon Y \to X$ is \'etale,
	the differential module $\Omega_{Y' / X'}^{1}$ is killed by a power $\pi^{r}$ of $\pi$.
	We want to show that any element $x' \in X'(\alg{O}_{K}(k')) \subset \alg{O}_{K}(k')^{m}$
	with term-wise difference
		\begin{equation} \label{eq: elements within Tougeron lemma range}
				x' - x
			\in
				\alg{p}_{K}^{2 r n + 1}(k')^{m}
		\end{equation}
	comes from $Y'(\alg{O}_{K}(k'))$,
	where $\alg{p}_{K}^{2 r n + 1}(k')^{m} \subset \alg{O}_{K}(k')^{m}$ is
	the set-theoretic product of $m$ copies of
	the $(2 r n + 1)$-st power of the ideal $\alg{p}_{K}(k')$.
	Let $f'^{-1}(x')$ be the fiber of $f' \colon Y' \to X'$ over $x' \in X'(\alg{O}_{K}(k'))$.
	Then $\Omega_{f'^{-1}(x') / \alg{O}_{K}(k')}^{1}$ is killed by $\pi^{r}$.
	We have a commutative diagram with cartesian squares of finitely presented $\alg{O}_{K}(k')$-schemes
		\[
			\begin{CD}
					f'^{-1}(x')
				@>> \incl >
					Y'
				@>> \incl >
					\Affine_{\alg{O}_{K}(k')}^{n}
				@>> P >
					\Affine_{\alg{O}_{K}(k')}^{l}
				\\
				@VVV
				@VV f' V
				@VV f' V
				@.
				\\
					x'
				@> \incl >>
					X'
				@> \incl >>
					\Affine_{\alg{O}_{K}(k')}^{m}.
			\end{CD}
		\]
	
	Let $y' = (y'_{1}, \dots, y'_{n})$ be the coordinates of $\Affine_{\alg{O}_{K}(k')}^{n}$.
	We view $f' \colon \Affine_{\alg{O}_{K}(k')}^{n} \to \Affine_{\alg{O}_{K}(k')}^{m}$
	as a system of $m$ polynomials in the $n$ variables $(y'_{1}, \dots, y'_{n})$
	and $x'$ as an element of $\alg{O}_{K}(k')^{m}$ (or $m$ constant polynomials).
	Define
		\[
				S
			=
				\alg{O}_{K}(k')[y'] / (f'(y') - x', P(y')),
		\]
	where the ideal on the right is generated by
	all the $m + l$ polynomials in the polynomial system $(f'(y') - x', P(y'))$.
	We have $f'^{-1}(x') = \Spec S$.
	Let $J$ be the Jacobian matrix of the system $(f'(y') - x', P(y'))$,
	which has entries in $\alg{O}_{K}(k')[y']$.
	Then $\Omega_{f'^{-1}(x') / \alg{O}_{K}(k')}^{1}$ is given by
	the cokernel of the $S$-module homomorphism $S^{m + l} \to S^{n}$
	corresponding to $J$.
	Let $F$ be the ideal of $\alg{O}_{K}(k')[y']$
	generated by all minors of size $n$ of $J$.
	Then the image $\bar{F}$ of $F$ in $S$ is
	the zeroth Fitting ideal of $\Omega_{f'^{-1}(x') / \alg{O}_{K}(k')}^{1}$ \cite[XIX, \S 2]{Lan02}.
	Therefore $\bar{F}$ contains the $n$-th power of
	the annihilator of $\Omega_{f'^{-1}(x') / \alg{O}_{K}(k')}^{1}$
	by \cite[XIX, \S 2, Prop.\ 2.5]{Lan02}
	and hence contains $\pi^{r n}$.
	Hence
		\[
				\pi^{r n}
			\in
				F + (f'(y') - x', P(y'))
			\subset
				\alg{O}_{K}(k')[y'].
		\]
	Evaluate this at $y' = y$.
	Since $y \in Y'(\alg{O}_{K}(k'))$, we have $P(y) = 0$.
	Also $f'(y) = x$.
	Hence
		\[
				\pi^{r n}
			\in
				F(y) + (x - x')
			\subset
				\alg{O}_{K}(k'),
		\]
	where $F(y)$ is the image of $F$ under the evaluation map
	$\alg{O}_{K}(k')[y'] \onto \alg{O}_{K}(k')$ at $y$.
	With this and using \eqref{eq: elements within Tougeron lemma range},
	we know that $\pi^{r n} \in F(y)$.
	By \eqref{lem: one of the elements divides} below,
	we may assume that the value $\delta(y)$ of a single minor $\delta$ divides
	$\pi^{r n}$ in $\alg{O}_{K}(k')$,
	so $\alg{p}_{K}^{r n}(k') \subset \delta(y) \alg{O}_{K}(k')$.
	Hence
		\begin{equation} \label{eq: valuation of minor in Jacobian}
				\alg{p}_{K}^{2 r n + 1}(k')
			\subset
				\delta(y)^{2} \alg{p}_{K}(k').
		\end{equation}
	
	Now we have
		\[
				f'(y) - x'
			\in
				\delta(y)^{2} \alg{p}_{K}(k')^{m},
			\quad
				P(y) = 0
		\]
	by \eqref{eq: elements within Tougeron lemma range}
	and \eqref{eq: valuation of minor in Jacobian}.
	Hence the system $(f'(y') - x', P)$ has a root $y' \in \alg{O}_{K}(k')^{n}$
	by Toug\'eron's lemma \cite[Lem.\ 5.10]{Art69}.
	For this $y'$, we have $f'(y') = x'$ and $y' \in Y'(\alg{O}_{K}(k'))$.
\end{proof}

\begin{Lem} \label{lem: one of the elements divides}
	If the ideal of $\alg{O}_{K}(k')$ generated by $q$ elements
	$\delta_{1}, \dots, \delta_{q} \in \alg{O}_{K}(k')$
	contains $\pi^{s}$, then there exists a disjoint Zariski covering
		\[
			\Spec k' = \bigsqcup_{i = 1}^{q} \Spec k'[1 / e_{i}]
		\]
	with idempotents $e_{i}$ such that
	the image of $\delta_{i}$ in $\alg{O}_{K}(k'[1 / e_{i}])$ divides $\pi^{s}$ for each $i$.
\end{Lem}

\begin{proof}
	Let $v$ be the normalized valuation of $\Order_{K}$.
	For each $\ideal{m} \in \Spec k'$,
	the ideal of $\alg{O}_{K}(k' / \ideal{m})$
	generated by $\delta_{1}(\ideal{m}), \dots, \delta_{q}(\ideal{m})$
	contains $\pi^{s}$.
	Therefore the sets
		\[
				U_{i}
			=
				\{
					\ideal{m} \in \Spec k'
			\mid
					v(\delta_{i}(\ideal{m})) \le s
				\}
		\]
	for $i = 1, \dots q$ cover $\Spec k'$.
	For each $i$, if $\delta_{i} = \sum_{n} \omega(a_{i, n}) \pi^{n}$,
	then $U_{i}$ is the union of the open sets
	$\Spec k'[1 / a_{i, s}], \Spec k'[1 / a_{i, s - 1}], \ldots$.
	Hence $\{U_{i}\}$ is an open covering of $\Spec k'$.
	Refine it by a disjoint Zariski covering
	$\Spec k' = \sqcup_{i = 1}^{q} \Spec k'[1 / e_{i}]$
	with idempotents $e_{i}$.
	This choice does the job.
\end{proof}

Now we can prove the following proposition.
This is useful especially to prove that
the first cohomology of an abelian variety over $K$ is ind-algebraic
and all of the higher cohomology is zero as sheaves over $k$.
The proof below and the use of the above topological statements
are inspired by the proof of \cite[Prop.\ 3.5 (a) and Lem.\ 5.3]{Ces15}.

\begin{Prop} \label{prop: higher cohomology with smooth coefficients}
	Let $A$ be a smooth group scheme over $K$
	and $k' \in k^{\ind\rat}$.
	Then we have
		\[
				H^{n}(\alg{K}(k')_{\et}, A)
			=
				H^{n}(\alg{K}^{\fp}(k')_{\et}, A)
		\]
	for any $n \ge 1$.
	This group is torsion.
\end{Prop}

\begin{proof}
	Let $f \colon \Spec \alg{K}(k') \to \Spec \alg{K}^{\fp}(k')$ be the natural morphism.
	Then
		\[
				f_{\ast}
			\colon
				\Ab(\alg{K}(k')_{\et})
			\to
				\Ab(\alg{K}^{\fp}(k')_{\et})
		\]
	is exact by \cite[Lem.\ 2.5.5]{Suz13}, hence
	$H^{n}(\alg{K}(k')_{\et}, A) = H^{n}(\alg{K}^{\fp}(k')_{\et}, f_{\ast} A)$.
	Note that $\alg{K}^{\fp}(k')$ is a filtered union of finite products of complete discrete valuation fields.
	The \'etale cohomology of a field can be calculated by Galois cohomology or \v{C}ech cohomology.
	Hence the cohomology of $\alg{K}^{\fp}(k')$ can be calculated by \v{C}ech cohomology.
	The isomorphism to be proved is thus
		\[
				\dirlim_{L' / \alg{K}^{\fp}(k')}
					\check{H}^{n}(L' / \alg{K}^{\fp}(k'), f_{\ast} A)
			=
				\dirlim_{L' / \alg{K}^{\fp}(k')}
					\check{H}^{n}(L' / \alg{K}^{\fp}(k'), A),
		\]
	where $L'$ runs through faithfully flat \'etale $\alg{K}^{\fp}(k')$-algebras
	and $\check{H}$ denotes \v{C}ech cohomology.
	It is enough to prove the isomorphism for each $L'$ before passing to the limit:
		\[
				\check{H}^{n}(L' / \alg{K}^{\fp}(k'), f_{\ast} A)
			=
				\check{H}^{n}(L' / \alg{K}^{\fp}(k'), A).
		\]
	Write $k' = \bigcup k'_{\lambda}$ with $k'_{\lambda} \in k^{\rat}$.
	For some $\lambda$, such $L'$ can be written as $L \tensor_{\alg{K}(k'_{\lambda})} \alg{K}^{\fp}(k')$
	for some faithfully flat \'etale $\alg{K}(k'_{\lambda})$-algebra $L$.
	Let $g \colon \Spec \alg{K}(k') \to \Spec \alg{K}(k'_{\lambda})$ and
	$h \colon \Spec \alg{K}^{\fp}(k') \to \Spec \alg{K}(k'_{\lambda})$
	be the natural morphisms.
	Then the \v{C}ech complex of $L' / \alg{K}^{\fp}(k')$ with values in $f_{\ast} A$ (resp.\ $A$)
	is the \v{C}ech complex of $L / \alg{K}(k'_{\lambda})$ with values in $g_{\ast} A$ (resp.\ $h_{\ast} A$).
	Hence the isomorphism to be proven is
		\[
				\check{H}^{n}(L / \alg{K}(k'_{\lambda}), g_{\ast} A)
			=
				\check{H}^{n}(L / \alg{K}(k'_{\lambda}), h_{\ast} A).
		\]
	By replacing $k'_{\lambda}$ with $k$, it is enough to show that
		\[
				\check{H}^{n}(L / K, g_{\ast} A)
			=
				\check{H}^{n}(L / K, h_{\ast} A)
		\]
	(for any perfect field $k$),
	where $L$ is a finite Galois extension of $K$
	and $g \colon \Spec \alg{K}(k') \to \Spec K$, $h \colon \Spec \alg{K}^{\fp}(k') \to \Spec K$
	the natural morphisms.
	
	Let $C^{n}$ be the Weil restriction of $A$ from the $(n + 1)$-fold tensor product
	$L \tensor_{K} \cdots \tensor_{K} L$ to $K$,
	which is representable by a smooth $K$-scheme.
	The \v{C}ech complex of $L / K$ with coefficients in
	$g_{\ast} A$ (resp.\ $h_{\ast} A$)
	is the $\alg{K}(k')$-valued (resp.\ $\alg{K}^{\fp}(k')$-valued) points of the complex
	$\{C^{n}\}$ of group schemes over $K$ with the usual coboundary maps $\{d^{n}\}$.
	Let $Z^{n}$ be the kernel of $d^{n} \colon C^{n} \to C^{n + 1}$,
	which is a $K$-scheme locally of finite type.
	We know that $d^{n - 1} \colon C^{n - 1} \to Z^{n}$ is a smooth morphism
	as shown in the proof of \cite[III, Thm.\ 3.9]{Mil80}.
	Consider the commutative diagram
		\[
			\begin{CD}
					C^{n - 1}(\alg{K}^{\fp}(k'))
				@>> d^{n - 1} >
					Z^{n}(\alg{K}^{\fp}(k'))
				\\
				@VVV
				@VVV
				\\
					C^{n - 1}(\alg{K}(k'))
				@> d^{n - 1} >>
					Z^{n}(\alg{K}(k')).
			\end{CD}
		\]
	The vertical maps are injective with dense image by
	\eqref{prop: Greenberg approximation}.
	The horizontal ones have open image by
	\eqref{prop: smooth morphism induces open map on points}.
	Therefore the map
		$
				\check{H}^{n}(L / K, h_{\ast} A)
			\to
				\check{H}^{n}(L / K, g_{\ast} A)
		$
	induced on the cokernels of the horizontal maps is an isomorphism.
	
	Galois cohomology is torsion in positive degrees.
	Hence $\check{H}^{n}(L / K, h_{\ast} A)$ is torsion (killed by $[L : K]$),
	and $H^{n}(\alg{K}^{\fp}(k')_{\et}, A)$ is torsion.
\end{proof}

The quotient in the following proposition will appear in the next subsection
as the first cohomology of $\Order_{K}$
with support on the closed point.

\begin{Prop} \label{prop: K-points mod O-points commutes with direct limits}
	Let $A$ be a separated group scheme locally of finite type over $\Order_{K}$
	and $k' \in k^{\ind\rat}$.
	Then we have
		\[
				A(\alg{K}(k')) / A(\alg{O}_{K}(k'))
			=
				A(\alg{K}^{\fp}(k')) / A(\alg{O}_{K}^{\fp}(k')).
		\]
\end{Prop}

\begin{proof}
	Consider the commutative diagram
		\[
			\begin{CD}
					A(\alg{O}_{K}^{\fp}(k'))
				@>>>
					A(\alg{K}^{\fp}(k'))
				\\
				@VVV
				@VVV
				\\
					A(\alg{O}_{K}(k'))
				@>>>
					A(\alg{K}(k')).
			\end{CD}
		\]
	As in the proof of the previous proposition,
	the vertical maps are injective with dense image.
	The horizontal inclusions are open immersions
	by \eqref{prop: integral points are open}.
	Hence the map induced on the cokernels of the horizontal maps is an isomorphism.
\end{proof}

We will use the following to see that
the group $N(\Order_{K})$ for a finite flat group scheme $N$ over $\Order_{K}$
can be viewed as an \'etale group over $k$.
A direct algebraic proof is not difficult,
but a proof using topology is clearer.

\begin{Prop} \label{prop: Gamma of finite schemes over local fields is etale}
	Let $X$ be a locally quasi-finite separated $\Order_{K}$-scheme with locally finite part $X^{f}$.
	Let $K^{\ur}$ be a maximal unramified extension of $K$.
	Let $x$ and $x^{f}$ be the sets $X(K^{\ur})$ and $X^{f}(K^{\ur})$, respectively,
	viewed as \'etale schemes over $k$.
	Then $X(\alg{O}_{K}(k')) = x^{f}(k')$ and $X(\alg{K}(k')) = x(k')$
	for any $k' \in k^{\ind\rat}$.
\end{Prop}

\begin{proof}
	Note that $X$ is a disjoint union of the $\Spec$'s of
	finite local $\Order_{K}$-algebras and/or finite local $K$-algebras.
	Write $X = X^{f} \sqcup Y$,
	where $Y$ is a locally finite $K$-scheme.
	We first show that for any $k' \in k^{\ind\rat}$,
	any $\Order_{K}$-scheme morphism $\Spec \alg{O}_{K}(k') \to X$ factors through $X^{f}$.
	By \eqref{prop: top space str from affine to general}
	\eqref{ass: K-points split into disjoint unions},
	we can write $k' = k'_{1} \times k'_{2}$ with $k'_{1}, k'_{2} \in k^{\ind\rat}$ such that
	the morphism $\Spec \alg{O}_{K}(k') \to X$ is the disjoint union of
	morphisms $\Spec \alg{O}_{K}(k'_{1}) \to X^{f}$ and
	$\Spec \alg{O}_{K}(k'_{2}) \to Y$.
	Then $\alg{O}_{K}(k'_{2})$ has to be a $K$-algebra.
	Therefore if $k'_{2}$ were non-zero,
	then the discrete valuation ring $\alg{O}_{K}(k'_{2} / \ideal{m})$
	for $\ideal{m} \in \Spec k'_{2}$ would have to be a $K$-algebra, which is absurd.
	Hence $k'_{2} = 0$,
	and $\Spec \alg{O}_{K}(k') \to X$ factors through $X^{f}$.
	
	Hence $X(\alg{O}_{K}(k')) = X^{f}(\alg{O}_{K}(k'))$.
	This is further equal to $X^{f}(\alg{K}(k'))$
	by \eqref{prop: schemes over local fields with ind-rational base}
	\eqref{ass: valuative criterion of properness}.
	Hence the first statement $X(\alg{O}_{K}(k')) = X^{f}(\alg{K}(k')) = x^{f}(k')$
	is reduced to the second $X(\alg{K}(k')) = x(k')$.
	
	To prove $X(\alg{K}(k')) = x(k')$,
	we may assume that $X$ is connected
	by \eqref{prop: top space str from affine to general}
	\eqref{ass: K-points split into disjoint unions}.
	We may further assume that $X$ is reduced,
	so $X = \Spec L$ for some finite extension $L / K$.
	Let $k_{L}$ be the residue field of $L$.
	Then $x = \Spec k_{L}$ if $L / K$ is unramified
	and $x = \emptyset$ otherwise.
	We also have $\Hom_{K}(L, \alg{K}^{\fp}(k')) = \Hom_{k}(k_{L}, k')$
	if $L / K$ is unramified
	and $\Hom_{K}(L, \alg{K}^{\fp}(k')) = \Hom_{k}(0, k')$ otherwise.
	Hence $X(\alg{K}^{\fp}(k')) = x(k')$,
	which is finite discrete.
	Since $X$ is separated, the topological space $X(\alg{K}(k'))$ is Hausdorff.
	By \eqref{prop: Greenberg approximation},
	we know that $X(\alg{K}^{\fp}(k'))$ is dense in $X(\alg{K}(k'))$.
	Therefore
	$X(\alg{K}(k')) = X(\alg{K}^{\fp}(k')) = x(k')$.
\end{proof}


\numberwithin{equation}{subsection}
\subsection{The relative fppf site of a local field}
\label{sec: The relative fppf site of a local field}

We recall from \cite[\S 2.3-4]{Suz13} the relative fppf site of $K$ over $k$,
the fppf structure morphism
and the cup product formalism in this setting.
We provide more on cohomology of $\Order_{K}$ with or without support.
The site used in \cite[loc.\ cit.]{Suz13} was $\Spec k^{\ind\rat}_{\et}$.
We will apply pro-\'etale sheafification to obtain a similar formalism over $\Spec k^{\ind\rat}_{\pro\et}$.

Let $K / k^{\ind\rat}$ be the category of pairs $(S, k_{S})$,
where $k_{S} \in k^{\ind\rat}$ and $S$ a finitely presented $\alg{K}(k_{S})$-algebra.
A morphism $(S, k_{S}) \to (S', k_{S'})$ consists of a morphism $k_{S} \to k_{S'}$ in $k^{\ind\rat}$
and a ring homomorphism $S \to S'$ such that the diagram
	\[
		\begin{CD}
				\alg{K}(k_{S})
			@>>>
				\alg{K}(k_{S'})
			\\
			@VVV
			@VVV
			\\
				S
			@>>>
				S'
		\end{CD}
	\]
commutes.
The composite of two morphisms is defined in the obvious way.
We say that a morphism $(S, k_{S}) \to (S', k_{S'})$ is flat/\'etale if
$k_{S} \to k_{S'}$ is \'etale and $S \to S'$ is flat.
We endow the category $K / k^{\ind\rat}$ with the topology
where a covering of an object $(S, k_{S})$ is a finite family $\{(S_{i}, k_{i})\}$
of objects flat/\'etale over $(S, k_{S})$ such that
$\prod S_{i}$ is faithfully flat over $S$.
The resulting site is the relative fppf site
$\Spec K_{\fppf} / k^{\ind\rat}_{\et}$ of $K$ over $k$ \cite[Def.\ 2.3.2]{Suz13}.
The cohomology of an object $(S, k_{S})$ with coefficients in $A \in \Ab(K_{\fppf} / k^{\ind\rat}_{\et})$
is given by the fppf cohomology of $S$:
	\[
			R \Gamma((S, k_{S}), A)
		=
			R \Gamma(S_{\fppf}, f_{\ast} A),
	\]
where $f \colon \Spec K_{\fppf} / k^{\ind\rat}_{\et} \to \Spec S_{\fppf}$
is given by the functor sending a finitely presented $S$-algebra $S'$
to the object $(S', k_{S})$ (\cite[Prop.\ 2.3.4]{Suz13}).
We also define the category $\Order_{K} / k^{\ind\rat}$ and
the relative fppf site $\Spec \Order_{K, \fppf} / k^{\ind\rat}_{\et}$
of $\Order_{K}$ over $k$ in a similar way,
using $\alg{O}_{K}$ instead of $\alg{K}$.
Its cohomology theory is similarly described.

The functors
	\begin{gather*}
			k^{\ind\rat}
		\to
			\Order_{K} / k^{\ind\rat},
		\quad
			k'
		\mapsto
			(\alg{O}_{K}(k'), k'),
		\\
			\Order_{K} / k^{\ind\rat}
		\to
			K / k^{\ind\rat},
		\quad
			(S, k_{S})
		\mapsto
			(S \tensor_{\Order_{K}} K, k_{S})
	\end{gather*}
and their composite
	\[
			k^{\ind\rat}
		\to
			K / k^{\ind\rat},
		\quad
			k'
		\mapsto
			(\alg{K}(k'), k'),
	\]
define morphisms
	\[
			\pi_{K / k}
		\colon
			\Spec K_{\fppf} / k^{\ind\rat}_{\et}
		\stackrel{j}{\into}
			\Spec \Order_{K, \fppf} / k^{\ind\rat}_{\et}
		\xrightarrow{\pi_{\Order_{K} / k}}
			\Spec k^{\ind\rat}_{\et}
	\]
of sites (see \cite[Def.\ 2.4.2]{Suz13} for $\pi_{K / k}$).
We call $\pi_{K / k}$ (resp.\ $\pi_{\Order_{K} / k}$)
the fppf structure morphism of $K$ (resp.\ $\Order_{K}$) over $k$.
We denote
	\begin{gather*}
				\alg{\Gamma}(K, \;\cdot\;)
			=
				(\pi_{K / k})_{\ast},
			\quad
				\alg{H}^{n}(K, \;\cdot\;)
			=
				R^{n} (\pi_{K / k})_{\ast}
			\colon
				\Ab(K_{\fppf} / k^{\ind\rat}_{\et})
			\to
				\Ab(k^{\ind\rat}_{\et}),
		\\
				R \alg{\Gamma}(K, \;\cdot\;)
			=
				R (\pi_{K / k})_{\ast}
			\colon
				D(K_{\fppf} / k^{\ind\rat}_{\et})
			\to
				D(k^{\ind\rat}_{\et}),
	\end{gather*}
and similarly for $\Order_{K}$.
Again, see \cite[Chap.\ 18]{KS06} for the formalism of
pushforward, pullback, sheaf-Hom and tensor products
in unbounded derived categories and arbitrary morphisms of sites.%
\footnote{
	We work with unbounded derived categories
	when the general theory in \cite[Chap.\ 18]{KS06} allows us to do so
	without extra effort.
	This eliminates some unnecessary assumptions on cohomological or Tor dimensions
	traditionally needed for morphisms of functoriality of derived pushforward
	and derived cup products.
	However, we do not emphasize the use of unbounded derived categories.
	Most of the complexes we actually need turn out to be bounded
	after non-trivial (but essentially classical) calculations,
	as seen in the next subsection.
	Therefore, with some care, it is possible avoid unbounded derived categories
	in order to prove the main theorems of this paper.
}
For any $A \in \Ab(K_{\fppf} / k^{\ind\rat}_{\et})$ and $k' \in k^{\ind\rat}$,
we have
	\begin{equation} \label{eq: cohomology of RGamma}
			R \Gamma \bigl(
				k'_{\et},
				R \alg{\Gamma}(K, A)
			\bigr)
		=
			R \Gamma(\alg{K}(k')_{\fppf}, A),
	\end{equation}
where we view $A$ as an fppf sheaf on $\alg{K}(k')$
by identifying it with the functor that sends a $\alg{K}(k')$-algebra $S$ of finite presentation to
$A(S, k')$.
If $A$ is a group scheme locally of finite type over $K$,
then this identification is consistent with the identification $A(S, k') = A(S)$.
The sheaf $\alg{H}^{n}(K, A)$ for any $A \in \Ab(K_{\fppf} / k^{\ind\rat}_{\et})$
is the \'etale sheafification of the presheaf
	\begin{equation} \label{eq: cohomology as a presheaf over residue field}
			k'
		\mapsto
			H^{i}(\alg{K}(k')_{\fppf}, A),
	\end{equation}
A similar description exists over $\Order_{K}$.

The sheaf-Hom functor for $\Spec K_{\fppf} / k^{\ind\rat}_{\et}$ is simply denoted by $\sheafhom_{K}$.
A morphism of fppf sheaves over $K$ (in the usual sense) induces
a morphism of sheaves over $\Spec K_{\fppf} / k^{\ind\rat}_{\et}$.
The derived versions are similarly denoted: $\sheafext_{K}^{n}$ and $R \sheafhom_{K}$.
There are versions $\sheafhom_{\Order_{K}}$, $\sheafext_{\Order_{K}}^{n}$ and
$R \sheafhom_{\Order_{K}}$ over $\Spec \Order_{K, \fppf} / k^{\ind\rat}_{\et}$.

We define a left exact functor
	\begin{gather*}
				\alg{\Gamma}_{x}(\Order_{K}, \;\cdot\;)
			\colon
				\Ab(\Order_{K, \fppf} / k^{\ind\rat}_{\et})
			\to
				\Ab(k^{\ind\rat}_{\et}),
		\\
				A
			\mapsto
				\Ker \bigl(
						\alg{\Gamma}(\Order_{K}, A)
					\to
						\alg{\Gamma}(K, j^{\ast} A)
				\bigr).
	\end{gather*}
This is the version of $\alg{\Gamma}(\Order_{K}, \;\cdot\;)$
with support on the closed point $x = \Spec k \into \Spec \Order_{K}$.
Its derived functor is denoted by $R \alg{\Gamma}_{x}(\Order_{K}, \;\cdot\;)$
with cohomology $\alg{H}^{n}_{x}(\Order_{K}, \;\cdot\;)$.
As in \cite[\S 2.2]{Bes78}, \cite[III, \S 0, Cohomology with support on closed subscheme]{Mil06},
this fits in the following localization triangle:

\begin{Prop} \label{prop: localization sequence}
	We have a distinguished triangle
		\[
				R \alg{\Gamma}_{x}(\Order_{K}, A)
			\to
				R \alg{\Gamma}(\Order_{K}, A)
			\to
				R \alg{\Gamma}(K, A)
		\]
	in $D(k^{\ind\rat}_{\et})$ for any $A \in D(\Order_{K, \fppf} / k^{\ind\rat}_{\et})$.
	(We frequently omit $j^{\ast}$ from the notation.)
\end{Prop}

\begin{proof}
	The site $\Spec K_{\fppf} / k^{\ind\rat}_{\et}$ is
	the localization \cite[III, \S 5]{AGV72a} of the site $\Spec \Order_{K, \fppf} / k^{\ind\rat}_{\et}$
	at the object $(K, k)$.
	Hence the functor $j^{\ast}$ admits an exact left adjoint functor $j_{!}$
	by \cite[IV, Prop.\ 11.3.1]{AGV72a}.
	The sheaf $\Z_{K} := j_{!} \Z$ is represented by
	the usual \'etale $\Order_{K}$-scheme of extension by zero of $\Z$,
	i.e.\ $\Spec \Order_{K} \sqcup \bigsqcup_{n \in \Z \setminus \{0\}} \Spec K$.
	Recall from \cite[III, \S 0, Lem.\ 0.2]{Mil06} that there is an exact sequence
		\[
			0 \to \Z_{K} \to \Z \to \Z_{x} \to 0
		\]
	of \'etale group schemes over $\Order_{K}$,
	where $\Z_{x}$ is supported on the closed point $x$ with special fiber $\Z$,
	i.e.\ the non-separated scheme obtained by glueing infinitely many copies of $\Spec \Order_{K}$
	by the common open subscheme $\Spec K$.
	For any $A \in \Ab(\Order_{K, \fppf} / k^{\ind\rat}_{\et})$,
	the induced exact sequence
		\[
				0
			\to
				\alg{\Gamma} \bigl(
					\Order_{K},
					\sheafhom_{\Order_{K}}(\Z_{x}, A)
				\bigr)
			\to
				\alg{\Gamma} \bigl(
					\Order_{K},
					\sheafhom_{\Order_{K}}(\Z, A)
				\bigr)
			\to
				\alg{\Gamma} \bigl(
					\Order_{K},
					\sheafhom_{\Order_{K}}(\Z_{K}, A)
				\bigr)
		\]
	can be identified with
		\[
				0
			\to
				\alg{\Gamma}_{x}(\Order_{K}, A)
			\to
				\alg{\Gamma}(\Order_{K}, A)
			\to
				\alg{\Gamma}(K, j^{\ast} A).
		\]
	This extends to term-wise exact sequences of complexes
	when $A$ is a complex in $\Ab(\Order_{K, \fppf} / k^{\ind\rat})$.
	
	We want to pass to the derived category.
	Let $A \isomto I$ be a quasi-isomorphism to a K-injective complex
	(called a homotopically injective complex in \cite[Def.\ 14.1.4 (i)]{KS06};
	for the existence of such $I$, see \cite[Thm.\ 14.3.1 (ii)]{KS06}).
	The existence of the exact left adjoint $j_{!}$ of $j^{\ast}$
	implies that $j^{\ast}$ preserves K-injectives by the definition of K-injectivity.
	Since K-injectives calculate any right derived functors by \cite[Thm.\ 14.3.1 (vi)]{KS06},
	the terms in the sequence of complexes
		\[
				0
			\to
				\alg{\Gamma}_{x}(\Order_{K}, I)
			\to
				\alg{\Gamma}(\Order_{K}, I)
			\to
				\alg{\Gamma}(K, j^{\ast} I)
		\]
	represents
		\[
				R \alg{\Gamma}_{x}(\Order_{K}, A),
			\quad
				R \alg{\Gamma}(\Order_{K}, A),
			\quad
				R \alg{\Gamma}(K, j^{\ast} A),
		\]
	respectively.
	
	On the other hand,
	the object $R \sheafhom_{\Order_{K}}(B, A)$ for any complex $B$ can be represented by
	the total complex $\sheafhom_{\Order_{K}}(B, I)$ of the sheaf-Hom double complex
	by \cite[Prop.\ 18.4.5]{KS06}.
	This complex is K-limp in the sense of \cite[Def.\ 5.11 (a)]{Spa88}
	by \cite[Prop.\ 5.14 and \S 5.12]{Spa88}.%
	\footnote{
		The setting in \cite{Spa88} is over topological spaces,
		but can be generalized to sites in the style of \cite[Chap.\ 18]{KS06}.
		As a quick definition (cf.\ \cite[Cor.\ 5.17]{Spa88}),
		we say that a complex of sheaves of abelian groups $C$ on a site $S$ is K-limp
		if the natural morphism $\Gamma(X, C) \to R \Gamma(X, C)$ in $D(S)$
		is an isomorphism for any object $X$ of $S$,
		where $\Gamma(X, C)$ is the complex $C$ with $\Gamma(X, \;\cdot\;)$ applied term-wise.
		A bounded below complex of acyclic sheaves
		(see the paragraph before \eqref{prop: misc on ind-rational pro-etale topology})
		is K-limp.
	}
	Hence its derived pushforward $R \alg{\Gamma}(\Order_{K}, \;\cdot\;)$ can be calculated
	by applying $\alg{\Gamma}(\Order_{K}, \;\cdot\;)$ term-wise
	by \cite[Prop.\ 6.7 (a)]{Spa88} (generalized to sites).
	Therefore the sequence of complexes
		\[
				0
			\to
				\alg{\Gamma} \bigl(
					\Order_{K},
					\sheafhom_{\Order_{K}}(\Z_{x}, I)
				\bigr)
			\to
				\alg{\Gamma} \bigl(
					\Order_{K},
					\sheafhom_{\Order_{K}}(\Z, I)
				\bigr)
			\to
				\alg{\Gamma} \bigl(
					\Order_{K},
					\sheafhom_{\Order_{K}}(\Z_{K}, I)
				\bigr)
		\]
	represents the distinguished triangle
		\begin{align*}
					R \alg{\Gamma} \bigl(
						\Order_{K},
						R \sheafhom_{\Order_{K}}(\Z_{x}, A)
					\bigr)
			&	\to
					R \alg{\Gamma} \bigl(
						\Order_{K},
						R \sheafhom_{\Order_{K}}(\Z, A)
					\bigr)
			\\
			&	\to
					R \alg{\Gamma} \bigl(
						\Order_{K},
						R \sheafhom_{\Order_{K}}(\Z_{K}, A)
					\bigr)
		\end{align*}
	coming from $0 \to \Z_{K} \to \Z \to \Z_{x} \to 0$.
	
	The above identification and distinguished triangle yield the required distinguished triangle.
\end{proof}

Relative fppf sites have \'etale counterparts (\cite[Introduction]{Suz13}).
The category $K_{\et} / k^{\ind\rat}$ is the full subcategory of $K / k^{\ind\rat}$
consisting of objects $(L, k_{L})$ with $L$ \'etale over $\alg{K}(k_{L})$.
An \'etale morphism $(L, k_{L}) \to (L', k_{L'})$ is a morphism with
$k_{L} \to k_{L'}$ (hence also $L \to L'$ this case) is \'etale.
This defines \'etale coverings in the category $K_{\et} / k^{\ind\rat}$
and hence the relative \'etale site $\Spec K_{\et} / k^{\ind\rat}_{\et}$ of $K$ over $k$.
The relative \'etale site
$\Spec \Order_{K, \et} / k^{\ind\rat}_{\et}$ of $\Order_{K}$ over $k$
is defined similarly.
We have the \'etale versions
	\[
			\Spec K_{\et} / k^{\ind\rat}_{\et}
		\into
			\Spec \Order_{K, \et} / k^{\ind\rat}_{\et}
		\to
			\Spec k^{\ind\rat}_{\et}.
	\]
of the structure morphisms.
The notation for their pushforward functors are
$R \alg{\Gamma}(K_{\et}, \;\cdot\;)$ and $R \alg{\Gamma}(\Order_{K, \et}, \;\cdot\;)$.
We mostly consider fppf cohomology for $K$ and $\Order_{K}$,
so when there is no subscript, it means fppf cohomology.

Recall from \cite[Prop.\ 2.4.3]{Suz13} that we have a natural morphism
	\[
			R \alg{\Gamma} \bigl(
				K,
				R \sheafhom_{K}(A, B)
			\bigr)
		\to
			R \sheafhom_{k^{\ind\rat}_{\et}} \bigl(
				R \alg{\Gamma}(K, A), R \alg{\Gamma}(K, B)
			\bigr)
	\]
of functoriality of $R \alg{\Gamma}$ in $D(k^{\ind\rat}_{\et})$
for any $A, B \in D(K_{\fppf} / k^{\ind\rat}_{\et})$.
Here $R \sheafhom_{K}$ is the derived sheaf-Hom for the site $\Spec K_{\fppf} / k^{\ind\rat}_{\et}$.
As we saw in loc.cit.\ the morphism of functoriality is equivalent to the cup-product pairing
	\[
				R \alg{\Gamma}(K, A)
			\tensor_{k}^{L}
				R \alg{\Gamma}(K, C)
		\to
				R \alg{\Gamma}(K, A \tensor_{K}^{L} C)
	\]
(where $\tensor_{k}^{L}$ and $\tensor_{K}^{L}$ denote
the derived tensor products over
$\Spec k^{\ind\rat}_{\et}$ and $\Spec K_{\fppf} / k^{\ind\rat}_{\et}$, respectively)
by the derived tensor-hom adjunction \cite[Thm.\ 18.6.4 (vii)]{KS06}
via the change of variables
$R \sheafhom_{K}(A, B) \leadsto C$ and
$A \tensor_{K}^{L} C \leadsto B$.

We need a version for $R \alg{\Gamma}_{x}$.
The proof of its existence is slightly different from \cite[Prop.\ 2.4.3]{Suz13},
so we prove it here.

\begin{Prop}
	The functoriality of $R \alg{\Gamma}_{x}$ induces natural morphisms
		\begin{gather}
				\label{eq: functoriality of cohom with support}
					R \alg{\Gamma} \bigl(
						\Order_{K},
						R \sheafhom_{\Order_{K}}(A, B)
					\bigr)
				\to
					R \sheafhom_{k^{\ind\rat}_{\et}} \bigl(
						R \alg{\Gamma}_{x}(\Order_{K}, A),
						R \alg{\Gamma}_{x}(\Order_{K}, B)
					\bigr),
			\\
				\label{eq: functoriality of cohom with support, interchanged}
					R \alg{\Gamma}_{x} \bigl(
						\Order_{K},
						R \sheafhom_{\Order_{K}}(A, B)
					\bigr)
				\to
					R \sheafhom_{k^{\ind\rat}_{\et}} \bigl(
						R \alg{\Gamma}(\Order_{K}, A),
						R \alg{\Gamma}_{x}(\Order_{K}, B)
					\bigr),
		\end{gather}
	for any $A, B \in D(\Order_{K, \fppf} / k^{\ind\rat}_{\et})$,
	where $R \sheafhom_{\Order_{K}}$ denotes derived sheaf-Hom
	for the site $\Spec \Order_{K, \fppf} / k^{\ind\rat}_{\et}$.
\end{Prop}

\begin{proof}
	For \eqref{eq: functoriality of cohom with support},
	let $A \isomto I$ and $B \isomto J$ be quasi-isomorphisms
	to K-injective complexes.
	As we saw in the proof of \eqref{prop: localization sequence}, the object
		$
			R \alg{\Gamma} \bigl(
				\Order_{K},
				R \sheafhom_{\Order_{K}}(A, B)
			\bigr)
		$
	is represented by
		$
			\alg{\Gamma} \bigl(
				\Order_{K},
				\sheafhom_{\Order_{K}}(I, J)
			\bigr)
		$.
	The functoriality of $\alg{\Gamma}_{x}$ induces a natural morphism
		\[
					\alg{\Gamma} \bigl(
						\Order_{K},
						\sheafhom_{\Order_{K}}(I, J)
					\bigr)
				\to
					\sheafhom_{k^{\ind\rat}_{\et}} \bigl(
						\alg{\Gamma}_{x}(\Order_{K}, I),
						\alg{\Gamma}_{x}(\Order_{K}, J)
					\bigr)
		\]
	of double complexes in $\Ab(k^{\ind\rat}_{\et})$.
	In $D(k^{\ind\rat}_{\et})$, there is a natural morphism
	from the right-hand side to
		\[
					R \sheafhom_{k^{\ind\rat}_{\et}} \bigl(
						\alg{\Gamma}_{x}(\Order_{K}, I),
						\alg{\Gamma}_{x}(\Order_{K}, J)
					\bigr)
				=
					R \sheafhom_{k^{\ind\rat}_{\et}} \bigl(
						R \alg{\Gamma}_{x}(\Order_{K}, A),
						R \alg{\Gamma}_{x}(\Order_{K}, B)
					\bigr).
		\]
	By composing, we obtain the required morphism.
	
	For \eqref{eq: functoriality of cohom with support, interchanged}, we have
		\begin{align*}
					R \alg{\Gamma}(\Order_{K}, A)
			&	\to
					R \alg{\Gamma} \Bigl(
						\Order_{K},
						R \sheafhom_{k^{\ind\rat}_{\et}} \bigl(
							R \sheafhom_{k^{\ind\rat}_{\et}}(A, B),
							B
						\bigr)
					\Bigr)
			\\
			&	\to
					R \sheafhom_{k^{\ind\rat}_{\et}} \Bigl(
						R \alg{\Gamma}_{x} \bigl(
							\Order_{K},
							R \sheafhom_{k^{\ind\rat}_{\et}}(A, B)
						\bigr),
						R \alg{\Gamma}_{x}(\Order_{K}, B)
					\Bigr),
		\end{align*}
	where the first morphism is the natural evaluation morphism
	and the second morphism is \eqref{eq: functoriality of cohom with support}.
	In general, morphisms of the form
	$C \to R \sheafhom_{k^{\ind\rat}_{\et}}(D, E)$ and
	$D \to R \sheafhom_{k^{\ind\rat}_{\et}}(C, E)$
	are both equivalent to
	$C \tensor^{L} D \to E$
	by the derived tensor-hom adjunction \cite[Thm.\ 18.6.4 (vii)]{KS06}.
	Hence the above yields the desired morphism.
\end{proof}

These morphisms of functoriality are compatible in the following sense.

\begin{Prop} \label{prop: localization and functoriality}
	Let $A, B \in D(\Order_{K, \fppf} / k^{\ind\rat}_{\et})$.
	To simplify the notation,
	we denote
		\[
					[\;\cdot\;, \;\cdot\;]_{\Order_{K}}
				=
					R \sheafhom_{\Order_{K}},
			\quad
					[\;\cdot\;, \;\cdot\;]_{K}
				=
					R \sheafhom_{K},
			\quad
					[\;\cdot\;, \;\cdot\;]_{k}
				=
					R \sheafhom_{k^{\ind\rat}_{\et}},
		\]
		\[
					R \alg{\Gamma}_{x}
				=
					R \alg{\Gamma}_{x}(\Order_{K}, \;\cdot\;),
			\quad
					R \alg{\Gamma}_{\Order_{K}}
				=
					R \alg{\Gamma}(\Order_{K}, \;\cdot\;),
			\quad
					R \alg{\Gamma}_{K}
				=
					R \alg{\Gamma}(K, \;\cdot\;).
		\]
	Then we have a morphism of distinguished triangles
		\[
			\begin{CD}
					R \alg{\Gamma}_{x} [A, B]_{\Order_{K}}
				@>>>
					R \alg{\Gamma}_{\Order_{K}} [A, B]_{\Order_{K}}
				@>>>
					R \alg{\Gamma}_{K} [A, B]_{K}
				\\
				@VVV
				@VVV
				@VVV
				\\
					[R \alg{\Gamma}_{\Order_{K}} A,
					R \alg{\Gamma}_{x} B]_{k}
				@>>>
					[R \alg{\Gamma}_{x} A,
					R \alg{\Gamma}_{x} B]_{k}
				@>>>
					[R \alg{\Gamma}_{K} A,
					R \alg{\Gamma}_{x} B]_{k}[1]
			\end{CD}
		\]
	in $D(k^{\ind\rat}_{\et})$,
	where the horizontal triangles are the localization triangles
	in \eqref{prop: localization sequence},
	the left two vertical morphisms are
	the morphisms of functoriality of $R \alg{\Gamma}_{x}$,
	and the right vertical morphism is the morphisms of functoriality of $R \alg{\Gamma}_{K}$
	composed with the connecting morphism
	$R \alg{\Gamma}_{K} B \to R \alg{\Gamma}_{x} B[1]$
	of the localization triangle.
\end{Prop}

\begin{proof}
	For complexes $C, D$ in $\Ab(\Order_{K, \fppf} / k^{\ind\rat}_{\et})$,
	we denote the total complex of the sheaf-Hom double complex $\sheafhom_{\Order_{K}}(C, D)$ by
	$[C, D]_{\Order_{K}}^{c}$.
	Note that there is a natural morphism
	$[C, D]_{\Order_{K}}^{c} \to [C, D]_{\Order_{K}}$
	in $D(k^{\ind\rat}_{\et})$.
	We use the notation $[\;\cdot\;, \;\cdot\;]_{K}^{c}$, $[\;\cdot\;, \;\cdot\;]_{k}^{c}$ similarly.
	We denote by $\alg{\Gamma}_{\Order_{K}} C$ the complex $\alg{\Gamma}(\Order_{K}, C)$,
	where $\alg{\Gamma}(\Order_{K}, \;\cdot\;)$ is applied term-wise.
	We use the notation $\alg{\Gamma}_{K}$ similarly.
	We denote the mapping fiber of the morphism
	$\alg{\Gamma}_{\Order_{K}} C \to \alg{\Gamma}_{K} C$ of complexes in $\Ab(k^{\ind\rat}_{\et})$
	by $\alg{\Gamma}_{x}^{c} C$.
	
	Let $A \isomto I$ and $B \isomto J$ be K-injective replacements.
	Then $[A, B]_{\Order_{K}}$ and $[A, B]_{K}$ can be represented by
	$[I, J]_{\Order_{K}}^{c}$ and $[I, J]_{K}^{c}$, respectively.
	Hence $R \alg{\Gamma}_{\Order_{K}} [A, B]_{\Order_{K}}$,
	$R \alg{\Gamma}_{K} [A, B]_{K}$ and therefore
	$R \alg{\Gamma}_{x} [A, B]_{\Order_{K}}$ can be represented by
	$\alg{\Gamma}_{\Order_{K}} [I, J]_{\Order_{K}}^{c}$,
	$\alg{\Gamma}_{K} [I, J]_{K}^{c}$ and
	$\alg{\Gamma}_{x}^{c} [I, J]_{\Order_{K}}^{c}$, respectively.
	Hence, if we show that the diagram
		\[
			\begin{CD}
					\alg{\Gamma}_{x}^{c} [I, J]_{\Order_{K}}^{c}
				@>>>
					\alg{\Gamma}_{\Order_{K}} [I, J]_{\Order_{K}}^{c}
				@>>>
					\alg{\Gamma}_{K} [I, J]_{K}^{c}
				@>>>
					\alg{\Gamma}_{x}^{c} [I, J]_{\Order_{K}}^{c}[1]
				\\
				@VVV
				@VVV
				@VVV
				@VVV
				\\
					[\alg{\Gamma}_{\Order_{K}} I,
					\alg{\Gamma}_{x}^{c} J]_{k}^{c}
				@>>>
					[\alg{\Gamma}_{x}^{c} I,
					\alg{\Gamma}_{x}^{c} J]_{k}^{c}
				@>>>
					[\alg{\Gamma}_{K} I,
					\alg{\Gamma}_{x}^{c} J]_{k}^{c}[1]
				@>>>
					[\alg{\Gamma}_{\Order_{K}} I,
					\alg{\Gamma}_{x}^{c} J]_{k}^{c}[1]
			\end{CD}
		\]
	of complexes in $\Ab(k^{\ind\rat}_{\et})$ is commutative up to homotopy,
	then we get the result by passing to the derived category and using the morphism 
	$[\;\cdot\;, \;\cdot\;]_{k}^{c} \to [\;\cdot\;, \;\cdot\;]_{k}$.
	On each square, it is routine to check that the square is commutative,
	or to construct a homotopy up to which the square is commutative.
\end{proof}

Now we pass to the pro-\'etale topology.
Let $P \colon \Spec k^{\ind\rat}_{\pro\et} \to \Spec k^{\ind\rat}_{\et}$
be the morphism defined by the identity.
The pullback $P^{\ast}$ is the pro-\'etale sheafification functor.
We denote the composite functor $P^{\ast} R \alg{\Gamma}(K, \;\cdot\;)$
by $R \Tilde{\alg{\Gamma}}(K, \;\cdot\;)$.%
\footnote{
	We cannot define this functor as a certain pushforward functor without sheafification.
	The functor $k' \mapsto (\alg{K}(k'), k')$ does not define a continuous map
	$\Spec K_{\fppf} / k^{\ind\rat}_{\et} \to \Spec k^{\ind\rat}_{\pro\et}$,
	and it is not clear whether there is a nice definition of
	the fppf site of $K$ relative to $\Spec k^{\ind\rat}_{\pro\et}$.
	For a faithfully flat ind-\'etale morphism $k' \to k''$ in $k^{\ind\rat}$,
	the corresponding morphism $\alg{O}_{K}(k') \to \alg{O}_{K}(k'')$ is not finitely presented,
	not even ind-\'etale,
	unless $k' \to k''$ is \'etale.
}
If $A$ is an object of $\Ab(K_{\fppf} / k^{\ind\rat}_{\et})$
such that $R \alg{\Gamma}(K, A)$ is P-acyclic, then we have
	\[
			R \Gamma(k'_{\pro\et}, R \Tilde{\alg{\Gamma}}(K, A))
		=
			R \Gamma(\alg{K}(k'), A)
	\]
for any $k' \in k^{\ind\rat}$,
the sheaf $\Tilde{\alg{H}}^{n}(K, A)$ is the \'etale (not pro-\'etale) sheafification of the presheaf
	\[
			k'
		\mapsto
			H^{n}(\alg{K}(k'), A),
	\]
and we have
	\[
			\Gamma(k', \Tilde{\alg{H}}^{n}(K, A))
		=
			H^{n}(\alg{K}(k'), A)
	\]
for any algebraically closed $k' \in k^{\ind\rat}$.
With this in mind, if $R \alg{\Gamma}(K, A)$ for $A \in D^{+}(K_{\fppf} / k^{\ind\rat}_{\et})$ is P-acyclic,
then we simply write
$R \Tilde{\alg{\Gamma}}(K, A) = R \alg{\Gamma}(K, A)$ following the convention made in
\S \ref{sec: Serre duality and P-acyclicity}.
Similar notation and convention will be applied to
$R \alg{\Gamma}(\Order_{K}, \;\cdot\;)$ and
$R \alg{\Gamma}_{x}(\Order_{K}, \;\cdot\;)$.

\begin{Prop} \label{prop: localization and functoriality in pro-etale}
	There is a localization distinguished triangle
		\[
				R \Tilde{\alg{\Gamma}}_{x}(\Order_{K}, A)
			\to
				R \Tilde{\alg{\Gamma}}(\Order_{K}, A)
			\to
				R \Tilde{\alg{\Gamma}}(K, A)
		\]
	in $D(k^{\ind\rat}_{\pro\et})$
	for $A \in D(\Order_{K, \fppf} / k^{\ind\rat}_{\et})$.
	The functoriality of $R \Tilde{\alg{\Gamma}}$ and $R \Tilde{\alg{\Gamma}}_{x}$ induce morphisms
		\[
				R \Tilde{\alg{\Gamma}} \bigl(
					K,
					R \sheafhom_{K}(A, B)
				\bigr)
			\to
				R \sheafhom_{k^{\ind\rat}_{\pro\et}} \bigl(
					R \Tilde{\alg{\Gamma}}(K, A), R \Tilde{\alg{\Gamma}}(K, B)
				\bigr)
		\]
	in $D(k^{\ind\rat}_{\pro\et})$ for $A, B \in D(K_{\fppf} / k^{\ind\rat}_{\et})$ and
		\begin{gather*}
					R \Tilde{\alg{\Gamma}} \bigl(
						\Order_{K},
						R \sheafhom_{\Order_{K}}(A, B)
					\bigr)
				\to
					R \sheafhom_{k^{\ind\rat}_{\pro\et}} \bigl(
						R \Tilde{\alg{\Gamma}}_{x}(\Order_{K}, A),
						R \Tilde{\alg{\Gamma}}_{x}(\Order_{K}, B)
					\bigr),
			\\
					R \Tilde{\alg{\Gamma}}_{x} \bigl(
						\Order_{K},
						R \sheafhom_{\Order_{K}}(A, B)
					\bigr)
				\to
					R \sheafhom_{k^{\ind\rat}_{\pro\et}} \bigl(
						R \Tilde{\alg{\Gamma}}(\Order_{K}, A),
						R \Tilde{\alg{\Gamma}}_{x}(\Order_{K}, B)
					\bigr),
		\end{gather*}
	in $D(k^{\ind\rat}_{\pro\et})$ for $A, B \in D(\Order_{K, \fppf} / k^{\ind\rat}_{\et})$.
	With a similar set of notation to
	\eqref{prop: localization and functoriality},
	we have a morphism of distinguished triangles
		\[
			\begin{CD}
					R \Tilde{\alg{\Gamma}}_{x} [A, B]_{\Order_{K}}
				@>>>
					R \Tilde{\alg{\Gamma}}_{\Order_{K}} [A, B]_{\Order_{K}}
				@>>>
					R \Tilde{\alg{\Gamma}}_{K} [A, B]_{K}
				\\
				@VVV
				@VVV
				@VVV
				\\
					[R \Tilde{\alg{\Gamma}}_{\Order_{K}} A,
					R \Tilde{\alg{\Gamma}}_{x} B]_{k}
				@>>>
					[R \Tilde{\alg{\Gamma}}_{x} A,
					R \Tilde{\alg{\Gamma}}_{x} B]_{k}
				@>>>
					[R \Tilde{\alg{\Gamma}}_{K} A,
					R \Tilde{\alg{\Gamma}}_{x} B]_{k}[1]
			\end{CD}
		\]
	in $D(k^{\ind\rat}_{\pro\et})$
	for $A, B \in D(\Order_{K, \fppf} / k^{\ind\rat}_{\et})$
	(with $[\;\cdot\;, \;\cdot\;]_{k}$ this times being
	$R \sheafhom_{k^{\ind\rat}_{\pro\et}}$),
	where the horizontal triangles are localization triangles
	and the vertical morphisms are the functoriality morphisms
	together with the connecting morphism
	$R \Tilde{\alg{\Gamma}}_{K} B \to R \Tilde{\alg{\Gamma}}_{x} B[1]$
	on the right lower term.
\end{Prop}

\begin{proof}
	We only prove the existence of the functoriality morphism of $R \Tilde{\alg{\Gamma}}$.
	The others are treated similarly.
	By adjunction, we have
		\begin{align*}
					R \alg{\Gamma} \bigl(
						K,
						R \sheafhom_{K}(A, B)
					\bigr)
			&	\to
					R \sheafhom_{k^{\ind\rat}_{\et}} \bigl(
						R \alg{\Gamma}(K, A),
						R \alg{\Gamma}(K, B)
					\bigr)
			\\
			&	\to
					R \sheafhom_{k^{\ind\rat}_{\et}} \bigl(
						R \alg{\Gamma}(K, A),
						R P_{\ast} P^{\ast} R \alg{\Gamma}(K, B)
					\bigr)
			\\
			&	=
					R P_{\ast}
					R \sheafhom_{k^{\ind\rat}_{\pro\et}} \bigl(
						P^{\ast} R \alg{\Gamma}(K, A),
						P^{\ast} R \alg{\Gamma}(K, B)
					\bigr),
		\end{align*}
	in $D(k^{\ind\rat}_{\et})$.
	Adjoining again gives the result.
\end{proof}


\numberwithin{equation}{subsection}
\subsection{Cohomology as sheaves on the residue field}
\label{sec: Cohomology as sheaves on the residue field}
We compute $R \alg{\Gamma}$ of several group schemes over $K$ and $\Order_{K}$.
In most cases below, we obtain P-acyclic ind-proalgebraic groups.
In some cases, the groups are in $\Pro \Alg_{\uc} / k$ or $\Ind \Alg_{\uc} / k$ (having unipotent connected part),
so that they are Serre reflexive.
In this subsection, all sheaves over $k$,
their exact sequences and distinguished triangles
are considered in $\Ab(k^{\ind\rat}_{\et})$, $D(k^{\ind\rat}_{\et})$
unless the pro-\'etale topology is explicitly mentioned.
As soon as objects are proved to be P-acyclic,
we can regard them as objects on $\Spec k^{\ind\rat}_{\pro\et}$
without losing any information
and apply the results of the previous section.

\begin{Prop} \label{prop: integral etale cohomology vanishes}
	The pushforward functor $\Ab(\Order_{K, \et} / k^{\ind\rat}_{\et}) \to \Ab(k^{\ind\rat}_{\et})$
	for the morphism
		$		\Spec \Order_{K, \et} / k^{\ind\rat}_{\et}
			\to
				\Spec k^{\ind\rat}_{\et}
		$
	is exact.
	Hence $\alg{H}^{n}(\Order_{K, \et}, \;\cdot\;) = 0$ as functors for all $n \ge 1$.
\end{Prop}

\begin{proof}
	We need to show that an \'etale covering $(S, k_{S})$ of an object of the form $(\alg{O}_{K}(k'), k')$
	with $k' \in k^{\ind\rat}$ can be refined by a covering of the form
	$(\alg{O}_{K}(k''), k'')$ with $k''$ faithfully flat \'etale over $k'$.
	By the $\alg{O}_{K}(k')$-version of \cite[Lem.\ 2.5.3]{Suz13},
	the \'etale covering $S$ of $\alg{O}_{K}(k')$ can be refined by a covering of the form
	$\Order'' \tensor_{\alg{O}_{K}^{\fp}(k')} \alg{O}_{K}(k')$,
	where $\Order''$ is a faithfully flat \'etale $\alg{O}_{K}^{\fp}(k')$-algebra.
	We can write $\Order'' = \Order_{0}'' \tensor_{\alg{O}_{K}(k'_{0})} \alg{O}_{K}^{\fp}(k')$
	with $k'_{0}$ a rational $k$-subalgebra of $k'$ and
	$\Order_{0}''$ a faithfully flat \'etale $\alg{O}_{K}(k'_{0})$-algebra.
	Since $\alg{O}_{K}(k'_{0})$ is a finite product of complete discrete valuation rings
	with residue ring $k'_{0}$,
	we may assume (by refining) that $\Order_{0}''$ is finite over $\alg{O}_{K}(k'_{0})$
	and hence written as $\alg{O}_{K}(k''_{0})$ with $k''_{0}$ faithfully flat \'etale over $k'_{0}$.
	Then we have $\Order'' = \alg{O}_{K}^{\fp}(k'')$
	with $k'' = k''_{0} \tensor_{k'_{0}} k'$ a faithfully flat \'etale $k'$-algebra.
	Hence we have a refinement $(S, k_{S}) \to (\alg{O}_{K}(k''), k_{S})$
	of coverings of $(\alg{O}_{K}(k'), k')$.
	The $k'$-algebra homomorphism $k_{S} \to k''$ is \'etale
	since both $k_{S}$ and $k''$ are \'etale over $k'$.
	Therefore the morphism $(\alg{O}_{K}(k''), k_{S}) \to (\alg{O}_{K}(k''), k'')$ is an \'etale covering,
	even though $k_{S} \to k''$ is not necessarily faithfully flat
	(\cite[Prop.\ 2.3.3]{Suz13}).
	The composite $(S, k_{S}) \to (\alg{O}_{K}(k''), k_{S}) \to (\alg{O}_{K}(k''), k'')$
	gives a desired refinement.
\end{proof}

\begin{Prop} \BetweenThmAndList \label{prop: cohomology of integers}
	\begin{enumerate}
		\item \label{ass: cohomology of smooth group over integers}
			Let $A$ be a smooth group scheme over $\Order_{K}$
			and $A_{x}$ its special fiber.
			Let $\alg{\Gamma}(\ideal{p}_{K}, A)$ be the kernel of the reduction morphism
			$\alg{\Gamma}(\Order_{K}, A) \to A_{x}$.
			Then the sequence
				\[
						0
					\to
						\alg{\Gamma}(\ideal{p}_{K}, A)
					\to
						\alg{\Gamma}(\Order_{K}, A)
					\to
						A_{x}
					\to
						0
				\]
			in $\Ab(k^{\ind\rat}_{\et})$ is exact.
			All the terms are P-acyclic.
			The group $\alg{\Gamma}(\ideal{p}_{K}, A)$ is connected pro-unipotent.
			In particular, we have $\pi_{0}(\alg{\Gamma}(\Order_{K}, A)) = \pi_{0}(A_{x})$,
			and if $A_{x}$ is of finite type,
			then $\alg{\Gamma}(\Order_{K}, A) \in \Pro \Alg / k$.
			We have $\alg{H}^{n}(\Order_{K}, A) = 0$ for all $n \ge 1$.
			In particular, $R \alg{\Gamma}(\Order_{K}, A)$ is P-acyclic.
		\item \label{ass: cohomology of finite flats over integers}
			Let $N$ be a finite flat group scheme over $\Order_{K}$.
			Then $\alg{\Gamma}(\Order_{K}, N) = \alg{\Gamma}(K, N) \in \FEt / k$,
			$\alg{H}^{1}(\Order_{K}, N)$ is connected pro-unipotent,
			and $\alg{H}^{n}(\Order_{K}, N) = 0$ for all $n \ge 2$.
			The complex $R \alg{\Gamma}(\Order_{K}, N)$ is P-acyclic and Serre reflexive.
	\end{enumerate}
\end{Prop}

\begin{proof}
	\eqref{ass: cohomology of smooth group over integers}.
	For each $m \ge 1$,
	the functor
		$
				\alg{\Gamma}(\Order_{K} / \ideal{p}_{K}^{m}, A)
			\colon
				k' \mapsto
			A(\alg{O}_{K} / \alg{p}_{K}^{m}(k'))
		$
	is represented by the perfection of the Greenberg transform of $A$ of level $m$ (\cite{Gre61}).
	\eqref{prop: schemes over local fields with ind-rational base}
	\eqref{ass: integral points are inverse limits}
	implies that
		$
				\alg{\Gamma}(\Order_{K}, A)
			=
				\invlim_{m}
					\alg{\Gamma}(\Order_{K} / \ideal{p}_{K}^{m}, A).
		$
	The reduction map $A(\alg{O}_{K}(k')) \to A_{x}(k')$ is surjective
	for any $k' \in k^{\ind\rat}$ by smoothness.
	Hence the reduction morphism $\alg{\Gamma}(\Order_{K}, A) \to A_{x}$ is surjective.
	The kernel of the surjection
		$
				\alg{\Gamma}(\Order_{K} / \ideal{p}_{K}^{m + 1}, A)
			\onto
				\alg{\Gamma}(\Order_{K} / \ideal{p}_{K}^{m}, A)
		$
	is the perfection of a vector group by \cite[Lem.\ 4.1.1, 2]{Beg81}%
	\footnote{
		Note that this lemma by B\'egueri is true only after perfection.
		See \cite[Rmk.\ 14.22, 15.9]{BGA13}.
		In our case, it is enough to check this lemma for perfect-field-valued points only
		by the following reason.
		The sheaf
		$\alg{\Gamma}(\Order_{K} / \ideal{p}_{K}^{m}, A)$
		on $\Spec k^{\ind\rat}_{\et}$ is locally of finite presentation
		and the kernel of the surjection mentioned here is also locally of finite presentation.
		Hence it is enough to treat them as sheaves on $\Spec k^{\rat}_{\et}$.
		For perfect-field-valued points, the lemma is classical.
	}
	and the proof of \cite[\S 1.1, Lem.\ 1.1 (ii)]{Bes78}.
	Therefore
		\[
				\alg{\Gamma}(\ideal{p}_{K}, A)
			=
				\invlim_{m}
					\Ker \bigl(
							\alg{\Gamma}(\Order_{K} / \ideal{p}_{K}^{m}, A)
						\to
							A_{x}
					\bigr)
		\]
	is connected pro-unipotent and P-acyclic by \eqref{prop: criterion of P-acyclicity}.
	The quasi-algebraic group $A_{x}$ is P-acyclic by the same proposition.
	Being an extension of P-acyclics, the group $\alg{\Gamma}(\Order_{K}, A)$ is P-acyclic.
	We have
		$
				\alg{H}^{n}(\Order_{K}, A)
			=
				\alg{H}^{n}(\Order_{K, \et}, A)
		$
	for all $n$ since $A$ is smooth and
	the fppf cohomology with smooth group scheme coefficients agrees with the \'etale cohomology
	by \cite[III, Rmk.\ 3.11 (b)]{Mil80}.
	We have $\alg{H}^{n}(\Order_{K, \et}, A) = 0$ for $n \ge 1$
	by the previous proposition.
	
	\eqref{ass: cohomology of finite flats over integers}
	The functors $\alg{\Gamma}(\Order_{K}, N) \colon k' \mapsto N(\alg{O}_{K}(k'))$
	and $\alg{\Gamma}(K, N) \colon k' \mapsto N(\alg{K}(k'))$
	for finite flat $N$ are represented by the same finite \'etale $k$-scheme by
	\eqref{prop: Gamma of finite schemes over local fields is etale}.
	
	For cohomology of degree $\ge 1$, recall from \cite[Prop.\ 2.2.1]{Beg81} that
	there is an exact sequence
	$0 \to N \to G \to H \to 0$
	of group schemes over $\Order_{K}$ with $G, H$ smooth affine with connected fibers
	(more specifically, $G$ is the Weil restriction of $\Gm$
	from the Cartier dual $N^{\CDual}$ to $\Spec \Order_{K}$
	and $N \into G$ is the natural embedding).
	The long exact sequence shows that $\alg{H}^{n}(\Order_{K}, N) = 0$ for $n \ge 2$
	and yields an exact sequence
		\[
				0
			\to
				\alg{\Gamma}(\Order_{K}, N)
			\to
				\alg{\Gamma}(\Order_{K}, G)
			\to
				\alg{\Gamma}(\Order_{K}, H)
			\to
				\alg{H}^{1}(\Order_{K}, N)
			\to
				0
		\]
	(in $\Ab(k^{\ind\rat}_{\et})$ a priori).
	The first three terms are P-acyclic.
	Hence so is the fourth.
	Therefore the sequence is exact also in $\Ab(k^{\ind\rat}_{\pro\et})$.
	Since the first term is finite and
	the second and third terms are connected proalgebraic,
	the fourth term is connected proalgebraic.
	Since $\alg{H}^{1}(\Order_{K}, N)$ is killed by the order of $N$,
	it does not have semi-abelian part.
	Hence $\alg{H}^{1}(\Order_{K}, N)$ is pro-unipotent.
	Therefore $R \alg{\Gamma}(\Order_{K}, N)$ is Serre reflexive
	by \eqref{prop: Serre duality}
	\eqref{ass: Serre duality for unip and etale}.
\end{proof}

\begin{Prop} \BetweenThmAndList \label{prop: cohomology of local fields}
	\begin{enumerate}
		\item \label{ass: higher cohomology of smooth group over local field is P-acyclic}
			Let $A$ be a smooth group scheme over $K$.
			Then $\alg{H}^{n}(K, A)$ for any $n \ge 1$ is torsion and locally of finite presentation
			as a functor on $k^{\ind\rat}$.
			In particular, it is P-acyclic.
		\item \label{ass: cohomology of finite flats over local fields}
			Let $N$ be a finite flat group scheme over $K$.
			Then $\alg{\Gamma}(K, N) \in \FEt / k$,
			$\alg{H}^{1}(K, N) \in \Ind \Pro \Alg_{\uc} / k$,
			$\alg{H}^{n}(K, N) = 0$ for all $n \ge 2$,
			and $R \alg{\Gamma}(K, N)$ is P-acyclic and Serre reflexive.
			The group $\alg{H}^{1}(K, N)$ is in $\Ind \Alg_{\uc} / k$ if $N$ is \'etale
			and in $\Pro \Alg_{\uc} / k$ if $N$ is multiplicative.
			If $K$ has mixed characteristic,
			then $\alg{H}^{1}(K, N) \in \Alg_{\uc} / k$.
		\item \label{ass: cohomology of finite free etale over local fields}
			Let $Y$ be a lattice over $K$
			(i.e., a finite free abelian group with a Galois action of $K$).
			Then $\alg{\Gamma}(K, Y)$ is a lattice over $k$,
			$\alg{H}^{1}(K, Y) \in \FEt / k$,
			$\alg{H}^{2}(K, Y) \in \Ind \Alg_{\uc} / k$,
			$\alg{H}^{n}(K, Y) = 0$ for $n \ge 3$,
			and $R \alg{\Gamma}(K, Y)$ is P-acyclic and Serre reflexive.
		\item \label{ass: cohomology of abelian varieties over local fields}
			Let $A$ be an abelian variety over $K$ with N\'eron model $\mathcal{A}$.
			Then $\alg{\Gamma}(K, A) = \alg{\Gamma}(\Order_{K}, \mathcal{A}) \in \Pro \Alg / k$,
			which is described by the previous proposition,
			$\alg{H}^{1}(K, A) \in \Ind \Alg_{\uc} / k$,
			$\alg{H}^{n}(K, A) = 0$ for all $n \ge 2$,
			and $R \alg{\Gamma}(K, A)$ is P-acyclic.
		\item \label{ass: cohomology of tori over local fields}
			Let $T$ be a torus over $K$ with N\'eron model $\mathcal{T}$.
			Then $\alg{\Gamma}(K, T) = \alg{\Gamma}(\Order_{K}, \mathcal{T})$
			is an extension of an \'etale group by a P-acyclic proalgebraic group
			described by the previous proposition.
			We have $\alg{H}^{n}(K, T) = 0$ for all $n \ge 1$,
			and $R \alg{\Gamma}(K, T)$ is P-acyclic.
	\end{enumerate}
\end{Prop}

\begin{proof}
	\eqref{ass: higher cohomology of smooth group over local field is P-acyclic}
	\eqref{prop: higher cohomology with smooth coefficients} shows that
	the functor $k' \mapsto H^{n}(\alg{K}(k')_{\et}, A)$ is locally of finite presentation for $n \ge 1$
	and takes values in torsion groups.
	Hence so is its \'etale sheafification $\alg{H}^{n}(K, A)$.
	Therefore $\alg{H}^{n}(K, A)$ is torsion and P-acyclic by
	\eqref{prop: criterion of P-acyclicity}
	\eqref{ass: local finite presentation implies P-acyclic}.
	
	\eqref{ass: cohomology of finite flats over local fields}
	By \eqref{prop: Gamma of finite schemes over local fields is etale},
	the sheaf $\alg{\Gamma}(K, N)$ is a finite \'etale group over $k$
	whose $\closure{k}$-points is given by $\Gamma(K^{\ur}, N)$.
	(Note that $N$ does not have to extend to $\Order_{K}$
	and hence we are not using \eqref{prop: cohomology of integers} here.)
	
	Suppose that $N$ has order prime to $p$.
	Then it is classical to see that $H^{n}(K^{\ur}, N)$ is zero for $n \ge 2$
	and $H^{1}(K^{\ur}, N)$ is finite.
	These groups do not change under residue field extensions,
	i.e.\ for any algebraically closed field $k'$ over $k$,
	we have $H^{n}(\alg{K}(k'), N) = H^{n}(K^{\ur}, N)$ for any $n \ge 1$.
	Since $N$ is \'etale,
	the sheaf $\alg{H}^{n}(K, N)$ is locally of finite presentation for any $n \ge 1$
	by the previous assertion.
	Hence it is determined by $k'$-points for various algebraically closed fields $k'$.
	Therefore $\alg{H}^{n}(K, N) = 0$ for $n \ge 2$ and
	$\alg{H}^{1}(K, N)$ is the finite \'etale group given by $H^{1}(K^{\ur}, N)$.
	
	Suppose that $N$ is $p$-primary.
	Suppose also that $K$ has equal characteristic.
	We first show the statements for $N = \Z / p \Z$, $\mu_{p}$ and $\alpha_{p}$.
	If $N = \Z / p \Z$, then $\alg{H}^{n}(K, \Z / p \Z) = 0$ for $n \ge 2$ and
		\[
				\alg{H}^{1}(K, \Z / p \Z)
			=
				\alg{K} / \wp(\alg{K})
			\cong
				\Ga^{\bigoplus \N}
			\in
				\Ind \Alg_{\uc} / k
		\]
	by the Artin-Schreier sequence
	$0 \to \Z / p \Z \to \Ga \stackrel{\wp}{\to} \Ga \to 0$
	and that $\alg{H}^{n}(K, \Ga) = 0$ for all $n \ge 1$.
	In particular, $\alg{H}^{1}(K, \Z / p \Z)$ is P-acyclic and Serre reflexive
	by \eqref{prop: Serre duality}
	\eqref{ass: Serre duality for unip and etale}
	and \eqref{prop: criterion of P-acyclicity}
	\eqref{ass: local finite presentation implies P-acyclic}.
	For $N = \mu_{p}$, the paragraph before \cite[Lem.\ 2.7.4]{Suz13}
	and the proof of the cited lemma show that
	$\alg{H}^{n}(K, \mu_{p}) = 0$ for $n \ge 2$ and
		\[
				\alg{H}^{1}(K, \mu_{p})
			=
				\alg{K}^{\times} / (\alg{K}^{\times})^{p}
			\cong
				\Z / p \Z \times \Ga^{\N}
			\in
				\Pro \Alg_{\uc} / k.
		\]
	In particular, $\alg{H}^{1}(K, \mu_{p})$ is P-acyclic
	by \eqref{prop: criterion of P-acyclicity}
	and Serre reflexive by 
	\eqref{prop: Serre duality}
	\eqref{ass: Serre duality for unip and etale}.
	For $N = \alpha_{p}$, a calculation similar to the case of $\Z / p \Z$ shows that
	$\alg{H}^{n}(K, \alpha_{p}) = 0$ for $n \ge 2$ and
		\[
				\alg{H}^{1}(K, \alpha_{p})
			=
				\alg{K} / \alg{K}^{p}
			\cong
				\Ga^{\bigoplus \N} \times \Ga^{\N}
			\in
				\Ind \Pro \Alg_{\uc} / k.
		\]
	Each factor is P-acyclic and Serre reflexive by the same propositions.
	Hence so is the product.
	
	Next, if $0 \to N_{1} \to N_{2} \to N_{3} \to 0$ is an exact sequence of
	finite flat group schemes over $K$,
	and if we know that $N_{1}$ and $N_{3}$ satisfy the statements,
	then the long exact sequence
		\begin{align*}
			&
					0
				\to
					\alg{\Gamma}(K, N_{1})
				\to
					\alg{\Gamma}(K, N_{2})
				\to
					\alg{\Gamma}(K, N_{3})
			\\
			&	\to
					\alg{H}^{1}(K, N_{1})
				\to
					\alg{H}^{1}(K, N_{2})
				\to
					\alg{H}^{1}(K, N_{3})
				\to
					0
		\end{align*}
	and the finiteness of $\alg{\Gamma}(K, N_{i})$ ($i = 1, 2, 3$) implies that
	$N_{3}$ satisfies the statements.
	Therefore if $N$ has a filtration whose successive subquotients are $\Z / p \Z$, $\mu_{p}$ or $\alpha_{p}$,
	then $N$ satisfies the statements.
	
	Now, let $N$ be any $p$-primary finite flat group scheme over $K$.
	Then $N$ has a filtration whose successive subquotients are
	\'etale, multiplicative or $\alpha_{p}$
	by \cite[IV, \S 3, 5.8, 5.9]{DG70b}.
	Let $L$ be a finite Galois extension of $K$
	that trivializes the Galois actions on the \'etale part
	and the Cartier dual of the multiplicative part.
	It is enough to show the statements after a finite unramified extension,
	so we may assume that $L / K$ is totally ramified.
	Let $M$ be the intermediate field of $L / K$
	that corresponds to a (or the) $p$-Sylow subgroup of $\Gal(L / K)$.
	If a $p$-group acts on a non-zero $\F_{p}$-vector space, then it has a non-zero fixed part
	by \cite[IX, Lem.\ 4]{Ser79}.
	Hence the base-changed group $N \times_{K} M$ over $M$ has a filtration
	whose successive subquotients are $\Z / p \Z$, $\mu_{p}$ or $\alpha_{p}$.
	Therefore $N \times_{K} M$ satisfies the statements,
	i.e.\ $\alg{H}^{1}(M, N) \in \Ind \Pro \Alg / k$
	($\Ind \Alg / k$ if $N$ is \'etale,
	$\Pro \Alg / k$ if multiplicative),
	$\alg{H}^{n}(M, N) = 0$ for $n \ge 2$ and
	$R \alg{\Gamma}(M, N)$ is P-acyclic and Serre reflexive.
	We have $R \alg{\Gamma}(M, N) = R \alg{\Gamma}(K, \Res_{M / K} N)$,
	where $\Res_{M / K}$ denotes the Weil restriction functor.
	The composite of the inclusion map $N \into \Res_{M / K} N$
	and the norm map $\Res_{M / K} N \to N$ is the multiplication by $[M : K]$,
	which is an isomorphism since $N$ is $p$-primary and $[M : K]$ is prime to $p$.
	Hence $N$ is a direct summand of $\Res_{M / K} N$.
	Therefore $N$ (over $K$) satisfies the statements.
	
	Suppose next that $K$ has mixed characteristic.
	If $N = \mu_{p}$, then
	$\alg{H}^{1}(K, \mu_{p}) = \alg{K}^{\times} / (\alg{K}^{\times})^{p}$ similarly.
	We have $\alg{K}^{\times} \cong \Z \times \alg{U}_{K}$
	by \cite[the paragraph before Prop.\ 2.4.4]{Suz13}.
	The group $\alg{U}_{K} = \alg{O}_{K}^{\times} = \alg{\Gamma}(\Order_{K}, \Gm)$
	is P-acyclic by \eqref{prop: cohomology of integers}
	\eqref{ass: cohomology of smooth group over integers}.
	Since $\alg{\Gamma}(K, \mu_{p}) = 0$ or $\Z / p \Z$,
	The exact sequence
		\[
				0
			\to
				\alg{\Gamma}(K, \mu_{p})
			\to
				\alg{K}^{\times}
			\stackrel{p}{\to}
				\alg{K}^{\times}
			\to
				\alg{K}^{\times} / (\alg{K}^{\times})^{p}
			\to
				0
		\]
	shows that $\alg{K}^{\times} / (\alg{K}^{\times})^{p}$
	and $(\alg{K}^{\times})^{p}$ are P-acyclic.
	The logarithm map shows that $(\alg{K}^{\times})^{p}$ contains
	the group $\alg{U}_{K}^{m} = 1 + \alg{p}_{K}^{m}$ of $m$-th principal units for some $m$.
	Since $\alg{U}_{K} / \alg{U}_{K}^{m}$ is an $m$-dimensional quasi-algebraic group,
	its quotient
		$
				\alg{U}_{K} / (\alg{U}_{K})^{p}
			=
				\alg{U}_{K}^{1} / (\alg{U}_{K}^{1})^{p}
		$
	is quasi-algebraic unipotent.
	Hence $\alg{H}^{1}(K, \mu_{p}) \in \Alg_{\uc} / k$.
	Then the same process as the equal characteristic case shows that
	$\alg{H}^{1}(K, N) \in \Alg_{\uc} / k$ for any finite flat (hence \'etale) $N$
	and $N$ satisfies the statements.
	
	\eqref{ass: cohomology of finite free etale over local fields}
	We have an exact sequence $0 \to Y \to Y \tensor \Q \to Y \tensor \Q / \Z \to 0$.
	We have $H^{n}(K^{\ur}, Y \tensor \Q) = 0$ for $n \ge 1$.
	From this, since $Y \tensor \Q$ is \'etale and hence smooth,
	we know that $\alg{H}^{n}(K, Y \tensor \Q) = 0$ for $n \ge 1$
	by the same method as the second paragraph of the proof of the previous assertion.
	Therefore
		\[
				\alg{H}^{n}(K, Y)
			=
				\alg{H}^{n - 1}(K, Y \tensor \Q / \Z)
			=
				\dirlim_{m}
				\alg{H}^{n - 1}(K, Y \tensor \Z / m \Z)
		\]
	for $n \ge 2$.
	This sheaf is zero for $n \ge 3$ and in $\Ind \Alg_{\uc} / k$ for $n = 2$
	by the previous assertion.
	
	By \eqref{prop: Gamma of finite schemes over local fields is etale},
	we know that $\alg{\Gamma}(K, Y)$ is a lattice over $k$.
	Let $I$ be the inertia group of a finite Galois extension of $K$
	over which $Y$ becomes trivial.
	The same proposition shows that the exact sequence
		\[
				0
			\to
				\alg{\Gamma}(K, Y)
			\to
				\alg{\Gamma}(K, Y \tensor \Q)
			\to
				\alg{\Gamma}(K, Y \tensor \Q / \Z)
			\to
				\alg{H}^{1}(K, Y)
			\to
				0
		\]
	is identified with the exact sequence
		\[
				0
			\to
				H^{0}(I, Y)
			\to
				H^{0}(I, Y \tensor \Q)
			\to
				H^{0}(I, Y \tensor \Q / \Z)
			\to
				H^{1}(I, Y)
			\to
				0
		\]
	of group cohomology groups.
	Since $I$ is finite and $Y$ finitely generated,
	we know that $H^{1}(I, Y)$ is finite
	\cite[VIII, \S 2, Cor.\ 2]{Ser79}.
	Hence $\alg{H}^{1}(K, Y)$ is finite \'etale.
	
	Therefore $R \alg{\Gamma}(K, Y)$ has Serre reflexive cohomologies
	by \eqref{prop: Serre duality}
	\eqref{ass: Serre duality for unip and etale}
	and hence itself is Serre reflexive.
	
	\eqref{ass: cohomology of abelian varieties over local fields}
	We have $\alg{\Gamma}(K, A) = \alg{\Gamma}(\Order_{K}, \mathcal{A})$
	by \eqref{prop: schemes over local fields with ind-rational base}
	\eqref{ass: valuative criterion of properness}.
	This is in $\Pro \Alg / k$ by the previous proposition.
	We know that $\alg{H}^{n}(K, A)$ is torsion for any $n \ge 1$
	by \eqref{ass: higher cohomology of smooth group over local field is P-acyclic}.
	For each $m \ge 1$, the exact sequence
	$0 \to A[m] \to A \to A \to 0$ yields an exact sequence
		\[
				0
			\to
				\alg{\Gamma}(K, A) / m \alg{\Gamma}(K, A)
			\to
				\alg{H}^{1}(K, A[m])
			\to
				\alg{H}^{1}(K, A)[m]
			\to
				0.
		\]
	The first term is P-acyclic since $\alg{\Gamma}(K, A)$ and its $m$-torsion part
	$\alg{\Gamma}(K, A[m])$ are so.
	It is in $\Pro \Alg_{\uc} / k$,
	since its semi-abelian part is divisible and killed by $m$, hence zero.
	The middle term is in $\Ind \Pro \Alg_{\uc} / k$ and P-acyclic by the finite flat case.
	Therefore the right term is in $\Ind \Pro \Alg_{\uc} / k$ and P-acyclic.
	So is the filtered union $\alg{H}^{1}(K, A)$.
	We know that $\alg{H}^{1}(K, A)$ is locally of finite presentation by
	\eqref{ass: higher cohomology of smooth group over local field is P-acyclic}.
	Hence \eqref{lem: criteiron when ind-proalgebraic is ind-algebraic} below implies that
	$\alg{H}^{1}(K, A) \in \Ind \Alg_{\uc} / k$.
	Similarly the vanishing $\alg{H}^{n}(K, A[m]) = 0$ of the finite flat case
	implies $\alg{H}^{n}(K, A) = 0$ for $n \ge 2$.
	
	\eqref{ass: cohomology of tori over local fields}
	For $\alg{\Gamma}$, it is similar to abelian varieties
	(though in this case $\mathcal{T}$ is only locally of finite type).
	To show $\alg{H}^{n}(K, T) = 0$ for $n \ge 1$,
	it is enough to see that
	$H^{n}(\alg{K}(k'), T) = 0$ for any algebraically closed field $k'$ over $k$
	by \eqref{ass: higher cohomology of smooth group over local field is P-acyclic}.
	This is given in \cite[X \S 7, Application]{Ser79}.
	
\end{proof}

\begin{Lem} \label{lem: criteiron when ind-proalgebraic is ind-algebraic}
	If $A \in \Ind \Pro \Alg / k$ is
	locally of finite presentation as a functor on $k^{\ind\rat}$,
	then we have $A \in \Ind \Alg / k$.
\end{Lem}

\begin{proof}
	Let $A = \dirlim A_{\lambda}$ with $A_{\lambda} \in \Pro \Alg / k$.
	We need to show that any morphism $B \to A$ from an object $B \in \Pro \Alg / k$
	factors through an object of $\Alg / k$.
	Let $B = \invlim B_{\mu}$ with $B_{\mu} \in \Alg / k$.
	We may assume that the transition morphisms $B_{\mu'} \to B_{\mu}$ are surjective.
	Let $\xi_{B}$ be the generic point of $B$
	in the sense of \cite[Def.\ 3.2.1]{Suz13},
	namely $\xi_{B} = \invlim \xi_{B_{\mu}}$
	and $\xi_{B_{\mu}}$ is the disjoint union of
	the generic points of the irreducible components of $B_{\mu}$.
	As a scheme, $\xi_{B}$ is the $\Spec$ of an ind-rational $k$-algebra $k'_{B}$.
	For any $C \in \Ab(k^{\ind\rat}_{\pro\et})$, we denote $C(\xi_{B}) = C(k'_{B})$.
	We define a group homomorphism $\sigma \colon C(\xi_{B}) \to C(\xi_{B \times_{k} B})$
	by sending a morphism $f \colon \xi_{B} \to C$
	to the morphism $\xi_{B \times B} \to C$ given by
	$(b_{1}, b_{2}) \mapsto f(b_{1}) + f(b_{2}) - f(b_{1} + b_{2})$.
	Then for $C \in \Pro \Alg / k$, we have
		\begin{equation} \label{eq: birational group morphisms extend}
				\Hom(B, C)
			=
				\Ker(C(\xi_{B}) \stackrel{\sigma}{\to} C(\xi_{B \times B})),
		\end{equation}
	which means that a homomorphism of birational groups extends to
	a everywhere regular group homomorphism
	(\cite[V, \S 1.5, Lem.\ 6]{Ser88} plus a limit argument).
	Now let $C = A$.
	Then using the assumption, we have
		\begin{align*}
					\Hom(B, A)
			&	=
					\dirlim_{\lambda}
						\Hom(B, A_{\lambda})
			\\
			&	=
					\dirlim_{\lambda}
						\Ker(A_{\lambda}(\xi_{B}) \to A_{\lambda}(\xi_{B \times B}))
			\\
			&	=
					\dirlim_{\lambda, \mu}
						\Ker(A_{\lambda}(\xi_{B_{\mu}}) \to A_{\lambda}(\xi_{B_{\mu} \times B_{\mu}}))
			\\
			&	=
					\dirlim_{\lambda, \mu}
						\Hom(B_{\mu}, A_{\lambda})
			\\
			&	=
					\dirlim_{\mu}
						\Hom(B_{\mu}, A).
		\end{align*}
	Hence a morphism $B \to A$ factors through some $B_{\mu} \in \Alg / k$.
\end{proof}

\begin{Prop} \BetweenThmAndList \label{prop: cohomology with support}
	Let $N$ be a finite flat group scheme over $\Order_{K}$.
	Then $\alg{H}_{x}^{n}(\Order_{K}, N) = 0$ for $n \ne 2$.
	We have $\alg{H}_{x}^{2}(\Order_{K}, N) \in \Ind \Alg_{\uc} / k$,
	and $R \alg{\Gamma}_{x}(\Order_{K}, N)$ is P-acyclic and Serre reflexive.
	If $K$ has mixed characteristic,
	then $\alg{H}_{x}^{2}(\Order_{K}, N) \in \Alg_{\uc} / k$.
\end{Prop}

\begin{proof}
	We have $\alg{\Gamma}(\Order_{K}, N) = \alg{\Gamma}(K, N)$
	and $\alg{H}^{n}(\Order_{K}, N) = \alg{H}^{n}(K, N) = 0$ for $n \ge 2$
	by the previous two propositions.
	We show that the morphism $\alg{H}^{1}(\Order_{K}, N) \to \alg{H}^{1}(K, N)$
	is injective.
	Let $k' \in k^{\ind\rat}$ and $X$ an fppf $N$-torsor over $\Spec \alg{O}_{K}(k')$.
	Since $N$ is finite, we have $X(\alg{O}_{K}(k')) = X(\alg{K}(k'))$
	by \eqref{prop: schemes over local fields with ind-rational base}
	\eqref{ass: valuative criterion of properness}.
	Hence if $X$ maps to zero under $H^{1}(\alg{O}_{K}(k'), N) \to H^{1}(\alg{K}(k'), N)$,
	then it is zero.
	Therefore $H^{1}(\alg{O}_{K}(k'), N) \to H^{1}(\alg{K}(k'), N)$ is injective
	and $\alg{H}^{1}(\Order_{K}, N) \to \alg{H}^{1}(K, N)$ is injective.
	Hence the localization triangle \eqref{prop: localization sequence}
		\[
				R \alg{\Gamma}_{x}(\Order_{K}, N)
			\to
				R \alg{\Gamma}(\Order_{K}, N)
			\to
				R \alg{\Gamma}(K, N)
		\]
	in $D(k^{\ind\rat}_{\et})$ reduces to an exact sequence
		\[
				0
			\to
				\alg{H}^{1}(\Order_{K}, N)
			\to
				\alg{H}^{1}(K, N)
			\to
				\alg{H}_{x}^{2}(\Order_{K}, N)
			\to
				0
		\]
	in $\Ab(k^{\ind\rat}_{\et})$.
	Since the first two terms are P-acyclic, Serre reflexive and in $\Ind \Pro \Alg_{\uc} / k$ by
	\eqref{prop: cohomology of integers}
	\eqref{ass: cohomology of finite flats over integers}
	and \eqref{prop: cohomology of local fields}
	\eqref{ass: cohomology of finite flats over local fields},
	so is the third $\alg{H}_{x}^{2}(\Order_{K}, N)$.
	
	To deduce $\alg{H}_{x}^{2}(\Order_{K}, N) \in \Ind \Alg_{\uc} / k$,
	let $0 \to N \to G \to H \to 0$ be an exact sequence
	of group schemes over $\Order_{K}$ with $G, H$ smooth affine with connected fibers,
	as we took in the proof of
	\eqref{prop: cohomology of integers}
	\eqref{ass: cohomology of finite flats over integers}.
	Since $R \alg{\Gamma}(\Order_{K}, G) = R \alg{\Gamma}(\Order_{K, \et}, G)$
	is concentrated in degree $0$ by
	\eqref{prop: integral etale cohomology vanishes}
	and the morphism $\alg{\Gamma}(\Order_{K}, G) \to \alg{\Gamma}(K, G)$ is injective,
	we know that
		\begin{gather*}
					\alg{\Gamma}_{x}(\Order_{K}, G)
				=
					0,
				\quad
					\alg{H}_{x}^{1}(\Order_{K}, G)
				=
					\alg{\Gamma}(K, G) / \alg{\Gamma}(\Order_{K}, G),
			\\
					\alg{H}_{x}^{n}(\Order_{K}, G)
				=
					\alg{H}^{n - 1}(K, G),
				\quad
					n \ge 2.
		\end{gather*}
	These are locally of finite presentation
	by \eqref{prop: higher cohomology with smooth coefficients}
	and \eqref{prop: K-points mod O-points commutes with direct limits}.
	Similarly, $\alg{H}_{x}^{n}(\Order_{K}, H)$ is locally of finite presentation for any $n$.
	The distinguished triangle
		$
				R \alg{\Gamma}_{x}(\Order_{K}, N)
			\to
				R \alg{\Gamma}_{x}(\Order_{K}, G)
			\to
				R \alg{\Gamma}_{x}(\Order_{K}, H)
		$
	then shows that
	$\alg{H}_{x}^{2}(\Order_{K}, N)$ is
	locally of finite presentation.
	We saw above that $\alg{H}_{x}^{2}(\Order_{K}, N) \in \Ind \Pro \Alg_{\uc} / k$,
	Hence \eqref{lem: criteiron when ind-proalgebraic is ind-algebraic} implies that
	$\alg{H}_{x}^{2}(\Order_{K}, N) \in \Ind \Alg_{\uc} / k$.
	
	If $K$ has mixed characteristic,
	then $\alg{H}^{1}(K, N) \in \Alg_{\uc} / k$ implies
	$\alg{H}^{1}(\Order_{K}, N) \in \Alg_{\uc} / k$,
	so $\alg{H}_{x}^{2}(\Order_{K}, N) \in \Alg_{\uc} / k$.
\end{proof}

\begin{Rmk}
	Another method to compute $R \alg{\Gamma}(\Order_{K}, N)$,
	$R \alg{\Gamma}(K, N)$ and $R \alg{\Gamma}_{x}(\Order_{K}, N)$ for finite flat $N$
	over equal characteristic $K$
	is to use the two exact sequences of \cite[III, \S 5]{Mil06}
	(see also \S \ref{sec: Proof of the finite flat duality} of this paper).
	Then these cohomology complexes can be calculated by
	the cohomology with coefficients in vector groups
	in the case $N$ or the Cartier dual $N^{\CDual}$ has height $1$.
	The general case follows by d\'evissage.
	This method is due to Artin-Milne \cite{AM76} in the global situation
	and Bester \cite{Bes78} in the local situation.
\end{Rmk}


\numberwithin{equation}{section}
\section{Statement of the duality theorem}
\label{sec: Statement of the duality theorem}
From now on throughout the paper,
all sheaves over $k$, their exact sequences and distinguished triangles
are considered in $\Ab(k^{\ind\rat}_{\pro\et})$, $D(k^{\ind\rat}_{\pro\et})$
unless otherwise noted.
We denote $R \alg{\Gamma}(\;\cdot\;) = R \alg{\Gamma}(K, \;\cdot\;)$
when there is no confusion.


\numberwithin{equation}{subsection}
\subsection{Formulation}
\label{sec: Formulation}

We formulate the duality theorem with coefficients in abelian varieties.
First, $R \alg{\Gamma}(\Gm)$ is P-acyclic by 
\eqref{prop: cohomology of local fields}
\eqref{ass: cohomology of tori over local fields},
so we write $R \Tilde{\alg{\Gamma}}(\Gm) = R \alg{\Gamma}(\Gm)$ in $D(k^{\ind\rat}_{\pro\et})$.
The same assertion or \cite[Prop.\ 2.4.4]{Suz13} shows that
	\[
			R \alg{\Gamma}(\Gm)
		=
			\alg{\Gamma}(\Gm)
		=
			\alg{K}^{\times}.
	\]
Recall from the paragraph before \cite[Prop.\ 2.4.4]{Suz13} that
there are the valuation map $\alg{K}^{\times} \to \Z$ as a morphism of sheaves
and a split exact sequence
	\[
			0
		\to
			\alg{U}_{K}
		\to
			\alg{K}^{\times}
		\to
			\Z
		\to
			0,
	\]
where $\alg{U}_{K} = \alg{O}_{K}^{\times}$.
An alternative definition of this sequence and the valuation map is the exact sequence
	\[
			0
		\to
			\alg{\Gamma}(\Order_{K}, \Gm)
		\to
			\alg{\Gamma}(\Order_{K}, \mathcal{G}_{m})
		=
			\alg{\Gamma}(K, \Gm)
		\to
			\alg{\Gamma}(\Order_{K}, \Z_{x})
		\to
			0
	\]
coming from the exact sequence
$0 \to \Gm \to \mathcal{G}_{m} \to \Z_{x} \to 0$
of group schemes over $\Order_{K}$,
where $\mathcal{G}_{m}$ is the N\'eron (lft) model of $\Gm$
and $\Z_{x}$ the \'etale group with support on $x = \Spec k$ and special fiber $\Z$,
and we used \eqref{prop: schemes over local fields with ind-rational base}
\eqref{ass: valuative criterion of properness} for the middle isomorphism
and \eqref{prop: integral etale cohomology vanishes} for the exactness.
On $k$-points, it is the usual sequence $0 \to U_{K} \to K^{\times} \to \Z \to 0$,
where $U_{K} = \Order_{K}^{\times}$.
We call the composite
	\begin{equation} \label{eq: trace morphism}
		R \alg{\Gamma}(K, \Gm) = \alg{K}^{\times} \onto \Z
	\end{equation}
the \emph{trace morphism}.

Let $A$ be an abelian variety over $K$ with dual $A^{\vee}$.
Recall from \eqref{prop: cohomology of local fields}
\eqref{ass: cohomology of abelian varieties over local fields}
that $\alg{\Gamma}(A) \in \Pro \Alg / k$,
$\alg{H}^{1}(A) \in \Ind \Alg_{\uc} / k$,
$\alg{H}^{n}(A) = 0$ for $n \ge 2$,
and $R \alg{\Gamma}(A)$ is P-acyclic.
In particular, we write $R \Tilde{\alg{\Gamma}}(A) = R \alg{\Gamma}(A)$.
With the morphism of functoriality of $R \Tilde{\alg{\Gamma}}$
in \eqref{prop: localization and functoriality in pro-etale}
and the trace morphism above,
we have morphisms
	\begin{align*}
				R \alg{\Gamma}(A^{\vee})
		&	\to
				R \Tilde{\alg{\Gamma}}
				R \sheafhom_{K}(A, \Gm)[1]
		\\
		&	\to
				R \sheafhom_{k^{\ind\rat}_{\pro\et}}(
					R \alg{\Gamma}(A), R \alg{\Gamma}(\Gm)
				)[1]
		\\
		&	\to
				R \sheafhom_{k^{\ind\rat}_{\pro\et}}(
					R \alg{\Gamma}(A), \Z
				)[1]
			=
				R \alg{\Gamma}(A)^{\SDual}[1]
	\end{align*}
in $D(k^{\ind\rat}_{\pro\et})$.
We denote by $\vartheta_{A}$ the composite morphism:
	\[
			\vartheta_{A}
		\colon
			R \alg{\Gamma}(A^{\vee})
		\to
			R \alg{\Gamma}(A)^{\SDual}[1]
		\quad
			(=
				R \sheafhom_{k^{\ind\rat}_{\pro\et}}(
					R \alg{\Gamma}(A), \Q / \Z
				)
			).
	\]
Recall from \eqref{prop: cohomology of local fields}
\eqref{ass: cohomology of abelian varieties over local fields}
and \eqref{prop: cohomology of integers}
\eqref{ass: cohomology of smooth group over integers}
that there are surjections
	\[
			\alg{\Gamma}(A)
		\onto
			\mathcal{A}_{x}
		\onto
			\pi_{0}(\mathcal{A}_{x})
	\]
with connected kernels,
where $\mathcal{A}$ is the N\'eron model of $A$ and $\mathcal{A}_{x}$ its special fiber.
In particular, $\alg{\Gamma}(A)$ is not Serre reflexive in general
due to the possibly non-zero semi-abelian part of $\mathcal{A}_{x}$.
We take the Serre dual of $\vartheta_{A}$.
Replacing $A$ with $A^{\vee}$, we have a morphism
	\[
			\theta_{A}
		\colon
			R \alg{\Gamma}(A^{\vee})^{\SDual \SDual}
		\to
			R \alg{\Gamma}(A)^{\SDual}[1]
	\]
in $D(k^{\ind\rat}_{\pro\et})$.
Then $\vartheta_{A}$ can be written as the composite of the natural evaluation morphism
$\id \to \SDual \SDual$ and $\theta_{A}$:
	\[
			\vartheta_{A}
		\colon
			R \alg{\Gamma}(A^{\vee})
		\to
			R \alg{\Gamma}(A^{\vee})^{\SDual \SDual}
		\stackrel{\theta_{A}}{\to}
			R \alg{\Gamma}(A)^{\SDual}[1].
	\]
We have $\alg{\Gamma}(A)^{\SDual} \in D^{b}(\Ind \Alg_{\uc} / k)$ and
$\alg{H}^{1}(A)^{\SDual} \in D^{b}(\Pro \Alg_{\uc} / k)$
by \eqref{prop: Serre duality},
which are both Serre reflexive.
Therefore $R \alg{\Gamma}(A)^{\SDual}$ is Serre reflexive,
so is its Serre dual.
Therefore $\theta_{A}$ is a morphism between Serre reflexive complexes.
Now we can precisely state our duality theorem for abelian varieties.

\begin{Thm} \label{thm: duality for abelian varieties}
	The morphism $\theta_{A}$ defined above is an isomorphism.
\end{Thm}

We start proving this from the next subsection.
The proof finishes at \S \ref{sec: End of proof: Grothendieck's conjecture}.


\numberwithin{equation}{subsection}
\subsection{Reduction to components groups and the first cohomology}
\label{sec: Reduction to components groups and the first cohomology}

First, the sheaf
$\alg{H}^{1}(A) \in \Ind \Alg / k$ is P-acyclic as we saw.
Since $\alg{\Gamma}(A) \in \Pro \Alg / k$,
the sheaf $\sheafext_{k^{\ind\rat}_{\et}}^{1}(\alg{\Gamma}(A), \Q / \Z)$ is
locally of finite presentation
by what we saw after \eqref{prop: Serre duality}.
In particular, $\sheafext_{k^{\ind\rat}_{\et}}^{1}(\alg{\Gamma}(A), \Q / \Z)$ is P-acyclic
by \eqref{prop: criterion of P-acyclicity}
\eqref{ass: local finite presentation implies P-acyclic},
and $\sheafext_{k^{\ind\rat}_{\et}}^{1}(\alg{\Gamma}(A), \Q / \Z)$
and $\sheafext_{k^{\ind\rat}_{\pro\et}}^{1}(\alg{\Gamma}(A), \Q / \Z)$
are equal as functors on $k^{\ind\rat}$.
Therefore the $k'$-points of $\alg{H}^{1}(A)$ and
$\sheafext_{k^{\ind\rat}_{\pro\et}}^{1}(\alg{\Gamma}(A), \Q / \Z)$
for any algebraically closed field $k' \in k^{\ind\rat}$ are
$H^{1}(\alg{K}(k'), A)$ and
$\Ext_{k'^{\ind\rat}_{\pro\et}}^{1}(\alg{\Gamma}(A), \Q / \Z)$,
respectively.

\begin{Prop} \label{prop: two parts of the duality morphism}
	The morphism $\theta_{A}$ induces two morphisms
		\begin{gather*}
					\theta_{A}^{+0}
				\colon
					\pi_{0}(\mathcal{A}_{x}^{\vee})
				\to
					\pi_{0}(\mathcal{A}_{x})^{\PDual}
				\quad \text{in} \quad
					\FEt / k,
			\\
					\theta_{A}^{+1}
				\colon
					\alg{H}^{1}(A^{\vee})
				\to
					\sheafext_{k^{\ind\rat}_{\pro\et}}^{1}(\alg{\Gamma}(A), \Q / \Z)
				\quad \text{in} \quad
					\Ind \Alg_{\uc} / k.
		\end{gather*}
	For any algebraically closed field $k' \in k^{\ind\rat}$,
	denote by $\theta_{A}^{+1}(k')$ the morphism $\theta_{A}^{+1}$ induced on the $k'$-points:
		\[
				\theta_{A}^{+1}(k')
			\colon
				H^{1}(\alg{K}(k'), A^{\vee})
			\to
				\Ext_{k'^{\ind\rat}_{\pro\et}}^{1}(\alg{\Gamma}(A), \Q / \Z).
		\]
	Then the following are equivalent:
	\begin{itemize}
		\item
			$\theta_{A}$ is an isomorphism.
		\item
			$\theta_{A}^{+0}$, $\theta_{A}^{+1}$, $\theta_{A^{\vee}}^{+1}$ are isomorphisms.
		\item
			$\theta_{A}^{+0}$, $\theta_{A}^{+1}(k')$, $\theta_{A^{\vee}}^{+1}(k')$ are isomorphisms
			for any algebraically closed field $k' \in k^{\ind\rat}$.
	\end{itemize}
\end{Prop}

We prove this in this subsection.
Basically the morphisms will be obtained
by writing down the spectral sequence associated with $\theta_{A}$,
as we did in \S \ref{sec: Motivation of our constructions}.
More precisely, $\theta_{A}^{+0}$ is the morphism
induced on the $\pi_{0}$ of $H^{0}$ of the both sides of $\theta_{A}$
and $\theta_{A}^{+1}$ is the morphism
induced on the $H^{1}$.
The part for $H^{-1}$ and the identity component of $H^{0}$ contain no additional information by symmetry,
and $H^{n} = 0$ for the both sides for $n \ne -1, 0, 1$.
To clarify the symmetry mentioned and the treatment of the double-dual,
we split the construction and the proof into several steps.

Consider the morphisms
	\[
			\invlim_{n} \alg{\Gamma}(A)
		\to
			\alg{\Gamma}(A)
		\to
			R \alg{\Gamma}(A),
	\]
where the (non-derived) limit on the left is over multiplication by $n \ge 1$.
Note that $\invlim_{n} \alg{\Gamma}(A)$ is no longer P-acyclic.
There is a canonical choice of a mapping cone of
the composite
	$
			\invlim_{n} \alg{\Gamma}(A)
		\to
			R \alg{\Gamma}(A)
	$
in $D(k^{\ind\rat}_{\pro\et})$
since the former is concentrated in degree $0$ and the latter is concentrated in non-negative degrees.
We denote this mapping cone by
	$
		\bigl[
				\invlim_{n} \alg{\Gamma}(A)
			\to
				R \alg{\Gamma}(A)
		\bigr]
	$.

\begin{Prop} \label{prop: double dual of R Gamma of abelian variety}
	There is a canonical isomorphism
		\[
				\bigl[
						\invlim_{n} \alg{\Gamma}(A)
					\to
						R \alg{\Gamma}(A)
				\bigr]
			=
				R \alg{\Gamma}(A)^{\SDual \SDual}.
		\]
	The induced morphism
	$R \alg{\Gamma}(A) \to R \alg{\Gamma}(A)^{\SDual \SDual}$
	is the natural evaluation morphism.
\end{Prop}

\begin{proof}
	Since $\alg{\Gamma}(A) \in \Pro \Alg / k$,
	we have $(\invlim_{n} \alg{\Gamma}(A))^{\SDual} = 0$
	as seen in the proof of \eqref{prop: Serre duality}
	\eqref{ass: Serre duality for general proalgebraic groups}.
	Therefore we have a natural morphism and an isomorphism
		\[
				\bigl[
						\invlim_{n} \alg{\Gamma}(A)
					\to
						R \alg{\Gamma}(A)
				\bigr]
			\to
				\bigl[
						\invlim_{n} \alg{\Gamma}(A)
					\to
						R \alg{\Gamma}(A)
				\bigr]^{\SDual \SDual}
			=
				R \alg{\Gamma}(A)^{\SDual \SDual}.
		\]
	We need to show that the left term is Serre reflexive.
	Since $\alg{H}^{n}(A) = 0$ for $n \ge 2$,
	we have a distinguished triangle
		\[
				\bigl[
						\invlim_{n} \alg{\Gamma}(A)
					\to
						\alg{\Gamma}(A)
				\bigr]
			\to
				\bigl[
						\invlim_{n} \alg{\Gamma}(A)
					\to
						R \alg{\Gamma}(A)
				\bigr]
			\to
				\alg{H}^{1}(A)[-1].
		\]
	The left mapping cone is isomorphic to
	$\alg{\Gamma}(A)^{\SDual \SDual} \in D^{b}(\Pro \Alg_{\uc} / k)$
	by \eqref{prop: Serre duality}
	\eqref{ass: Serre duality for general proalgebraic groups},
	which is Serre reflexive.
	We have $\alg{H}^{1}(A) \in \Ind \Alg_{\uc} / k$, which is Serre reflexive.
	Therefore the middle term is Serre reflexive.
	Thus we get the required isomorphism.
\end{proof}

Let $\alg{\Gamma}(A^{\vee})_{0}$ be the identity component of $\alg{\Gamma}(A^{\vee})$.
We have two distinguished triangles
	\[
		\begin{CD}
				\alg{\Gamma}(A^{\vee})_{0}
			@>>>
				R \alg{\Gamma}(A^{\vee})
			@>>>
				\bigl[
						\alg{\Gamma}(A^{\vee})_{0}
					\to
						R \alg{\Gamma}(A^{\vee})
				\bigr]
			\\
				\alg{H}^{1}(A)^{\SDual}[2]
			@>>>
				R \alg{\Gamma}(A)^{\SDual}[1]
			@>>>
				\alg{\Gamma}(A)^{\SDual}[1]
		\end{CD}
	\]
and the morphism
	$
			\vartheta_{A}
		\colon
			R \alg{\Gamma}(A^{\vee})
		\to
			R \alg{\Gamma}(A)^{\SDual}[1]
	$
between the middle terms.

\begin{Prop} \label{prop: extend duality morphism to the plus-minus parts}
	There is a unique way to extend $\vartheta_{A}$ to a morphism of triangles
		\[
			\begin{CD}
					\alg{\Gamma}(A^{\vee})_{0}
				@>>>
					R \alg{\Gamma}(A^{\vee})
				@>>>
					\bigl[
							\alg{\Gamma}(A^{\vee})_{0}
						\to
							R \alg{\Gamma}(A^{\vee})
					\bigr]
				\\
				@VV \vartheta_{A}^{-} V
				@VV \vartheta_{A} V
				@VV \vartheta_{A}^{+} V
				\\
					\alg{H}^{1}(A)^{\SDual}[2]
				@>>>
					R \alg{\Gamma}(A)^{\SDual}[1]
				@>>>
					\alg{\Gamma}(A)^{\SDual}[1]
			\end{CD}
		\]
	(which in particular means that
	$\vartheta_{A}^{+}$ and $\vartheta_{A}^{-}$ are compatible
	with the connecting morphisms of the triangles).
	This diagram can further be extended uniquely to a morphism of triangles
		\begin{equation} \label{eq: two parts of the duality morphism}
			\begin{CD}
					\alg{\Gamma}(A^{\vee})_{0}^{\SDual \SDual}
				@>>>
					R \alg{\Gamma}(A^{\vee})^{\SDual \SDual}
				@>>>
					\bigl[
							\alg{\Gamma}(A^{\vee})_{0}
						\to
							R \alg{\Gamma}(A^{\vee})
					\bigr]
				\\
				@VV \theta_{A}^{-} V
				@VV \theta_{A} V
				@VV \theta_{A}^{+} V
				\\
					\alg{H}^{1}(A)^{\SDual}[2]
				@>>>
					R \alg{\Gamma}(A)^{\SDual}[1]
				@>>>
					\alg{\Gamma}(A)^{\SDual}[1]
			\end{CD}
		\end{equation}
	(where the middle $\theta_{A}$ has been defined earlier.)
\end{Prop}

\begin{proof}
	Since $\alg{\Gamma}(A) \in \Pro \Alg / k$, the complex
		$
				\alg{\Gamma}(A)^{\SDual}[1]
			=
				R \sheafhom_{k^{\ind\rat}_{\pro\et}}(
					\alg{\Gamma}(A), \Q / \Z
				)
		$
	is concentrated in degrees $0$ and $1$
	with $H^{0} = \pi_{0}(\alg{\Gamma}(A))^{\PDual} = \pi_{0}(\mathcal{A}_{x})^{\PDual} \in \FEt / k$
	by \eqref{prop: Serre duality}
	\eqref{ass: Ext of a proalg by Q mod Z is zero above degree 2},
	\eqref{ass: Serre duality for unip and etale}
	and \eqref{ass: Serre duality for general proalgebraic groups}.
	In particular, its ($-1$)-shift $\alg{\Gamma}(A)^{\SDual}$ is concentrated in degrees $1$ and $2$.
	Therefore
		\[
				\Hom_{D(k^{\ind\rat}_{\pro\et})} \bigl(
					\alg{\Gamma}(A^{\vee})_{0},
					\alg{\Gamma}(A)^{\SDual}[1]
				\bigr)
			=
				\Hom_{D(k^{\ind\rat}_{\pro\et})} \bigl(
					\alg{\Gamma}(A^{\vee})_{0},
					\alg{\Gamma}(A)^{\SDual}
				\bigr)
			=
				0,
		\]
	since a morphism in $D(k^{\ind\rat}_{\pro\et})$
	from an object of $\Ab(k^{\ind\rat}_{\pro\et})$
	to a complex concentrated in non-negative degrees factors through $H^{0}$.
	With this, we can uniquely extend $\vartheta_{A}$ to the first diagram
	by the general lemma below.
	For the second diagram, it is sufficient to note that the right upper term
	(whose $H^{0} = \pi_{0}(\alg{\Gamma}(A^{\vee}))$ is finite),
	the left bottom term, the middle bottom term and the right lower term
	are all Serre reflexive
	by \eqref{prop: Serre duality}
	\eqref{ass: Serre duality for unip and etale},
	\eqref{ass: Serre duality for general proalgebraic groups}.
\end{proof}

\begin{Lem}
	Let $\mathcal{D}$ be a triangulated category.
	Let $X \to Y \to Z$ and $X' \to Y' \to Z'$ be two distinguished triangles in $\mathcal{D}$.
	Assume that the homomorphisms
	$\Hom(X, X') \to \Hom(X, Y')$ and $\Hom(Z, Z') \to \Hom(Y, Z')$ are isomorphisms.
	Then any morphism $f \colon Y \to Y'$ can uniquely be extended to a morphism of the triangles.
	The assumption is satisfied if
	$\Hom(X, Z') = \Hom(X, Z'[-1]) = 0$.
\end{Lem}

\begin{proof}
	This is elementary and well-known.
	(See also \cite[Prop.\ 10.1.17]{KS06}.)
	We recall its proof.
	Let $g \colon X \to X'$ be the morphism corresponding to the composite of
	$X \to Y$ and $f \colon Y \to Y'$ under the isomorphism
	$\Hom(X, X') \isomto \Hom(X, Y')$.
	Then by an axiom of triangulated categories,
	there is a morphism $h \colon Z \to Z'$ such that
	the triple $(g, f, h)$ is a morphism of the triangles.
	This in particular means that the image of $h$
	under the isomorphism $\Hom(Z, Z') \isomto \Hom(Y, Z')$ is
	the composite of $f \colon Y \to Y'$ and $Y' \to Z'$.
	Hence such a morphism $h$ is unique.
	This shows that $f$ is uniquely extended to a morphism of the triangles.
	If $\Hom(X, Z') = \Hom(X, Z'[-1]) = 0$,
	then the assumption follows from the long exact sequence of $\Hom$.
\end{proof}

\begin{Prop} \label{prop: two parts of theta A plus}
	The morphism $\theta_{A}^{+}$ in \eqref{eq: two parts of the duality morphism} induces morphisms
		\begin{gather*}
					\theta_{A}^{+0}
				\colon
					\pi_{0}(\mathcal{A}_{x}^{\vee})
				\to
					\pi_{0}(\mathcal{A}_{x})^{\PDual}
				\quad \text{in} \quad
					\FEt / k,
			\\
					\theta_{A}^{+1}
				\colon
					\alg{H}^{1}(A^{\vee})
				\to
					\sheafext_{k^{\ind\rat}_{\pro\et}}^{1}(\alg{\Gamma}(A), \Q / \Z)
				\quad \text{in} \quad
					\Ind \Alg_{\uc} / k,
		\end{gather*}
	on $H^{0}$, $H^{1}$, respectively.
\end{Prop}

\begin{proof}
	This follows from 
		$
				\alg{\Gamma}(A)^{\SDual}[1]
			=
				R \sheafhom_{k^{\ind\rat}_{\pro\et}}(
					\alg{\Gamma}(A), \Q / \Z
				)
		$,
	$\pi_{0}(\alg{\Gamma}(A)) =  \pi_{0}(\mathcal{A}_{x})$,
	and \eqref{prop: Serre duality}
	\eqref{ass: Serre duality for unip and etale}.
\end{proof}

\begin{Prop}
	The morphism $\theta_{A}$ is an isomorphism
	if and only if $\theta_{A}^{+}$ and $\theta_{A}^{-}$ are so.
\end{Prop}

\begin{proof}
	The ``if'' part is trivial.
	We show the converse.
	We have $\invlim_{n} \alg{\Gamma}(A^{\vee}) = \invlim_{n} \alg{\Gamma}(A^{\vee})_{0}$
	since finite groups are killed by $\invlim_{n}$.
	By \eqref{prop: Serre duality}
	\eqref{ass: Serre duality for general proalgebraic groups} and
	\eqref{prop: double dual of R Gamma of abelian variety},
	we have
		\begin{gather*}
					\alg{\Gamma}(A^{\vee})_{0}^{\SDual \SDual}
				=
					\bigl[
							\invlim_{n} \alg{\Gamma}(A^{\vee})
						\to
							\alg{\Gamma}(A^{\vee})_{0}
					\bigr],
			\\
					R \alg{\Gamma}(A)^{\SDual \SDual}
				=
					\bigl[
							\invlim_{n} \alg{\Gamma}(A)
						\to
							R \alg{\Gamma}(A)
					\bigr].
		\end{gather*}
	From these, we can see that
	the upper triangle in the diagram \eqref{eq: two parts of the duality morphism}
	is of the form $X \to Y \to Z$ with:
	$X$ concentrated in degrees $-1, 0$;
	$Y$ in $-1, 0, 1$;
	and $Z$ in $0, 1$.
	For such a distinguished triangle, we have isomorphisms and an exact sequence
		\begin{gather*}
					H^{-1}(X) = H^{-1}(Y),
				\quad
					H^{1}(Y) = H^{1}(Z),
			\\
				0 \to H^{0}(X) \to H^{0}(Y) \to H^{0}(Z) \to 0.
		\end{gather*}
	The exact sequence in the second line is a connected-\'etale sequence in $\Pro \Alg_{\uc}$.
	
	We show a similar statement
	for the lower triangle in the diagram \eqref{eq: two parts of the duality morphism}:
	it has cohomologies in the same ranges with the same property
	(i.e.\ $H^{0}$ is a connected-\'etale sequence).
	Since $\alg{H}^{1}(A) \in \Ind \Alg_{\uc} / k$,
	we know that $\alg{H}^{1}(A)^{\SDual}[2]$ is concentrated in degrees $-1$, $0$
	whose $H^{0}$ is connected  pro-unipotent
	by \eqref{prop: Serre duality}
	\eqref{ass: Serre duality for unip and etale}.
	We already saw that $\alg{\Gamma}(A)^{\SDual}[1]$ is concentrated in degrees $0, 1$
	whose $H^{0}$ is finite \'etale.
	Therefore $R \alg{\Gamma}(A)^{\SDual}[1]$ is concentrated in degrees $-1, 0, 1$
	whose $H^{0}$ is in $\Pro \Alg_{\uc}$.
	Thus the lower triangle in \eqref{eq: two parts of the duality morphism}
	has the expected properties.
	
	Now if the morphism on the $Y$'s of the two triangles in \eqref{eq: two parts of the duality morphism}
	is an isomorphism,
	then so are the morphisms on the $X$'s and $Z$'s.
\end{proof}

Before the next proposition,
note that, by \eqref{prop: Serre duality}
\eqref{ass: Ext of a proalg by Q mod Z is zero above degree 2},
	\begin{align*}
				\alg{\Gamma}(A)_{0}^{\SDual}
		&	=
				R \sheafhom_{k^{\ind\rat}_{\pro\et}}(
					\alg{\Gamma}(A)_{0}, \Q / \Z
				)[-1]
		\\
		&	=
				\sheafext_{k^{\ind\rat}_{\pro\et}}^{1}(
					\alg{\Gamma}(A)_{0}, \Q / \Z
				)[-2].
		\\
		&	=
				\sheafext_{k^{\ind\rat}_{\pro\et}}^{1}(
					\alg{\Gamma}(A), \Q / \Z
				)[-2].
	\end{align*}
Hence the Serre dual (up to shift) of $\theta_{A}^{+1}$
in \eqref{prop: two parts of theta A plus}
can be written as a morphism
$\alg{\Gamma}(A)_{0}^{\SDual \SDual} \to \alg{H}^{1}(A^{\vee})^{\SDual}[2]$.
Since $\theta_{A^{\vee}}^{-}$ is also a morphism
$\alg{\Gamma}(A)_{0}^{\SDual \SDual} \to \alg{H}^{1}(A^{\vee})^{\SDual}[2]$,
it is meaningful to compare $\theta_{A^{\vee}}^{-}$ and $(\theta_{A}^{+1})^{\SDual}$.

\begin{Prop}
	We have $\theta_{A^{\vee}} = (\theta_{A})^{\SDual}$,
	$\theta_{A^{\vee}}^{-} = (\theta_{A}^{+1})^{\SDual}$
	and $\theta_{A^{\vee}}^{+0} = (\theta_{A}^{+0})^{\PDual}$.
\end{Prop}

\begin{proof}
	We first show $\theta_{A^{\vee}} = (\theta_{A})^{\SDual}$,
	or $\theta_{A} = (\theta_{A^{\vee}})^{\SDual}$.
	The morphisms $A^{\vee} \tensor^{L} A \to \Gm[1]$ and 
	$A^{\vee} \tensor^{L} A^{\vee \vee} \to \Gm[1]$ in $D(K_{\fppf})$
	coming from the Poincar\'e bundle (or $\Gm$-biextension)
	are compatible with the biduality isomorphism $A \isomto A^{\vee \vee}$.
	Hence we have a commutative diagram
		\[
			\begin{CD}
					A^{\vee} \tensor^{L} A
				@>>>
					\Gm[1]
				\\
				@VV \wr V
				@|
				\\
					A^{\vee} \tensor^{L} A^{\vee \vee}
				@>>>
					\Gm[1]
			\end{CD}
		\]
	in $D(K_{\fppf} / k^{\ind\rat}_{\et})$.
	Recall from the second paragraph after
	\eqref{prop: localization sequence}
	that the morphism of functoriality and the cup product pairing are equivalent
	under the derived tensor-hom adjunction \cite[Thm.\ 18.6.4 (vii)]{KS06}.
	This relation also holds after pro-\'etale sheafification.
	Applying $R \Tilde{\alg{\Gamma}}$ to the above diagram and using the cup product pairing,
	we have a commutative diagram
		\[
			\begin{CD}
						R \alg{\Gamma}(A^{\vee})
					\tensor^{L}
						R \alg{\Gamma}(A)
				@>>>
					R \alg{\Gamma}(\Gm)[1]
				@>>>
					\Z[1]
				\\
				@VV \wr V
				@|
				@|
				\\
						R \alg{\Gamma}(A^{\vee})
					\tensor^{L}
						R \alg{\Gamma}(A^{\vee \vee})
				@>>>
					R \alg{\Gamma}(\Gm)[1]
				@>>>
					\Z[1]
			\end{CD}
		\]
	in $D(k^{\ind\rat}_{\pro\et})$.
	The upper morphisms give a morphism
	$R \alg{\Gamma}(A^{\vee})^{\SDual \SDual} \to R \alg{\Gamma}(A)^{\SDual}$
	via the derived tensor-Hom adjunction,
	which is $\theta_{A}$.
	Similarly, the lower morphisms give a morphism
	$R \alg{\Gamma}(A^{\vee \vee}) \to R \alg{\Gamma}(A^{\vee})^{\SDual}$,
	which is $\vartheta_{A^{\vee}}$,
	and a morphism
	$R \alg{\Gamma}(A^{\vee})^{\SDual \SDual} \to R \alg{\Gamma}(A^{\vee \vee})^{\SDual}$,
	which is $(\theta_{A^{\vee}})^{\SDual}$.
	The commutativity of the diagram implies that
	$\theta_{A} = (\theta_{A^{\vee}})^{\SDual}$.
	
	For $\theta_{A^{\vee}}^{-} = (\theta_{A}^{+1})^{\SDual}$,
	or $(\theta_{A}^{-})^{\SDual} = \theta_{A^{\vee}}^{+1}$,
	apply $\SDual$ to the second diagram of
	\eqref{prop: extend duality morphism to the plus-minus parts}:
		\[
			\begin{CD}
					\alg{\Gamma}(A)^{\SDual \SDual}
				@>>>
					R \alg{\Gamma}(A)^{\SDual \SDual}
				@>>>
					\alg{H}^{1}(A)[-1]
				\\
				@VV (\theta_{A}^{+})^{\SDual} V
				@VV (\theta_{A})^{\SDual} V
				@VV (\theta_{A}^{-})^{\SDual} V
				\\
					\bigl[
							\alg{\Gamma}(A^{\vee})_{0}
						\to
							R \alg{\Gamma}(A^{\vee})
					\bigr]^{\SDual}[1]
				@>>>
					R \alg{\Gamma}(A^{\vee})^{\SDual}[1]
				@>>>
					\alg{\Gamma}(A^{\vee})_{0}^{\SDual}[1].
			\end{CD}
		\]
	The morphism $(\theta_{A}^{-})^{\SDual}$ is a morphism
	between complexes concentrated in degree $1$.
	The complex $\alg{\Gamma}(A)^{\SDual \SDual}$ is concentrated in degrees $-1, 0$
	by \eqref{prop: Serre duality}
	\eqref{ass: Serre duality for general proalgebraic groups}.
	Consider the distinguished triangle
		\[
				\alg{H}^{1}(A^{\vee})^{\SDual}[2]
			\to
				\bigl[
						\alg{\Gamma}(A^{\vee})_{0}
					\to
						R \alg{\Gamma}(A^{\vee})
				\bigr]^{\SDual}[1]
			\to
				\pi_{0}(\alg{\Gamma}(A^{\vee}))^{\PDual}.
		\]
	By $\alg{H}^{1}(A^{\vee}) \in \Ind \Alg_{\uc} / k$
	and \eqref{prop: Serre duality}
	\eqref{ass: Serre duality for unip and etale},
	we know that $\alg{H}^{1}(A^{\vee})^{\SDual}[2]$ is concentrated in degrees $-1, 0$.
	Hence the middle term
		$
			\bigl[
					\alg{\Gamma}(A^{\vee})_{0}
				\to
					R \alg{\Gamma}(A^{\vee})
			\bigr]^{\SDual}[1]
		$
	is concentrated in degrees $-1, 0$.
	Therefore $H^{1}((\theta_{A})^{\SDual}) = (\theta_{A}^{-})^{\SDual}$.
	By a similar observation, we have $H^{1}(\theta_{A^{\vee}}) = \theta_{A^{\vee}}^{+1}$.
	Hence $(\theta_{A})^{\SDual} = \theta_{A^{\vee}}$ proved above
	implies $(\theta_{A}^{-})^{\SDual} = \theta_{A^{\vee}}^{+1}$.
	
	Finally, the equality $\theta_{A^{\vee}} = (\theta_{A})^{\SDual}$
	induces the equality $\theta_{A^{\vee}}^{+0} = (\theta_{A}^{+0})^{\PDual}$
	on $\pi_{0} H^{0}$.
\end{proof}

The previous two propositions show that
the first two assertions in \eqref{prop: two parts of the duality morphism}
are equivalent.
To connect them to the third,
it is enough to show the following general lemma.

\begin{Lem} \label{lem: isom on indalgebraics can be checked on closed fields}
	Let $\varphi \colon B \to C$ be a morphism in $\Ind \Alg / k$.
	Assume that $\varphi \colon B(k') \to C(k')$ is an isomorphism
	for any algebraically closed field $k'$ over $k$.
	Then $\varphi \colon B \to C$ is an isomorhism.
\end{Lem}

\begin{proof}
	By considering the kernel and cokernel,
	it is enough to show that if $D \in \Ind \Alg / k$ satisfies $D(k') = 0$
	for any algebraically closed field $k'$ over $k$,
	then $D = 0$.
	For this, it is enough to show that $\Hom(E, D) = 0$
	for any $E \in \Alg / k$.
	Let $\xi_{E}$ be the generic point of $E$.
	Then by \eqref{eq: birational group morphisms extend}, we have
		\[
				\Hom(E, D)
			=
				\Ker(
						D(\xi_{E})
					\stackrel{\sigma}{\to}
						D(\xi_{E \times_{k} E})
				)
		\]
	(with the same notation as the cited proof).
	But the assumption implies  that $D(\xi_{E}) = 0$.
	Hence $\Hom(E, D) = 0$ and $D = 0$.
\end{proof}

\begin{Rmk}
	Here is another method to obtain \eqref{prop: two parts of the duality morphism} or,
	more precisely, \eqref{prop: extend duality morphism to the plus-minus parts}.
	The Poincar\'e biextension $A^{\vee} \times A \to \Gm[1]$ canonically extends to
	a biextension $\mathcal{A}_{0}^{\vee} \times \mathcal{A} \to \Gm[1]$
	by \cite[IX, 1.4.3]{Gro72},
	where $\mathcal{A}_{0}^{\vee}$ is the maximal open subgroup scheme of $\mathcal{A}^{\vee}$
	with connected special fiber.
	Hence we have a morphism $\mathcal{A}_{0}^{\vee} \to R \sheafhom_{\Order_{K}}(\mathcal{A}, \Gm[1])$.
	We have the trace isomorphism $R \alg{\Gamma}_{x}(\Order_{K}, \Gm[1]) = \Z$;
	see the second paragraph of \S \ref{sec: Bester's finite flat duality}.
	Hence we have a morphism
		\begin{align*}
					R \alg{\Gamma}_{x}(\Order_{K}, \mathcal{A}_{0}^{\vee})
			&	\to
					R \sheafhom_{k^{\ind\rat}_{\pro\et}} \bigl(
						R \alg{\Gamma}(\Order_{K}, \mathcal{A}),
						R \alg{\Gamma}_{x}(\Order_{K}, \Gm[1])
					\bigr)
			\\
			&	=
					R \sheafhom_{k^{\ind\rat}_{\pro\et}} \bigl(
						R \alg{\Gamma}(\Order_{K}, \mathcal{A}),
						\Z
					\bigr)
				=
					R \alg{\Gamma}(\Order_{K}, \mathcal{A})^{\SDual}.
		\end{align*}
	The localization exact sequences give a morphism
		\[
			\begin{CD}
					R \alg{\Gamma}(\Order_{K}, \mathcal{A}_{0}^{\vee})
				@>>>
					R \alg{\Gamma}(K, A^{\vee})
				@>>>
					R \alg{\Gamma}_{x}(\Order_{K}, \mathcal{A}_{0}^{\vee})[1]
				\\
				@VVV
				@VVV
				@VVV
				\\
					R \alg{\Gamma}_{x}(\Order_{K}, \mathcal{A})^{\SDual}
				@>>>
					R \alg{\Gamma}(K, A)^{\SDual}[1]
				@>>>
					R \alg{\Gamma}(\Order_{K}, \mathcal{A})^{\SDual}[1]
			\end{CD}
		\]
	between distinguished triangles
	(see \eqref{prop: localization and Serre dual, for finite flats}
	for the corresponding statement for finite flat group schemes).
	We can identify this diagram with the first diagram in
	\eqref{prop: extend duality morphism to the plus-minus parts}.
	We somewhat avoided the use of cohomology of $\Order_{K}$ in this subsection
	and constructed the diagram purely from the structure of $R \alg{\Gamma}(K, A)$
	as a complex of sheaves over $k$.
\end{Rmk}


\numberwithin{equation}{section}
\section{Relation to Grothendieck's and \v{S}afarevi\v{c}'s conjectures}
\label{sec: Comparison with other authors' work}

We compare our morphisms $\theta_{A}^{+0}, \theta_{A}^{+1}$
with Grothendieck's pairing, B\'egueri's isomorphism
and Bester-Bertapelle's isomorphism.


\numberwithin{equation}{subsection}
\subsection{Grothendieck's pairing}
\label{sec: Grothendieck's pairing}

\begin{Prop} \label{prop: comparison with Grothendieck's pairing}
	The morphism
		\[
				\theta_{A}^{+0}
			\colon
				\pi_{0}(\mathcal{A}_{x}^{\vee})
			\to
				\pi_{0}(\mathcal{A}_{x})^{\PDual},
		\]
	of \eqref{prop: two parts of the duality morphism}
	coincides with Grothendieck's pairing
	\cite[IX, 1.2.1]{Gro72}.
\end{Prop}

We prove this by writing down $\theta_{A}^{+0}$ explicitly.
Recall that the morphism $\theta_{A}$ is defined
by the functoriality of $R \Tilde{\alg{\Gamma}}$
applied to $R \sheafhom_{K}(A, \Gm)$
and the trace morphism $R \alg{\Gamma}(\Gm) \to \Z$
or the valuation map $\alg{K}^{\times} \onto \Z$.
To write $\theta_{A}^{+0}$,
we give an arbitrary element of $\Ext_{K}^{1}(A, \Gm)$,
which is an extension class,
and apply $R \Tilde{\alg{\Gamma}}$.

\begin{proof}
	We may assume that $k$ is algebraically closed.
	We first recall the construction of Grothendieck's pairing following \cite[Appendix C]{Mil06}.
	We denote by
	$j \colon \Spec K_{\sm} \into \Spec \Order_{K, \sm}$
	the natural morphism between the smooth sites.
	Let $0 \to \Gm \to X \to A \to 0$ be
	an element of $\Gamma(K, A^{\vee}) = \Ext_{K}^{1}(A, \Gm)$.
	Let $\mathcal{G}_{m} = j_{\ast} \Gm$,
	$\mathcal{X} = j_{\ast} X$, $\mathcal{A} = j_{\ast} A$ be the N\'eron models
	(or more precisely, the N\'eron lft models).
	Then the sequence
	$0 \to \mathcal{G}_{m} \to \mathcal{X} \to \mathcal{A} \to 0$
	in $\Ab(\Order_{K, \sm})$ is exact since $R^{1} j_{\ast} \Gm = 0$
	(see the proof of \cite[loc.cit., Lem.C.10]{Mil06}).
	We can view this as an exact sequence of group schemes over $\Order_{K}$.
	Hence we have an exact sequence
	$0 \to \mathcal{G}_{m} \to \mathcal{X} \to \mathcal{A} \to 0$
	in $\Ab(\Order_{K, \fppf})$
	(though these no longer represent the pushforward sheaves
	of the original algebraic groups over $K_{\fppf}$).
	We also have an exact sequence
	$0 \to \Gm \to \mathcal{G}_{m} \to i_{\ast} \Z \to 0$
	in $\Ab(\Order_{K, \fppf})$,
	where $i \colon \Spec k_{\fppf} \into \Spec \Order_{K, \fppf}$
	is the natural morphism.
	By pushing out the extension
	$0 \to \mathcal{G}_{m} \to \mathcal{X} \to \mathcal{A} \to 0$
	by the morphism $\mathcal{G}_{m} \onto i_{\ast} \Z$,
	we get an extension
	$0 \to i_{\ast} \Z \to \mathcal{X} / \Gm \to \mathcal{A} \to 0$
	in $\Ab(\Order_{K, \fppf})$.
	Pulling back by $i$, we have an exact sequence
	$0 \to \Z \to \mathcal{X}_{x} / \Gm \to \mathcal{A}_{x} \to 0$
	in $\Ab(k_{\fppf})$,
	where the subscript $x$ denotes the special fiber.
	Therefore we get an element of
		\[
				\Ext_{k_{\fppf}}^{1}(\mathcal{A}_{x}, \Z)
			=
				\Hom_{k_{\fppf}}(\mathcal{A}_{x}, \Q / \Z)
			=
				\pi_{0}(\mathcal{A}_{x})^{\PDual}.
		\]
	This defines a homomorphism
	$\Gamma(K, A^{\vee}) = \Gamma(\Order_{K}, \mathcal{A}^{\vee}) \to \pi_{0}(\mathcal{A}_{x})^{\PDual}$,
	which turns out to factor through $\pi_{0}(\mathcal{A}_{x}^{\vee})$
	(\cite[loc.cit., Lem.\ C.11]{Mil06}).
	The homomorphism $\pi_{0}(\mathcal{A}_{x}^{\vee}) \to \pi_{0}(\mathcal{A}_{x})^{\PDual}$
	thus obtained is Grothendieck's pairing.
	
	We next describe $\theta_{A}^{+0}$.
	By definition,
	the morphism $\vartheta_{A}$ after taking $R \Gamma(k, \;\cdot\;)$ and $H^{0}$ is
		\begin{align*}
				\Gamma(K, A^{\vee})
		&	\to
				\Hom_{D(K_{\fppf} / k^{\ind\rat}_{\et})}(A, \Gm[1])
		\\
		&	\stackrel{R \Tilde{\alg{\Gamma}}}{\longrightarrow}
				\Hom_{D(k^{\ind\rat}_{\pro\et})} \bigl(
					R \alg{\Gamma}(A), R \alg{\Gamma}(\Gm)[1]
				\bigr)
		\\
		&	\stackrel{\text{trace}}{\longrightarrow}
				\Hom_{D(k^{\ind\rat}_{\pro\et})} \bigl(
					\alg{\Gamma}(A), \Z[1]
				\bigr).
		\end{align*}
	Since $R \Hom_{k^{\ind\rat}_{\pro\et}}(\alg{\Gamma}(A), \Q) = 0$
	by \eqref{prop: comparison thm for ind-proalgebraic groups}
	\eqref{ass: RHom between IPAlg and Et},
	the final group in the above displayed equation is further isomorphic to
		\[
				\Hom_{k^{\ind\rat}_{\pro\et}} \bigl(
					\alg{\Gamma}(A), \Q / \Z
				\bigr)
			=
				\Hom_{k^{\ind\rat}_{\pro\et}} \bigl(
					\pi_{0}(\mathcal{A}_{x}), \Q / \Z
				\bigr).
		\]
	The composite
		$
				\Gamma(K, A^{\vee})
			\to
				\pi_{0}(\mathcal{A}_{x})^{\PDual}
		$
	factors through $\pi_{0}(\mathcal{A}_{x}^{\vee})$
	and the resulting homomorphism
	$\pi_{0}(\mathcal{A}_{x}^{\vee}) \to \pi_{0}(\mathcal{A}_{x})^{\PDual}$
	is the definition of $\theta_{A}^{+0}$.
	Therefore an explicit description of $\theta_{A}^{+0}$ is given as follows.
	Let $0 \to \Gm \to X \to A \to 0$ be as above
	(which corresponds to a morphism $A \to \Gm[1]$ in $D(K_{\fppf} / k^{\ind\rat}_{\et})$).
	We apply $R \Tilde{\alg{\Gamma}}$.
	We have $\alg{H}^{1}(\Gm) = 0$ by
	\eqref{prop: cohomology of local fields}
	\eqref{ass: cohomology of tori over local fields}.
	Hence we have an exact sequence
		\[
				0
			\to
				\alg{K}^{\times}
			\to
				\alg{\Gamma}(X)
			\to
				\alg{\Gamma}(A)
			\to
				0
		\]
	in $\Ab(k^{\ind\rat}_{\pro\et})$,
	which gives an element of
		$
			\Ext_{k^{\ind\rat}_{\pro\et}}^{1}(
				\alg{\Gamma}(A),
				\alg{K}^{\times}
			)
		$.
	By pushing it out by the valuation map $\alg{K}^{\times} \onto \Z$,
	we have an exact sequence
		$
				0
			\to
				\Z
			\to
				\alg{\Gamma}(X) / \alg{U}_{K}
			\to
				\alg{\Gamma}(A)
			\to
				0
		$.
	This gives a morphism $\alg{\Gamma}(A) \to \Z[1]$ in $D(k^{\ind\rat}_{\pro\et})$
	and hence a morphism $\pi_{0}(\mathcal{A}_{x}) \to \Q / \Z$,
	which is the value of $\theta_{A}^{+0}$.
	
	Now we compare the two constructions.
	The sequence
		$
				0
			\to
				\alg{K}^{\times}
			\to
				\alg{\Gamma}(X)
			\to
				\alg{\Gamma}(A)
			\to
				0
		$
	above can also be obtained by applying $\alg{\Gamma}(\Order_{K}, \;\cdot\;)$
	to the sequence $0 \to \mathcal{G}_{m} \to \mathcal{X} \to \mathcal{A} \to 0$ of N\'eron models
	by \eqref{prop: schemes over local fields with ind-rational base}
	\eqref{ass: valuative criterion of properness}.
	We have a commutative diagram with exact rows
		\[
			\begin{CD}
					0
				@>>>
					\alg{K}^{\times}
				@>>>
					\alg{\Gamma}(X)
				@>>>
					\alg{\Gamma}(A)
				@>>>
					0
				\\
				@.
				@VVV
				@VVV
				@VVV
				@.
				\\
					0
				@>>>
					(\mathcal{G}_{m})_{x}
				@>>>
					\mathcal{X}_{x}
				@>>>
					\mathcal{A}_{x}
				@>>>
					0
			\end{CD}
		\]
	in $\Ab(k^{\ind\rat}_{\pro\et})$,
	where the vertical morphisms are the reduction maps of the N\'eron models.
	The valuation map $\alg{K}^{\times} \onto \Z$ and the morphism $(\mathcal{G}_{m})_{x} \onto \Z$
	are compatible.
	Hence the pushouts by them yield the same morphism
	$\pi_{0}(\alg{\Gamma}(A)) = \pi_{0}(\mathcal{A}_{x}) \to \Q / \Z$.
	Thus our $\theta_{A}^{+0}$ coincides with Grothendieck's pairing.
\end{proof}

\begin{Rmk}
	In the paragraph before the cited lemma \cite[loc.cit., Lem.\ C.11]{Mil06},
	the cases $r = 0, 1$ of the equality
		\[
				\sheafext_{\Order_{K}}^{r}(
					\mathcal{A},
					i_{\ast} \Z
				)
			=
				i_{\ast}
				\sheafext_{k}^{r}(
					i^{\ast} \mathcal{A},
					\Z
				)
		\]
	are stated and used.
	In this remark, we show, as an illustration of our methods in \cite{Suz13},
	that this equality is true for any $r$.
	The tricky point is that
	the topologies on $\Order_{K}$ and $k$ here should be the smooth topologies
	and the $i$ here should be the natural continuous map
	$i_{\sm} \colon \Spec k_{\sm} \to \Spec \Order_{K, \sm}$.
	Hence, if $r \ge 2$,
	the equality above might look as if it required that the pullback $i_{\sm}^{\ast}$ be exact or,
	equivalently, the pushforward $i_{\sm, \ast}$ send injectives to injectives.
	However, $i_{\sm}^{\ast}$ is not exact,
	since the category of smooth schemes over $\Order_{K}$ is not closed under finite inverse limits.
	In fact, the morphism $\Ga \to \Ga$ over $\Order_{K}$
	given by multiplication by a prime element is injective in $\Ab(\Order_{K, \sm})$,
	while it becomes the zero map $\Ga \to \Ga$ after applying $i_{\sm}^{\ast}$.
	In the above proof of \eqref{prop: comparison with Grothendieck's pairing},
	we used the fact that $j_{\ast} X$ is representable (by the N\'eron model $\mathcal{X}$)
	to pass from the smooth site to the fppf site.
	Without the exactness of $i_{\sm}^{\ast}$,
	we can still prove the above equality for general $r$ as follows.
	
	The last paragraph before \eqref{prop: misc on ind-rational pro-etale topology}
	tells us that $i_{\sm, \ast}$ sends acyclic sheaves to acyclic sheaves.
	By \cite[Lem.\ 3.7.2]{Suz13},
	we know that
	$i_{\sm}^{\ast} \colon \Ab(\Order_{K, \sm}) \to \Ab(k_{\sm})$ admits a left derived functor
	$L i_{\sm}^{\ast} \colon D(\Order_{K, \sm}) \to D(k_{\sm})$,
	which is left adjoint to $R i_{\ast} = i_{\ast}$.
	It is enough to show that
	$L_{n} i^{\ast} \mathcal{A} = 0$ for any $n \ge 1$.
	Let $M(\mathcal{A})$ be Mac Lane's resolution of $\mathcal{A}$ in $\Ab(\Order_{K, \sm})$
	(\cite{Mac57}, \cite[\S 3.4]{Suz13}).
	Its $n$-th term is a direct summand of a direct sum of sheaves of the form
	$\Z[\mathcal{A}^{m}]$ for various $m \ge 0$,
	where $\Z[\mathcal{A}^{m}]$ is the sheafification of the presheaf
	that sends a smooth $\Order_{K}$-algebra $S$
	to the free abelian group generated by the set $\mathcal{A}^{m}(S)$
	(\cite[\S 3.4, 3.5]{Suz13}).
	Hence we have 
	$L i_{\sm}^{\ast} \mathcal{A} = i_{\sm}^{\ast} M(\mathcal{A})$,
	where $i_{\sm}^{\ast}$ on the right is applied term-wise.
	We have $i_{\sm}^{\ast} \Z[\mathcal{A}^{m}] = \Z[i_{\sm}^{\ast \set} (\mathcal{A}^{m})]$,
	where $i_{\sm}^{\ast \set} \colon \Set(\Order_{K, \sm}) \to \Set(k_{\sm})$
	is the pullback for sheaves of sets.
	Since $\mathcal{A}^{m}$ is in the underlying category of $\Spec \Order_{K, \sm}$,
	we know that $i_{\sm}^{\ast \set} (\mathcal{A}^{m})$ is the special fiber
	$\mathcal{A}^{m}_{x} = \mathcal{A}^{m} \times_{\Order_{K}} k$.
	Therefore
	$i_{\sm}^{\ast} M(\mathcal{A}) = M(\mathcal{A}_{x}) = M(i^{\ast} \mathcal{A})$
	is Mac Lane's resolution of $\mathcal{A}_{x}$ in $\Ab(k_{\sm})$,
	which is acyclic outside degree zero.
	Therefore $L_{n} i^{\ast} \mathcal{A} = 0$ for any $n \ge 1$.
	
	In the same way,
	we can prove that if $X$ is a scheme
	and $f$ is the continuous map from the big \'etale site of $X$
	to the smooth site of $X$ defined by the identity,
	then $f^{\ast}$ admits a left derived functor $L f^{\ast}$,
	and we have $L_{n} f^{\ast} T = 0$ for any smooth group scheme $T$ over $X$ and $n \ge 1$.
	This implies that \cite[Thm.\ 4.11]{Mil06} is true
	with the big \'etale site replaced by the smooth site
	by the same argument as above.
	This answers Milne's question made right after the cited theorem.
\end{Rmk}


\numberwithin{equation}{subsection}
\subsection{Bester-Bertapelle's isomorphism}
\label{sec: Bester-Bertapelle's isomorphism}

First note that
	\[
			\Ext_{k^{\ind\rat}_{\pro\et}}^{1}(\alg{\Gamma}(A), \Q / \Z)
		=
			\dirlim_{n}
			\Ext_{\Pro \Alg / k}^{1}(\alg{\Gamma}(A), \Z / n \Z)
	\]
by \eqref{prop: comparison thm for ind-proalgebraic groups}
\eqref{ass: RHom between IPAlg and IPAlg}.

\begin{Prop} \label{prop: comparison with Bester-Bertapelle}
	Assume that $k$ is algebraically closed.
	The morphism
		\[
				\theta_{A}^{+1}(k)
			\colon
				H^{1}(K, A^{\vee})
			\to
				\Ext_{k^{\ind\rat}_{\pro\et}}^{1}(\alg{\Gamma}(A), \Q / \Z)
		\]
	of \eqref{prop: two parts of the duality morphism}
	coincides with Bester-Bertapelle's isomorphism
	\cite[\S 2.7, Thm.\ 7.1]{Bes78}, \cite[Thm.\ 3]{Ber03}
	when $K$ has equal characteristic
	and $A$ has semistable reduction.
\end{Prop}

We prove this in this subsection.
In \cite[\S 2.3, Thm.\ 3.1]{Bes78},
Bester proved a duality theorem for cohomology of $\Order_{K}$
with coefficients in finite flat group schemes.
Based on this result, Bertapelle proved its generalization
for coefficients in the quasi-finite flat group scheme $\mathcal{A}[n]$
of torsion points of the N\'eron model of any semistable abelian variety $A$ (\cite[Thm.\ 1]{Ber03}).
Then she deduced the existence of the above isomorphism for such $A$ (\cite[Thm.\ 2]{Ber03}),
and deduced the existence of the above isomorphism for a general abelian variety.

In this subsection, we first give another proof of Bester's finite flat duality,
by giving another construction of the duality morphism and directly showing that it is an isomorphism.
The idea of proof is the same as Bester's (which is the local version of \cite{AM76}),
but we need to make it work within the formulation of this paper.
Then we use our finite flat duality isomorphism (instead of Bester's isomorphism)
as the input for Bertapelle's constructions.
This outputs an isomorphism
		\[
				\psi_{A}
			\colon
				H^{1}(K, A^{\vee})
			=
				\Ext_{k^{\ind\rat}_{\pro\et}}^{1}(\alg{\Gamma}(A), \Q / \Z)
		\]
as in the proof of \cite[Thm.\ 2]{Ber03}.
Finally we show that $\theta_{A}^{+1}(k)$ and $\psi_{A}$ are equal.
Hence we do not actually compare our isomorphism and Bester's isomorphism
(see \eqref{rmk: Bester duality} below for the reason).
Therefore the statement of \eqref{prop: comparison with Bester-Bertapelle} is
slightly imprecise in this sense.
See \eqref{prop: comparison of Bertapelle isomorphism and our isomorphism} below
for the precise statement.
We do not have to compare Bertapelle's isomorphism with $\theta_{A}^{+1}(k)$
for non-semistable abelian varieties
for the purpose of proving our main theorems.
But this can be done;
see \eqref{rmk: remarks after proof of main theorem}
\eqref{rmk: comparison with Bertapelle in non-semistable case} below.

Throughout this subsection, we assume that $K$ has equal characteristic.
Some statements can be formulated without assuming the residue field $k$ to be algebraically closed.
Therefore we assume that $k$ is a general perfect field.


\numberwithin{equation}{subsubsection}
\subsubsection{Bester's finite flat duality}
\label{sec: Bester's finite flat duality}

Let $N$ be a finite flat group scheme over $\Order_{K}$
with Cartier dual $N^{\CDual}$.
In this subsubsection,
we first formulate a duality theorem that relates the two complexes
	\[
			R \alg{\Gamma}_{x}(\Order_{K}, N^{\CDual}),
		\quad
			R \alg{\Gamma}(\Order_{K}, N)
	\]
each other.
This is written in \cite[Rmk.\ 2.7.6]{Suz13} without proof.
We show that this implies Bester's duality.
The proof of the duality theorem itself will be given in the next subsubsection.

We have $R \alg{\Gamma}(K, \Gm) = \alg{K}^{\times}$ as recalled in the previous section,
and $R \alg{\Gamma}(\Order_{K}, \Gm) = \alg{\Gamma}(\Order_{K}, \Gm) = \alg{U}_{K}$ by
\eqref{prop: cohomology of integers}
\eqref{ass: cohomology of smooth group over integers}.
Hence we have
	\begin{equation} \label{eq: trace isomorphism}
		R \alg{\Gamma}_{x}(\Order_{K}, \Gm) = \Z[-1],
	\end{equation}
which we call the \emph{trace isomorphism}.

Let $N$ be a finite flat group scheme over $\Order_{K}$
with Cartier dual $N^{\CDual}$.
Recall from \eqref{prop: cohomology of integers}
\eqref{ass: cohomology of finite flats over integers}
and \eqref{prop: cohomology with support} that
$R \alg{\Gamma}(\Order_{K}, N) \in D^{b}(\Pro \Alg_{\uc} / k)$,
$R \alg{\Gamma}_{x}(\Order_{K}, N) \in D^{b}(\Ind \Alg_{\uc} / k)$ are P-acyclic and Serre reflexive,
$\alg{\Gamma}(\Order_{K}, N)$ is finite,
$\alg{H}^{1}(\Order_{K}, N)$ is connected pro-unipotent,
$\alg{H}_{x}^{2}(\Order_{K}, N) \in \Ind \Alg_{\uc} / k$,
and the cohomology of $R \alg{\Gamma}(\Order_{K}, N)$ and $R \alg{\Gamma}_{x}(\Order_{K}, N)$
vanishes at all other degrees.
The morphism of functoriality of $R \Tilde{\alg{\Gamma}}$ and $R \Tilde{\alg{\Gamma}}_{x}$
in \eqref{prop: localization and functoriality in pro-etale}
and the trace isomorphism give morphisms
	\begin{align*}
		&
				R \alg{\Gamma}_{x}(\Order_{K}, N^{\CDual})
			\to
				R \Tilde{\alg{\Gamma}}_{x} \bigl(
					\Order_{K},
					R \sheafhom_{\Order_{K}}(N, \Gm)
				\bigr)
		\\
		&	\quad
			\to
				R \sheafhom_{k^{\ind\rat}_{\pro\et}} \bigl(
					R \alg{\Gamma}(\Order_{K}, N),
					\Z[-1]
				\bigr)
			=
				R \alg{\Gamma}(\Order_{K}, N)^{\SDual}[-1]
	\end{align*}

\begin{Thm} \label{thm: duality for finite flats over integers}
	The morphism
		\[
				R \alg{\Gamma}_{x}(\Order_{K}, N^{\CDual})
			\to
				R \alg{\Gamma}(\Order_{K}, N)^{\SDual}[-1]
		\]
	in $D(k^{\ind\rat}_{\pro\et})$ defined above is an isomorphism.
\end{Thm}

This is stated in \cite[Rmk.\ 2.7.6]{Suz13} without proof.
We will prove this in the next subsubsection.
We show here that this implies Bester's duality.

The both sides are in $D^{b}(\Ind \Alg_{\uc} / k)$.
The functor $R \Gamma(\closure{k}_{\pro\et}, \;\cdot\;)$ restricted to $D^{b}(\Ind \Alg_{\uc} / k)$
takes values in the derived category of torsion discrete $\Gal(\closure{k} / k)$-modules,
since unipotent quasi-algebraic groups in positive characteristic are torsion.
Applying this functor to the morphism in the theorem, we obtain a morphism
	\[
			R \Gamma_{x}(\Order_{K}^{\ur}, N^{\CDual})
		\to
			R \Hom_{\closure{k}^{\ind\rat}_{\pro\et}} \bigl(
				R \alg{\Gamma}(\Order_{K}, N),
				\Q / \Z
			\bigr)[-2]
	\]
in the derived category of torsion discrete $\Gal(\closure{k} / k)$-modules,
where $\Order_{K}^{\ur}$ is the maximal unramified extension of $\Order_{K}$.
We have
	\begin{align*}
				R \Hom_{\closure{k}^{\ind\rat}_{\pro\et}} \bigl(
					R \alg{\Gamma}(\Order_{K}, N),
					\Q / \Z
				\bigr)
		&	=
				R \Hom_{\Pro \Alg / \closure{k}} \bigl(
					R \alg{\Gamma}(\Order_{K}, N),
					\Q / \Z
				\bigr)
		\\	
		&	\Bigl(:=
				\dirlim_{n}
				R \Hom_{\Pro \Alg / \closure{k}} \bigl(
					R \alg{\Gamma}(\Order_{K}, N),
					\Z / n \Z
				\bigr)
			\Bigr)
	\end{align*}
by \eqref{prop: comparison thm for ind-proalgebraic groups}
\eqref{ass: RHom between IPAlg and IPAlg}.
Let $\Pro \FEt / k$ be the category of pro-finite-\'etale group schemes over $k$ and
$\pi_{0} \colon \Pro \Alg / k \to \Pro \FEt / k$ the maximal pro-finite-\'etale quotient functor as before.
We saw in \S \ref{sec: Generalities on derived categories of ind-procategories} that
$\Pro \Alg / k$ (and also $\Pro \FEt / k$) is the opposite of a Grothendieck category
and hence has enough (K-)projectives.
Hence $\pi_{0}$ admits a left derived functor
$L \pi_{0} \colon D(\Pro \Alg / k) \to D(\Pro \FEt / k)$.
The $n$-th functor $\pi_{n} := L_{n} \pi_{0}$ is
the $n$-th homotopy group functor defined by Serre \cite[\S 5.3]{Ser60}.
The functor $\pi_{0}$ is left adjoint to the inclusion functor
$\Pro \FEt / k \into \Pro \Alg / k$, which is exact.
Hence $\pi_{0}$ sends projectives to projectives.
By \cite[Prop.\ 13.3.13 (ii)]{KS06}, we have
	\[
			R \Hom_{\Pro \Alg / \closure{k}} \bigl(
				R \alg{\Gamma}(\Order_{K}, N),
				\Q / \Z
			\bigr)
		=
			R \Hom_{\Pro \FEt / \closure{k}} \bigl(
				L \pi_{0}
				R \alg{\Gamma}(\Order_{K}, N),
				\Q / \Z
			\bigr).
	\]
We have a distinguished triangle
	\[
			L \pi_{0} \alg{\Gamma}(\Order_{K}, N)
		\to
			L \pi_{0} R \alg{\Gamma}(\Order_{K}, N)
		\to
			L \pi_{0} \alg{H}^{1}(\Order_{K}, N)[-1].
	\]
Recall from \cite[\S 10.2, Thm.\ 2]{Ser60} that
$\pi_{n} = L_{n} \pi_{0} = 0$ for all $n \ge 2$.
The finiteness of $\alg{\Gamma}(\Order_{K}, N)$ implies that
$\pi_{1} \alg{\Gamma}(\Order_{K}, N) = 0$ and hence
$L \pi_{0} \alg{\Gamma}(\Order_{K}, N) = \alg{\Gamma}(\Order_{K}, N)$.
The connectedness of $\alg{H}^{1}(\Order_{K}, N)$ implies that
$L \pi_{0} \alg{H}^{1}(\Order_{K}, N) = \pi_{1} \alg{H}^{1}(\Order_{K}, N)[1]$.
Therefore $L \pi_{0} R \alg{\Gamma}(\Order_{K}, N)$ is concentrated in degree $0$
and the above triangle reduces to an exact sequence
	\[
			0
		\to
			\alg{\Gamma}(\Order_{K}, N)
		\to
			H^{0}
			L \pi_{0}
			R \alg{\Gamma}(\Order_{K}, N)
		\to
			\pi_{1} \alg{H}^{1}(\Order_{K}, N)
		\to
			0.
	\]
As in the proof of \eqref{prop: cohomology of integers}
\eqref{ass: cohomology of finite flats over integers},
let $0 \to N \to G \to H \to 0$ be an exact sequence
of group schemes over $\Order_{K}$ with $G, H$ smooth affine with connected fibers.
Then \eqref{prop: cohomology of integers}
\eqref{ass: cohomology of smooth group over integers} shows that
$R \alg{\Gamma}(\Order_{K}, G)$ is concentrated in degree zero
and $\alg{\Gamma}(\Order_{K}, G) \in \Pro \Alg / k$ is connected.
Hence
$L \pi_{0} R \alg{\Gamma}(\Order_{K}, G) = \pi_{1} \alg{\Gamma}(\Order_{K}, G)[1]$
and $L \pi_{0} R \alg{\Gamma}(\Order_{K}, H) = \pi_{1} \alg{\Gamma}(\Order_{K}, H)[1]$.
Therefore the distinguished triangle
	\[
			L \pi_{0}
			R \alg{\Gamma}(\Order_{K}, N)
		\to
			L \pi_{0}
			\alg{\Gamma}(\Order_{K}, G)
		\to
			L \pi_{0}
			\alg{\Gamma}(\Order_{K}, H)
	\]
is reduced to an exact sequence
	\[
			0
		\to
			\pi_{1} \alg{\Gamma}(\Order_{K}, G)
		\to
			\pi_{1} \alg{\Gamma}(\Order_{K}, H)
		\to
			H^{0}
			L \pi_{0}
			R \alg{\Gamma}(\Order_{K}, N)
		\to
			0.
	\]
This description shows the following intermediate result:

\begin{Prop} \label{prop: Bester group is the homotopy of cohomology}
	The complex $L \pi_{0} R \alg{\Gamma}(\Order_{K}, N)$ is concentrated in degree $0$
	whose cohomology is canonically identified with Bester's group $\mathscr{F}(N)$
	\cite[\S 1.3]{Bes78}.
\end{Prop}

Therefore our morphism after $R \Gamma(\closure{k}_{\pro\et}, \;\cdot\;)$
takes the form
	\[
			R \Gamma_{x}(\Order_{K}^{\ur}, N^{\CDual})
		\to
			\Hom_{\Pro \FEt / \closure{k}} \bigl(
				\mathscr{F}(N),
				\Q / \Z
			\bigr)[-2].
	\]
Recall again that the left-hand side is concentrated in degree $2$.

\begin{Thm} \label{thm: Bester's isomorphism}
	The morphism
		\[
				H_{x}^{2}(\Order_{K}^{\ur}, N^{\CDual})
			\to
				\mathscr{F}(N)^{\PDual}
		\]
	of torsion discrete $\Gal(\closure{k} / k)$-modules thus obtained is an isomorphism.
\end{Thm}

This is stated in \cite[\S 2.3, Thm.\ 3.1]{Bes78}.
We will directly prove this by proving \eqref{thm: duality for finite flats over integers}
in the next subsubsection.

\begin{Rmk} \label{rmk: Bester duality}
	Comparing our finite flat duality isomorphism above and Bester's isomorphism,
	which we omit here, is highly technical and complicated.
	We have to bring many constructions in Bester's work into derived categories of sheaves.
	Several technical fixes that we do not discuss here in detail are also necessary.
	In this paper, we avoid the comparison and fixes to save pages and prove Bester's duality independently.
	
	Below we merely indicate one issue that needs a fix.
	Assume for simplicity that the residue field $k$ is algebraically closed.
	Recall that the finite flat site $\Spec \Order_{K, \ffl}$ of $\Order_{K}$ is
	the category of finite flat $\Order_{K}$-algebras
	where a cover of an object $S$ is a finite family $\{S_{i}\}$ of finite flat $S$-algebras
	such that $S \to \prod S_{i}$ is faithfully flat.
	For a finite flat group scheme $N$ over $\Order_{K}$,
	Bester defined his group $\mathscr{F}(N)$ in \cite[\S 1.3, Def.\ 3.3]{Bes78}
	not only as a pro-finite-\'etale group over $k$
	but also as a \emph{sheaf} on the finite flat site $\Spec \Order_{K, \ffl}$.
	Let us denote this sheaf by $\mathscr{F}_{\sh}(N)$ for clarity.
	He stated the existence of a trace isomorphism
		\[
				H_{x}^{2}(\Order_{K}, \mathscr{F}_{\sh}(\mu_{p^{n}}))
			=
				\Z / p^{n} \Z
		\]
	for any $n \ge 1$ in \cite[\S 2.6, Lem.\ 6.2]{Bes78} and its proof.
	Here $\mathscr{F}_{\sh}(\mu_{p^{n}})$ is pulled back to $\Spec \Order_{K, \fppf}$
	(\cite[\S 1.2, Rmk.\ 2.2]{Bes78})
	and then taken the fppf cohomology with support (\cite[\S 2.2]{Bes78}).
	However, we can show that $\mathscr{F}_{\sh}(\mu_{p^{n}}) = 0$
	as a sheaf on $\Spec \Order_{K, \ffl}$,
	which contradicts to his statement.
	
	To prove $\mathscr{F}_{\sh}(\mu_{p^{n}}) = 0$,
	we first recall the definition of this sheaf.
	For a flat group scheme $G$ over $\Order_{K}$,
	recall from \cite[\S 1.2, i)]{Bes78} that
	another sheaf denoted by $\pi_{1}(\Tilde{G})$ on $\Spec \Order_{K, \ffl}$ is defined as follows.
	Assume for simplicity that $G$ is affine.
	For a finite flat local $\Order_{K}$-algebra $S$,
	let $\Res_{S / \Order_{K}}$ be the Weil restriction functor.
	Consider the Greenberg transform $\alg{\Gamma}(\Order_{K}, \Res_{S / \Order_{K}} G)$
	of $\Res_{S / \Order_{K}} G$
	and denote it by $\alg{\Gamma}(S, G)$.
	This is a proalgebraic group over $k$ representing the functor
		\[
				R
			\mapsto
				\Gamma(R[[T]], \Res_{S / \Order_{K}} G)
			=
				\Gamma(R[[T]] \tensor_{k[[T]]} S, G)
		\]
	on the category of $k$-algebras,
	where we identified $\Order_{K} = k[[T]]$.
	We have
		\[
				\alg{\Gamma}(S, G)
			=
				\invlim_{n} \Res_{(S / T^{n} S) / k} G.
		\]
	Define $\pi_{1}(\Tilde{G})$ to be the presheaf
	that assigns to each a finite flat $\Order_{K}$-algebra $S$
	and its decomposition $\prod S_{i}$ into local rings the group
		\[
			\prod_{i} \pi_{1} \alg{\Gamma}(S_{i}, G).
		\]
	This is in fact a sheaf by the faithfully flat descent and
	the left exactness of $\pi_{1}$ (\cite[\S 1.2, Lem.\ 2.1]{Bes78}).
	Here we treat $\pi_{1}(\Tilde{G})$ as a single set of notation
	and do not decompose it into the $\pi_{1}$ of anything in the usual sense.
	Let $N$ be a finite flat group scheme over $\Order_{K}$.
	Take an exact sequence $0 \to N \to G \to H \to 0$ of group schemes over $\Order_{K}$
	with $G$, $H$ smooth affine with connected fibers,
	as we took in the proof of \eqref{prop: cohomology of integers}
	\eqref{ass: cohomology of finite flats over integers}.
	Then the sheaf $\mathscr{F}_{\sh}(N) \in \Ab(\Order_{K, \ffl})$,
	as defined in \cite[\S 1.3, Def.\ 3.3]{Bes78},
	is the cokernel of the morphism of sheaves
		\begin{equation} \label{eq: definition of Bester sheaf}
				\pi_{1}(\Tilde{G}) \to \pi_{1}(\Tilde{H})
			\quad \text{in} \quad
				\Ab(\Order_{K, \ffl}).
		\end{equation}
	
	We show that $\mathscr{F}_{\sh}(\mu_{p^{n}}) = 0$ for any $n \ge 1$
	as a sheaf on $\Spec \Order_{K, \ffl}$.
	For this, it is enough to show that for any finite flat local $\Order_{K}$-algebra $S$,
	there exists a finite faithfully flat local $S$-algebra $S'$
	such that the homomorphism from the cokernel of
		\[
				\pi_{1} \alg{\Gamma}(S, \Gm)
			\stackrel{p^{n}}{\to}
				\pi_{1} \alg{\Gamma}(S, \Gm)
		\]
	to the cokernel of
		\[
				\pi_{1} \alg{\Gamma}(S', \Gm)
			\stackrel{p^{n}}{\to}
				\pi_{1} \alg{\Gamma}(S', \Gm)
		\]
	is zero.
	The group $\alg{\Gamma}(S, \Gm)$ is connected.
	Since $\Res_{R / k} \mu_{p^{n}}$ is connected for any finite $k$-algebra $R$,%
	\footnote{
		To see this,
		it is enough to show that
		the scheme-theoretically isomorphic group $\Res_{R / k} \alpha_{p^{n}}$ is connected.
		Let $\{r_{i}\}_{i = 1}^{d}$ be a $k$-basis of $R$.
		Write $r_{j}^{p^{n}} = \sum_{i} c_{i j} r_{i}$ with $c_{i j} \in k$.
		Let $C \colon \Ga^{d} \to \Ga^{d}$ be the $k$-linear morphism
		given by the matrix with entries $c_{i j}$.
		Then $\Res_{R / k} \alpha_{p^{n}} = \Ker(C F^{n})$ as a subgroup of $\Ga^{d} = \Res_{R / k} \Ga$,
		where $F \colon \Ga^{d} \to \Ga^{d}$ is the Frobenius over $k$.
		Hence $\Res_{R / k} \alpha_{p^{n}}$ is an extension of the vector group $\Ker C$ by $\alpha_{p^{n}}^{d}$,
		thus connected.
	}
	it follows that $\alg{\Gamma}(S, \mu_{p^{n}})$ is also connected
	(cf.\ \cite[\S 1.1, Lem.\ 1.2]{Bes78}).
	This implies
		\[
				\bigl(
					\pi_{1} \alg{\Gamma}(S, \Gm)
				\bigr) / p^{n}
			=
				\pi_{1} \bigl(
					\alg{\Gamma}(S, \Gm) / p^{n}
				\bigr),
		\]
	where $/ p^{n}$ denotes the cokernel of multiplication by $p^{n}$.
	Similar for $\alg{\Gamma}(S', \;\cdot\;)$.
	Hence it is enough to show that
	the morphism $\alg{\Gamma}(S, \Gm) \to \alg{\Gamma}(S', \Gm)$
	factors through the image $D$
	of the $p^{n}$-th power map $\alg{\Gamma}(S', \Gm) \to \alg{\Gamma}(S', \Gm)$.
	The groups $\alg{\Gamma}(S, \Gm)$, $\alg{\Gamma}(S', \Gm)$ and $D$ are
	inverse limits of smooth algebraic groups over $k$.
	Therefore it is enough to check the statement for $k$-points
	($k$ being algebraically closed).
	The requirement for $S'$ is now that
	every element of $\Gamma(S, \Gm)$ should become a $p^{n}$-th power in $\Gamma(S', \Gm)$.
	If $S = k[[T]][x_{1}, \dots, x_{n}] / (f_{1}, \dots, f_{m})$
	for some polynomials $f_{1}, \dots, f_{m}$ of the polynomial ring,
	then $S' = k[[T^{1 / p^{n}}]][x_{1}^{1 / p^{n}}, \dots, x_{n}^{1 / p^{n}}] / (f_{1}, \dots, f_{m})$
	does the job.
	This choice of $S'$ was suggested by Bertapelle.
	
	We should instead consider the mapping cone
	of the morphism \eqref{eq: definition of Bester sheaf} in $D(\Order_{K, \ffl})$.
	This is the first step for a fix,
	but we should stop here.
\end{Rmk}

\numberwithin{equation}{subsubsection}
\subsubsection{Proof of the finite flat duality}
\label{sec: Proof of the finite flat duality}

Let $N$ be a finite flat group scheme over $K$.
The morphism of functoriality of $R \Tilde{\alg{\Gamma}}$
\eqref{prop: localization and functoriality in pro-etale}
and the trace morphism \eqref{eq: trace morphism} give morphisms
	\begin{align*}
		&
				R \alg{\Gamma}(K, N^{\CDual})
			\to
				R \Tilde{\alg{\Gamma}} \bigl(
					K,
					R \sheafhom_{K}(N, \Gm)
				\bigr)
		\\
		&	\quad
			\to
				R \sheafhom_{k^{\ind\rat}_{\pro\et}} \bigl(
					R \alg{\Gamma}(K, N),
					\Z
				\bigr)
			=
				R \alg{\Gamma}(K, N)^{\SDual}
	\end{align*}

\begin{Thm} \label{thm: duality for finite flats over local fields}
	The morphism
		\[
				R \alg{\Gamma}(K, N^{\CDual})
			\to
				R \alg{\Gamma}(K, N)^{\SDual}
		\]
	in $D(k^{\ind\rat}_{\pro\et})$ defined above is an isomorphism.
\end{Thm}

This is proved in \cite[Thm.\ 2.7.1]{Suz13}
for the case $N$ does not have connected unipotent part.
We prove \eqref{thm: duality for finite flats over local fields} for general $N$,
in addition to and at the same time proving \eqref{thm: duality for finite flats over integers}.
Note that P\'epin \cite[\S 7.6.1]{Pep14} also proved a result equivalent to
\eqref{thm: duality for finite flats over local fields}
for general $N$.

\begin{Prop} \label{prop: localization and Serre dual, for finite flats}
	For a finite flat group scheme $N$ over $\Order_{K}$ with Cartier dual $N^{\CDual}$,
	we have a morphism of distinguished triangles
		\[
			\begin{CD}
					R \alg{\Gamma}(\Order_{K}, N^{\CDual})
				@>>>
					R \alg{\Gamma}(K, N^{\CDual})
				@>>>
					R \alg{\Gamma}_{x}(\Order_{K}, N^{\CDual})[1]
				\\
				@VVV
				@VVV
				@VVV
				\\
					R \alg{\Gamma}_{x}(\Order_{K}, N)^{\SDual}[-1]
				@>>>
					R \alg{\Gamma}(K, N)^{\SDual}
				@>>>
					R \alg{\Gamma}(\Order_{K}, N)^{\SDual}
			\end{CD}
		\]
	in $D(k^{\ind\rat}_{\pro\et})$.
\end{Prop}

\begin{proof}
	Apply the diagram in \eqref{prop: localization and functoriality in pro-etale}
	for $A = N$ and $B = \Gm$
	and note that the trace (iso)morphisms
	$R \alg{\Gamma}(K, \Gm) \to \Z$ and $R \alg{\Gamma}_{x}(\Order_{K}, \Gm)[1] \to \Z$
	are compatible with the morphism
	$R \alg{\Gamma}(K, \Gm) \to R \alg{\Gamma}_{x}(\Order_{K}, \Gm)[1]$
	by construction.
\end{proof}

\begin{Prop}
	To prove \eqref{thm: duality for finite flats over integers}
	and \eqref{thm: duality for finite flats over local fields},
	it is enough to show \eqref{thm: duality for finite flats over integers}
	for the case that the finite flat group scheme $N$ or $N^{\CDual}$
	has height $1$,
	i.e.\ has zero Frobenius.
\end{Prop}

\begin{proof}
	First note that if $0 \to N_{1} \to N_{2} \to N_{3} \to 0$ is
	an exact sequence of finite flat group schemes (over $\Order_{K}$ or $K$)
	and if the statements are true for any two of them, then so is for the other
	by the five lemma.
	The statements are classical and elementary if $N$ has order prime to $p$
	(cf.\ \cite[\S 3.2]{Ber03}).
	Assume $N$ (over $\Order_{K}$ or $K$) has $p$-power order.
	Then it has a filtration whose successive subquotients are either height $1$
	or have Cartier dual of height $1$ (\cite[\S 2.6, Lem.\ 6.1]{Bes78}).
	A finite flat group scheme over $K$ of height $1$ can be extended
	to a finite flat group scheme over $\Order_{K}$ of height $1$
	by \cite[III, Prop.\ B.4]{Mil06}.
\end{proof}

Let $\Omega_{\Order_{K}}^{1} = \Omega_{\Order_{K} / k}^{1}$ and $\Omega_{K}^{1} = \Omega_{K / k}^{1}$ be
the first differential modules.
We identify them as group schemes isomorphic to $\Ga$.
We have $R \alg{\Gamma}(\Order_{K}, \Ga) = \alg{O}_{K}$,
$R \alg{\Gamma}(K, \Ga) = \alg{K}$, and hence
$R \alg{\Gamma}_{x}(\Order_{K}, \Ga) = \alg{K} / \alg{O}_{K}[-1]$.
The residue map $\sum a_{n} T^{n} d T / T \mapsto a_{0}$ may be viewed as a morphism
	\[
			\Res
		\colon
			R \alg{\Gamma}(K, \Omega_{K}^{1})
		\to
			R \alg{\Gamma}_{x}(\Order_{K}, \Omega_{\Order_{K}}^{1})[1]
		\to
			\Ga
	\]
in $D(k^{\ind\rat}_{\pro\et})$.
Together with the coboundary map $\Ga \to \Z / p \Z[1]$ of the Artin-Schreier sequence
$0 \to \Z / p \Z \to \Ga \to \Ga \to 0$
and the coboundary map $\Z / p \Z[1] \to \Z[2]$ of the sequence
$0 \to \Z \to \Z \to \Z / p \Z \to 0$,
we have morphisms
	\[
			R \alg{\Gamma}(K, \Omega_{K}^{1})
		\to
			R \alg{\Gamma}_{x}(\Order_{K}, \Omega_{\Order_{K}}^{1})[1]
		\to
			\Z[2]
	\]
in $D(k^{\ind\rat}_{\pro\et})$,
which we call the \emph{additive trace morphisms}
(as opposed to the \emph{multiplicative} trace morphisms
$R \alg{\Gamma}(K, \Gm) \to R \alg{\Gamma}_{x}(\Order_{K}, \Gm)[1] = \Z$).

\begin{Prop} \label{prop: additive duality for each term}
	Let $M \times L \to \Omega_{\Order_{K}}^{1}$
	(resp.\ $M \times L \to \Omega_{K}^{1}$) be a perfect pairing of
	finite free $\Order_{K}$-modules (resp.\ finite free $K$-modules).
	Consider the morphisms
		\begin{align*}
					R \alg{\Gamma}_{x}(\Order_{K}, M)
			&	\to
					R \alg{\Gamma}_{x} \bigl(
						\Order_{K},
						R \sheafhom_{\Order_{K}}(L, \Omega_{\Order_{K}}^{1})
					\bigr)
			\\
			&	\to
					R \sheafhom_{k^{\ind\rat}_{\pro\et}} \bigl(
						R \alg{\Gamma}(\Order_{K}, L),
						R \alg{\Gamma}_{x}(\Order_{K}, \Omega_{\Order_{K}}^{1})
					\bigr)
			\\
			&	\to
					R \alg{\Gamma}(\Order_{K}, L)^{\SDual}[1]
		\end{align*}
	and similarly
		\[
				R \alg{\Gamma}(\Order_{K}, M)
			\to
				R \alg{\Gamma}_{x}(\Order_{K}, L)^{\SDual}[1]
		\]
	(resp.\
		\begin{align*}
					R \alg{\Gamma}(K, M)
			&	\to
					R \alg{\Gamma} \bigl(
						K,
						R \sheafhom_{K}(L, \Omega_{K}^{1})
					\bigr)
			\\
			&	\to
					R \sheafhom_{k^{\ind\rat}_{\pro\et}} \bigl(
						R \alg{\Gamma}(K, L),
						R \alg{\Gamma}(K, \Omega_{K}^{1})
					\bigr)
			\\
			&	\to
					R \alg{\Gamma}(K, L)^{\SDual}[2])
		\end{align*}
	in $D(k^{\ind\rat}_{\pro\et})$
	defined by the functoriality of $R \Tilde{\alg{\Gamma}}_{x}(\Order_{K}, \;\cdot\;)$
	(resp.\ $R \Tilde{\alg{\Gamma}}(K, \;\cdot\;)$)
	and the additive trace morphism.
	These are isomorphisms.
\end{Prop}

\begin{proof}
	The latter isomorphism (for cohomology of $K$)
	follows from the former two isomorphisms (for cohomology of $\Order_{K}$)
	by a morphism of distinguished triangles similar to the one in
	\eqref{prop: localization and Serre dual, for finite flats}.
	For the former two, it is enough to treat the morphism
		$
				R \alg{\Gamma}_{x}(\Order_{K}, M)
			\to
				R \alg{\Gamma}(\Order_{K}, L)^{\SDual}[1]
		$
	since $\alg{O}_{K} \cong \Ga^{\N}$, $\alg{K} / \alg{O}_{K} \cong \Ga^{\bigoplus \N}$ and
	$\alg{K} \cong \Ga^{\N} \oplus \Ga^{\bigoplus \N}$ are all Serre reflexive
	by \eqref{prop: Serre duality}
	\eqref{ass: Serre duality for unip and etale}.
	We may assume that $M = \Order_{K}$, $L = \Omega_{\Order_{K}}^{1}$
	and the pairing $M \times L \to \Omega_{\Order_{K}}^{1}$ is the multiplication.
	Then the morphism can be written as the composite
		\[
				\alg{K} / \alg{O}_{K}
			\to
				\sheafhom_{k^{\ind\rat}_{\pro\et}} \bigl(
					\alg{\Gamma}(\Order_{K}, \Omega_{\Order_{K}}^{1}),
					\Ga
				\bigr)
			\to
				\alg{\Gamma}(\Order_{K}, \Omega_{\Order_{K}}^{1})^{\SDual}[2],
		\]
	where the first morphism comes from the $k$-linear map
	$K / \Order_{K} \to \Hom_{\text{$k$-mod}}(\Omega_{\Order_{K}}^{1}, k)$,
	the residue map.
	We have
		\[
				\sheafhom_{k^{\ind\rat}_{\pro\et}} \bigl(
					\alg{\Gamma}(\Order_{K}, \Omega_{\Order_{K}}^{1}),
					\Ga
				\bigr)
			=
				\dirlim_{n \ge 1}
				\sheafhom_{k^{\ind\rat}_{\pro\et}} \bigl(
					\alg{\Gamma}(\Order_{K} / \ideal{p}_{K}^{n}, \Omega_{\Order_{K}}^{1}),
					\Ga
				\bigr)
		\]
	by \eqref{prop: comparison thm for ind-proalgebraic groups}
	\eqref{ass: RHom between IPAlg and LAlg}
	and \eqref{prop: schemes over local fields with ind-rational base}
	\eqref{ass: integral points are inverse limits}.
	Hence the mentioned morphism can also be written as the direct limit in $n \ge 1$ of
	the morphisms
		\[
				\alg{p}_{K}^{- n} / \alg{O}_{K}
			\to
				\sheafhom_{k^{\ind\rat}_{\pro\et}} \bigl(
					\alg{\Gamma}(\Order_{K} / \ideal{p}_{K}^{n}, \Omega_{\Order_{K}}^{1}),
					\Ga
				\bigr),
		\]
	where $\alg{p}_{K}^{- n}(k') = \ideal{p}_{K}^{- n} \tensor_{\Order_{K}} \alg{O}_{K}(k')$
	for $k' \in k^{\ind\rat}$.
	For each $n$, this morphism comes from the $k$-linear morphism
		\[
				\Res
			\colon
				\ideal{p}_{K}^{- n} / \Order_{K}
			\to
				\Hom_{\text{$k$-mod}} \bigl(
					\Gamma(\Order_{K} / \ideal{p}_{K}^{n}, \Omega_{\Order_{K}}^{1}),
					k
				\bigr),
		\]
	which is well-known (and easily seen) to be an isomorphism.
	Hence the Breen-Serre duality \cite[III, Lem.\ 0.13 (c)]{Mil06} (where the perfect \'etale site is used)
	and \eqref{prop: Serre duality}
	\eqref{ass: Serre duality for unip and etale} shows that
		\[
				\alg{p}_{K}^{- n} / \alg{O}_{K}
			\isomto
				\alg{\Gamma} \bigl(
					\Order_{K} / \ideal{p}_{K}^{n},
					\Omega_{\Order_{K}}^{1}
				\bigr)^{\SDual}[2].
		\]
	The direct limit in $n$ gives the result.
\end{proof}

We denote $K' = K^{1 / p}$ and $\Order_{K}' = \Order_{K}^{1 / p}$,
which are viewed as a $K$-algebra and an $\Order_{K}$-algebra, respectively,
via inclusions $K \into K'$ and $\Order_{K} \into \Order_{K}'$.
Let $F \colon \Spec K' \to \Spec K$ and $F \colon \Spec \Order_{K}' \to \Spec \Order_{K}$
be the natural morphisms.

Let $N$ be a finite flat group scheme over $\Order_{K}$ of height $1$.
Recall from \cite[Prop.\ 1.1, Lem.\ 2.2]{AM76} that there exist canonical exact sequences
	\[
			0
		\to
			N^{\CDual}
		\to
			V^{0}(N^{\CDual})
		\to
			V^{1}(N^{\CDual})
		\to
			0,
	\]
	\[
			0
		\to
			N
		\to
			F_{\ast} N
		\to
			U^{0}(N)
		\to
			U^{1}(N)
		\to
			0
	\]
of group schemes over $\Order_{K}$.
The terms $U^{i}(N)$, $V^{i}(N^{\CDual})$ are vector groups.
The second exact sequence above for $N = \mu_{p}$ is explicitly given by
	\[
			0
		\to
			\mu_{p}
		\to
			F_{\ast} \mu_{p}
		\to
			F_{\ast} \Omega_{\Order_{K}'}^{1}
		\to
			\Omega_{\Order_{K}}^{1}
		\to
			0,
	\]
where the middle morphism is $F_{\ast} \dlog$
and the right morphism is the $\Order$-linear Cartier operator $C$
minus the formal $p$-th power $W^{\ast}$
(\cite[Lem.\ 2.1, 2.2]{AM76}).
Let $U(N)$ be the complex $U^{0}(N) \to U^{1}(N)$ in degrees $0$ and $1$
and $V(N^{\CDual})$ the complex $V^{0}(N^{\CDual}) \to V^{1}(N^{\CDual})$ in degrees $0$ and $1$.
The morphism $U(N) \to N[1]$ induces isomorphisms
	\begin{gather*}
				R \Gamma(\alg{O}_{K}(k'), U(N))
			\isomto
				R \Gamma(\alg{O}_{K}(k'), N)[1],
		\\
				R \Gamma(\alg{K}(k'), U(N))
			\isomto
				R \Gamma(\alg{K}(k'), N)[1]
	\end{gather*}
for any $k' \in k^{\ind\rat}$ by \cite[Prop.\ 2.4]{AM76}.
(One checks that the proof there also works for our situation
$\Spec \alg{O}_{K}(k') \to \Spec k'$ and $\Spec \alg{K}(k') \to \Spec k'$.)
Therefore we have
	\begin{equation} \label{eq: resolution by vector groups for height one}
		\begin{gathered}
					R \alg{\Gamma}(\Order_{K}, U(N))
				\isomto
					R \alg{\Gamma}(\Order_{K}, N)[1],
			\\
					R \alg{\Gamma}(K, U(N))
				\isomto
					R \alg{\Gamma}(K, N)[1],
			\\
					R \alg{\Gamma}_{x}(\Order_{K}, U(N))
				\isomto
					R \alg{\Gamma}_{x}(\Order_{K}, N)[1].
		\end{gathered}
	\end{equation}
Also
	\begin{equation} \label{eq: resolution by vector groups for dual height one}
		\begin{gathered}
					R \alg{\Gamma}(\Order_{K}, V(N^{\CDual}))
				\cong
					R \alg{\Gamma}(\Order_{K}, N^{\CDual}),
			\\
					R \alg{\Gamma}(K, V(N^{\CDual}))
				\cong
					R \alg{\Gamma}(K, N^{\CDual}),
			\\
					R \alg{\Gamma}_{x}(\Order_{K}, V(N^{\CDual}))
				\cong
					R \alg{\Gamma}_{x}(\Order_{K}, N^{\CDual}).
		\end{gathered}
	\end{equation}

By \cite[Prop.\ 3.4]{AM76}, there exists a canonical pairing
	\[
		V(N^{\CDual}) \times U(N) \to U(\mu_{p})
	\]
of complexes of group schemes over $\Order_{K}$.
The parts
	\[
				V^{1}(N^{\CDual}) \times U^{0}(N)
			\to
				U^{1}(\mu_{p})
			=
				\Omega_{\Order_{K}}^{1},
		\quad
				V^{0}(N^{\CDual}) \times U^{1}(N)
			\to
				U^{1}(\mu_{p})
			=
				\Omega_{\Order_{K}}^{1}
	\]
are given by perfect pairings of finite free $\Order_{K}$-modules.
Let $Z(N)$ be the complex $F_{\ast} N \to U^{0}(N) \to U^{1}(N)$ of group schemes in degrees $0, 1, 2$,
which is a resolution of $N$.
By \cite[(4.7)]{AM76}, the diagram
	\begin{equation} \label{eq: Cartier dual pairing and coherent pairing}
		\begin{array}{ccccc}
			N & \times & N^{\CDual} & \longrightarrow & \mu_{p} \\
			\bigg \downarrow & & \bigg \| & & \bigg \downarrow \\
			Z(N) & \times & N^{\CDual} & \longrightarrow & Z(\mu_{p}) \\
			\bigg \uparrow & & \bigg \downarrow & & \bigg \uparrow \\
			U(N)[-1] & \times & V(N^{\CDual}) & \longrightarrow & U(\mu_{p})[-1]
		\end{array}
	\end{equation}
of pairings of complexes of group schemes over $\Order_{K}$ is commutative.

The complete discrete valuation field $K' = K^{1 / p}$ has residue field $k$.
We can define a sheaf of rings
$\alg{K}' = \alg{\Gamma}(K', \Ga) = \alg{\Gamma}(K, F_{\ast} \Ga)$ on $\Spec k^{\ind\rat}_{\pro\et}$
in a way similar to $\alg{K}$.
We also have a residue map
$\alg{\Gamma}(K, F_{\ast} \Omega_{K'}^{1}) = \alg{\Gamma}(K', \Omega_{K'}^{1}) \to \Ga$ for $K'$.

\begin{Prop} \label{prop: compatibility of additive and multiplicative trace morphisms}
	The residue map gives a morphism of complexes from
		\[
				R \alg{\Gamma}_{x}(\Order_{K}, U(\mu_{p}))
			=
				\left[
						\frac{
							\alg{\Gamma}(K, F_{\ast} \Omega_{K'}^{1})
						}{
							\alg{\Gamma}(\Order_{K}, F_{\ast} \Omega_{\Order_{K}'}^{1})
						}
					\to
						\frac{
							\alg{\Gamma}(K, \Omega_{K}^{1})
						}{
							\alg{\Gamma}(\Order_{K}, \Omega_{\Order_{K}}^{1})
						}
				\right][-2]
		\]
	to $[\Ga \stackrel{1 - F}{\to} \Ga][-2] \cong \Z / p \Z [-1]$.
	The morphism
		\[
				R \alg{\Gamma}_{x}(\Order_{K}, U(\mu_{p}))
			\to
				\Z / p \Z[-1]
		\]
	thus obtained and the morphism
		\[
				R \alg{\Gamma}_{x}(\Order_{K}, \mu_{p})[1]
			\to
				\Z / p \Z[-1]
		\]
	coming from $R \alg{\Gamma}(\Order_{K}, \mu_{p}) \cong \alg{U}_{K} / (\alg{U}_{K})^{p}[-1]$
	and
		$
				R \alg{\Gamma}(K, \mu_{p})
			\cong
				\alg{K}^{\times} / (\alg{K}^{\times})^{p}[-1]
			\stackrel{v_{K}}{\to}
				\Z / p \Z[-1]
		$
	are compatible with the identification
	\eqref{eq: resolution by vector groups for height one}.
\end{Prop}

\begin{proof}
	By \cite[Lem.\ 2.1]{AM76}, we have an exact sequence
		\[
				0
			\to
				\Gm
			\to
				F_{\ast} \Gm
			\stackrel{\dlog}{\to}
				F_{\ast} \Omega_{K'}^{1}
			\stackrel{C - W^{\ast}}{\to}
				\Omega_{K}^{1}
			\to
				0
		\]
	of group schemes over $K$.
	Since $\alg{H}^{1}(K, \Gm) = 0$,
	applying $R \alg{\Gamma}(K, \;\cdot\;)$ gives the exact sequence
	in the top row of the following diagram:
		\[
			\begin{CD}
					0
				@>>>
					\alg{K}^{\times}
				@>> \text{incl} >
					\alg{K}'^{\times}
				@>> \dlog >
					\alg{\Gamma}(K', \Omega_{K'}^{1})
				@>> C - W^{\ast} >
					\alg{\Gamma}(K, \Omega_{K}^{1})
				\\
				@.
				@VVV
				@VV v_{K'} V
				@VV \Res V
				@VV \Res V
				@.
				\\
				@.
					0
				@>>>
					\Z / p \Z
				@>>>
					\Ga
				@> 1 - F >>
					\Ga
				@>>>
					0.
			\end{CD}
		\]
	Here $v_{K'}$ is the normalized valuation for $K'$
	and $\alg{\Gamma}(K', \Omega_{K'}^{1}) \to \Ga$ is the residue map for $K'$.
	The result follows if we check that the squares are commutative.
	The left two squares are easily seen to be commutative.
	For the right square, let $T$ be a prime element of $K$ and write
		\[
				F_{\ast} \Omega_{K}^{1}
			=
				\bigoplus_{i = 0}^{p - 1} \Ga T^{1 / p} d T^{1 / p} / T^{1 / p}
			\cong
				\Ga^{p}
		\]
	as group schemes over $K$.
	The morphism $C - W^{\ast}$ sends an element
	$\sum f_{i} T^{i / p} d T^{1 / p} / T^{1 / p}$ to
		\[
				f_{0} d T / T
			-
				\sum (f_{i}^{p} T^{i}) d T / T.
		\]
	Its residue is $a_{0} - a_{0}^{p}$, where $a_{0}$ is the constant term of $f_{0}$,
	or the residue of $f_{0}$.
	This proves the commutativity.
\end{proof}

We also call the morphism
	\[
			R \alg{\Gamma}_{x}(\Order_{K}, U(\mu_{p}))
		\to
			\Z / p \Z[-1]
		\to
			\Z
	\]
thus obtained the additive trace morphism.

\begin{Prop} \label{prop: duality for vector group resolutions}
	Let $N$ be a finite flat group scheme of height $1$ over $\Order_{K}$.
	Then the morphisms
		\begin{align*}
					R \alg{\Gamma}_{x}(\Order_{K}, V(N^{\CDual}))
			&	\to
					R \sheafhom_{k^{\ind\rat}_{\pro\et}} \bigl(
						R \alg{\Gamma}(\Order_{K}, U(N)),
						R \alg{\Gamma}_{x}(\Order_{K}, U(\mu_{p}))
					\bigr)
			\\
			&	\to
					R \alg{\Gamma}(\Order_{K}, U(N))^{\SDual}
		\end{align*}
	and
		\begin{align*}
					R \alg{\Gamma}(\Order_{K}, V(N^{\CDual}))
			&	\to
					R \sheafhom_{k^{\ind\rat}_{\pro\et}} \bigl(
						R \alg{\Gamma}_{x}(\Order_{K}, U(N)),
						R \alg{\Gamma}_{x}(\Order_{K}, U(\mu_{p}))
					\bigr)
			\\
			&	\to
					R \alg{\Gamma}(\Order_{K}, U(N))^{\SDual}
		\end{align*}
	in $D(k^{\ind\rat}_{\pro\et})$ defined by the pairing
	$V(N^{\CDual}) \times U(N) \to U(\mu_{p})$,
	the functoriality of $R \Tilde{\alg{\Gamma}}_{x}$ and the additive trace morphism
	are isomorphisms.
\end{Prop}

\begin{proof}
	We only treat the first morphism as the second morphism can be treated similarly.
	We have a morphism between distinguished triangles
		\[
			\begin{CD}
					R \alg{\Gamma}_{x}(\Order_{K}, V(N^{\CDual}))
				@>>>
					R \alg{\Gamma}_{x}(\Order_{K}, V^{0}(N^{\CDual}))
				@>>>
					R \alg{\Gamma}_{x}(\Order_{K}, V^{1}(N^{\CDual}))
				\\
				@VVV
				@VVV
				@VVV
				\\
					R \alg{\Gamma}(\Order_{K}, U(N))^{\SDual}
				@>>>
					R \alg{\Gamma}(\Order_{K}, U^{1}(N))^{\SDual}[1]
				@>>>
					R \alg{\Gamma}(\Order_{K}, U^{0}(N))^{\SDual}[1]
			\end{CD}
		\]
	The right two morphisms are isomorphisms by
	\eqref{prop: additive duality for each term}.
	So is the left one.
\end{proof}

\begin{Prop}
	Let $N$ be a finite flat group scheme of height $1$ over $\Order_{K}$.
	The isomorphisms
		\[
					R \alg{\Gamma}_{x}(\Order_{K}, V(N^{\CDual}))
				\to
					R \alg{\Gamma}(\Order_{K}, U(N))^{\SDual},
			\quad
					R \alg{\Gamma}(\Order_{K}, V(N^{\CDual}))
				\to
					R \alg{\Gamma}(\Order_{K}, U(N))^{\SDual}
		\]
	in the previous proposition and the morphisms
		\[
					R \alg{\Gamma}_{x}(\Order_{K}, N^{\CDual})
				\to
					R \alg{\Gamma}(\Order_{K}, N)^{\SDual}[-1],
			\quad
					R \alg{\Gamma}(\Order_{K}, N^{\CDual})
				\to
					R \alg{\Gamma}(\Order_{K}, N)^{\SDual}[-1]
		\]
	given in \eqref{thm: duality for finite flats over integers}
	are compatible under the identifications
	\eqref{eq: resolution by vector groups for height one}
	and \eqref{eq: resolution by vector groups for dual height one}.
\end{Prop}

\begin{proof}
	Applying the functoriality of $R \Tilde{\alg{\Gamma}}_{x}$
	to the diagram \eqref{eq: Cartier dual pairing and coherent pairing},
	we have commutative diagrams
		\[
			\begin{CD}
					R \alg{\Gamma}_{x}(\Order_{K}, N^{\CDual})
				@>>>
					R \sheafhom_{k} \bigl(
						R \alg{\Gamma}(\Order_{K}, N),
						R \alg{\Gamma}_{x}(\Order_{K}, \mu_{p})
					\bigr)
				\\
				@|
				@|
				\\
					R \alg{\Gamma}_{x}(\Order_{K}, V(N^{\CDual}))
				@>>>
					R \sheafhom_{k} \bigl(
						R \alg{\Gamma}(\Order_{K}, U(N)),
						R \alg{\Gamma}_{x}(\Order_{K}, U(\mu_{p}))
					\bigr)
			\end{CD}
		\]
	and
		\[
			\begin{CD}
					R \alg{\Gamma}(\Order_{K}, N^{\CDual})
				@>>>
					R \sheafhom_{k} \bigl(
						R \alg{\Gamma}_{x}(\Order_{K}, N),
						R \alg{\Gamma}_{x}(\Order_{K}, \mu_{p})
					\bigr)
				\\
				@|
				@|
				\\
					R \alg{\Gamma}(\Order_{K}, V(N^{\CDual}))
				@>>>
					R \sheafhom_{k} \bigl(
						R \alg{\Gamma}_{x}(\Order_{K}, U(N)),
						R \alg{\Gamma}_{x}(\Order_{K}, U(\mu_{p}))
					\bigr)
			\end{CD}
		\]
	in $D(k^{\ind\rat}_{\pro\et})$.
	The morphism in \eqref{thm: duality for finite flats over integers}
	is given by applying the morphism
		\[
				R \alg{\Gamma}_{x}(\Order_{K}, \mu_{p})[1]
			\to
				\Z / p \Z[-1].
		\]
	The morphism in \eqref{prop: duality for vector group resolutions}
	is given by applying the morphism
		\[
				R \alg{\Gamma}_{x}(\Order_{K}, U(\mu_{p}))
			\to
				\Z / p \Z[-1].
		\]
	As they are compatible
	by \eqref{prop: compatibility of additive and multiplicative trace morphisms},
	the result follows.
\end{proof}

This finishes the proof of
\eqref{thm: duality for finite flats over integers}
and \eqref{thm: duality for finite flats over local fields}.


\numberwithin{equation}{subsubsection}
\subsubsection{Bertapelle's isomorphism}
\label{sec: Bertapelle's isomorphism}

\begin{Prop} \label{prop: comparison of Bertapelle isomorphism and our isomorphism}
	Assume that $K$ has equal characteristic and $k$ is algebraically closed.
	Let $A$ be a semistable abelian variety over $K$.
	Suppose that we take the morphism constructed in \eqref{thm: Bester's isomorphism}
	as the definition of Bester's isomorphism
	that Bertapelle used in \cite[Thm.\ 1]{Ber03}.
	Then the resulting isomorphism
		\[
				\psi_{A}
			\colon
				H^{1}(K, A^{\vee})
			\isomto
				\Ext_{k^{\ind\rat}_{\pro\et}}^{1}(\alg{\Gamma}(A), \Q / \Z)
		\]
	Bertapelle constructed in \cite[Thm.\ 2]{Ber03} coincides
	with our morphism $\theta_{A}^{+1}(k)$.
\end{Prop}

We prove this below.
For the moment, $k$ is assumed to be a general perfect field.
First we need a preparation about the finite flat site of $\Order_{K}$.

\begin{Prop} \label{prop: local rings are Henselian}
	Let $k' \in k^{\ind\rat}$.
	The local ring (for the Zariski topology) of $\alg{O}_{K}(k')$
	at any maximal ideal is Henselian.
\end{Prop}

\begin{proof}
	By \eqref{prop: basic properties of local fields with ind-rational base}
	\eqref{ass: maximal ideals of local fields with ind-rational base},
	a maximal ideal $\ideal{n}$ of $\alg{O}_{K}(k')$ is
	of the form $\alg{p}_{K}(k') + \alg{O}_{K}(\ideal{m})$
	for some $\ideal{m} \in \Spec k'$, with residue field $k' / \ideal{m}$.
	By \eqref{prop: basic properties of local fields with ind-rational base}
	\eqref{ass: Zariski topology of local fields with ind-rational base},
	the local ring $\alg{O}_{K}(k')_{\ideal{n}}$
	is given by the filtered direct limit of $\alg{O}_{K}(k'[1 / e]) = \alg{O}_{K}(k' / (1 - e))$
	for the idempotents $e \in k' \setminus \ideal{m}$.
	
	Let $f_{1}, \dots, f_{n} \in \alg{O}_{K}(k')[x_{1}, \dots, x_{n}]$
	and $\alpha = (\alpha_{1}, \dots, \alpha_{n}) \in \alg{O}_{K}(k')^{n}$.
	Suppose that the images $\Bar{f}_{i}(\alpha) \in k' / \ideal{m}$ are zero for all $i$
	and $\det (\partial \Bar{f}_{i} / \partial x_{j})(\alpha) \ne 0$.
	We need to show that the polynomial system $(f_{1}, \dots, f_{n})$ has a root
	in $\alg{O}_{K}(k')_{\ideal{n}}$ whose reduction is $\Bar{\alpha} \in k' / \ideal{m}$.
	It is enough to show the existence of a root in $\alg{O}_{K}(k')$
	after replacing $k'$ by $k'[1 / e]$ for an idempotent $e \in k' \setminus \ideal{m}$.
	Since the element $\det (\partial f_{i} / \partial x_{j})(\alpha) \in \alg{O}_{K}(k')$
	has non-zero image in $k' / \ideal{m}$,
	\eqref{prop: basic properties of local fields with ind-rational base}
	\eqref{ass: invertibility of series in ind-rational coefficients}
	shows that $\det (\partial f_{i} / \partial x_{j})(\alpha) \in \alg{U}_{K}(k')$
	after replacing $k'$.
	Since $\alg{O}_{K}(k')$ is $\alg{p}_{K}(k')$-adically complete,
	Hensel's lemma shows the existence of a desired root.
\end{proof}

Recall that a morphism of schemes is finite locally free
if and only if it is finite flat locally of finite presentation
(\cite[Chap.\ II, \S 5, No.\ 2, Cor.\ 2 to Thm.\ 1]{Bou98},
\cite[Prop.\ 1.4.7]{Gro64}).
For a commutative ring $S$,
we denote by $\Spec S_{\ffl}$ the finite flat site of $S$,
namely the category of finite locally free $S$-algebras
where a covering of an object $S'$ is
a finite family $\{S'_{i}\}$ of finite locally free $S'$-algebras
such that $\prod S'_{i}$ is faithfully flat over $S'$.

\begin{Prop} \label{prop: push from fppf to ffl is exact}
	Let $k' \in k^{\ind\rat}$
	and $S$ a finite $\alg{O}_{K}(k')$-algebra.
	Then any fppf covering of $\Spec S$ can be refined by a finite locally free covering.
	That is, for any faithfully flat $S$-algebra $S'$ of finite presentation,
	there exist a faithfully flat finite locally free $S$-algebra $S''$
	and an $S$-algebra homomorphism $S' \to S''$.
	In particular, the continuous map $g \colon \Spec S_{\fppf} \to \Spec S_{\ffl}$ of sites
	defined by the identity induces an exact pushforward functor
	$g_{\ast} \colon \Set(S_{\fppf}) \to \Set(S_{\ffl})$.
\end{Prop}

\begin{proof}
	Refine $\Spec S'$ by a quasi-finite flat covering $\Spec S'''$ of finite presentation
	(\cite[17.16.2]{Gro67}).
	Since $S$ is finite, the previous proposition shows that
	the local ring of $S$ at any maximal ideal of $S$ is Henselian.
	Hence $\Spec S'''$ can be refined by a finite locally free covering.
\end{proof}

The following is stated in \cite[Rmk.\ 2.7.6 (3)]{Suz13}.

\begin{Prop}
	Let $j \colon \Spec K_{\fppf} / k^{\ind\rat} \into \Spec O_{K, \fppf} / k^{\ind\rat}$
	be the morphism induced by the open immersion $\Spec K \into \Spec \Order_{K}$.
	Denote by $j_{!}$ the zero extension functor by $j$ (\cite[Exp.\ III, 5.3, 3)]{AGV72a}).
	Then $R \alg{\Gamma}(\Order_{K}, j_{!} A) = 0$
	for any $A \in \Ab(K_{\fppf} / k^{\ind\rat}_{\et})$.
\end{Prop}

\begin{proof}
	It is enough to show that
	$R \Gamma(\alg{O}_{K}(k'), j_{!} A) = 0$ for any $k' \in k^{\ind\rat}$.
	Let $g \colon \Spec \alg{O}_{K}(k')_{\fppf} \to \Spec \alg{O}_{K}(k')_{\ffl}$
	be the continuous map defined by the identity.
	We have $R \Gamma(\alg{O}_{K}(k'), j_{!} A) = R \Gamma(\alg{O}_{K}(k')_{\ffl}, g_{\ast} j_{!} A)$
	by the previous proposition.
	The $j$ here should be understood as the open immersion
	$\Spec \alg{K}(k') \into \Spec \alg{O}_{K}(k')$.
	Let $S$ be a non-zero finite locally free $\alg{O}_{K}(k')$-algebra of (locally constant) rank $r \ge 1$.
	Then the norm $N_{S / \alg{O}_{K}(k')} \pi = \pi^{r}$
	does not belong to $\alg{U}_{K}(k')$.
	Hence the homomorphism $\alg{O}_{K}(k') \to S$ cannot factor through $\alg{K}(k')$.
	Therefore $g_{\ast} j_{!} = 0$
	by the construction of the zero-extension functor $j_{!}$.
\end{proof}

Bertapelle extended Bester's group $\mathscr{F}(N)$
for quasi-finite flat separated group schemes $N$ over $\Order_{K}$.
For such a group $N$, we denote its finite part by $N^{f}$.
By \cite[\S 3.1 Example]{Ber03}, we have $\mathscr{F}(N) = \mathscr{F}(N^{f})$.

\begin{Prop} \label{prop: cohomology of quasi-finite flats is given by the finite part}
	Let $N$ be a quasi-finite flat separated group scheme over $\Order_{K}$
	with finite part $N^{f}$.
	Then we have $R \alg{\Gamma}(\Order_{K}, N) = R \alg{\Gamma}(\Order_{K}, N^{f}) \in D(\Pro \Alg / k)$
	and it is P-acyclic.
	In particular, there exists a canonical isomorphism
	$L \pi_{0} R \alg{\Gamma}(\Order_{K}, N) \cong \mathscr{F}(N)$
	compatible with the isomorphism
	$L \pi_{0} R \alg{\Gamma}(\Order_{K}, N^{f}) \cong \mathscr{F}(N^{f})$
	of \eqref{prop: Bester group is the homotopy of cohomology}.
\end{Prop}

\begin{proof}
	Let $Y = N \setminus N^{f}$ be the part of $N$ finite over $K$.
	Let $k' \in k^{\ind\rat}$ and $S$ a finite locally free $\alg{O}_{K}(k')$-algebra.
	If there is an $\alg{O}_{K}(k')$-scheme morphism $\Spec S \to Y$,
	then $S$ has to be a $\alg{K}(k')$-algebra,
	which happens only when $S = 0$ by
	the same reasoning as the proof of the previous proposition.
	Hence $\Gamma(S, N^{f}) = \Gamma(S, N)$.
	Therefore $N^{f} = N$ as sheaves on $\Spec \alg{O}_{K}(k')_{\ffl}$.
	Hence \eqref{prop: push from fppf to ffl is exact} implies that
	$R \Gamma(\alg{O}_{K}(k'), N^{f}) = R \Gamma(\alg{O}_{K}(k'), N)$,
	and we have $R \alg{\Gamma}(\Order_{K}, N^{f}) = R \alg{\Gamma}(\Order_{K}, N)$.
\end{proof}

Now we recall the definition of Bertapelle's isomorphism
in the case the abelian variety $A$ has semistable reduction.
Assume that $k$ is algebraically closed.
Let $\mathcal{A}$ be the N\'eron model of $A$
and $\mathcal{A}_{0}$ the maximal open subgroup scheme with connected special fiber.
Let $n \ge 1$ be an integer that kills $\pi_{0}(\mathcal{A}_{x})$.
Let $\mathcal{A}[n]^{f}$ be the finite part of the $n$-torsion part $\mathcal{A}[n]$
and $(\mathcal{A}[n]^{f})_{K} = \mathcal{A}[n]^{f} \times_{\Order_{K}} K$.
The inclusion $(\mathcal{A}[n]^{f})_{K} \into A[n]$ defines a surjection
$A[n]^{\CDual} \onto (\mathcal{A}[n]^{f})_{K}^{\CDual}$.
By \cite[Lem.\ 14]{Ber03}, this canonically extends to a morphism
	\[
			\mathcal{A}^{\vee}[n]
		\to
			(\mathcal{A}[n]^{f})^{\CDual}
	\]
of group schemes over $\Order_{K}$, which induces an isomorphism
	\[
			H_{x}^{2}(\Order_{K}, \mathcal{A}^{\vee}[n])
		\isomto
			H_{x}^{2}(\Order_{K}, (\mathcal{A}[n]^{f})^{\CDual}).
	\]
We denote by $\theta_{\mathcal{A}[n]^{f}}^{2}$ the isomorphism
	\[
			H_{x}^{2}(\Order_{K}, (\mathcal{A}[n]^{f})^{\CDual})
		\isomto
			\mathscr{F}(\mathcal{A}[n]^{f})^{\PDual}
	\]
in \eqref{thm: Bester's isomorphism} for $N = \mathcal{A}[n]^{f}$.
Combining the above two isomorphisms and $\mathscr{F}(\mathcal{A}[n]^{f}) \cong \mathscr{F}(\mathcal{A}[n])$,
we have an isomorphism
	\[
			H_{x}^{2}(\Order_{K}, \mathcal{A}^{\vee}[n])
		\isomto
			\mathscr{F}(\mathcal{A}[n])^{\PDual}.
	\]
We denote this isomorphism by $\psi_{\mathcal{A}[n]}$.
Note that
	\[
			H_{x}^{2}(\Order_{K}, \mathcal{A}_{0}^{\vee})
		=
			H_{x}^{2}(\Order_{K}, \mathcal{A}^{\vee})
		=
			H^{1}(K, A^{\vee}).
	\]
Hence the exact sequence
	$
			0
		\to
			\mathcal{A}^{\vee}[n]
		\to
			\mathcal{A}^{\vee}
		\stackrel{n}{\to}
			\mathcal{A}_{0}^{\vee}
		\to
			0
	$
induces a surjection
	\[
			H_{x}^{2}(\Order_{K}, \mathcal{A}^{\vee}[n])
		\onto
			H^{1}(K, A^{\vee})[n].
	\]
By \cite[\S 3.1, Examples]{Ber03},
the map $\mathscr{F}(\mathcal{A}_{0}[n]) \to \mathscr{F}(\mathcal{A}[n])$ is injective,
and we have
	\begin{align*}
				\mathscr{F}(\mathcal{A}_{0}[n])
		&	=
				\Coker \bigl(
						\pi_{1} \alg{\Gamma}(\Order_{K}, \mathcal{A}_{0})
					\stackrel{n}{\to}
						\pi_{1} \alg{\Gamma}(\Order_{K}, \mathcal{A}_{0})
				\bigr)
		\\
		&	=
				\Coker \bigl(
						\pi_{1} \alg{\Gamma}(\Order_{K}, \mathcal{A})
					\stackrel{n}{\to}
						\pi_{1} \alg{\Gamma}(\Order_{K}, \mathcal{A})
				\bigr)
		\\
		&	=
				\Coker \bigl(
						\pi_{1} \alg{\Gamma}(K, A))
					\stackrel{n}{\to}
						\pi_{1} \alg{\Gamma}(K, A))
				\bigr).
	\end{align*}
By the proof of \cite[Thm.\ 2]{Ber03}, the isomorphism
	$
			\psi_{\mathcal{A}[n]}
		\colon
			H_{x}^{2}(\Order_{K}, \mathcal{A}^{\vee}[n])
		\isomto
			\mathscr{F}(\mathcal{A}[n])^{\PDual}
	$
induces an isomorphism
	\[
			H^{1}(K, A^{\vee})[n]
		\isomto
			(\pi_{1} \alg{\Gamma}(K, A))^{\PDual}[n]
	\]
on the quotients.
The resulting isomorphism
	\[
			H^{1}(K, A^{\vee})
		\isomto
			(\pi_{1} \alg{\Gamma}(K, A))^{\PDual}
	\]
is the definition of Bertapelle's isomorphism in the semistable case.
We denote it by $\psi_{A}$.

Hence, to prove \eqref{prop: comparison of Bertapelle isomorphism and our isomorphism},
it is enough to show the following.

\begin{Prop}
	Let $A$ and $n$ be as above.
	There exists a diagram (commutativity to be mentioned below)
		\[
			\begin{CD}
					H^{1}(K, A^{\vee}[n])
				@>>>
					H_{x}^{2}(\Order_{K}, \mathcal{A}^{\vee}[n])
				@>>>
					H^{1}(K, A^{\vee})[n]
				\\
				@VV \theta_{A[n]}^{1} V
				@VV \psi_{\mathcal{A}[n]} V
				@VV \theta_{A}^{+1} \text{ or } \psi_{A} V
				\\
					H^{0} \bigl(
						\mathscr{F}(A[n])^{\PDual}
					\bigr)
				@>>>
					\mathscr{F}(\mathcal{A}[n])^{\PDual}
				@>>>
					(\pi_{1} \alg{\Gamma}(K, A))^{\PDual}[n].
			\end{CD}
		\]
	Here we define
		\[
				\mathscr{F}(A[n])^{\PDual}
			=
				R \Hom_{k}(R \alg{\Gamma}(K, A[n]), \Q / \Z).
		\]
	The left vertical morphism $\theta_{A[n]}^{1}$ is
	the isomorphism induced on $H^{1}$ by the isomorphism given in
	\eqref{thm: duality for finite flats over local fields}.
	We denoted $\theta_{A}^{+1} = \theta_{A}^{+1}(k)$.
	The horizontal homomorphisms in the left square are to be constructed below,
	and those in the right square are already mentioned.
	The upper horizontal homomorphisms are surjective.
	
	The left square is commutative.
	If we use Bertapelle's isomorphism $\psi_{A}$ for the right vertical arrow,
	then the right square is commutative (by definition).
	If we use our morphism $\theta_{A}^{+1}$,
	then the total square (omitting the middle $\psi_{\mathcal{A}[n]}$) is commutative.
\end{Prop}

As a consequence, we have $\theta_{A}^{+1} = \psi_{A}$
and the right square is commutative.

\begin{proof}
	The upper horizontal homomorphism in the left square is defined as the coboundary map
	of the localization sequence for $\mathcal{A}^{\vee}[n]$.
	It is surjective since
		\[
				R \Gamma(\Order_{K}, \mathcal{A}^{\vee}[n])
			=
				R \Gamma(\Order_{K}, \mathcal{A}^{\vee}[n]^{f})
		\]
	by \eqref{prop: cohomology of quasi-finite flats is given by the finite part}
	and this is concentrated in degrees $0$ and $1$.
	The composite of the upper two horizontal homomorphisms is induced by
	the inclusion $A^{\vee}[n] \into A^{\vee}$.
	By \eqref{prop: cohomology of quasi-finite flats is given by the finite part}
	and by the same calculation as the second paragraph after
	\eqref{thm: duality for finite flats over integers},
	we have
		\begin{align*}
					\mathscr{F}(\mathcal{A}[n])^{\PDual}
			&	=
					R \Hom_{\Pro \FEt / k} \bigl(
						L \pi_{0} R \alg{\Gamma}(\Order_{K}, \mathcal{A}[n]),
						\Q / \Z
					\bigr)
			\\
			&	=
					R \Hom_{k} \bigl(
						R \alg{\Gamma}(\Order_{K}, \mathcal{A}[n]),
						\Q / \Z
					\bigr).
		\end{align*}
	The lower horizontal homomorphism in the left square is defined by dualizing
		$
				R \alg{\Gamma}(\Order_{K}, \mathcal{A}[n])
			\to
				R \alg{\Gamma}(K, A[n])
		$.
	The left square is decomposed into
		\[
			\begin{CD}
					H^{1}(K, A^{\vee}[n])
				@>>>
					H^{1} \bigl(
						K, (\mathcal{A}[n]^{f})_{K}^{\CDual}
					\bigr)
				@>>>
					H_{x}^{2} \bigl(
						\Order_{K}, (\mathcal{A}[n]^{f})^{\CDual}
					\bigr)
				\\
				@VV \theta_{A[n]}^{1} V
				@VV \theta_{(\mathcal{A}[n]^{f})_{K}}^{1} V
				@VV \theta_{\mathcal{A}[n]^{f}}^{2} V
				\\
					\mathscr{F}(A[n])^{\PDual}
				@>>>
					\mathscr{F} \bigl(
						(\mathcal{A}[n]^{f})_{K}
					\bigr)^{\PDual}
				@>>>
					\mathscr{F} \bigl(
						\mathcal{A}[n]^{f}
					\bigr)^{\PDual}.
			\end{CD}
		\]
	The commutativity of the left square is the functoriality of
	the isomorphisms $\theta_{N}^{1}$ for finite flat group schemes $N$ over $K$,
	which is easy to check.
	That of the right square follows from
	\eqref{prop: localization and Serre dual, for finite flats}.
	
	We show the commutativity of the total square
	with $\theta_{A}^{+1}$ used for the right vertical arrow.
	With the identification
	$\mathscr{F}(\mathcal{A}_{0}[n])^{\PDual} = \pi_{1} \alg{\Gamma}(K, A)^{\PDual}[n]$
	used earlier,
	the lower two horizontal homomorphisms can be written as
	the $H^{0}$ of the Pontryagin dual of the morphisms
		\[
				L \pi_{0} R \alg{\Gamma}(\Order_{K}, \mathcal{A}_{0}[n])
			\to
				L \pi_{0} R \alg{\Gamma}(\Order_{K}, \mathcal{A}[n])
			\to
				L \pi_{0} R \alg{\Gamma}(K, A[n]).
		\]
	The $H^{0}$ of the composite of these morphisms is
		\[
				\Coker \bigl(
						\pi_{1} \alg{\Gamma}(K, A)
					\stackrel{n}{\to}
						\pi_{1} \alg{\Gamma}(K, A)
				\bigr)
			\to
				H^{0} L \pi_{0} R \alg{\Gamma}(K, A[n]).
		\]
	This comes from the $H^{0}$ of the distinguished triangle
		\[
				L \pi_{0} R \alg{\Gamma}(K, A)[-1]
			\stackrel{n}{\to}
				L \pi_{0} R \alg{\Gamma}(K, A)[-1]
			\to
				L \pi_{0} R \alg{\Gamma}(K, A[n]).
		\]
	Therefore it is enough to construct a morphism of distinguished triangles
		\[
			\begin{CD}
					R \alg{\Gamma}(K, A^{\vee}[n])
				@>>>
					R \alg{\Gamma}(K, A^{\vee})
				@>> n >
					R \alg{\Gamma}(K, A^{\vee})
				\\
				@VVV
				@VVV
				@VVV
				\\
					R \alg{\Gamma}(K, A[n])^{\SDual}
				@>>>
					R \alg{\Gamma}(K, A)^{\SDual}[1]
				@> n >>
					R \alg{\Gamma}(K, A)^{\SDual}[1].
			\end{CD}
		\]
	(The left square gives the total square in the statement.)
	For this, note that we have a morphism of triangles from the short exact sequence
	$0 \to A^{\vee}[n] \to A^{\vee} \to A^{\vee} \to 0$
	to
		\[
				R \sheafhom_{K}(A[n], \Gm)
			\to
				R \sheafhom_{K}(A, \Gm)[1]
			\to
				R \sheafhom_{K}(A, \Gm)[1].
		\]
	Applying the morphism of functoriality of $R \alg{\Gamma}$ to it
	and using the trace morphism $R \alg{\Gamma}(K, \Gm) \to \Z$,
	we get the desired morphism of the distinguished triangles.
\end{proof}

This completes the proof of
\eqref{prop: comparison of Bertapelle isomorphism and our isomorphism}
and hence \eqref{prop: comparison with Bester-Bertapelle}.


\numberwithin{equation}{subsection}
\subsection{B\'egueri's isomorphism}
\label{sec: Begueri's isomorphism}

\begin{Prop} \label{prop: comparison with Begueri}
	Assume that $k$ is algebraically closed.
	The morphism
		\[
				\theta_{A}^{+1}(k)
			\colon
				H^{1}(K, A^{\vee})
			\to
				\Ext_{k^{\ind\rat}_{\pro\et}}^{1}(\alg{\Gamma}(A), \Q / \Z)
		\]
	of \eqref{prop: two parts of the duality morphism}
	coincides with B\'egueri's isomorphism
	\cite[Thm.\ 8.3.6]{Beg81}
	when $K$ has mixed characteristic.
\end{Prop}

\begin{proof}
	Let $n \ge 1$.
	We have morphisms
		\begin{align*}
					R \Hom_{K}(\Z / n \Z, A^{\vee})
			&	\to
					R \Hom_{K} \bigl(
						\Z / n \Z, R \sheafhom_{K}(A, \Gm)
					\bigr)[1]
			\\
			&	\to
					R \Hom_{K} \bigl(
						A, R \sheafhom_{K}(\Z / n \Z, \Gm)
					\bigr)[1]
			\\
			&	=
					R \Hom_{K}(A, \mu_{n})[1]
			\\
			&	\to
					R \Hom_{k^{\ind\rat}_{\pro\et}}(
						R \alg{\Gamma}(A), R \alg{\Gamma}(\mu_{n})
					)[1]
			\\
			&	\to
					R \Hom_{k^{\ind\rat}_{\pro\et}}(
						R \alg{\Gamma}(A), \Z / n \Z
					)
		\end{align*}
	in $D(\Ab)$, where the last morphism comes from the Kummer sequence
	and the trace morphism $R \alg{\Gamma}(\Gm) = \alg{K}^{\times} \onto \Z$.
	Using the sequence $0 \to \Z \to \Z \to \Z / n \Z \to 0$,
	this and our duality morphism fit in the following commutative diagram:
		\[
			\begin{CD}
					R \Hom_{K}(\Z / n \Z, A^{\vee})
				@>>>
					R \Hom_{k^{\ind\rat}_{\pro\et}}(
						R \alg{\Gamma}(A), \Z / n \Z
					)
				\\
				@VVV
				@VVV
				\\
					R \Gamma(A^{\vee})
				@>>>
					R \Hom_{k^{\ind\rat}_{\pro\et}}
						(R \alg{\Gamma}(A), \Z
					)[1].
			\end{CD}
		\]
	
	By construction, the top horizontal morphism in degree $1$ is given as follows.
	First, for $B, C \in \Ab(K_{\fppf} / k^{\ind\rat}_{\et})$,
	the sheaf
		$
				\Hom_{\alg{K}}(B, C)
			:=
				\Tilde{\alg{\Gamma}} \sheafhom_{K}(B, C)
		$
	on $\Spec k^{\ind\rat}_{\pro\et}$ is the pro-\'etale sheafification of the \'etale sheaf
	$k' \mapsto \Hom_{(\alg{K}(k'), k')}(B, C)$,
	where $\Hom_{(\alg{K}(k'), k')}$ is the Hom functor for the localization of
	$\Spec K_{\fppf} / k^{\ind\rat}_{\et}$ at the object $(\alg{K}(k'), k')$.
	For any $n$, the $n$-th cohomology $\Ext_{\alg{K}}^{n}(B, C)$ of the derived functor
		$
				R \Hom_{\alg{K}}(B, C)
			=
				R \Tilde{\alg{\Gamma}} R \sheafhom_{K}(B, C)
		$
	is the pro-\'etale sheafification of the presheaf
	$k' \mapsto \Ext_{(\alg{K}(k'), k')}^{n}(B, C)$.
	Now let $0 \to A^{\vee} \to X \to \Z / n \Z \to 0$ be an extension over $K$.
	The long exact sequence for $R \Hom_{\alg{K}}$ gives
		\[
				0
			\to
				\Ext_{\alg{K}}^{1}(\Z / n \Z, \Gm)
			\to
				\Ext_{\alg{K}}^{1}(X, \Gm)
			\to
				\Ext_{\alg{K}}^{1}(A^{\vee}, \Gm)
			\to
				0,
		\]
	or
		\[
				0
			\to
				\alg{H}^{1}(\mu_{n})
			\to
				\Ext_{\alg{K}}^{1}(X, \Gm)
			\to
				\alg{\Gamma}(A)
			\to
				0,
		\]
	hence an element of $\Ext_{k^{\ind\rat}_{\pro\et}}^{1}(\alg{\Gamma}(A), \alg{H}^{1}(\mu_{n}))$,
	hence an element of $\Ext_{k^{\ind\rat}_{\pro\et}}^{1}(\alg{\Gamma}(A), \Z / n \Z)$.
	This homomorphism fits in the following commutative diagram:
		\[
			\begin{CD}
					\Ext_{K}^{1}(\Z / n \Z, A^{\vee})
				@>>>
					\Ext_{k^{\ind\rat}_{\pro\et}}^{1}(\alg{\Gamma}(A), \Z / n \Z)
				\\
				@VVV
				@VVV
				\\
					H^{1}(A^{\vee})
				@>>>
					\Ext_{k^{\ind\rat}_{\pro\et}}^{1}(\alg{\Gamma}(A), \Q / \Z),
			\end{CD}
		\]
	where the bottom arrow is our homomorphism $\theta_{A}^{+1}(k)$.
	The description above shows that the top arrow is the same as B\'egueri's homomorphism
	(\cite[Lem.\ 8.2.2, Thm.\ 8.3.6]{Beg81}).
	This completes the comparison.
\end{proof}


\numberwithin{equation}{section}
\section{Galois descent}
\label{sec: Galois descent}

\begin{Prop} \label{prop: reduction over a finite extension}
	Let $A$ be an abelian variety over $K$
	and $L / K$ a finite Galois extension.
	If \eqref{thm: duality for abelian varieties} is for $A$ over $L$,
	then so is for $A$ over $K$.
\end{Prop}

We prove this below.
We need notation and lemmas on group (co)homology.
The basic reference is \cite[VII, VIII, IX]{Ser79}.
We need Tate cohomology in the setting of derived categories of sheaves of $G$-modules on sites.
For the purpose of this section,
available literature seems to be either less general
or more topological (in the sense of equivariant stable homotopy theory)
than what is needed here.
Hence we include some basics for the convenience of the reader.

Let $G$ be a finite (abstract) group and $S$ a site.
We denote the category of (left) $G$-modules (or $\Z[G]$-modules) by $\Mod{G}$,
so that $\Mod{G}(S)$ is
the (abelian) category of sheaves of $G$-modules on $S$.
For a complex $M \in D(\Mod{G}(S))$,
we define its group cohomology, group homology by
	\begin{gather*}
				R \Gamma(G, M)
			=
				R \sheafhom_{\Mod{G}(S)}(\Z, M),
			\quad
				L \Delta(G, M)
			=
				\Z \tensor_{\Z[G]}^{L} M,
	\end{gather*}
respectively.%
\footnote{
	Using $\Delta$, the Greek letter next to $\Gamma$, to denote homology is non-standard.
	There seems to be no widely used notation for homology in derived categories
	that is parallel to cohomology $R \Gamma$.
	Perhaps $\Lambda$ instead,
	in accordance with the definition of the homology $X \mapsto E \wedge X$ of a spectrum $E$?
}
They are objects of $D(S) = D(\Ab(S))$.
Let $C(G)$ be the standard resolution
	\[
			\cdots
		\to
			\Z[G^{3}]
		\to
			\Z[G^{2}]
		\to
			\Z[G]
	\]
of the trivial $G$-module $\Z$
(\cite[VII, \S 3]{Ser79}).
This is viewed as a complex $\{C(G)^{i}\}_{i \le 0}$
concentrated in non-positive degrees
(in cohomological grading, as for all the complexes in this paper).
Let $C^{\vee}(G) = \Hom_{\Ab}(C(G), \Z)$.
Define the standard complete resolution $\hat{C}(G)$ of $\Z$
\cite[I, \S 0, ``Tate (modified) cohomology groups'']{Mil06}
to be the complex
	\[
			\cdots
		\to
			\Z[G^{2}]
		\to
			\Z[G]
		\stackrel{N}{\to}
			\Z[G]
		\to
			\Z[G^{2}]
		\to
			\cdots,
	\]
where the map $N$ from the degree $0$ term to the degree $1$ term
is the norm map $\sum_{\sigma \in G} \sigma$,
and the non-positive degree part of the complex is $C(G)$
and the positive degree part $C^{\vee}(G)[-1]$.
The morphism of complexes
	$
			C(G)
		\to
			C^{\vee}(G)
	$
induced by $N \colon \Z[G] \to \Z[G]$ in degree zero
is also denoted by $N$.
Then $\hat{C}(G)$ is the mapping fiber of $N \colon C(G) \to C^{\vee}(G)$.
For a bounded complex $M$ of sheaves of $G$-modules on $S$,
we define the Tate cohomology as the sheaf-Hom (total) complex
	\begin{equation} \label{eq: Tate cohomology}
			R \hat{\Gamma}(G, M)
		=
			\sheafhom_{\Mod{G}(S)}(\hat{C}(G), M)
	\end{equation}
viewed as an object of $D(S)$.
Note that $R \sheafhom_{\Mod{G}(S)}(\hat{C}(G), M) = 0$
since $\hat{C}(G)$ is an exact complex.

We relate the three functors $R \Gamma$, $L \Delta$ and $R \Hat{\Gamma}$
and show that $R \Hat{\Gamma}$ factors through the derived category.
We have
	$
			R \Gamma(G, M)
		=
			R \sheafhom_{\Mod{G}(S)}(C(G), M)
	$.
Consider the hyperext spectral sequence
	\[
			E_{1}^{i j}
		=
			\prod_{i' + i'' = i}
			\sheafext_{\Mod{G}(S)}^{j}(C(G)^{- i'}, M^{i''})
		\Longrightarrow
			H^{i + j} R \sheafhom_{\Mod{G}(S)}(C(G), M),
	\]
constructed in the same way as \cite[Thm.\ 12.2]{ML63}.
This is convergent since $C(G)$ is bounded above and $M$ bounded below.
Since each term of $C(G)$ is finite free over $\Z[G]$,
we have $E_{1}^{i j} = 0$ for any $i, j$ with $j \ge 1$.
Hence $R \sheafhom_{\Mod{G}(S)}(C(G), M)$ is represented by
the total complex $\sheafhom_{\Mod{G}(S)}(C(G), M)$.
Thus
	\begin{equation} \label{eq: group cohomology using standard complex}
			R \Gamma(G, M)
		=
			\sheafhom_{\Mod{G}(S)}(C(G), M).
	\end{equation}
Similarly we have
	$
			L \Delta(G, M)
		=
			C(G) \tensor_{\Z[G]} M
	$.
The $n$-th term of $C(G) \tensor_{\Z[G]} M$ for each $n$ is
	\begin{align*}
				\bigoplus_{i + j = n}
					C(G)^{i} \tensor_{\Z[G]} M^{j}
		&	=
				\bigoplus_{i + j = n}
					\sheafhom_{\Mod{G}(S)}(C^{\vee}(G)^{-i}, M^{j})
		\\
		&	=
				\prod_{i + j = n}
					\sheafhom_{\Mod{G}(S)}(C^{\vee}(G)^{-i}, M^{j}),
	\end{align*}
where the second equality comes from the property that
$C^{\vee}(G)$ is bounded below and $M$ bounded above.
The last term is the $n$-th term of
$\sheafhom_{\Mod{G}(S)}(C^{\vee}(G), M)$.
Thus
	\begin{equation} \label{eq: group homology using dual standard complex}
			L \Delta(G, M)
		=
			\sheafhom_{\Mod{G}(S)}(C^{\vee}(G), M).
	\end{equation}
In particular, the norm map $N \colon C(G) \to C^{\vee}(G)$
induces a morphism
$L \Delta(G, M) \to R \Gamma(G, M)$ in $D(S)$,
which we denote by the same symbol $N$.
Combining \eqref{eq: Tate cohomology}, \eqref{eq: group cohomology using standard complex},
\eqref{eq: group homology using dual standard complex}
and the mapping fiber distinguished triangle
	\[
		\hat{C}(G) \to C(G) \stackrel{N}{\to} C^{\vee}(G),
	\]
we have a distinguished triangle
	\begin{equation} \label{eq: norm-cofiber sequence}
			L \Delta(G, M)
		\stackrel{N}{\to}
			R \Gamma(G, M)
		\to
			R \hat{\Gamma}(G, M).
	\end{equation}
In particular, if $M$ is an exact complex,
then $L \Delta(G, M)$, $R \Gamma(G, M)$ and hence $R \hat{\Gamma}(G, M)$ are all zero in $D(S)$.
Therefore the assignment $M \mapsto R \hat{\Gamma}(G, M)$ defines
a well-defined triangulated functor $D^{b}(\Mod{G}(S)) \to D(S)$
by \cite[Prop.\ 10.3.3]{KS06}.%
\footnote{
	If $M$ is unbounded,
	the above definition of $R \hat{\Gamma}(G, M)$ does not factor through the derived category
	$D(\Mod{G}(S))$ and hence is not ``correct''.
	In this case, we need to define $R \hat{\Gamma}(G, M)$ to be the mapping cone of the norm map
	$C(G) \tensor_{\Z[G]} I \to \sheafhom_{\Mod{G}(S)}(C(G), I)$,
	where $M \isomto I$ is a K-injective replacement in $\Mod{G}(S)$.
	Below we use bounded $M$ only.
}

We need to know how dual of group cohomology and group cohomology of dual are related.
Let $M \in D^{b}(\Mod{G}(S))$ and $P \in D(S)$.
The complex $R \sheafhom_{S}(M, P)$ can be viewed as an object of $D(\Mod{G}(S))$
by giving $M$ a right $G$-action by $m g := g^{-1} m$
and $P$ a trivial $G$-action.
We have
	\begin{align*}
				R \sheafhom_{\Mod{G}(S)} \bigl(
					\Z, R \sheafhom_{S}(M, P)
				\bigr)
		&	=
				R \sheafhom_{S}(
					M \tensor_{\Z[G]}^{L} \Z, P
				)
		\\
		&	=
				R \sheafhom_{S}(
					\Z \tensor_{\Z[G]}^{L} M, P
				)
	\end{align*}
by the derived tensor-hom adjunction \cite[Rmk.\ 18.6.11]{KS06}
and interchanging the tensor factors.
In our notation, this means
	\begin{equation} \label{eq: universal coefficient theorem}
			R \Gamma \bigl(
				G, R \sheafhom_{S}(M, P)
			\bigr)
		=
			R \sheafhom_{S} \bigl(
				L \Delta(G, M), P
			\bigr)
	\end{equation}
(cf.\ the universal coefficient theorem).
Hence the triangle \eqref{eq: norm-cofiber sequence} induces a distinguished triangle
	\begin{equation} \label{eq: Tate cohomology and Serre dual}
		\begin{split}
				R \sheafhom_{S} \bigl(
					R \Hat{\Gamma}(G, M), P
				\bigr)
		&	\to
				R \sheafhom_{S} \bigl(
					R \Gamma(G, M), P
				\bigr)
		\\
		&	\to
				R \Gamma \bigl(
					G, R \sheafhom_{S}(M, P)
				\bigr).
		\end{split}
	\end{equation}

\begin{Prop} \label{prop: Galois descent technique}
	Let $G$ be a finite group, $S$ a site and
	$M \in D^{b}(\Mod{G}(S))$.
	Assume that $G$ is cyclic and
	$R \Gamma(G, M)$ is bounded.
	Then we have $R \hat{\Gamma}(G, M) = 0$,
	and the norm map gives an isomorphism
	$L \Delta(G, M) \isomto R \Gamma(G, M)$
	between homology and cohomology.
	In particular, the triangle \eqref{eq: Tate cohomology and Serre dual} for any $P \in D(S)$
	reduces to an isomorphism
		\[
				R \sheafhom_{S} \bigl(
					R \Gamma(G, M), P
				\bigr)
			=
				R \Gamma \bigl(
					G, R \sheafhom_{S}(M, P)
				\bigr)
		\]
	in $D(S)$.
\end{Prop}

\begin{proof}
	Let $\sigma$ be a generator of the cyclic group $G$.
	Then $\hat{C}(G)$ is chain homotopic to
	the periodic complex
		$
			\cdots \stackrel{\sigma  - 1}{\to} \Z[G]
			\stackrel{N}{\to} \Z[G] \stackrel{\sigma - 1}{\to}
			\Z[G] \stackrel{N}{\to} \cdots
		$
	by \cite[VIII, Prop.\ 6]{Ser79}
	(see also \cite[VI, Prop.\ 3.3]{Bro82}).
	Therefore we have $R \hat{\Gamma}(G, M) = R \hat{\Gamma}(G, M)[2]$.
	On the other hand, the boundedness of $M$ and $R \Gamma(G, M)$ implies that
	$R \hat{\Gamma}(G, M)$ is acyclic in large degrees.
	Therefore $R \hat{\Gamma}(G, M) = 0$.
	Hence $L \Delta(G, M) \isomto R \Gamma(G, M)$.
\end{proof}

When $S = \Spec k^{\ind\rat}_{\pro\et}$,
the last isomorphism in the above proposition for $P = \Z$ may be written as
$R \Gamma(G, M)^{\SDual} = R \Gamma(G, M^{\SDual})$.

Next we give a variant of the Hochschild-Serre spectral sequence.
Let $L$ be a totally ramified (for simplicity) finite Galois extension of $K$ with Galois group $G$.
For any $A \in \Ab(K_{\fppf} / k^{\ind\rat}_{\et})$ and $k' \in k^{\ind\rat}$,
the $G$-action on $L$ induces a $G$-action on $\Gamma((\alg{L}(k'), k'), A)$.
Hence the functor
$\alg{\Gamma}(L, \;\cdot\;) \colon \Ab(K_{\fppf} / k^{\ind\rat}_{\et}) \to \Ab(k^{\ind\rat}_{\et})$
factors through $\Mod{G}(k^{\ind\rat}_{\et})$
and the functor 
$R \alg{\Gamma}(L, \;\cdot\;) \colon D(K_{\fppf} / k^{\ind\rat}_{\et}) \to D(k^{\ind\rat}_{\et})$
factors through $D(\Mod{G}(k^{\ind\rat}_{\et}))$.
An object of $\Ab(k^{\ind\rat}_{\et})$ can be viewed as an object of
$\Mod{G}(k^{\ind\rat}_{\et})$ by putting the trivial $G$-action.
The resulting functor $D(k^{\ind\rat}_{\et}) \to D(\Mod{G}(k^{\ind\rat}_{\et}))$
is left adjoint to the derived $G$-invariants $R \Gamma(G, \;\cdot\;)$.
For any $A \in D(K_{\fppf} / k^{\ind\rat}_{\et})$,
the natural morphism $R \alg{\Gamma}(K, A) \to R \alg{\Gamma}(L, A)$ in $D(k^{\ind\rat}_{\et})$
factors through $R \Gamma(G, R \alg{\Gamma}(L, A))$ by adjunction.
Similarly, the inclusion
$R \Tilde{\alg{\Gamma}}(K, A) \to R \Tilde{\alg{\Gamma}}(L, A)$ in $D(k^{\ind\rat}_{\pro\et})$
factors through $R \Gamma(G, R \Tilde{\alg{\Gamma}}(L, A))$.

\begin{Prop} \label{prop: Hochschild-Serre}
	Let $L$ be a totally ramified finite Galois extension of $K$ with Galois group $G$.
	Let $A \in D(K_{\fppf} / k^{\ind\rat}_{\et})$.
	The morphism
		\[
				R \Gamma \bigl(
					G, R \alg{\Gamma}(L, A)
				\bigr)
			\gets
				R \alg{\Gamma}(K, A)
		\]
	in $D(k^{\ind\rat}_{\et})$ defined above is an isomorphism.
	The morphism
		\[
				R \Gamma \bigl(
					G, R \Tilde{\alg{\Gamma}}(L, A)
				\bigr)
			\gets
				R \Tilde{\alg{\Gamma}}(K, A)
		\]
	in $D(k^{\ind\rat}_{\pro\et})$ defined above is an isomorphism if $A$ is bounded below.
\end{Prop}

\begin{proof}
	We first treat the first morphism.
	For any $k' \in k^{\ind\rat}$,
	the right-hand side after applying $R \Gamma(k'_{\et}, \;\cdot\;)$ is
		$
				R \Gamma(\alg{K}(k'), A)
		$,
	where we view $A$ as an fppf sheaf on $\Spec \alg{K}(k')$
	as in \eqref{eq: cohomology of RGamma}.
	The left-hand side after applying $R \Gamma(k'_{\et}, \;\cdot\;)$ is
		\begin{align*}
					R \Gamma \Bigl(
						k'_{\et},
						R \Gamma \bigl(
							G, R \alg{\Gamma}(L, A)
						\bigr)
					\Bigr)
			&	=
					R \Gamma \Bigl(
						G,
						R \Gamma \bigl(
							k'_{\et}, R \alg{\Gamma}(L, A)
						\bigr)
					\Bigr)
			\\
			&	=
					R \Gamma \bigl(
						G, R \Gamma(\alg{L}(k'), A)
					\bigr),
		\end{align*}
	where the $R \Gamma(G, \;\cdot\;)$ in the second and third terms are the usual group cohomology.
	Hence the first morphism in the statement after applying $R \Gamma(k'_{\et}, \;\cdot\;)$ is
		\[
				R \Gamma \bigl(
					G, R \Gamma(\alg{L}(k'), A)
				\bigr)
			\gets
				R \Gamma(\alg{K}(k'), A)
		\]
	in $D(\Ab)$.
	This is an isomorphism by the usual Hochschild-Serre spectral sequence
	\cite[III, Rmk.\ 2.21 (a)]{Mil80}
	since the morphism $\Spec \alg{L}(k') \to \Spec \alg{K}(k')$ is a $G$-covering.
	This implies that the first morphism in the proposition is an isomorphism.
	
	For the second, it is enough to show that
	$R \Gamma(G, M)^{\sim} = R \Gamma(G, \Tilde{M})$
	if $M \in D^{+}(\Mod{G}(k^{\ind\rat}_{\et}))$,
	where $\sim$ denotes pro-\'etale sheafification.
	By \eqref{eq: group cohomology using standard complex}
	(which does not require $M$ to be bounded above),
	the $n$-th term of $R \Gamma(G, M)$ is given by
		\[
				\prod_{i + j = n}
					\sheafhom_{k^{\ind\rat}_{\et}}(C(G)^{-i}, M^{j})
			=
				\prod_{i + j = n}
					C^{\vee}(G)^{i} \tensor_{\Z[G]} M^{j}.
		\]
	The final product is a finite product since $C^{\vee}(G)$ and $M$ are bounded below.
	Hence the $n$-th term of $R \Gamma(G, M)^{\sim}$ is
		\[
			\prod_{i + j = n}
				\sheafhom_{k^{\ind\rat}_{\pro\et}}(C(G)^{-i}, \Tilde{M}^{j}),
		\]
	which is the $n$-th term of $R \Gamma(G, \Tilde{M})$.
\end{proof}

\begin{proof}[Proof of \eqref{prop: reduction over a finite extension}]
	Let $A$ be an abelian variety over $K$.
	Assume that \eqref{thm: duality for abelian varieties} is true for $A$
	over a finite Galois extension $L$ of $K$:
		\[
				\theta_{A \times_{K} L}
			\colon
				R \alg{\Gamma}(L, A^{\vee})^{\SDual \SDual}
			\isomto
				R \alg{\Gamma}(L, A)^{\SDual}.
		\]
	in $D(k'^{\ind\rat}_{\pro\et})$,
	where $k'$ is the residue field of $L$.
	We want to deduce the corresponding statement for $A$ over $K$.
	If $L / K$ is unramified, then the morphism above for $A$ over $L$
	is nothing but the morphism for $A$ over $K$
	restricted from $D(k^{\ind\rat}_{\pro\et})$ to $D(k'^{\ind\rat}_{\pro\et})$.
	Therefore the invertibility of these morphisms is equivalent.
	Hence we may assume that $L / K$ is totally ramified.
	Since it is a solvable extension \cite[IV, Cor.\ 5 to Prop.\ 7]{Ser79},
	we may further assume that $L / K$ is cyclic.
	
	Let $G = \Gal(L / K)$.
	As we saw,
	the complexes $R \alg{\Gamma}(L, A)$ and $R \alg{\Gamma}(L, A^{\vee})$
	may be viewed as objects of $D(\Mod{G}(k^{\ind\rat}_{\pro\et}))$.
	We show that the morphism $\theta_{A \times_{K} L}$ is $G$-equivariant,
	i.e.\ a morphism in $D(\Mod{G}(k^{\ind\rat}_{\pro\et}))$.
	From the construction of the (normalized) functorial valuation map
	$v_{L} \colon \alg{L}^{\times} \to \Z$ for $L$
	given in the paragraph before \cite[Prop.\ 2.4.4]{Suz13},
	we see that $v_{L}$ is $G$-equivariant,
	where we put a trivial $G$-action on $\Z$.
	Recall from \S \ref{sec: Formulation} that
	the morphism $\theta_{A \times_{K} L}$ is defined as the Serre dual of the composite
		\begin{align*}
					R \alg{\Gamma}(L, A)
			&	\to
					R \Tilde{\alg{\Gamma}} \bigl(
						L, R \sheafhom_{L}(A^{\vee}, \Gm)
					\bigr)[1]
			\\
			&	\to
					R \sheafhom_{k^{\ind\rat}_{\pro\et}} \bigl(
						R \alg{\Gamma}(L, A^{\vee}), R \alg{\Gamma}(L, \Gm)
					\bigr)[1]
			\\
			&	\to
					R \sheafhom_{k^{\ind\rat}_{\pro\et}} \bigl(
						R \alg{\Gamma}(L, A^{\vee}), \Z
					\bigr)[1]
				=
					R \alg{\Gamma}(L, A^{\vee})^{\SDual}[1].
		\end{align*}
	All these morphisms,
	including the trace morphism
	$R \alg{\Gamma}(L, \Gm) = \alg{L}^{\times} \stackrel{v_{L}}{\to} \Z$ over $L$,
	are $G$-equivariant.
	
	We apply $R \Gamma(G, \;\cdot\;)$ to $\theta_{A \times_{K} L}$.
	We have
		\[
				R \Gamma \bigl(
					G,
					R \alg{\Gamma}(L, A)
				\bigr)
			=
					R \alg{\Gamma}(K, A)
		\]
	in $D(k^{\ind\rat}_{\pro\et})$ by the Hochschild-Serre spectral sequence
	\eqref{prop: Hochschild-Serre}.
	The both complexes $R \alg{\Gamma}(L, A)$ and $R \alg{\Gamma}(K, A)$ are
	concentrated in degrees $0$ and $1$.
	Hence we can apply
	\eqref{prop: Galois descent technique}
	for $M = R \alg{\Gamma}(L, A)$ and $P = \Z$.
	This yields
		\[
				R \Gamma \bigl(
					G,
					R \alg{\Gamma}(L, A)^{\SDual}
				\bigr)
			=
				R \alg{\Gamma}(K, A)^{\SDual}.
		\]
	The both complexes $R \alg{\Gamma}(L, A)^{\SDual}$ and $R \alg{\Gamma}(K, A)^{\SDual}$ are
	concentrated in degrees $0$ to $2$.
	Applying the same proposition again, we have
		\[
				R \Gamma \bigl(
					G,
					R \alg{\Gamma}(L, A)^{\SDual \SDual}
				\bigr)
			=
				R \alg{\Gamma}(K, A)^{\SDual \SDual}.
		\]
	Now we apply $R \Gamma(G, \;\cdot\;)$ to the both sides of $\theta_{A \times_{K} L}$
	to get an isomorphism
		\[
				\theta_{A \times_{K} L}^{G}
			\colon
				R \alg{\Gamma}(K, A^{\vee})^{\SDual \SDual}
			\isomto
				R \alg{\Gamma}(K, A)^{\SDual}.
		\]
	
	We show that this isomorphism $\theta_{A \times_{K} L}^{G}$ is equal to $\theta_{A}$.
	Consider the morphism
		\begin{equation} \label{eq: cup product over L}
				R \alg{\Gamma}(L, A^{\vee})
			\tensor^{L}
				R \alg{\Gamma}(L, A)
			\to
				R \alg{\Gamma}(L, \Gm)[1]
		\end{equation}
	in $D(\Mod{G}(k^{\ind\rat}_{\pro\et}))$.
	Apply $R \Gamma(G, \;\cdot\;)$.
	The same proof as \cite[Prop.\ 2.4.3 and the paragraph after]{Suz13} shows that
	there is a cup product pairing
		\begin{equation} \label{eq: group cohomology of cup product for L}
				R \Gamma \bigl(
					G,
					R \alg{\Gamma}(L, A^{\vee})
				\bigr)
			\tensor^{L}
				R \Gamma \bigl(
					G,
					R \alg{\Gamma}(L, A)
				\bigr)
			\to
				R \Gamma \bigl(
					G,
					R \alg{\Gamma}(L, \Gm)
				\bigr)[1],
		\end{equation}
	which can be identified with
		\[
				R \alg{\Gamma}(K, A^{\vee})
			\tensor^{L}
				R \alg{\Gamma}(K, A)
			\to
				R \alg{\Gamma}(K, \Gm)[1].
		\]
	This and the trace morphism
	$R \alg{\Gamma}(K, \Gm)[1] = \alg{K}^{\times} \stackrel{v_{K}}{\to} \Z[1]$
	for $K$ leads to the morphism $\theta_{A}$
	via the derived tensor-hom adjunction (used before).
	On the other hand, \eqref{eq: cup product over L} gives morphisms
		\begin{align*}
					R \alg{\Gamma}(L, A^{\vee})
			&	\to
					R \sheafhom_{k^{\ind\rat}_{\pro\et}} \bigl(
						R \alg{\Gamma}(L, A),
						R \alg{\Gamma}(L, \Gm)
					\bigr)[1]
			\\
			&	\to
					R \sheafhom_{k^{\ind\rat}_{\pro\et}} \Bigl(
						R \alg{\Gamma}(L, A),
						L \Delta \bigl(
							G, R \alg{\Gamma}(L, \Gm)
						\bigr)
					\Bigr)[1]
		\end{align*}
	in $D(\Mod{G}(k^{\ind\rat}_{\pro\et}))$,
	where the second morphism is the natural morphism
	(i.e.\ $M \to \Z \tensor_{\Z[G]}^{L} M$ given by $m \mapsto 1 \tensor m$)
	and $L \Delta(G, R \alg{\Gamma}(L, \Gm))$ here is given the trivial $G$-action.
	Applying $R \Gamma(G, \;\cdot\;)$
	and using \eqref{eq: universal coefficient theorem}
	for
		$
				P
			=
				L \Delta \bigl(
					G,
					R \alg{\Gamma}(L, \Gm)
				\bigr)
		$,
	we have a morphism
		\begin{align*}
			&
					R \Gamma \bigl(
						G,
						R \alg{\Gamma}(L, A^{\vee})
					\bigr)
			\\
			&	\to
					R \sheafhom_{k^{\ind\rat}_{\pro\et}} \Bigl(
						L \Delta \bigl(
							R \alg{\Gamma}(L, A)
						\bigr),
						L \Delta \bigl(
							G, R \alg{\Gamma}(L, \Gm)
						\bigr)
					\Bigr)[1].
		\end{align*}
	By the derived tensor-hom adjunction,
	we have a morphism
		\begin{equation} \label{eq: group homology of cup product for L}
				R \Gamma \bigl(
					G,
					R \alg{\Gamma}(L, A^{\vee})
				\bigr)
			\tensor^{L}
				L \Delta \bigl(
					G,
					R \alg{\Gamma}(L, A)
				\bigr)
			\to
				L \Delta \bigl(
					G,
					R \alg{\Gamma}(L, \Gm)
				\bigr)[1].
		\end{equation}
	This and the trace morphism
	$L \Delta(G, R \alg{\Gamma}(L, \Gm)) = L \Delta(G, \alg{L}^{\times}) \stackrel{v_{L}}{\to} \Z$
	for $L$ leads to the morphism $\theta_{A \times_{K} L}^{G}$.
	The two morphisms \eqref{eq: group cohomology of cup product for L}
	and \eqref{eq: group homology of cup product for L} and the two trace morphisms are compatible:
		\[
			\begin{CD}
					R \Gamma \bigl(
						G,
						R \alg{\Gamma}(L, A^{\vee})
					\bigr)
				\tensor^{L}
					L \Delta \bigl(
						G,
						R \alg{\Gamma}(L, A)
					\bigr)
			@>>>
				L \Delta \bigl(
					G,
					R \alg{\Gamma}(L, \Gm)
				\bigr)[1]
			@>>>
				\Z
			\\
			@VV \id \tensor N V
			@VV N V
			@|
			\\
					R \Gamma \bigl(
						G,
						R \alg{\Gamma}(L, A^{\vee})
					\bigr)
				\tensor^{L}
					R \Gamma \bigl(
						G,
						R \alg{\Gamma}(L, A)
					\bigr)
			@>>>
				R \Gamma \bigl(
					G,
					R \alg{\Gamma}(L, \Gm)
				\bigr)[1]
			@>>>
				\Z.
			\end{CD}
		\]
	The commutativity of the left square is trivial.
	That of the right comes from the equality
	$v_{L} = v_{K} \compose N: \alg{L}^{\times} \to \alg{K}^{\times} \to \Z$
	of the normalized valuations of $K$ and $L$,
	which is true as $L / K$ is totally ramified.
	This implies $\theta_{A} = \theta_{A \times_{K} L}^{G}$.
	Hence the invertibility of $\theta_{A \times_{K} L}$ implies that of $\theta_{A}$.
\end{proof}

\begin{Rmk} \BetweenThmAndList \label{rmk: remarks after proof of main theorem}
	\begin{enumerate}
		\item
			The assumptions in
			\eqref{prop: Galois descent technique} can be weakened to make it parallel to
			Nakayama's theorem (\cite[IX, \S 5, Thm.\ 8]{Ser79}).
			The group $G$ may be any finite group.
			The complex $M$ may just be bounded,
			with no restrictions on the number of non-zero terms or their positions.
			The assumption for $R \hat{\Gamma}(G, M)$ should now be that
			for any prime number $l$ and an $l$-Sylow subgroup $G_{l}$ of $G$,
			the complex $R \hat{\Gamma}(G_{l}, M)$ be acyclic in two consecutive degrees.
			The conclusion is that $R \hat{\Gamma}(H, M) = 0$ for any subgroup $H$ of $G$.
			To prove this, first notice that the statement is reduced to the corresponding statement
			for a complex of usual $G$-modules (not sheaves of $G$-modules),
			by considering the Hom-complex $\Hom_{k^{\ind\rat}_{\pro\et}}(M, I)$
			for any injective sheaf $I \in \Ab(k^{\ind\rat}_{\pro\et})$.
			Then the statement was proved by Koya \cite[Prop.\ 4]{Koy93}.
			Alternatively, we may assume that $M$ is concentrated in degree zero,
			since the bounded derived category of $\Z[G]$-modules modulo bounded perfect complexes
			and the stable category of $\Z[G]$-modules are equivalent (\cite[Prop.\ 8.2]{BIK13}).
			Then the statement is equivalent to Nakayama's theorem.
		\item \label{rmk: comparison with Bertapelle in non-semistable case}
			\eqref{prop: comparison with Bester-Bertapelle} is true
			without assuming $A$ to be semistable.
			We can see this by the following Galois descent argument.
			Assume that $k$ is algebraically closed and
			$A$ has semistable reduction over a finite Galois extension $L$ of $K$.
			Let $G = \Gal(L / K)$.
			Let $\mathcal{A}_{L}$ be the N\'eron model of $A_{L} = A \times_{K} L$.
			The group $\mathcal{A}_{L}$ has a $G$-equivariant sheaf structure
			coming from the descent data of $A_{L}$.
			Therefore, for an integer $n \ge 1$ that kills $\pi_{0}((\mathcal{A}_{L})_{x})$,
			the finite part $\mathcal{A}_{L}[n]^{f}$ of the $n$-torsion part of $\mathcal{A}_{L}$
			also has a $G$-equivariant sheaf structure.
			Since the pairing
				$
						(\mathcal{A}_{L}[n]^{f})^{\CDual}
					\times_{\Order_{L}}
						\mathcal{A}_{L}[n]^{f}
					\to
						\Gm
				$
			is $G$-equivariant and the trace isomorphism
			$R \alg{\Gamma}_{x}(\Order_{L}, \Gm) = \Z[-1]$ is $G$-invariant,
			the isomorphism
				\[
						R \alg{\Gamma}_{x}(\Order_{K}, (\mathcal{A}_{L}[n]^{f})^{\CDual})
					\to
						R \alg{\Gamma}(\Order_{K}, \mathcal{A}_{L}[n]^{f})^{\SDual}[-1]
				\]
			in \eqref{thm: duality for finite flats over integers}
			is an isomorphism in $D(\Mod{G}(k^{\ind\rat}_{\pro\et}))$.
			From this, we can see that all the steps in the construction of Bertapelle's isomorphism
				\[
						\psi_{A_{L}}
					\colon
						H^{1}(L, A^{\vee})
					\isomto
						\Ext_{k}^{1}(\alg{\Gamma}(L, A), \Q / \Z)
				\]
			is $G$-equivariant.
			Hence $\psi_{A_{L}}$ itself is $G$-equivariant.
			For $n \ge 0$, let $L^{\tensor n}$ be the tensor product of $n$ copies of $L$ over $K$
			and $\Res_{L^{\tensor n} / K}$ be the Weil restriction functor.
			Then we have an exact sequence
				\[
						0 \to A \to \Res_{L / K} A
					\to
						\Res_{L^{\tensor 2} / K} A
					\to
						\Res_{L^{\tensor 3} / K} A
					\to
						\cdots
				\]
			of abelian varieties over $K$ coming from the \v{C}ech complex
			(cf.\ the proof of \cite[III, Thm.\ 3.9]{Mil80}).
			The cokernel of any morphism in the sequence is an abelian variety.
			Dualizing, we have an exact sequence
				\[
						\cdots
					\to
						\Res_{L^{\tensor 3} / K} A^{\vee}
					\to
						\Res_{L^{\tensor 2} / K} A^{\vee}
					\to
						\Res_{L / K} A^{\vee} \to A^{\vee} \to 0
				\]
			of abelian varieties over $K$.
			Since $H^{2}(K, \text{abelian variety}) = 0$,
			this induces an exact sequence
				\[
						H^{1}(K, \Res_{L^{\tensor 2} / K} A^{\vee})
					\to
						H^{1}(K, \Res_{L / K} A^{\vee})
					\to
						H^{1}(K, A^{\vee})
					\to
						0,
				\]
			or
				\[
						\prod_{\sigma \in G}
						H^{1}(L, A^{\vee})
					\stackrel{\prod \sigma}{\longrightarrow}
						H^{1}(L, A^{\vee})
					\stackrel{\mathrm{Cores}}{\longrightarrow}
						H^{1}(K, A^{\vee})
					\to
						0.
				\]
			Hence the corestriction identifies $H^{1}(K, A^{\vee})$
			as the $G$-coinvariants of $H^{1}(L, A^{\vee})$.
			Similarly, by \eqref{prop: Serre duality}
			\eqref{ass: Ext of a proalg by Q mod Z is zero above degree 2},
			the inclusion $\alg{\Gamma}(K, A) \into \alg{\Gamma}(L, A)$
			identifies $\Ext_{k}^{1}(\alg{\Gamma}(K, A), \Q / \Z)$
			as the $G$-coinvariants of $\Ext_{k}^{1}(\alg{\Gamma}(L, A), \Q / \Z)$.
			The construction of Bertapelle's isomorphism
				\[
						\psi_{A}
					\colon
						H^{1}(K, A^{\vee})
					\isomto
						\Ext_{k^{\ind\rat}_{\pro\et}}^{1}(\alg{\Gamma}(K, A), \Q / \Z)
				\]
			given in \cite[Thm.\ 3]{Ber03} is the same
			as to take the $G$-coinvariants of $\psi_{A_{L}}$.
			The same is true for our morphisms $\theta_{A_{L}}^{+1}(k)$ and $\theta_{A}^{+1}(k)$
			by the functoriality of our constructions.
			Hence the equality $\theta_{A_{L}}^{+1}(k) = \psi_{A_{L}}$ implies
			the equality $\theta_{A}^{+1}(k) = \psi_{A}$.
	\end{enumerate}
\end{Rmk}


\numberwithin{equation}{section}
\section{End of proof: Grothendieck's conjecture}
\label{sec: End of proof: Grothendieck's conjecture}
Now we prove \eqref{thm: duality for abelian varieties} and Grothendieck's conjecture
by summarizing what we have done so far.
Recall from Introduction that
\v{S}afarevi\v{c}'s conjecture,
proved by B\'egueri \cite[Thm.\ 8.3.6]{Beg81},
Bester \cite[\S 2.7, Thm.\ 7.1]{Bes78} and Bertapelle \cite[Thm.\ 3]{Ber03},
states that there exists a canonical isomorphism
	\[
			H^{1}(A^{\vee})
		\to
			\Ext_{k}^{1}(\alg{\Gamma}(A), \Q / \Z)
	\]
for an abelian variety $A$ over $K$ when $k$ is algebraically closed.
To be clear, in the proposition, the theorem and their proofs below,
we mean by ``\v{S}afarevi\v{c}'s conjecture for $A$'' the statement that
the morphism $\theta_{A}^{+1}$ in
\eqref{prop: two parts of the duality morphism}
is an isomorphism.

\begin{Prop}
	Let $A$ be an abelian variety over $K$.
	Then \eqref{thm: duality for abelian varieties} for $A$ is equivalent to
	the conjunction of the following three statements:
	\begin{itemize}
		\item
			Grothendieck's conjecture for $A$,
		\item
			\v{S}afarevi\v{c}'s conjecture for $A \times_{K} \alg{K}(k')$
			for any algebraically closed field $k' \in k^{\ind\rat}$, and
		\item
			 \v{S}afarevi\v{c}'s conjecture for $A^{\vee} \times_{K} \alg{K}(k')$
			for any algebraically closed field $k' \in k^{\ind\rat}$.
	\end{itemize}
\end{Prop}

\begin{proof}
	Immediate from
	\eqref{prop: two parts of the duality morphism}
	and \eqref{prop: comparison with Grothendieck's pairing}.
\end{proof}

\begin{Thm}
	\eqref{thm: duality for abelian varieties}
	and Grothendieck's conjecture are both true for any abelian variety $A$ over $K$.
	That is, the morphism
		\[
				\theta_{A}
			\colon
				R \alg{\Gamma}(A^{\vee})^{\SDual \SDual}
			\to
				R \alg{\Gamma}(A)^{\SDual}
		\]
	in $D(k)$ is an isomorphism, and Grothendieck's pairing
		\[
				\pi_{0}(\mathcal{A}_{x}^{\vee})
			\times
				\pi_{0}(\mathcal{A}_{x})
			\to
				\Q / \Z
		\]
	is perfect.
\end{Thm}

\begin{proof}
	By the semistable reduction theorem,
	there exists a finite Galois extension $L / K$ such that $A$ has semistable reduction over $L$.
	Then Grothendieck's conjecture is true for $A$ over $L$ by Werner's result \cite{Wer97}.
	Also \v{S}afarevi\v{c}'s conjecture is true for $A$ over $L$
	by \eqref{prop: comparison with Bester-Bertapelle}
	and \eqref{prop: comparison with Begueri}
	and the results of B\'egueri, Bester and Bertapelle cited above.
	Hence \eqref{thm: duality for abelian varieties} is true for $A$ over $L$
	by the previous proposition.
	Hence it is true for $A$ over $K$
	by \eqref{prop: reduction over a finite extension}.
	Hence Grothendieck's conjecture is true for $A$ over $K$
	by the previous proposition.
\end{proof}

This completes the proof of
Theorems \ref{mainthm: duality for abelian varieties},
\ref{mainthm: relation with Gro and Sha conjectures}
and \ref{mainthm: Gro conj}.

\begin{Rmk} \label{rmk: where we need P-acyclicity for abelian variety}
		Since $R \alg{\Gamma}(A^{\vee})$ is P-acyclic,
		we know that its $R \Gamma(k_{\pro\et}, \;\cdot\;)$ is the usual cohomology complex
		$R \Gamma(A^{\vee}) = R \Gamma(K, A^{\vee})$.
		Therefore the duality
			\[
					\Bigl[
							\invlim_{n} \alg{\Gamma}(A^{\vee})
						\to
							R \alg{\Gamma}(A^{\vee})
					\Bigr]
				\isomto
					R \sheafhom_{k^{\ind\rat}_{\pro\et}} \bigl(
						R \alg{\Gamma}(A),
						\Q / \Z
					\bigr),
			\]
		true in $D(k^{\ind\rat}_{\pro\et})$,
		gives a statement about $R \Gamma(A^{\vee})$
		by taking the $R \Gamma(k_{\pro\et}, \;\cdot\;)$ of the both sides,
		even if $k$ is not algebraically closed.
		
		To make this explicit,
		let $(\mathcal{A}_{x}^{\vee})_{\sAb}$ be the maximal semi-abelian quotient of
		$(\mathcal{A}_{x}^{\vee})_{0}$.
		and $\closure{(\mathcal{A}_{x}^{\vee})_{\sAb}}$ its universal covering.
		Then
			\[
					\invlim_{n}
						\alg{\Gamma}(A^{\vee})
				=
					\invlim_{n}
						(\mathcal{A}_{x}^{\vee})_{\sAb}
				=
					\closure{(\mathcal{A}_{x}^{\vee})_{\sAb}}
			\]
		by \eqref{prop: Serre duality}
		\eqref{ass: Serre duality for general proalgebraic groups},
		where $\invlim_{n}$ denotes the inverse limit for multiplication by $n \ge 1$.
		These groups do not change under replacing $\invlim_{n}$ by $R \invlim_{n}$
		by \eqref{prop: misc on ind-rational pro-etale topology}
		\eqref{ass: Rlim vanishes on proalgebraic groups}.
		We have
			\[
					R \Gamma(
						k_{\pro\et},
						R \invlim_{n}
						(\mathcal{A}_{x}^{\vee})_{\sAb}
					)
				=
					R \invlim_{n}
					R \Gamma(
						k_{\pro\et},
						(\mathcal{A}_{x}^{\vee})_{\sAb}
					)
				=
					R \invlim_{n}
					R \Gamma(
						k_{\et},
						(\mathcal{A}_{x}^{\vee})_{\sAb}
					)
			\]
		by \eqref{prop: limits and cohomology commute}
		\eqref{ass: limits and pushforward}
		and \eqref{prop: misc on ind-rational pro-etale topology}
		\eqref{ass: profppf, proet and et cohom of quasi-algebraic groups etc}.
		Hence
			\[
					R \Gamma(
						k_{\pro\et},
						\invlim_{n} \alg{\Gamma}(A^{\vee})
					)
				=
					R \Gamma(
						k_{\pro\et},
						\closure{(\mathcal{A}_{x}^{\vee})_{\sAb}}
					)
				=
					R \invlim_{n}
					R \Gamma(
						k_{\et},
						(\mathcal{A}_{x}^{\vee})_{\sAb}
					).
			\]
		Denote these isomorphic objects by $Q_{k, A^{\vee}}$.
		It is a complex of $\Q$-vector spaces.
		We have
			\begin{align*}
						R \Gamma \Bigl(
							k_{\pro\et},
							R \sheafhom_{k^{\ind\rat}_{\pro\et}} \bigl(
								R \alg{\Gamma}(A),
								\Q / \Z
							\bigr)
						\Bigr)
				&	=
						R \Hom_{k^{\ind\rat}_{\pro\et}} \bigl(
							R \alg{\Gamma}(A),
							\Q / \Z
						\bigr)
				\\
				&	=
						R \Hom_{\Ind \Pro \Alg / k} \bigl(
							R \alg{\Gamma}(A),
							\Q / \Z
						\bigr)
			\end{align*}
		by \eqref{prop: comparison thm for ind-proalgebraic groups}.
		Therefore the duality isomorphism in $D(k^{\ind\rat}_{\pro\et})$ above gives an isomorphism
			\[
					\Bigl[
							Q_{k, A^{\vee}}
						\to
							R \Gamma(A^{\vee})
					\Bigr]
				\isomto
					R \Hom_{\Ind \Pro \Alg / k} \bigl(
						R \alg{\Gamma}(A),
						\Q / \Z
					\bigr)
			\]
		in $D(\Ab)$.
\end{Rmk}


\numberwithin{equation}{section}
\section{Duality with coefficients in tori}
\label{sec: Duality with coefficients in tori}

In this section, we give an analogue of \eqref{thm: duality for abelian varieties} for tori
and describe it.
In this and the next sections,
we write $\sheafhom_{k} = \sheafhom_{k^{\ind\rat}_{\pro\et}}$
and similarly for $\sheafext$ and $R \sheafhom$,
and $R \alg{\Gamma}(\;\cdot\;) = R \alg{\Gamma}(K, \;\cdot\;)$,
when there is no confusion.

Let $T$ be a torus over $K$
with N\'eron model $\mathcal{T}$ and Cartier dual $T^{\CDual}$.
Let $\mathcal{T}_{x}$ be the special fiber of $\mathcal{T}$.
By \eqref{prop: cohomology of local fields}
\eqref{ass: cohomology of tori over local fields},
we know that
$\alg{\Gamma}(T)$ is an extension of an \'etale group by a P-acyclic proalgebraic group,
$\alg{H}^{n}(T) = 0$ for $n \ge 1$,
and $R \alg{\Gamma}(T) = \alg{\Gamma}(T)$ is P-acyclic.
Also by \eqref{prop: cohomology of local fields}
\eqref{ass: cohomology of finite free etale over local fields},
$\alg{\Gamma}(T^{\CDual})$ is a lattice over $k$,
$\alg{H}^{1}(T^{\CDual}) \in \FEt / k$,
$\alg{H}^{2}(T^{\CDual}) \in \Ind \Alg_{\uc} / k$,
$\alg{H}^{n}(T^{\CDual}) = 0$ for $n \ge 3$,
and $R \alg{\Gamma}(T^{\CDual})$ is P-acyclic and Serre reflexive.
Therefore we will write
$R \Tilde{\alg{\Gamma}}(T) = R \alg{\Gamma}(T)$ and
$R \Tilde{\alg{\Gamma}}(T^{\CDual}) = R \alg{\Gamma}(T^{\CDual})$.
The morphism of functoriality of $R \Tilde{\alg{\Gamma}}$
\eqref{prop: localization and functoriality in pro-etale}
and the trace morphism $R \alg{\Gamma}(\Gm) \to \Z$ \eqref{eq: trace morphism}
yield morphisms
	\begin{align*}
				R \alg{\Gamma}(T^{\CDual})
		&	\to
				R \Tilde{\alg{\Gamma}}
				R \sheafhom_{K}(T, \Gm)
		\\
		&	\to
				R \sheafhom_{k}(
					R \alg{\Gamma}(T), R \alg{\Gamma}(\Gm)
				)
		\\
		&	\to
				R \sheafhom_{k}(
					R \alg{\Gamma}(T), \Z
				)
			=
				R \alg{\Gamma}(T)^{\SDual}
	\end{align*}
in $D(k^{\ind\rat}_{\pro\et})$.
The following is the duality with coefficients in tori.
See also \cite[7.2]{Beg81} in the case of mixed characteristic $K$ with algebraically closed $k$.

\begin{Thm} \label{thm: duality with coefficients in tori}
	The morphism
		\[
				\theta_{T}
			\colon
				R \alg{\Gamma}(T^{\CDual})
			\to
				R \alg{\Gamma}(T)^{\SDual}
		\]
	in $D(k^{\ind\rat}_{\pro\et})$ defined above is an isomorphism.%
	\footnote{
		Since $R \alg{\Gamma}(T^{\CDual})$ is Serre reflexive as we saw above,
		there is no point putting $\SDual \SDual$ to it
		and using a different letter $\vartheta_{T}$ to denote this morphism,
		as opposed to the case of abelian varieties.
	}
	Its Serre dual yields isomorphisms
		\[
				\theta_{T^{\CDual}}
			\colon
				\bigl[
						\invlim_{n} \alg{\Gamma}(T)
					\to
						R \alg{\Gamma}(T)
				\bigr]
			=
				R \alg{\Gamma}(T)^{\SDual \SDual}
			=
				R \alg{\Gamma}(T^{\CDual})^{\SDual}.
		\]
\end{Thm}

\begin{proof}
	We first treat $\theta_{T}$.
	If $T = \Gm$, then it is an isomorphism by \cite[Thm.\ 2.6.1]{Suz13}
	(which is essentially Serre's local class field theory \cite{Ser61}).
	For general $T$, let $L / K$ be a finite Galois extension with Galois group $G$ that splits $T$,
	so that we have
		\[
				\theta_{T \times_{K} L}
			\colon
				R \alg{\Gamma}(L, T^{\CDual})
			\isomto
				R \alg{\Gamma}(L, T)^{\SDual}.
		\]
	We may assume that $L / K$ is totally ramified and cyclic
	as in the proof of
	\eqref{prop: reduction over a finite extension}.
	The complexes $R \alg{\Gamma}(L, T)$ and
		\[
				R \Gamma \bigl(
					G, R \alg{\Gamma}(L, T)
				\bigr)
			=
				R \alg{\Gamma}(K, T)
		\]
	are both concentrated in degree zero.
	Therefore we can apply \eqref{prop: Galois descent technique},
	obtaining
		\[
				R \Gamma \bigl(
					G, R \alg{\Gamma}(L, T)^{\SDual}
				\bigr)
			=
				R \alg{\Gamma}(K, T)^{\SDual}.
		\]
	Applying $R \Gamma(G, \;\cdot\;)$ to $\theta_{T \times_{K} L}$,
	we obtain an isomorphism
		$
				R \alg{\Gamma}(K, T^{\CDual})
			\isomto
				R \alg{\Gamma}(K, T)^{\SDual}
		$.
	We can see that this isomorphism is equal to $\theta_{T}$
	by the same method as the last part of the proof of
	\eqref{prop: reduction over a finite extension}.
	
	Before treating $\theta_{T^{\CDual}}$,
	we show that $\pi_{0}(\mathcal{T}_{x})$ is a finitely generated \'etale group over $k$.
	This is well-known
	(see \cite[Prop.\ 3.5]{HN11} for example),
	but can be deduced as a corollary of the method of proof in the previous paragraph, as follows.
	By \eqref{prop: Galois descent technique},
	we know that the norm map gives an isomorphism
		\[
				L \Delta(G, \alg{\Gamma}(L, T))
			\stackrel{N}{\isomto}
				R \Gamma(G, \alg{\Gamma}(L, T))
			=
				\alg{\Gamma}(K, T)
		\]
	if $L / K$ is finite Galois totally ramified and cyclic.
	The cyclicity can be removed by d\'evissage.
	Since the zeroth group homology is the coinvariants,
	we know that the $G$-coinvariants of $\alg{\Gamma}(L, T)$ is $\alg{\Gamma}(K, T)$.
	The group $\pi_{0}(\alg{\Gamma}(K, T))$ is the N\'eron component group of $T$.
	Hence the $G$-coinvariants of the N\'eron component group of $T \times_{K} L$
	is the N\'eron component group of $T$.
	Since the N\'eron component group of a split torus is obviously finite free,
	this implies that $\pi_{0}(\mathcal{T}_{x})$ is finitely generated.
	
	For $\theta_{T^{\CDual}}$, recall again that $R \alg{\Gamma}(T) = \alg{\Gamma}(T)$.
	Hence the only non-trivial part (after the invertibility of $\theta_{T}$ shown above) is
		\[
				\alg{\Gamma}(T)^{\SDual \SDual}
			=
				\bigl[
						\invlim_{n} \alg{\Gamma}(T)
					\to
						\alg{\Gamma}(T)
				\bigr].
		\]
	Let $\alg{\Gamma}(T)_{0} \in \Pro \Alg / k$ be the identity component of $\alg{\Gamma}(T)$.
	By \eqref{prop: Serre duality}
	\eqref{ass: Serre duality for general proalgebraic groups}, we have
		\[
				\alg{\Gamma}(T)_{0}^{\SDual \SDual}
			=
				\bigl[
						\invlim_{n} \alg{\Gamma}(T)_{0}
					\to
						\alg{\Gamma}(T)_{0}
				\bigr]
			\in
				D^{b}(\Pro \Alg_{\uc}).
		\]
	Since $\pi_{0}(\alg{\Gamma}(T))$ is finitely generated,
	we have $\invlim_{n} \alg{\Gamma}(T) = \invlim_{n} \alg{\Gamma}(T)_{0}$,
	which is a uniquely divisible proalgebraic group.
	Hence
		\[
				(\invlim_{n} \alg{\Gamma}(T))^{\SDual}
			=
				(\invlim_{n} \alg{\Gamma}(T)_{0})^{\SDual}
			=
				0
		\]
	as seen in the proof of
	\eqref{prop: Serre duality}
	\eqref{ass: Serre duality for general proalgebraic groups}.
	Therefore
		\[
				\alg{\Gamma}(T)^{\SDual \SDual}
			=
				\bigl[
						\invlim_{n} \alg{\Gamma}(T)
					\to
						\alg{\Gamma}(T)
				\bigr]^{\SDual \SDual}.
		\]
	We do not need $\SDual \SDual$ on the right-hand side,
	since the distinguished triangle
		\[
				\bigl[
						\invlim_{n} \alg{\Gamma}(T)_{0}
					\to
						\alg{\Gamma}(T)_{0}
				\bigr]
			\to
				\bigl[
						\invlim_{n} \alg{\Gamma}(T)
					\to
						\alg{\Gamma}(T)
				\bigr]
			\to
				\pi_{0}(\alg{\Gamma}(T))
		\]
	shows that the middle term is Serre reflexive.
\end{proof}

We deduce concrete consequences from this theorem,
as in \eqref{prop: two parts of the duality morphism}.
Note that if $G$ is a connected proalgebraic group over $k$,
then $\sheafext_{k}^{n}(G, \Z) = 0$ for $n = 0, 1$ since
$R \sheafhom_{k}(G, \Z) = R \sheafhom_{k}(G, \Q / \Z)[-1]$
by \eqref{prop: comparison thm for ind-proalgebraic groups}
\eqref{ass: RHom between IPAlg and Et}
and $\sheafhom_{k}(G, \Q / \Z) = 0$ by
\eqref{prop: Serre duality}
\eqref{ass: Ext of a proalg by Q mod Z is zero above degree 2}.
This implies that
	\begin{equation} \label{eq: ext of proalg by Z in deg 0 and 1}
			\sheafext_{k}^{n}(\alg{\Gamma}(T), \Z)
		=
			\sheafext_{k}^{n} \bigl(
				\pi_{0}(\alg{\Gamma}(T)), \Z
			\bigr)
		=
			\sheafext_{k}^{n}(\pi_{0}(\mathcal{T}_{x}), \Z)
	\end{equation}
for $n = 0, 1$.
Hence the isomorphism $\theta_{T}$ in the above theorem induces,
in degrees $0$ and $1$, isomorphisms
	\begin{gather*}
				\theta_{T}^{0}
			\colon
				\alg{\Gamma}(T^{\CDual})
			\isomto
				\sheafhom_{k}(\pi_{0}(\mathcal{T}_{x}), \Z),
		\\
				\theta_{T}^{1}
			\colon
				\alg{H}^{1}(T^{\CDual})
			\isomto
				\sheafext_{k}^{1}(\pi_{0}(\mathcal{T}_{x}), \Z)
	\end{gather*}
in $\FGEt / k$.
We will put them together, and make explicit the isomorphism in degree $2$.
We need notation to do this.

Let $I_{K}$ be the inertia group of the absolute Galois group of $K$,
$X^{\ast}(T) = T^{\CDual}$ the character lattice,
$X_{\ast}(T) = \sheafhom_{K}(\Gm, T)$ the cocharacter lattice and
$\hat{T}_{\tor} = X^{\ast}(T) \tensor_{\Z} \Q / \Z$ the torsion part of the dual torus.%
\footnote{
	Here we treat the torsion part of the dual torus algebraically
	without using $\C$ or $\closure{\Q}_{l}$.
	Therefore we ignore the Tate twist $\Q / \Z(1)$,
	which has nothing to do with the Galois action of $K$.
}
For an \'etale group $X$ over $K$,
the $I_{K}$-coinvariants $X(K^{\sep}) / I_{K}$ is a discrete $\Gal(\closure{k} / k)$-module.
We denote the corresponding \'etale group over $k$
by $X / I_{K}$ by abuse of notation.

\begin{Prop} \label{prop: content of duality of tori}
	From $H^{0}$ and $H^{1}$ of the isomorphism $\theta_{T}$,
	we can deduce an isomorphism
		\[
				\theta_{T}^{01}
			\colon
				\pi_{0}(\mathcal{T}_{x})
			\isomto
				X_{\ast}(T) / I_{K}
		\]
	in $\FGEt / k$,
	and from $H^{2}$, an isomorphism
		\[
				\theta_{T}^{2}
			\colon
				\alg{H}^{1}(\hat{T}_{\tor})
			\isomto
				\sheafext_{k}^{1}(\alg{\Gamma}(T), \Q / \Z)
		\]
	in $\Ind \Alg_{\uc} / k$.
\end{Prop}

We need a lemma.

\begin{Lem}
	For any finitely generated \'etale group $A$ over $K$, we have a natural isomorphism
		\[
				\tau_{\le 1}
				R \alg{\Gamma}
				R \sheafhom_{K}(A, \Z)
			=
				R \sheafhom_{k}(A / I_{K}, \Z)
		\]
	in $D(k)$, where $\tau$ denotes the truncation functor.
\end{Lem}

\begin{proof}
	First note that $R \sheafhom_{K}$ between \'etale groups over $K$
	is concentrated in degrees $0$ and $1$
	since \'etale groups are locally just abelian groups
	and $\Ext_{\Ab}^{\ge 2} = 0$.
	The same is true for $R \sheafhom_{k}$.
	Let $0 \to B \to F \to A \to 0$ be an exact sequence
	with $F$ a finite free $\Z[G]$-module
	for some finite Galois extension $L / K$ and $G = \Gal(L / K)$.
	Let $I_{G}$ be the inertia subgroup of $G$.
	
	First we show that the distinguished triangle
		\[
				R \alg{\Gamma} R \sheafhom_{K}(A, \Z)
			\to
				R \alg{\Gamma} R \sheafhom_{K}(F, \Z)
			\to
				R \alg{\Gamma} R \sheafhom_{K}(B, \Z)
		\]
	induces a canonical morphism
		\[
				\bigl[
					\alg{\Gamma} \sheafhom_{K}(F, \Z) \to \alg{\Gamma} \sheafhom_{K}(B, \Z)
				\bigr][-1]
			\to
				R \alg{\Gamma} R \sheafhom_{K}(A, \Z),
		\]
	where $[\,\cdot\,]$ denotes the mapping cone.
	Let $\Z \isomto J$ be an injective resolution in $\Ab(K_{\fppf} / k^{\ind\rat}_{\et})$.
	Then in the category of complexes in $\Ab(k^{\ind\rat}_{\pro\et})$,
	we have natural morphism and isomorphism
		\begin{align*}
			&
					\bigl[
						\alg{\Gamma} \sheafhom_{K}(F, \Z) \to \alg{\Gamma} \sheafhom_{K}(B, \Z)
					\bigr][-1]
			\\
			&	\to
					\bigl[
							\Tilde{\alg{\Gamma}} \sheafhom_{K}(F, J)
						\to
							\Tilde{\alg{\Gamma}} \sheafhom_{K}(B, J)
					\bigr][-1]
			\\
			&	=
					\Tilde{\alg{\Gamma}} \sheafhom_{K} \bigl(
						[B \to F], J
					\bigr).
		\end{align*}
	For another injective resolution $\Z \isomto J'$ and a homotopy equivalence $J \isomto J'$ over $\Z$,
	we have a commutative diagram
		\[
			\begin{CD}
					\bigl[
						\alg{\Gamma} \sheafhom_{K}(F, \Z) \to \alg{\Gamma} \sheafhom_{K}(B, \Z)
					\bigr][-1]
				@>>>
					\Tilde{\alg{\Gamma}} \sheafhom_{K} \bigl(
						[B \to F], J
					\bigr)
				\\
				@|
				@V \wr VV
				\\
					\bigl[
						\alg{\Gamma} \sheafhom_{K}(F, \Z) \to \alg{\Gamma} \sheafhom_{K}(B, \Z)
					\bigr][-1]
				@>>>
					\Tilde{\alg{\Gamma}} \sheafhom_{K} \bigl(
						[B \to F], J'
					\bigr),
			\end{CD}
		\]
	where the right vertical arrow is a homotopy equivalence.
	All homotopy equivalences $J \isomto J'$ over $\Z$ are homotopic to each other,
	so they induce the same homotopy class for the right vertical arrow.
	We saw in the proof of \eqref{prop: localization sequence} that
	the complex $\alg{\Gamma} \sheafhom_{K}([B \to F], J)$ represents
	$R \alg{\Gamma} R \sheafhom_{K}([B \to F], \Z)$,
	or $R \alg{\Gamma} R \sheafhom_{K}(A, \Z)$.
	Hence in $D(k^{\ind\rat}_{\pro\et})$, we have a canonical morphism
		\[
				\bigl[
					\alg{\Gamma} \sheafhom_{K}(F, \Z) \to \alg{\Gamma} \sheafhom_{K}(B, \Z)
				\bigr][-1]
			\to
				R \alg{\Gamma}
				R \sheafhom_{K}(A, \Z).
		\]
	
	Since the left-hand side of this morphism is concentrated in degrees $\le 1$,
	the morphism factors through the truncation $\tau_{\le 1}$ of the right-hand side:
		\[
				\bigl[
					\alg{\Gamma} \sheafhom_{K}(F, \Z) \to \alg{\Gamma} \sheafhom_{K}(B, \Z)
				\bigr][-1]
			\to
				\tau_{\le 1}
				R \alg{\Gamma}
				R \sheafhom_{K}(A, \Z).
		\]
	This is an isomorphism
	since $F$ is finite free over $\Z[G]$,
	$\alg{H}^{1}(K, \Z[G]) = \alg{H}^{1}(L, \Z) = 0$ and hence
	$R \alg{\Gamma} R \sheafhom_{K}(F, \Z) = R \alg{\Gamma} \sheafhom_{K}(F, \Z)$ has trivial $H^{1}$.
	The sheaf $\alg{\Gamma} \sheafhom_{K}(B, \Z)$ is a lattice over $k$ by
	\eqref{prop: cohomology of local fields}
	\eqref{ass: cohomology of finite free etale over local fields}.
	Its $\closure{k}$-points is $\Hom_{K^{\ur}}(B, \Z) = \Hom_{\closure{k}}(B / I_{G}, \Z)$
	by \eqref{prop: Gamma of finite schemes over local fields is etale}.
	Therefore $\alg{\Gamma} \sheafhom_{K}(B, \Z) = \sheafhom_{k}(B / I_{G}, \Z)$.
	Similarly we have $\alg{\Gamma} \sheafhom_{K}(F, \Z) = \sheafhom_{k}(F / I_{G}, \Z)$.
	Therefore we have
		\begin{align*}
					\tau_{\le 1}
					R \alg{\Gamma}
					R \sheafhom_{K}(A, \Z)
			&	=
					\bigl[
						\alg{\Gamma} \sheafhom_{K}(F, \Z) \to \alg{\Gamma} \sheafhom_{K}(B, \Z)
					\bigr][-1]
			\\
			&	=
					\bigl[
						\sheafhom_{k}(F / I_{G}, \Z) \to \sheafhom_{k}(B / I_{G}, \Z)
					\bigr][-1].
		\end{align*}
	Consider the following part of the long exact sequence of group homology:
		\[
				H_{1}(I_{G}, A)
			\to
				B / I_{G}
			\to
				F / I_{G}
			\to
				A / I_{G}
			\to
				0.
		\]
	This is an exact sequence of \'etale groups over $k$.
	Let $X$ be the kernel of $F / I_{G} \to A / I_{G}$.
	In the short exact sequence $0 \to X \to F / I_{G} \to A / I_{G} \to 0$,
	the middle term $F / I_{G}$ is a lattice over $k$.
	Hence the same argument as above yields an isomorphism
		\[
				\bigl[
					\sheafhom_{k}(F / I_{G}, \Z) \to \sheafhom_{k}(X, \Z)
				\bigr][-1]
			=
				R \sheafhom_{k}(A / I_{G}, \Z).
		\]
	(No truncation $\tau_{\le 1}$ is necessary on the right.)
	Consider the exact sequence
		\[
				0
			\to
				\sheafhom_{k}(X, \Z)
			\to
				\sheafhom_{k}(B / I_{G}, \Z)
			\to
				\sheafhom_{k} \bigl(
					H_{1}(I_{G}, A), \Z
				\bigr).
		\]
	Since $G$ is finite,
	the group $H_{1}(I_{G}, A)$ is torsion (\cite[VIII, Cor.\ 1 to Prop.\ 4]{Ser79}),
	so the final term is zero.
	With $A / I_{G} = A / I_{K}$, we get
		\[
				\bigl[
					\sheafhom_{k}(F / I_{G}, \Z) \to \sheafhom_{k}(B / I_{G}, \Z)
				\bigr][-1]
			=
				R \sheafhom_{k}(A / I_{K}, \Z).
		\]
	Combining all the above, we get the required isomorphism.
	This isomorphism is independent of the choice of the sequence
	$0 \to B \to F \to A \to 0$.
\end{proof}

\begin{proof}[Proof of \eqref{prop: content of duality of tori}]
	We construct $\theta_{T}^{01}$.
	By \eqref{eq: ext of proalg by Z in deg 0 and 1}, the natural morphism
		\[
				R \sheafhom_{k} \bigl(
					\pi_{0}(\alg{\Gamma}(T)),
					\Z
				\bigr)
			\to
				\tau_{\le 1}
				R \sheafhom_{k} \bigl(
					\alg{\Gamma}(T),
					\Z
				\bigr)
		\]
	is an isomorphism.
	Hence we have
		\[
				\tau_{\le 1} (
					R \alg{\Gamma}(T)^{\SDual}
				)
			=
				R \sheafhom_{k} \bigl(
					\pi_{0}(\alg{\Gamma}(T)),
					\Z
				\bigr)
			=
				R \sheafhom_{k} \bigl(
					\pi_{0}(\mathcal{T}_{x}),
					\Z
				\bigr).
		\]
	Since $T^{\CDual} = X^{\ast}(T) = \sheafhom_{K}(X_{\ast}(T), \Z)$, we have
		\[
				\tau_{\le 1}
				R \alg{\Gamma}(T^{\CDual})
			=
				\tau_{\le 1}
				R \alg{\Gamma}
				\sheafhom_{K}(X_{\ast}(T), \Z)
			=
				R \sheafhom_{k}(X_{\ast}(T) / I_{K}, \Z)
		\]
	using the lemma.
	Hence the isomorphism $\theta_{T}$ induces an isomorphism
		\[
				R \sheafhom_{k}(X_{\ast}(T) / I_{K}, \Z)
			\isomto
				R \sheafhom_{k} \bigl(
					\pi_{0}(\mathcal{T}_{x}),
					\Z
				\bigr)
		\]
	in $D^{b}(\FGEt / k)$.
	Apply $\SDual$ to the both sides.
	Then \eqref{prop: Serre duality}
	\eqref{ass: Serre duality for unip and etale} gives an isomorphism
		$
				\theta_{T}^{01}
			\colon
				\pi_{0}(\mathcal{T}_{x})
			\isomto
				X_{\ast}(T) / I_{K}
		$.
	
	We construct $\theta_{T}^{2}$.
	Applying the direct limit by multiplication by $n \ge 1$ for $\theta_{T}$, we have
		\[
				\theta_{T} \tensor \Q
			\colon
				\dirlim_{n}
				R \alg{\Gamma}(T^{\CDual})
			\isomto
				\dirlim_{n}
				R \sheafhom_{k}(R \alg{\Gamma}(T), \Z).
		\]
	For any $k' \in k^{\ind\rat}$, we have
		\[
				\dirlim_{n} R \Gamma(\alg{K}(k'), T^{\CDual})
			=
				R \Gamma(\alg{K}(k'), \dirlim_{n} T^{\CDual})
			=
				R \Gamma(\alg{K}(k'), T^{\CDual} \tensor_{\Z} \Q)
		\]
	by \cite[III, Rmk.\ 3.6]{Mil80},
	so $\dirlim_{n} R \alg{\Gamma}(T^{\CDual}) = R \alg{\Gamma}(T^{\CDual} \tensor_{\Z} \Q)$.
	Also
		\[
				\dirlim_{n}
				R \sheafhom_{k}(\alg{\Gamma}(T)_{0}, \Z)
			=
				R \sheafhom_{k}(\alg{\Gamma}(T)_{0}, \Q)
			=
				0
		\]
	by \eqref{thm: Ext of proalgebraic groups as sheaves}
	\eqref{ass: Ext commutes with direct limits},
	\eqref{ass: Ext by uniquely divisible is zero}
	and \eqref{prop: comparison thm in the pro-etale topology}
	since $\alg{\Gamma}(T)_{0} \in \Pro \Alg / k$,
	and
		\[
				\dirlim_{n}
				R \sheafhom_{k} \bigl(
					\pi_{0}(\alg{\Gamma}(T)), \Z
				\bigr)
			=
				R \sheafhom_{k} \bigl(
					\pi_{0}(\alg{\Gamma}(T)), \Q
				\bigr)
		\]
	since $\pi_{0}(\alg{\Gamma}(T)) \in \FGEt / k$
	and $\Ext_{\Ab}^{n}(G, \;\cdot\;)$ for finitely generated $G$ commutes with filtered direct limits.
	Hence
		$
				\dirlim_{n}
				R \sheafhom_{k}(\alg{\Gamma}(T), \Z)
			=
				R \sheafhom_{k}(\alg{\Gamma}(T), \Q)
		$.
	Therefore
		\[
				\theta_{T} \tensor \Q
			\colon
				R \alg{\Gamma}(T^{\CDual} \tensor_{\Z} \Q)
			\isomto
				R \sheafhom_{k}(R \alg{\Gamma}(T), \Q).
		\]
	Together with $\theta_{T}$, we have a morphism
		\[
			\begin{CD}
					R \alg{\Gamma}(T^{\CDual})
				@>>>
					R \alg{\Gamma}(T^{\CDual} \tensor_{\Z} \Q)
				@>>>
					R \alg{\Gamma}(T^{\CDual} \tensor_{\Z} \Q / \Z)
				\\
				@VVV
				@VVV
				@VVV
				\\
					R \sheafhom_{k}(\alg{\Gamma}(T), \Z)
				@>>>
					R \sheafhom_{k}(\alg{\Gamma}(T), \Q)
				@>>>
					R \sheafhom_{k}(\alg{\Gamma}(T), \Q / \Z).
			\end{CD}
		\]
	of distinguished triangles,
	whose left two vertical arrows are isomorphisms.
	Thus we get an isomorphism
		\begin{equation} \label{eq: torsion version of duality for tori}
				R \alg{\Gamma}(\hat{T}_{\tor})
			\isomto
				R \sheafhom_{k}(\alg{\Gamma}(T), \Q / \Z).
		\end{equation}
	In degree $1$, it gives $\theta_{T}^{2}$.
\end{proof}

The existence of an isomorphism
	$
			\pi_{0}(\mathcal{T}_{x})
		\cong
			X_{\ast}(T) / I_{K}
	$
was also obtained by Bertapelle and Gonz\'alez-Avil\'es \cite{BGA15} in another method.
Note that the result of Xarles \cite{Xar93} has been recovered in the course of the above proof.

We describe $\theta_{T}^{01}$.
First we need a preparation about not necessarily totally ramified extensions
(cf.\ \cite[\S 4.2]{SY12}).
For a finite Galois extension $L / K$ with residue extension $k' / k$,
let $\alg{L}^{\times} \in \Ab(k'^{\ind\rat}_{\pro\et})$ be
the sheaf defined in the same way as $\alg{K}^{\times}$:
	\[
			\alg{L}^{\times}(k'')
		=
			\bigl(
				(W(k'') \Hat{\tensor}_{W(k')} \Order_{L}) \tensor_{\Order_{L}} L
			\bigr)^{\times}
	\]
for $k'' \in k'^{\ind\rat}$.
On the other hand, we denote
$\alg{L}_{k}^{\times} = \alg{\Gamma}(K, \Res_{L / K} \Gm)$,
where $\Res_{L / K}$ is the Weil restriction.
These sheaves are related in the following way:
for any $k'' \in k^{\ind\rat}$, we have
	\[
			\alg{L}_{k}^{\times}(k'')
		=
			(\alg{K}(k'') \tensor_{K} L)^{\times}
		=
			\alg{L}^{\times}(k'' \tensor_{k} k')
	\]
(see \cite[\S 4.2]{SY12}; the notation there for $\alg{L}^{\times}$ was $\alg{L}_{k'}^{\times}$).
Hence $\alg{L}_{k}^{\times} = \Res_{k' / k} \alg{L}^{\times}$.
Let $\Res_{k' / k} \Z = \Z[\Gal(k' / k)]$ be the group ring viewed as an \'etale group over $k$.
Applying $\Res_{k' / k}$ to the split exact sequence
$0 \to \alg{U}_{L} \to \alg{L}^{\times} \to \Z \to 0$ in $\Ab(k'^{\ind\rat}_{\pro\et})$
(where $\alg{U}_{L}$ is defined similarly as $\alg{U}_{K}$),
we have a split exact sequence
	\[
			0
		\to
			\Res_{k' / k} \alg{U}_{L}
		\to
			\alg{L}_{k}^{\times}
		\to
			\Z[\Gal(k' / k)]
		\to
			0
	\]
in $\Ab(k^{\ind\rat}_{\pro\et})$.
The group $\Res_{k' / k} \alg{U}_{L}$ is connected,
so $\pi_{0}(\alg{L}_{k}^{\times}) = \Z[\Gal(k' / k)]$.

Now we define a map that is going to be inverse to $\theta_{T}^{01}$.
Let $\lambda \colon \Gm \to T$ be a morphism over $L$.
The duality for the finite \'etale morphism $\Spec L \to \Spec K$ gives a morphism
$\Res_{L / K} \Gm \to T$ over $K$.
In other words, this is the composite of $\lambda \colon \Res_{L / K} \Gm \to \Res_{L / K} T$
with the norm map $\Res_{L / K} T \to T$ for $T$.
Applying $\alg{\Gamma}$, we have a morphism
$\alg{L}_{k}^{\times} \to \alg{\Gamma}(K, T)$ over $k$,
Applying $\pi_{0}$, we have a morphism
$\Z[\Gal(k' / k)] \to \pi_{0}(\alg{\Gamma}(T)) = \pi_{0}(\mathcal{T}_{x})$ over $k$.
The image of $1 \in \Z[\Gal(k' / k)]$ corresponds to a $k'$-point of $\pi_{0}(\mathcal{T}_{x})$.
Thus we have a homomorphism
$\Hom_{L}(\Gm, T) \to \Gamma(k', \pi_{0}(\mathcal{T}_{x}))$.
This is compatible with the action of $\Gal(L / K)$
(which factors through $\Gal(k' / k)$ on the right).
It is also compatible with extending $L$.
Therefore we obtain a homomorphism
$X_{\ast}(T) / I_{K} \to \pi_{0}(\mathcal{T}_{x})$ of \'etale groups over $k$,
which we denote by $\varphi$.

\begin{Prop} \label{prop: description of Bitan isomorphism}
	The homomorphism $\varphi \colon X_{\ast}(T) / I_{K} \to \pi_{0}(\mathcal{T}_{x})$ just defined
	is the inverse map of the isomorphism $\theta_{T}^{01}$.
\end{Prop}

\begin{proof}
	We may assume that $k$ is algebraically closed,
	so $G_{K} := I_{K}$ is the absolute Galois group of $K$.
	It is enough to compare the homomorphisms induced on the $\Q / \Z$-duals
	$\Hom_{\Ab}(X_{\ast}(T) / I_{K}, \Q / \Z)$
	and $\Hom_{\Ab}(\pi_{0}(\mathcal{T}_{x}), \Q / \Z)$.
	
	By \eqref{eq: torsion version of duality for tori},
	we have
		\[
				R \Gamma(\hat{T}_{\tor})
			\isomto
				R \Hom_{k}(\alg{\Gamma}(T), \Q / \Z).
		\]
	Therefore
		\begin{equation} \label{eq: cup for tori in degree zero}
				\Gamma(\hat{T}_{\tor})
			\isomto
				\Hom_{k}(\alg{\Gamma}(T), \Q / \Z)
			\isomto
				\Hom_{\Ab} \bigl(
					\pi_{0}(\alg{\Gamma}(T)), \Q / \Z
				\bigr).
		\end{equation}
	Also, since $\Hat{T}_{\tor} = X^{\ast}(T) \tensor_{\Z} \Q / \Z = \sheafhom_{K}(X_{\ast}(T), \Q / \Z)$,
	we have a natural isomorphism
		\begin{align} \label{eq: char-cochar duality}
					\Gamma(\hat{T}_{\tor})
			&	\isomto
					\Hom_{K} \bigl(
						X_{\ast}(T), \Q / \Z
					\bigr)
			\\ \notag
			&	\isomto
					\Hom_{\Ab} \bigl(
						X_{\ast}(T) / G_{K}, \Q / \Z
					\bigr).
		\end{align}
	By construction, 
	the $\Q / \Z$-dual of $\theta_{T}^{01}$
	is given by the composite
		\[
				\Hom_{\Ab} \bigl(
					X_{\ast}(T) / G_{K}, \Q / \Z
				\bigr)
			\isomfrom
				\Gamma(\hat{T}_{\tor})
			\isomto
				\Hom_{\Ab} \bigl(
					\pi_{0}(\alg{\Gamma}(T)), \Q / \Z
				\bigr).
		\]
	It is enough to show that the morphism
		\[
				\Gamma(\hat{T}_{\tor})
			\stackrel{\eqref{eq: cup for tori in degree zero}}{\isomto}
				\Hom_{\Ab} \bigl(
					\pi_{0}(\alg{\Gamma}(T)), \Q / \Z
				\bigr)
			\stackrel{\varphi}{\to}
				\Hom_{\Ab} \bigl(
					X_{\ast}(T) / G_{K}, \Q / \Z
				\bigr)
		\]
	coincides with \eqref{eq: char-cochar duality}.
	
	Let $\chi \tensor r \in \hat{T}_{\tor} = X^{\ast}(T) \tensor \Q / \Z$ be $G_{K}$-invariant,
	where $\chi \colon T \to \Gm$ is a morphism over a finite Galois extension $L / K$
	and $r \in \Q / \Z$.
	Then we have a morphism
		\[
				\alg{\Gamma}(L, T)
			\stackrel{\chi}{\to}
				\alg{L}^{\times}
			\stackrel{v_{L}}{\onto}
				\Z
			\stackrel{r}{\to}
				\Q / \Z
		\]
	over $k$.
	The norm map $N \colon \alg{\Gamma}(L, T) \to \alg{\Gamma}(K, T)$ identifies
	$\alg{\Gamma}(K, T)$ as the $G_{K}$-coinvariants of $\alg{\Gamma}(L, T)$
	as we saw in the second paragraph of the proof of
	\eqref{thm: duality with coefficients in tori}.
	The $G_{K}$-invariance of $\chi \tensor r$ implies that
	the above $r v_{L} \compose \chi$ factors through $\alg{\Gamma}(T) = \alg{\Gamma}(K, T)$
	and hence through its $\pi_{0}$.
	This defines a homomorphism
		\[
				r v_{L} \compose \chi \compose N^{-1}
			\colon
				\pi_{0}(\mathcal{T}_{x})
			\to
				\Q / \Z.
		\]
	This is the element given by \eqref{eq: cup for tori in degree zero}.
	
	We calculate the image (or pullback) of this element by $\varphi$.
	If $\lambda \colon \Gm \to T$ is a morphism over $L$,
	then $\varphi(\lambda) = N \lambda(\pi_{L})$
	by the construction of $\varphi$,
	where $\pi_{L}$ is a prime element of $L$.
	Hence the pullback by $\varphi$ of $r v_{L} \compose \chi \compose N^{-1}$
	sends $\lambda$ to
		\begin{align*}
					r v_{L} \compose \chi \compose N^{-1} \compose N \compose \lambda(\pi_{L})
			&	=
					r v_{L} \compose \chi \compose \lambda(\pi_{L})
			\\
			&	=
					r v_{L}(\pi_{L}^{\langle \chi, \lambda \rangle})
			\\
			&	=
					r \langle \chi, \lambda \rangle
				\in
					\Q / \Z,
		\end{align*}
	where $\langle \chi, \lambda \rangle \in \Z$ is the natural pairing between a character and a cocharacter.
	The obtained assignment
	$\lambda \mapsto r \langle \chi, \lambda \rangle$ agrees with
	the image of $\chi \tensor r$ under the homomorphism
	\eqref{eq: char-cochar duality}.
\end{proof}

We describe $\theta_{T}^{2}$.
One method to do this is to go back to the proof of \eqref{thm: duality with coefficients in tori},
namely use Galois descent to reduce to Serre's local class field theory;
cf.\ \eqref{rmk: remarks after proof of main theorem}
\eqref{rmk: comparison with Bertapelle in non-semistable case}.
Another method is to use the sequence $0 \to T[n] \to T \stackrel{n}{\to} T \to 0$
for each $n \ge 1$ to reduce it to the description of the duality
with coefficients in finite flat group schemes $T[n]$.
We give here yet another, more direct method.
It is enough to describe the morphism induced on $k$-points
since the groups involved with $\theta_{T}^{2}$ are ind-algebraic and $k$ can be any perfect field.

An element of
	\[
			H^{1}(\hat{T}_{\tor})
		=
			H^{2}(T^{\CDual})
		=
			\Ext_{K}^{2}(\Z, T^{\CDual})
	\]
corresponds to an extension class
	\[
		E \colon 0 \to T^{\CDual} \to X_{2} \to X_{1} \to \Z \to 0.
	\]
The terms $X_{i}$ can be chosen more explicitly as follow.
If the element lives in $H^{2}(G, \Gamma(L, T^{\CDual}))$
with $G = \Gal(L / K)$ for a finite Galois extension $L / K$,
then it gives a morphism of complexes of $G$-modules
from the standard resolution of $\Z$ to $\Gamma(L, T^{\CDual})[2]$.
By pushing out the standard resolution
	\[
			\cdots
		\to
			\Z[G^{4}]
		\to
			\Z[G^{3}]
		\to
			\Z[G^{2}]
		\to
			\Z[G]
		\to
			\Z
		\to
			0
	\]
(of homogeneous chains) at the degree $-2$ term $\Z[G^{3}]$ to $\Gamma(L, T^{\CDual})$,
we get a choice of (the $L$-valued points of) the extension class $E$:
	\[
			0
		\to
			\Gamma(L, T^{\CDual})
		\to
			X
		\to
			\Z[G]
		\to
			\Z
		\to
			0.
	\]
All these terms are lattices over $K$.
The image of $E$ in $\Ext_{K}^{2}(T, \Gm)$ is given by taking its Cartier dual
	\[
			0
		\to
			\Gm
		\to
			\Res_{L / K} \Gm
		\to
			X^{\CDual}
		\to
			T
		\to
			0.
	\]
Apply $R \Tilde{\alg{\Gamma}}$.
Since the terms are all tori, we have an exact sequence
	\[
			0
		\to
			\alg{K}^{\times}
		\to
			\alg{L}_{k}^{\times}
		\to
			\alg{\Gamma}(X^{\CDual})
		\to
			\alg{\Gamma}(T)
		\to
			0
	\]
in $\Ab(k^{\ind\rat}_{\pro\et})$
by \eqref{prop: cohomology of local fields}.
By pushing it out by the valuation map $\alg{K}^{\times} \onto \Z$,
we get an extension
	\[
			0
		\to
			\Z
		\to
			\alg{L}_{k}^{\times} / \alg{U}_{K}
		\to
			\alg{\Gamma}(X^{\CDual})
		\to
			\alg{\Gamma}(T)
		\to
			0.
	\]
This is the element of
	$
		\Ext_{k}^{2} \bigl(
			\alg{\Gamma}(T),
			\Z
		\bigr)
	$
corresponding to $\theta_{T}$.
Let $\psi$ be the middle morphism
	$
			\alg{L}_{k}^{\times}
		\to
			\alg{\Gamma}(X^{\CDual})
	$.
Recall the morphism $\alg{L}_{k}^{\times} \onto \Z[\Gal(k' / k)]$
defined before \eqref{prop: description of Bitan isomorphism}.
The maximal constant quotient of the \'etale group $\Z[\Gal(k' / k)]$ is $\Z$.
Let $\alg{U}_{L, \, k' / k}$ be the kernel of the composite
$\alg{L}_{k}^{\times} \onto \Z[\Gal(k' / k)] \onto \Z$.
Let $n = [L : K]$.
Then the above extension class comes from the extension
	\[
			0
		\to
			\Z / n \Z
		\to
				\alg{\Gamma}(X^{\CDual})
			/
				\psi(\alg{U}_{L, \, k' / k})
		\to
			\alg{\Gamma}(T)
		\to
			0
	\]
in
	$
		\Ext_{k}^{1} \bigl(
			\alg{\Gamma}(T),
			\Z / n \Z
		\bigr)
	$.
By pushing it out by $\Z / n \Z \into \Q / \Z$,
we get an corresponding extension class in
	$
		\Ext_{k}^{1} \bigl(
			\alg{\Gamma}(T),
			\Q / \Z
		\bigr)
	$.
This is the element given by $\theta_{T}^{2}$.

\begin{Rmk} \label{rmk: where we need P-acyclicity for tori}
	As in \eqref{rmk: where we need P-acyclicity for abelian variety},
	we can give a statement on $R \Gamma(T)$ and $R \Gamma(T^{\CDual})$
	even if $k$ is not algebraically closed,
	thanks to the P-acyclicity of $R \alg{\Gamma}(T)$ and $R \alg{\Gamma}(T^{\CDual})$.
	Let $(\mathcal{T}_{x})_{\mathrm{t}}$ be the maximal torus of
	the special fiber $\mathcal{T}_{x}$
	and set $Q_{k, T} = R \invlim_{n} R \Gamma(k_{\et}, (\mathcal{T}_{x})_{\mathrm{t}})$.
	Then by the same reasoning as the cited remark, we have
		\begin{gather*}
					R \Gamma(T^{\CDual})
				=
					R \Hom_{k^{\ind\rat}_{\et}} \bigl(
						\alg{\Gamma}(T),
						\Z
					\bigr),
			\\
					\bigl[
							Q_{k, T}
						\to
							R \Gamma(T)
					\bigr]
				=
					R \Hom_{k^{\ind\rat}_{\et}} \bigl(
						R \alg{\Gamma}(T^{\CDual}),
						\Z
					\bigr)
		\end{gather*}
	in $D(\Ab)$.
	By the isomorphism \eqref{eq: torsion version of duality for tori},
	we have
		\[
				R \Gamma(\hat{T}_{\tor})
			=
				R \Hom_{k^{\ind\rat}_{\et}} \bigl(
					\alg{\Gamma}(T),
					\Q / \Z
				\bigr).
		\]
	If we look at the maximal subgroups of $\alg{\Gamma}(T)$ and $\alg{\Gamma}(T^{\CDual})$
	whose quotients are lattices,
	then the corresponding parts of the right-hand sides of these isomorphisms
	may be written by $R \Hom_{\Pro \Alg / k}$ and $R \Hom_{\Ind \Alg / k}$.
\end{Rmk}


\numberwithin{equation}{section}
\section{Duality with coefficients in $1$-motives}
\label{sec: Duality with coefficients in 1-motives}

Next we treat the case of coefficients in $1$-motives.
We only formulate and prove the duality,
without giving its full description.
Let $M = [Y \to G]$ be a $1$-motive over $K$ in the sense of Deligne,
where $Y$ is a lattice placed in degree $-1$ and $G$ a semi-abelian variety.
We can naturally regard $M$ as an object in $D(K_{\fppf})$
and hence in $D(K_{\fppf} / k^{\ind\rat}_{\et})$.
We calculate $R \alg{\Gamma}(M)$ by
\eqref{prop: cohomology of local fields}.
Applying $R \alg{\Gamma}$ on $M$, we have a long exact sequence
	\begin{align*}
		&
				0
			\to
				\alg{H}^{-1}(M)
			\to
				\alg{\Gamma}(Y)
			\to
				\alg{\Gamma}(G)
			\to
				\alg{H}^{0}(M)
		\\
		&	\to
				\alg{H}^{1}(Y)
			\to
				\alg{H}^{1}(G)
			\to
				\alg{H}^{1}(M)
			\to
				\alg{H}^{2}(Y)
			\to
				0
	\end{align*}
in $\Ab(k^{\ind\rat}_{\et})$
and $\alg{H}^{n}(M) = 0$ for $n \ne -1, 0, 1$.
We know that $\alg{\Gamma}(Y)$ is a lattice over $k$,
$\alg{\Gamma}(G)$ is an extension of a lattice
by a P-acyclic proalgebraic group
(which can be checked by reducing it to the case of an abelian variety and a torus,
using the fact that an extension of a proalgebraic group by $\Q$ splits
by \eqref{prop: comparison thm for ind-proalgebraic groups}),
$\alg{H}^{1}(Y) \in \FEt / k$, and
$\alg{H}^{1}(G), \alg{H}^{2}(Y) \in \Ind \Alg_{\uc} / k$.
Hence $\alg{H}^{-1}(M)$ is a lattice over $k$
and $\alg{H}^{1}(M) \in \Ind \Alg_{\uc} / k$.
Since all the terms but $\alg{H}^{0}(M)$ are P-acyclic,
we know that $\alg{H}^{0}(M)$ is P-acyclic,
Therefore $R \alg{\Gamma}(M)$ is P-acyclic,
and we will write $R \Tilde{\alg{\Gamma}}(M) = R \alg{\Gamma}(M) \in D(k^{\ind\rat}_{\pro\et})$.

Let $M^{\vee}$ be the dual $1$-motive of $M$.
In $D(K_{\fppf})$,
the complex $M^{\vee}$ is identified with $\tau_{\le 0} R \sheafhom_{K}(M, \Gm[1])$.
Note the shift: if $M$ is a torus $T$,
then $M^{\vee}$ is the character lattice $T^{\CDual}$ placed in degree $-1$.
We have a morphism
	\begin{align*}
				R \alg{\Gamma}(M^{\vee})
		&	\to
				R \Tilde{\alg{\Gamma}}
				R \sheafhom_{K}(M, \Gm[1])
		\\
		&	\to
				R \sheafhom_{k} \bigl(
					R \alg{\Gamma}(M),
					\Z
				\bigr)[1]
			=
				R \alg{\Gamma}(M)^{\SDual}[1]
	\end{align*}
in $D(k^{\ind\rat}_{\pro\et})$ and hence a morphism
	\[
			R \alg{\Gamma}(M^{\vee})^{\SDual \SDual}
		\to
			R \alg{\Gamma}(M)^{\SDual}[1]
	\]
by taking $\SDual$ of the both sides and replacing $M$ with $M^{\vee}$.

\begin{Thm} \label{thm: duality for one-motives}
	The above defined morphism
		\[
				R \alg{\Gamma}(M^{\vee})^{\SDual \SDual}
			\to
				R \alg{\Gamma}(M)^{\SDual}[1]
		\]
	in $D(k^{\ind\rat}_{\pro\et})$ is an isomorphism.
\end{Thm}

\begin{proof}
	Note that the both sides are triangulated functors on $M$.
	Therefore if $M_{1} \to M_{2} \to M_{3}$ is a distinguished triangle of $1$-motives
	and the theorem is true for any two of them,
	then so is the third.
	Let $M = [Y \to G]$ with $Y$ a lattice over $K$,
	$G$ semi-abelian with torus part $T$ and $A = G / T$.
	We have the weight filtration $W_{n} M$ of $M$,
	where
	$W_{0} M = M$, $W_{-1} M = [0 \to G]$,
	$W_{-2} M = [0 \to T]$ and $W_{-3} M = 0$.
	The distinguished triangles
		\[
				W_{-1} M
			\to
				M
			\to
				[Y \to 0],
			\quad
				W_{-2} M
			\to
				W_{-1} M
			\to
				[0 \to A]
		\]
	show that the statement for $M$ is reduced to that of $A$, $T$ and $Y$.
	These are \eqref{thm: duality for abelian varieties}
	and \eqref{thm: duality with coefficients in tori}.
\end{proof}

\begin{Rmk}
	In the proof of \eqref{thm: duality for abelian varieties},
	we have used the case of semistable abelian varieties.
	The above formulation for $1$-motives allows us to deduce this case
	from the case of good reduction abelian varieties, as follows.
	
	Let $A$ be a semistable abelian variety over $K$.
	Then by rigid analytic uniformization,
	we have an exact sequence
	$0 \to Y \to G \to A \to 0$ of rigid $K$-groups,
	where $Y$ is a lattice with unramified Galois action
	and $G$ an extension of a good reduction abelian variety by an unramified torus.
	Set $M = [Y \to G]$.
	Let $K^{\fp}_{\et} / k^{\ind\rat}$ be the modification of the category $K_{\et} / k^{\ind\rat}$
	where objects are pairs $(S, k_{S})$ with $k_{S} \in k^{\ind\rat}$ and
	$S$ an \'etale $\alg{K}^{\fp}(k_{S})$-algebra.
	Denote $\hat{S} = S \tensor_{\alg{K}^{\fp}(k_{S})} \alg{K}(k_{S})$.
	Let $\Spec K^{\fp}_{\et} / k^{\ind\rat}_{\et}$ be the \'etale site on $K^{\fp}_{\et} / k^{\ind\rat}$.
	The functors $k' \mapsto (\alg{K}^{\fp}(k'), k')$ and
	$(S, k_{S}) \mapsto (\hat{S}, k_{S})$ define morphisms
		\[
				\Spec K_{\et} / k^{\ind\rat}_{\et}
			\stackrel{f}{\to}
				\Spec K^{\fp}_{\et} / k^{\ind\rat}_{\et}
			\stackrel{\pi^{\fp}}{\to}
				\Spec k^{\ind\rat}_{\et}
		\]
	of sites.
	The pushforward $f_{\ast}$ is exact by \cite[Lem.\ 2.5.5]{Suz13}.
	Note that the ring $\hat{S}$
	for any \'etale $\alg{K}^{\fp}(k_{S})$-algebra $S$ is a complete topological $K$-algebra.
	We can show that
	the sequence $0 \to Y \to G \to A \to 0$ of rigid $K$-groups
	induces an exact sequence
	$0 \to f_{\ast} Y \to f_{\ast} G \to f_{\ast} A \to 0$ in $\Ab(K^{\fp}_{\et} / k^{\ind\rat}_{\et})$,
	by a rigid analytic version of the approximation argument used in the proof of the cited lemma.
	Hence we have $f_{\ast} M \isomto f_{\ast} A$ in $\Ab(K^{\fp}_{\et} / k^{\ind\rat}_{\et})$.
	Applying $R \pi^{\fp}_{\ast}$ and using the exactness of $f_{\ast}$,
	we have an isomorphism
		\[
				R \alg{\Gamma}(M)
			\isomto
				R \alg{\Gamma}(A).
		\]
	Hence \eqref{thm: duality for abelian varieties} for $A$ is reduced to
	\eqref{thm: duality for one-motives} for $M$.
	The cases of $\Gm$ and $\Z$ were explained in \eqref{thm: duality with coefficients in tori},
	which was essentially Serre's local class field theory.
	Hence we are reduced to the case of abelian varieties with good reduction.
	
	Note that the case of abelian varieties with good reduction
	is further reduced to the case of finite flat group schemes over $\Order_{K}$
	obtained as the torsion part,
	as explained in B\'egueri \cite[\S 7.1]{Beg81} and Bester \cite[\S 2.7]{Bes78}.
	With this argument, we can avoid Bertapelle's and Werner's results,
	and prove Grothendieck's and \v{S}afarevi\v{c}'s conjectures directly
	from the B\'egueri-Bester duality for $\Order_{K}$ with coefficients in finite flat group schemes.
\end{Rmk}


\numberwithin{equation}{section}
\section{Connection with classical statements for finite residue fields}
\label{sec: Connection with classical statements for finite residue fields}

We show that the duality theorems in this paper induce classical duality statements
when the residue field $k$ is finite.
Basically this can be done as follows.
Take the derived global section $R \Gamma(k_{\pro\et}, \;\cdot\;)$
of the both sides of our duality
to translate $R \alg{\Gamma}(K, A) \in D(k^{\ind\rat}_{\pro\et})$ into
$R \Gamma(K, A) \in D(\Ab)$.
Then use Lang's theorem (see below)
	\[
			\Ext_{k}^{1}(C, \Q / \Z)
		=
			\Hom_{\Ab}(C(k), \Q / \Z)
	\]
for a connected quasi-algebraic group $C$ over $k$
to translated the Serre dual to the Pontryagin dual.
We can actually treat not only Galois group cohomology but also Weil group cohomology,
and thanks to our setting of the pro-\'etale topology,
recover classical dualities that take profinite group structures on cohomology groups into account.

We need to define and study several notions for this.
Assume that $k = \F_{q}$.
Let $\closure{k}$ be an algebraic closure of $k$.
We denote by $F$ the $q$-th power Frobenius homomorphism on any $k$-algebra.
It induces an action on any object of $\Ab(k^{\ind\rat}_{\pro\et})$.
For $A \in D(k^{\ind\rat}_{\pro\et})$,
we denote the mapping cone of $F - 1 \colon A \to A$ by
$[A \stackrel{F - 1}{\to} A]$.
With a shift, we have a triangulated functor
	\begin{equation} \label{eq: Frobenius mapping cone}
			D(k^{\ind\rat}_{\pro\et})
		\to
			D(k^{\ind\rat}_{\pro\et}),
		\quad
			A
		\mapsto
			[A \to A][-1].
	\end{equation}
This is the mapping fiber of $F - 1 \colon A \to A$.

Next, as in \S \ref{sec: Sites and algebraic groups: setup and first properties},
let $k_{\pro\et}$ be the category of ind-\'etale $k$-algebras
and $\Spec k_{\pro\et}$ the pro-\'etale site on $k_{\pro\et}$.
Let $k_{\pro\zar}$ be the full subcategory of $k_{\pro\et}$
consisting of filtered unions $\bigcup k'_{\lambda}$
where each $k'_{\lambda}$ is a finite product of copies of $k$.
For $k' \in k_{\pro\zar}$,
we say that a finite family $\{k'_{i}\}$ of objects of $k_{\pro\zar}$ over $k'$ is a covering
if $k' \to \prod_{i} k'_{i}$ is faithfully flat.
This defines a site, which we call the pro-Zariski site of $k$
and denote by $\Spec k_{\pro\zar}$.
This is equivalent to the pro-\'etale site
$\Spec \closure{k}_{\pro\et}$ of $\closure{k}$
via the functor $k' \in k_{\pro\zar} \mapsto k' \tensor_{k} \closure{k} \in \closure{k}_{\pro\et}$.
It is also equivalent to the site of profinite sets given in \cite[Ex.\ 4.1.10]{BS15}.
By composing the continuous map of sites
	$
			\Spec k^{\ind\rat}_{\pro\et}
		\to
			\Spec \closure{k}_{\pro\et}
	$
defined by the identity with the identification
$\Spec \closure{k}_{\pro\et} \cong \Spec k_{\pro\zar}$,
we have a continuous map of sites
	\[
			f
		\colon
			\Spec k^{\ind\rat}_{\pro\et}
		\to
			\Spec k_{\pro\zar},
	\]
which is defined by the functor
$k_{\pro\zar} \to k_{\pro\et}$
given by $k' \mapsto k' \tensor_{k} \closure{k}$.
Its pushforward functor $f_{\ast} \colon \Ab(k^{\ind\rat}_{\pro\et}) \to \Ab(k_{\pro\zar})$ is exact,
which induces a triangulated functor
	\begin{equation} \label{eq: geometric points functor}
			f_{\ast}
		\colon
			D(k^{\ind\rat}_{\pro\et})
		\to
			D(k_{\pro\zar}).
	\end{equation}

The functor $f_{\ast}$ is a sort of a geometric points functor in the following sense.
For any $A \in \Ab(k^{\ind\rat}_{\pro\et})$ we have
$(f_{\ast} A)(k) = A(\closure{k})$.
If $A$ is locally of finite presentation,
we may view the abelian group $A(\closure{k})$ as a constant sheaf on $\Spec k_{\pro\zar}$,
and for any $k' = \bigcup k'_{\lambda} \in k_{\pro\zar}$ with $k_{\lambda}$ a finite product of copies of $k$,
we have
	\[
			(f_{\ast} A)(k')
		=
			\dirlim A(k'_{\lambda} \tensor_{k} \closure{k})
		=
			\dirlim A(\closure{k})(k'_{\lambda})
		=
			A(\closure{k})(k').
	\]
Thus $f_{\ast} A = A(\closure{k})$ in $\Ab(k_{\pro\zar})$ if $A$ is locally of finite presentation.
Note that $f_{\ast}$ commutes with filtered direct limits and arbitrary inverse limits.
Hence if $A = \dirlim_{\lambda} \invlim_{\lambda'} A_{\lambda \lambda'} \in \Ind \Pro \Alg / k$
with $A_{\lambda \lambda'} \in \Alg / k$, then
$f_{\ast} A = \dirlim_{\lambda} \invlim_{\lambda'} (A_{\lambda \lambda'}(\closure{k}))$
in $\Ab(k_{\pro\zar})$.

Now we define a functor $R \Gamma(k_{W}, \;\cdot\;)$
to be the composite of \eqref{eq: Frobenius mapping cone} and
\eqref{eq: geometric points functor}:
	\begin{gather*}
				R \Gamma(k_{W}, \;\cdot\;)
			\colon
				D(k^{\ind\rat}_{\pro\et})
			\to
				D(k_{\pro\zar}),
		\\
				R \Gamma(k_{W}, A)
			=
				f_{\ast} [A \stackrel{F - 1}{\to} A].
	\end{gather*}
For $A \in \Ab(k^{\ind\rat}_{\pro\et})$,
we define $H^{n}(k_{W}, A) \in \Ab(k_{\pro\zar})$ to be the $n$-th cohomology of
$R \Gamma(k_{W}, A)$,
and set $\Gamma(k_{W}, A) = H^{0}(k_{W}, A)$.
Obviously $H^{n}(k_{W}, A) = 0$ for $n \ne 0, 1$.
We have an exact sequence
	\[
			0
		\to
			\Gamma(k_{W}, A)(k)
		\to
			A(\closure{k})
		\stackrel{F - 1}{\to}
			A(\closure{k})
		\to
			H^{1}(k_{W}, A)(k)
		\to
			0,
	\]
so that the global section of $R \Gamma(k_{W}, A) \in D(k_{\pro\zar})$ is
$R \Gamma(W_{k}, A(\closure{k})) \in D(\Ab)$,
where $W_{k} \cong \Z$ is the Weil group of $k$.
Thus $R \Gamma(k_{W}, \;\cdot\;)$ is an enhancement of the Weil group cohomology of $k$
as a functor valued in $D(k_{\pro\zar})$.

Let $\Fin$ be the category of finite abelian groups.
Let $\Pro \Fin$, $\Ind \Fin$, $\Ind \Pro \Fin$
its procategory, indcategory, ind-procategory, respectively.
Note that $\Ind \Fin$ is equivalent to the category of torsion abelian groups.
A finite set $X$ can be identified with a finite $k$-scheme $\bigsqcup_{x \in X} \Spec k$
and hence a finite product $\prod_{x \in X} k$ of copies of $k$.
This extends to a functor from the category of profinite sets to $k_{\pro\zar} \subset k^{\ind\rat}$.
Hence we have an additive functor
$\Ind \Pro \Fin \to \Ab(k_{\pro\zar}) \subset \Ab(k^{\ind\rat}_{\pro\et})$.
There is also an obvious functor $\Ab \to \Ab(k_{\pro\zar})$
that sends an abelian group $A$ to $\bigsqcup_{a \in A} \Spec k$.

\begin{Prop} \BetweenThmAndList \label{prop: ind-profinite groups and ind-proalgebraic groups}
	\begin{enumerate}
		\item \label{ass: fully faithful embedding for ind-profinite groups}
			We have fully faithful embeddings
				\[
					\begin{CD}
							D^{b}(\Fin)
						@>> \subset >
							D^{b}(\Ind \Fin)
						@>> \subset >
							D^{b}(\Ab)
						\\
						@VV \cap V
						@V \cap VV
						@V \cap VV
						\\
							D^{b}(\Pro \Fin)
						@> \subset >>
							D^{b}(\Ind \Pro \Fin)
						@> \subset >>
							D^{b}(k_{\pro\zar}),
					\end{CD}
				\]
			of triangulated categories.
		\item \label{ass: Weil group cohomology of ind-proalgebraic groups}
			If $A \in \Alg / k$, $\Pro \Alg / k$, $\Ind \Alg / k$ or $\Ind \Pro \Alg / k$,
			then the kernel and cokernel of $F - 1 \colon A \to A$ are in
			$\Fin$, $\Pro \Fin$, $\Ind \Fin$ or $\Ind \Pro \Fin$, respectively.
			We have $H^{n}(k_{W}, A) \in \Fin$, $\Pro \Fin$, $\Ind \Fin$ or $\Ind \Pro \Fin$,
			respectively, for any $n$.
	\end{enumerate}
\end{Prop}

\begin{proof}
	\eqref{ass: fully faithful embedding for ind-profinite groups}
	For any $B \in \Ab$ and $n \ge 1$,
	we have $\Ext_{k_{\pro\zar}}^{n}(\Z, B) = H^{n}(k_{\pro\zar}, B) = 0$
	by the same proof as \cite[Prop.\ 3.3.1]{Suz13}
	(Zariski cohomology of a field is trivial).
	It follows that $R \Hom_{k_{\pro\zar}}(A, B) = R \Hom_{\Fin}(A, B)$ if $A, B \in \Fin$.
	From this, we can deduce the rest of the statements
	in the same way as we did in the proof of
	\eqref{prop: fully faithful embedding for ind-proalgebraic groups}
	(which ultimately will need the cubical construction
	for the limit arguments \cite[Lem.\ 3.8.2]{Suz13}).
	
	Alternatively,
	a (finite) set also corresponds to a (finite) \'etale $\closure{k}$-scheme.
	Therefore it is enough to show the corresponding statement
	that replaces $k_{\pro\zar}$ with $\closure{k}_{\pro\et}$,
	$\Fin$ with $\FEt / \closure{k}$
	and $\Ab$ with $\Et / \closure{k}$.
	Let $\Bar{f} \colon \Spec \closure{k}^{\ind\rat}_{\pro\et} \to \Spec \closure{k}_{\pro\et}$
	be the morphism defined by the identity.
	Then $\Bar{f}_{\ast} \Bar{f}^{\ast} = \id$ on $D(\closure{k}_{\pro\et})$,
	so $\Bar{f}^{\ast} \colon D(\closure{k}_{\pro\et}) \to D(\closure{k}^{\ind\rat}_{\pro\et})$ is fully faithful.
	The fully faithful embeddings
		\[
				D^{b}(\Ind \Pro \Alg / \closure{k}),
				D^{b}(\Et / \closure{k})
			\into
				D^{b}(\closure{k}^{\ind\rat}_{\pro\et})
		\]
	in \eqref{prop: fully faithful embedding for ind-proalgebraic groups}
	then restrict to fully faithful embeddings
		\[
				D^{b}(\Ind \Pro \FEt / \closure{k}),
				D^{b}(\Et / \closure{k})
			\into
				D^{b}(\closure{k}_{\pro\et}).
		\]
	
	\eqref{ass: Weil group cohomology of ind-proalgebraic groups}
	It is enough to show the statements for $A \in \Alg / k$.
	By Lang's theorem \cite[VI, Prop.\ 3]{Ser88},
	we have an exact sequence
		\[
				0
			\to
				A(k)
			\to
				A
			\stackrel{F - 1}{\to}
				A
			\to
				\pi_{0}(A)_{F}
			\to
				0,
		\]
	where $\pi_{0}(A)_{F}$ is the coinvariants under $F$.
	The groups $A(k), \pi_{0}(A)_{F}$ are finite constant groups and,
	in particular, locally of finite presentation.
	Hence $\Gamma(k_{W}, A) = A(k) \in \Fin$
	and $H^{1}(k_{W}, A) = \pi_{0}(A)_{F} \in \Fin$.
\end{proof}

Since $F = 1$ on $\Z$, we have
$R \Gamma(k_{W}, \Z) = \Z \oplus \Z[-1]$.
Hence for any $A \in D(k^{\ind\rat}_{\pro\et})$, we have natural morphisms
	\begin{align*}
				R \Gamma(k_{W}, A^{\SDual})
		&	=
				R \Gamma \bigl(
					k_{W},
					R \sheafhom_{k^{\ind\rat}_{\pro\et}}(A, \Z)
				\bigr)
		\\
		&	\to
				R \sheafhom_{k_{\pro\zar}} \bigl(
					R \Gamma(k_{W}, A),
					R \Gamma(k_{W}, \Z)
				\bigr)
		\\
		&	\to
				R \sheafhom_{k_{\pro\zar}} \bigl(
					R \Gamma(k_{W}, A), \Z)
				\bigr)[-1]
		\\
		&	=
				R \Gamma(k_{W}, A)^{\LDual}[-1]
		\\
		&	\gets
				R \Gamma(k_{W}, A)^{\PDual}[-2],
	\end{align*}
where we denote $B^{\LDual} = R \sheafhom_{k_{\pro\zar}}(B, \Z)$
(linear dual)
and $B^{\PDual} = R \sheafhom_{k_{\pro\zar}}(B, \Q / \Z)$
(Pontryagin dual) for $B \in D(k_{\pro\zar})$.

\begin{Prop} \label{prop: Weil group cohomology of Serre dual}
	If $A \in \Ind \Pro \Alg / k$,
	then the above defined morphisms
		\[
				R \Gamma(k_{W}, A^{\SDual})
			\to
				R \Gamma(k_{W}, A)^{\LDual}[-1]
			\gets
				R \Gamma(k_{W}, A)^{\PDual}[-2]
		\]
	are isomorphisms.
	If $A \in \FGEt / k$, then the morphism
		\[
				R \Gamma(k_{W}, A^{\SDual})
			\to
				R \Gamma(k_{W}, A)^{\LDual}[-1]
		\]
	is an isomorphism.
\end{Prop}

\begin{proof}
	For the moment,
	let $A \in \Ab(k^{\ind\rat}_{\pro\et})$ be arbitrary.
	We use the identification
	$\Spec \closure{k}_{\pro\et} \cong \Spec k_{\pro\zar}$.
	Hence we view the map $f$ used in \eqref{eq: geometric points functor} as a continuous map of sites
	$\Spec k^{\ind\rat}_{\pro\et} \to \Spec \closure{k}_{\pro\et}$
	defined by the identity.
	Let $g \colon \Spec k^{\ind\rat}_{\pro\et} \to \Spec k_{\pro\et}$ be the morphism defined by the identity.
	Let $j \colon \Spec \closure{k}_{\pro\et} \to \Spec k_{\pro\et}$ be the natural morphism
	(defined by the functor
	$k' \in k_{\pro\et} \mapsto k' \tensor \closure{k} \in \closure{k}_{\pro\et}$).
	The pullback $j^{\ast}$ is the restriction functor $|_{\Spec \closure{k}}$.
	Hence $f_{\ast} = j^{\ast} g_{\ast}$ as functors
	$D(k^{\ind\rat}_{\pro\et}) \to D(\closure{k}_{\pro\et})$.
	
	We have
		\[
				R \Gamma(k_{W}, A^{\SDual})
			=
				j^{\ast} g_{\ast}
				R \sheafhom_{k^{\ind\rat}_{\pro\et}}(
					[A \stackrel{F - 1}{\to} A],
					\Z
				).
		\]
	On the other hand, the restriction $j^{\ast}$ commutes with derived sheaf-Hom's:
		\[
				j^{\ast} R \sheafhom_{k_{\pro\et}}(B, C)
			=
				R \sheafhom_{\closure{k}_{\pro\et}}(j^{\ast} B, j^{\ast} C)
		\]
	(\cite[Prop.\ 18.4.6]{KS06}).
	Together with the adjunction $g^{\ast} \leftrightarrow g_{\ast}$
	and $j^{\ast} \Z = \Z$ and $g_{\ast} \Z = \Z$,
	we have
		\begin{align*}
					R \Gamma(k_{W}, A)^{\LDual}[-1]
			&	=
					R \sheafhom_{\closure{k}_{\pro\et}}(
						j^{\ast} g_{\ast} [A \stackrel{F - 1}{\to} A],
						\Z
					)
			\\
			&	=
					j^{\ast} g_{\ast}
					R \sheafhom_{k^{\ind\rat}_{\pro\et}}(
						g^{\ast} g_{\ast} [A \stackrel{F - 1}{\to} A],
						\Z
					).
		\end{align*}
	
	Hence the morphism
		$
				R \Gamma(k_{W}, A^{\SDual})
			\to
				R \Gamma(k_{W}, A)^{\LDual}[-1]
		$
	is an isomorphism if $g^{\ast} g_{\ast} [A \to A] = [A \to A]$.
	For a profinite set $X$ viewed as an object of $k^{\ind\rat}$, we have
	$g^{\ast} g_{\ast} X = g^{\ast} X = X$
	since $X$ is representable in $k_{\pro\et}$.
	The same is true if $X$ is an ind-object of profinite sets (in particular, a set)
	since both $g^{\ast}$ and $g_{\ast}$ commute with filtered direct limits.
	We have $[A \to A] \in D^{b}(\Ind \Pro \Fin)$ if $A \in \Ind \Pro \Alg / k$
	by \eqref{prop: ind-profinite groups and ind-proalgebraic groups}
	\eqref{ass: Weil group cohomology of ind-proalgebraic groups},
	and $[A \to A] \in D^{b}(\Ab)$ if $A \in \FGEt / k$.
	Hence $g^{\ast} g_{\ast} [A \to A] = [A \to A]$ in either case.
	(Note that $g^{\ast} g_{\ast} \Ga$ is not $\Ga$
	but the \'etale group whose group of geometric points
	is the discrete $\Gal(\closure{k} / k)$-modules of the additive group $\closure{k}$).
	
	For the isomorphism
		$
				R \Gamma(k_{W}, A)^{\PDual}[-2]
			=
				R \Gamma(k_{W}, A)^{\LDual}[-1]
		$
	for $A \in \Ind \Pro \Alg / k$,
	it is enough to show that
	$R \sheafhom_{\closure{k}_{\pro\et}}(B, \Q) = 0$
	for any $B \in \Ind \Pro \Fin$.
	We can see this by the same proof as
	\eqref{thm: Ext of proalgebraic groups as sheaves}
	\eqref{ass: Ext by uniquely divisible is zero}
	or by reducing to it by
		$
				R \sheafhom_{\closure{k}_{\pro\et}}(B, \Q)
			=
				\Bar{f}_{\ast} R \sheafhom_{\closure{k}^{\ind\rat}_{\pro\et}}(B, \Q)
			=
				0
		$,
	where $\Bar{f} \colon \Spec \closure{k}^{\ind\rat}_{\pro\et} \to \closure{k}_{\pro\et}$
	is the morphism defined by the identity.
\end{proof}

Let $A \in \Ab(K_{\fppf} / k^{\ind\rat}_{\et})$.
Assume that $R \alg{\Gamma}(K, A)$ is P-acyclic.
Then
	\[
			R \Gamma(k_{W}, R \alg{\Gamma}(K, A))(k)
		=
			\bigl[
					R \Gamma(\hat{K}^{\ur}, A)
				\stackrel{F - 1}{\to}
					R \Gamma(\hat{K}^{\ur}, A)
			\bigr][-1].
	\]
This is equal to $R \Gamma(W_{K}, A(L))$
when the fppf cohomology $R \Gamma(\hat{K}^{\ur}, A)$ agrees with the \'etale cohomology
$R \Gamma(\hat{K}^{\ur}_{\et}, A)$ (such as when $A$ is a smooth group scheme over $K$),
where $W_{K}$ is the Weil group of $K$ and $L$ a separable closure of $\hat{K}^{\ur}$.
For this reason, we denote%
\footnote{
	The notation $R \Gamma(K_{W}, \;\cdot\;)$ reads
	``Weil-fppf cohomology of the scheme $\Spec K$'',
	while $R \Gamma(W_{K}, \;\cdot\;)$ is group cohomology of the topological group $W_{K}$.
	It is not good to choose the latter as general notation for the cohomology theory defined here
	since fppf cohomology is not always described as group cohomology.
}
	\begin{gather*}
				R \Gamma(K_{W}, A)
			=
				R \Gamma(k_{W}, R \Tilde{\alg{\Gamma}}(K, A))
			\in
				D(k_{\pro\zar}),
		\\
				H^{n}(K_{W}, A)
			=
				H^{n} R \Gamma(K_{W}, A)
			\in
				\Ab(k_{\pro\zar})
	\end{gather*}
for any $A \in \Ab(K_{\fppf} / k^{\ind\rat}_{\et})$.
Whenever $R \alg{\Gamma}(K, A)$ is P-acyclic and
in $D^{b}(\Ind \Pro \Alg / k)$ (resp.\ $D^{b}(\Et / k)$),
which happens in many cases by
\eqref{prop: cohomology of local fields},
the complex $R \Gamma(K_{W}, A)$
is thus an enhancement of the Weil group cohomology of $K$
as a complex of ind-profinite groups
(resp.\ finitely generated abelian groups).
The above proposition says that
the Serre duality for $R \Tilde{\alg{\Gamma}}(K, A)$ is translated into
the linear duality or Pontryagin duality
for $R \Gamma(K_{W}, A)$ up to shift.

For example, let $A$ be an abelian variety over $K$.
Then our duality theorem \eqref{thm: duality for abelian varieties}
and \eqref{prop: double dual of R Gamma of abelian variety}
show that
	\[
			\bigl[
					\invlim_{n} \alg{\Gamma}(K, A^{\vee})
				\to
					R \alg{\Gamma}(K, A^{\vee})
			\bigr]
		=
			R \alg{\Gamma}(K, A)^{\SDual}[1].
	\]
Apply $R \Gamma(k_{W}, \;\cdot\;)$.
If $\alg{\Gamma}(K, A^{\vee})_{\mathrm{sAb}}$ denotes
the semi-abelian quotient of $\alg{\Gamma}(K, A^{\vee})_{0}$,
which is quasi-algebraic,
then
	$
		R \Gamma \bigl(
			k_{W},
			\alg{\Gamma}(K, A^{\vee})_{\mathrm{sAb}}
		\bigr)
	$
is finite by
\eqref{prop: ind-profinite groups and ind-proalgebraic groups}.
Therefore
	\[
			R \Gamma \bigl(
				k_{W},
				\invlim_{n} \alg{\Gamma}(K, A^{\vee})
			\bigr)
		=
			R \invlim_{n}
			R \Gamma \bigl(
				k_{W},
				\alg{\Gamma}(K, A^{\vee})_{\mathrm{sAb}}
			\bigr)
		=
			0.
	\]
Hence \eqref{prop: Weil group cohomology of Serre dual}
shows that we have an isomorphism
	\[
			R \Gamma(K_{W}, A^{\vee})
		=
			R \Gamma(K_{W}, A)^{\PDual}[-1]
	\]
in $D^{b}(\Ind \Pro \Fin)$.
The description of $\alg{H}^{n}(K, A)$
in \eqref{prop: cohomology of local fields}
\eqref{ass: cohomology of abelian varieties over local fields}
shows that $H^{0}(K_{W}, A) = A(K)$ is a profinite group,
$H^{1}(K_{W}, A) = H^{1}(K, A)$ an indfinite (i.e., torsion) group,
and $H^{n}(K_{W}, A) = 0$ for $n \ne 0, 1$.
This recovers Tate-Milne's duality \cite{Tat58} \cite{Mil70} \cite{Mil72}
(see also \cite[I, Cor.\ 3.4 and III, Thm.\ 7.8]{Mil06}).

If $T$ is a torus over $K$, then the same procedure gives isomorphisms
	\begin{gather*}
				R \Gamma(K_{W}, T^{\CDual})
			=
				R \Gamma(K_{W}, T)^{\LDual}[-1]
		\\
				R \Gamma(K_{W}, T)
			=
				R \Gamma(K_{W}, T^{\CDual})^{\LDual}[-1]
	\end{gather*}
in $D^{b}(k_{\pro\zar})$.
We see that $H^{0}(K_{W}, T) = T(K)$
is an extension of a finite free abelian group by a profinite group,
$H^{1}(K_{W}, T)$ a finitely generated abelian group
and $H^{n}(K_{W}, T) = 0$ for $n \ne 0, 1$.
We also see that $H^{0}(K_{W}, T^{\CDual})$ and $H^{1}(K_{W}, T^{\CDual})$ are finitely generated abelian groups,
$H^{2}(K_{W}, T^{\CDual}) = H^{1}(K, \hat{T}_{\tor})$ an indfinite (i.e. torsion) group
and $H^{n} = 0$ for $n \ne 0, 1, 2$.
Let $C$ be an algebraically closed field of characteristic zero.
We view $C^{\times}$ as a constant group scheme over $k$ or $K$
and let $\Hat{T}(C) = T^{\CDual} \tensor_{\Z} C^{\times}$ be the $C$-value points of the dual torus.
Taking $\tensor_{\Z}^{L} C^{\times}$ of the above isomorphism,
we have
	\[
			R \Gamma(K_{W}, \Hat{T}(C))
		=
			R \sheafhom_{k_{\pro\zar}}(R \Gamma(K_{W}, T), C^{\times})[-1]
	\]
by the same method as the proof of
\eqref{prop: content of duality of tori}.
From $H^{1}$, we get
	\[
			H^{1}(K_{W}, \Hat{T}(C))
		=
			\sheafhom_{k_{\pro\zar}}(\Gamma(K, T), C^{\times}).
	\]
This recovers the local Langlands correspondence for tori \cite{Lan97}.

If $M$ is a $1$-motive over $K$, then
we have
	\[
			R \Gamma(
				k_{W},
				R \alg{\Gamma}(K, M)^{\SDual \SDual}
			)
		=
			R \Gamma(k_{W}, R \alg{\Gamma}(K, M))
	\]
by reducing it to the previous cases of abelian varieties, tori and lattices
by the weight filtration, and an isomorphism
	\[
			R \Gamma(K_{W}, M^{\vee})
		=
			R \Gamma(K_{W}, M)^{\LDual}
	\]
in $D^{b}(k_{\pro\zar})$.
This case, it is a little bit complicated to describe each $H^{n}(K_{W}, M)$
due to its possible uniquely divisible part.
This is essentially a result of Harari-Szamuely \cite[Thm.\ 2.3]{HS05}.

If $N$ is a finite flat group scheme over $K$,
then the B\'egueri-Bester-Kato duality
	\[
			R \alg{\Gamma}(K, N^{\CDual})
		=
			R \alg{\Gamma}(K, N)^{\SDual}
	\]
(see Thm.\ \eqref{thm: duality for finite flats over local fields} and
\cite[Thm.\ 2.7.1]{Suz13} for this form of the statement) implies
	\[
			R \Gamma(K_{W}, N^{\CDual})
		=
			R \Gamma(K_{W}, N)^{\PDual}[-2]
	\]
in $D^{b}(\Ind \Pro \Fin)$.
We see that $H^{0}(K_{W}, N) = H^{0}(K, N)$ and $H^{2}(K_{W}, N) = H^{2}(K, N)$ are finite,
$H^{1}(K_{W}, N)$ the group $H^{1}(K, N)$ endowed with an ind-profinite group structure,
and $H^{n} = 0$ for $n \ne 0, 1, 2$.
If $K$ has mixed characteristic, then $H^{1}(K_{W}, N)$ is finite.
This recovers Tate-Shatz's duality \cite[II, \S 5, Thm.\ 2]{Ser02}, \cite{Sha64}.

\begin{Rmk}
	As in \cite{Sha64},
	the group $H^{1}(K, N)$ for equal characteristic $K$ is not only ind-profinite but locally compact
	(see also \cite[III, \S 6]{Mil06} and \cite{Ces15}),
	i.e., it can be written as a filtered union $\bigcup B_{\lambda}$ of profinite groups
	such that $B_{\mu} / B_{\lambda}$ is finite for any $\mu \ge \lambda$.
	In our setting,
	we can deduce this from the fact (true for general perfect $k$) that
	$\alg{H}^{1}(K, N) \in \Ind \Pro \Alg / k$ can be written as
	a filtered union $\bigcup C_{\lambda}$ of proalgebraic groups
	such that $C_{\mu} / C_{\lambda}$ is quasi-algebraic for any $\mu \ge \lambda$.
	For the moment, let us say such an ind-proalgebraic group algebraically growing.
	To see that $\alg{H}^{1}(K, N)$ is algebraically growing,
	note that this is true for $N = \alpha_{p}, \mu_{p}, \Z / n \Z$
	as seen from the proof of
	\eqref{prop: cohomology of local fields}.
	Let $0 \to C_{1} \to C_{2} \to C_{3} \to 0$ be an exact sequence in $\Ind \Pro \Alg / k$.
	If $C_{2}$ is algebraically growing, so is $C_{1}$.
	If $C_{1}$ and $C_{3}$ are both algebraically growing, so is $C_{2}$.
	If $C_{2}$ is algebraically growing and $C_{1} \in \Pro \Alg / k$,
	then $C_{3}$ is algebraically growing.
	Then the general finite flat group case follows
	from the proof of the same proposition.
\end{Rmk}


\appendix

\numberwithin{equation}{section}
\section{The pro-fppf topology for proalgebraic proschemes}
\label{sec: The pro-fppf topology for proalgebraic proschemes}

In this appendix, we briefly explain
how \eqref{thm: Ext of proalgebraic groups as sheaves},
originally proved in \cite[\S 3]{Suz13} for affine proalgebraic groups $A$,
can be extended for arbitrary proalgebraic groups.
For abelian varieties, the same proof as \cite{Suz13} works.
Although the general case including infinite products of abelian varieties
is not needed in the main body of this paper,
we treat the general case in this appendix for completeness.
The problem is that infinite products of abelian varieties are not schemes, but proschemes.
Let $k^{\pro\alge}$ be the procategory of the category of quasi-algebraic $k$-schemes
(i.e., perfections of $k$-schemes of finite type).
We call an object of $k^{\pro\alge}$ a proalgebraic $k$-proscheme.
Examples are infinite products of perfections of proper varieties.
Proalgebraic groups are group objects in $k^{\pro\alge}$.
In \cite[\S 3.1]{Suz13}, we introduced the pro-fppf topology
on the category of perfect affine $k$-schemes.
Below we extend it to the category $k^{\pro\alge}$ of proalgebraic $k$-proschemes
in the paragraph after \eqref{prop: composite of flat of finite presentation}.
We also show in \eqref{prop: topos equivalence for proalgebraic proschemes and perfect affine schemes}
that the pro-fppf site of proalgebraic $k$-proschemes
has an equivalent topos to the pro-fppf site of perfect affine $k$-schemes.
Then most part of \cite[\S 3]{Suz13} goes without large modifications.

Generalizing \cite[\S 3.1]{Suz13} from the affine case, we say that
a morphism of quasi-compact quasi-separated perfect $k$-schemes $Y \to X$ is
\emph{of finite presentation}
if it is the perfection of a $k$-scheme morphism $Y_{0} \to X$ of finite presentation in the usual sense.
If moreover $Y_{0} \to X$ can be chosen to be (faithfully) flat,
we say that $Y \to X$ is \emph{(faithfully) flat of finite presentation} (in the perfect sense).%
\footnote{
	The notion of morphisms of finite presentation in the sense here is
	Zariski local in the source and the target,
	as shown in \cite[\S 1, Cor.\ to Prop.\ 9]{Ser60} when $X = \Spec k$ and $Y$ separated,
	and in \cite[Prop.\ 3.13]{BS17} in general.
	But it is not clear whether the notion of flat morphisms of finite presentation is Zariski local or not.
	The definition given here is sufficient
	for defining a nice enough category of sheaves.
}
A quasi-compact quasi-separated perfect $k$-scheme can naturally be viewed as an object of $k^{\pro\alge}$
since it is a filtered inverse limit of quasi-algebraic $k$-schemes with affine transition morphisms
(use the absolute noetherian approximation \cite[Thm.\ C.9]{TT90} and apply the perfection).
We say that a morphism $Y \to X$ in $k^{\pro\alge}$ is
\emph{of finite presentation}
if it can be written as the base change of a morphism of finite presentation $Y_{0} \to X_{0}$
in the above perfect sense
between quasi-compact quasi-separated perfect $k$-schemes
by some morphism $X \to X_{0}$ in $k^{\pro\alge}$.
If $Y_{0} \to X_{0}$ can be chosen to be (faithfully) flat of finite presentation (in the above perfect sense),
then we say that $Y \to X$ is \emph{(faithfully) flat of finite presentation}.
We say that a morphism $Y \to X$ in $k^{\pro\alge}$ is
\emph{(faithfully) flat of profinite presentation}
if there exists a filtered inverse system $\{Y_{\lambda}\}$ in $k^{\pro\alge}$
such that each morphism $Y_{\lambda} \to X$ is (faithfully) flat of finite presentation
and $Y$ is isomorphic to $\invlim_{\lambda} Y_{\lambda}$ over $X$ in $k^{\pro\alge}$.
Here are two basic permanence properties of morphisms (faithfully) flat of finite presentation.

\begin{Prop}
	If a morphism $Y \to X$ in $k^{\pro\alge}$ is (faithfully) flat of finite presentation,
	then it can be written as the base change of a morphism $Y_{0} \to X_{0}$
	(faithfully) flat of finite presentation between quasi-algebraic $k$-schemes
	by some morphism $X \to X_{0}$ in $k^{\pro\alge}$.
\end{Prop}

\begin{proof}
	We may assume that $X$ and $Y$ are quasi-compact quasi-separated perfect $k$-schemes.
	Write $Y \to X$ as the perfection of a morphism $Y_{0}' \to X$
	(faithfully) flat of finite presentation in the usual sense.
	By \cite[Thm.\ C.9]{TT90}, we can write $X$ as a filtered inverse limit of algebraic $k$-schemes
	with affine transition morphisms.
	Hence the permanence property of flatness (resp.\ surjectivity)
	under passage to limits \cite[Thm.\ 11.2.6]{Gro66} (resp.\ \cite[Thm.\ 8.10.5]{Gro66}) shows that
	$Y_{0}' \to X$ can be written as the base change of a morphism $Y_{00} \to X_{00}$
	(faithfully) flat of finite presentation in the usual sense
	between algebraic $k$-schemes
	by some morphism $X \to X_{00}$.
	Let $Y_{0} \to X_{0}$ be the perfection of $Y_{00} \to X_{00}$,
	which is (faithfully) flat of finite presentation (in the perfect sense)
	between quasi-algebraic $k$-schemes.
	Its base change by $X \to X_{0}$ is $Y \to X$.
\end{proof}

\begin{Prop}
	Let $\{X_{\lambda}\}$ be a filtered inverse system in $k^{\pro\alge}$
	with inverse limit $X$.
	Let $Y \to X$ be a (faithfully) flat morphism of finite presentation in $k^{\pro\alge}$.
	Then there exist an index $\lambda$, an object $Y_{\lambda}$ of $k^{\pro\alge}$
	and a (faithfully) flat morphism of finite presentation $Y_{\lambda} \to X_{\lambda}$ in $k^{\pro\alge}$
	such that the base change of $Y_{\lambda} \to X_{\lambda}$ by $X \to X_{\lambda}$
	gives the original morphism $Y \to X$.
\end{Prop}

\begin{proof}
	As in the previous proposition,
	let $Y_{0} \to X_{0}$ be a (faithfully) flat morphism of finite presentation
	between quasi-algebraic $k$-schemes
	whose base change to $X$ by some morphism $X \to X_{0}$ gives the morphism $Y \to X$.
	Then the morphism $X \to X_{0}$ factors through some $X_{\lambda}$.
	Let $Y_{\lambda} \to X_{\lambda}$ be the base change of $Y_{0} \to X_{0}$
	by the morphism $X_{\lambda} \to X_{0}$.
	Then $Y_{\lambda} \to X_{\lambda}$ is (faithfully) flat of finite presentation,
	and its base change to $X$ is the morphism $Y \to X$.
\end{proof}

A fact already used above is that
a base change of a morphism (faithfully) flat of finite presentation is also.
The same is true for a morphism (faithfully) flat of profinite presentation.
For composites, we have the following.

\begin{Prop} \label{prop: composite of flat of finite presentation}
	Let $Z \to Y \to X$ be morphisms in $k^{\pro\alge}$.
	If $Z \to Y$ and $Y \to X$ are (faithfully) flat of finite presentation,
	then so is $Z \to X$.
	If $Z \to Y$ and $Y \to X$ are (faithfully) flat of profinite presentation,
	then so is $Z \to X$.
\end{Prop}

\begin{proof}
	Assume that $Z \to Y$ and $Y \to X$ are (faithfully) flat of finite presentation.
	We first treat that case that $X, Y, Z$ are quasi-compact quasi-separated perfect $k$-schemes.
	Write $Z \to Y$ (resp.\ $Y \to X$) as the perfection
	of a morphism $Z_{0} \to Y$ (resp.\ $Y_{0} \to X$)
	(faithfully) flat of finite presentation in the usual sense.
	For any $n \ge 0$, let $Y_{0}^{(n)}$ be the scheme $Y_{0}$
	with a new $k$-scheme structure given by the composite
		\[
			Y_{0} \to \Spec k \isomto \Spec k,
		\]
	where the first morphism is the original $k$-scheme structure morphism of $Y_{0}$
	and the second morphism is the $p^{n}$-th power absolute Frobenius of $\Spec k$.
	The scheme $Y_{0}^{(n)}$ is (faithfully) flat of finite presentation
	over $X$ ($= X^{(n)}$) in the usual sense.
	Since $Y = \invlim_{n} Y_{0}^{(n)}$,
	the permanence property of flatness (resp.\ surjectivity)
	under passage to limits \cite[Thm.\ 11.2.6]{Gro66} (resp.\ \cite[Thm.\ 8.10.5]{Gro66}) shows that
	we can write $Z_{0} \to Y$ as the base change of
	some (faithfully) flat morphism of finite presentation $Z_{00} \to Y_{0}^{(n)}$ in the usual sense
	by the morphism $Y \to Y_{0}^{(n)}$.
	Then $Z \to X$ is the perfection of the composite $Z_{00} \to Y_{0}^{(n)} \to X$,
	which is (faithfully) flat of finite presentation in the usual sense.
	Hence $Z \to X$ is (faithfully) flat of finite presentation (in the perfect sense).
	
	We treat the general case.
	Write $X$ as a filtered inverse limit of quasi-algebraic $k$-schemes $X_{\lambda}$.
	Write $Y \to X$ as the base change of a morphism $Y_{\lambda} \to X_{\lambda}$
	(faithfully) flat of finite presentation (in the perfect sense)
	as in the previous proposition.
	Let $Y_{\mu} = Y_{\lambda} \times_{X_{\lambda}} X_{\mu}$ for $\mu \ge \lambda$.
	Then $Y = \invlim_{\mu \ge \lambda} Y_{\mu}$.
	Write $Z \to Y$ as the base change of a morphism $Z_{\mu} \to Y_{\mu}$
	(faithfully) flat of finite presentation
	as in the previous proposition.
	Then $Z \to X$ is the base change of the composite $Z_{\mu} \to Y_{\mu} \to X_{\mu}$,
	which is (faithfully) flat of finite presentation
	by the quasi-compact quasi-separated case treated earlier.
	
	Next assume that $Z \to Y$ and $Y \to X$ are (faithfully) flat of profinite presentation.
	Write $Z$ (resp.\ $Y$) as a filtered inverse limit
	of morphisms $Z_{\mu} \to Y$ (resp.\ $Y_{\lambda} \to X$) (faithfully) flat of finite presentation.
	By a similar argument as \cite[Lem.\ 1.5]{Swa98},
	it is enough to show that any morphism $Z \to Z'$ over $X$
	to an object $Z'$ over $X$ of finite presentation factors through
	an object over $X$ (faithfully) flat of finite presentation.
	The $X$-morphism $Z \to Z'$ factors through some $Z_{\mu}$.
	Write the morphism $Z_{\mu} \to Y$ as the base change of
	some morphism $Z_{\lambda \mu} \to Y_{\lambda}$ (faithfully) flat of finite presentation
	as in the previous proposition.
	Let $Z_{\lambda' \mu} = Z_{\lambda \mu} \times_{Y_{\lambda}} Y_{\lambda'}$ for $\lambda' \ge \lambda$.
	Then $Z_{\mu} = \invlim_{\lambda' \ge \lambda} Z_{\lambda' \mu}$.
	Then the $X$-morphism $Z \to Z'$ factors through $Z_{\lambda' \mu}$ for some $\lambda' \ge \lambda$.
	The composite $Z_{\lambda' \mu} \to Y_{\lambda'} \to X$ is (faithfully) flat of finite presentation
	by what we have proven above.
	Thus $Z \to X$ is (faithfully) flat of profinite presentation.
\end{proof}

With this, we can define the \emph{perfect fppf site of proalgebraic $k$-proschemes},
$\Spec k^{\pro\alge}_{\fppf}$,
to be the category $k^{\pro\alge}$ endowed with the topology
where a covering $\{Y_{i} \to X\}$ is a finite family of morphisms
such that $\bigsqcup Y_{i} \to X$ is faithfully flat of finite presentation
(which implies that each $Y_{i} \to X$ is flat of finite presentation).
Replacing ``faithfully flat of finite presentation'' by ``faithfully flat of profinite presentation'',
we can define the \emph{perfect pro-fppf site of proalgebraic $k$-proschemes}
$\Spec k^{\pro\alge}_{\pro\fppf}$ in the same way.

As in the main text, let $k^{\perf}$ be the category of perfect $k$-algebras or perfect affine $k$-schemes.
The identity functor defines a morphism of sites
$\Spec k^{\pro\alge}_{\pro\fppf} \to \Spec k^{\perf}_{\pro\fppf}$.

\begin{Prop} \label{prop: topos equivalence for proalgebraic proschemes and perfect affine schemes}
	The  morphism of sites
	$\Spec k^{\pro\alge}_{\pro\fppf} \to \Spec k^{\perf}_{\pro\fppf}$
	defined by the identity
	induces an equivalence on the topoi.
\end{Prop}

In particular, Ext groups and Ext sheaves between proalgebraic groups can be calculated using
either $\Spec k^{\pro\alge}_{\pro\fppf}$ or $\Spec k^{\perf}_{\pro\fppf}$.

To prove the proposition,
the key is that a proalgebraic $k$-proscheme is \emph{pro-Zariski locally affine}
(\eqref{prop: proalgebraic proschemes are pro-Zariski locally affine} below).
This is trivial for a proalgebraic $k$-proscheme $X = \invlim_{i} X_{i}$
indexed by integers $i \ge 1$:
we can find an affine ``pro-Zariski'' covering $Y = \invlim_{i} Y_{i}$ of $X$
by successively taking affine Zariski coverings $Y_{1} \onto X_{1}$,
$Y_{2} \onto Y_{1} \times_{X_{1}} X_{2}$,
$Y_{3} \onto Y_{2} \times_{X_{2}} X_{3}$ and so on.
Here we say that a surjective morphism of schemes $Y \onto Z$ is an (affine) Zariski covering
if there exists a decomposition $Y = \bigsqcup Y_{\lambda}$ into disjoint (affine) open subschemes
such that the composite $Y_{\lambda} \to Y \to Z$ is an open immersion for any $\lambda$.
The difficulties thus come from the complexity of the index set.

Recall that a poset (partially ordered set) is a pre-ordered set such that
$\lambda \le \lambda' \le \lambda$ implies $\lambda = \lambda'$.
A pre-ordered set $\Lambda$ is said to be cofinite if for every $\lambda \in \Lambda$,
there are only finitely many elements $\le \lambda$.
An element $\lambda$ of a pre-ordered set is said to be greatest
if $\lambda' \le \lambda$ for any $\lambda'$.
If $\Lambda$ is a directed set without a greatest element
and $M$ is the set of all finite subsets $\Lambda' \subset \Lambda$
having a unique greatest element $\max(\Lambda')$,
then $M$ is a cofinite directed poset by inclusion,
and the map $\Lambda' \mapsto \max(\Lambda')$ from $M$ to $\Lambda$ is cofinal
(which is implicit in the proof of \cite[I, Prop.\ 8.1.6]{AGV72a}).
If $\Lambda$ has a greatest element $\lambda$,
then $\{\lambda\} \into \Lambda$ is cofinal.
Therefore every filtered inverse system in a category is isomorphic (in the procategory)
to a system indexed by a cofinite directed poset
(again implicit in the proof of \cite[loc.\ cit.]{AGV72a}).
For a cofinite poset $\Lambda$ and an element $\lambda \in \Lambda$,
we define $n(\lambda)$ to be the largest integer $n \ge 1$ such that
there exists a chain $\lambda_{1} < \dots < \lambda_{n - 1} < \lambda$ of elements of $\Lambda$
(so $n(\lambda) = 1$ if $\lambda$ is minimal).
This function satisfies the property that if $\lambda \le \lambda'$,
then $\lambda = \lambda'$ or $n(\lambda) < n(\lambda')$.

\begin{Prop} \label{prop: proalgebraic proschemes are pro-Zariski locally affine}
	Let $\{X_{\lambda}\}_{\lambda \in \Lambda}$ be an inverse system
	of quasi-compact quasi-separated $k$-schemes.
	Assume that the index set $\Lambda$ is a cofinite poset.%
	\footnote{We do not need to assume that $\Lambda$ is directed,
	though we only need the directed case below.}
	Then there exists an inverse system $\{Y_{\lambda}\}$ of affine $k$-schemes
	and a system $\{Y_{\lambda} \to X_{\lambda}\}$ of $k$-scheme morphisms
	such that each $Y_{\lambda} \onto X_{\lambda}$ is a Zariski covering.
\end{Prop}

\begin{proof}
	We will construct a system
		\[
			\bigl\{
					\cdots
				\onto
					Y_{\lambda}^{n + 1}
				\onto
					Y_{\lambda}^{n}
				\onto
					\cdots
				\onto
					Y_{\lambda}^{0}
				=
					X_{\lambda}
			\bigr\}_{\lambda}
		\]
	of sequences of Zariski coverings of quasi-compact quasi-separated $k$-schemes such that
	for each $n$ and $\lambda$ with $n(\lambda) \le n$,
	the scheme $Y_{\lambda}^{n}$ is affine
	and the morphisms $\dots \to Y_{\lambda}^{n + 1} \to Y_{\lambda}^{n}$ after the $n$-th term are isomorphisms.
	Then, letting $Y_{\lambda} = Y_{\lambda}^{n}$ for $n(\lambda) \le n$,
	we will obtain a required system $\{Y_{\lambda}\}$ of coverings.
	
	Fix an integer $n \ge 0$.
	Suppose that we have such a sequence
	$\{Y_{\lambda}^{n} \onto \cdots \onto Y_{\lambda}^{0}\}_{\lambda}$ up to $n$.
	For each $\lambda$ with $n(\lambda) = n + 1$, we choose an affine Zariski covering
	$Y_{\lambda}^{n + 1} \onto Y_{\lambda}^{n}$.
	For general $\lambda$, we define
		\[
				Y_{\lambda}^{n + 1}
			=
				\sideset{}{_{Y_{\lambda}^{n}}}
				\prod_{\lambda' \le \lambda,\, n(\lambda') = n + 1}
				(Y_{\lambda'}^{n + 1} \times_{Y_{\lambda'}^{n}} Y_{\lambda}^{n}),
		\]
	where $\sideset{}{_{Y_{\lambda}^{n}}} \prod$ denotes the fiber product over $Y_{\lambda}^{n}$.
	This is a finite product since $\Lambda$ is cofinite.
	The property of the function $n(\lambda)$ above shows that
	this definition is consistent with $Y_{\lambda}^{n + 1}$ taken early for $n(\lambda) = n + 1$,
	and the natural morphism $Y_{\lambda}^{n + 1} \to Y_{\lambda}^{n}$ is an isomorphism for $n(\lambda) \le n$.
	The morphism $Y_{\lambda}^{n + 1} \to Y_{\lambda}^{n}$ for general $\lambda$
	is a Zariski covering since it is a finite fiber product of Zariski coverings.
	Each $Y_{\lambda}^{n + 1}$ is a quasi-compact quasi-separated $k$-scheme
	since it is a fiber product of such.
	The morphisms $Y_{\lambda}^{n + 1} \to Y_{\lambda}^{n}$ form a system.
	By induction, we obtain a required sequence.
\end{proof}

We also need to compare the topology on $k^{\perf}$ induced from $\Spec k^{\pro\alge}_{\pro\fppf}$
and the topology of $\Spec k^{\perf}_{\pro\fppf}$
defined in \S \ref{sec: Sites and algebraic groups: setup and first properties}.
The following shows that they agree up to refinement.

\begin{Prop}
	If a morphism $\Spec S \to \Spec R$
	between perfect affine (hence quasi-compact quasi-separated) $k$-schemes
	is (faithfully) flat of profinite presentation in the sense of this appendix,
	then there exists a homomorphism $S \to T$ faithfully flat of ind-finite presentation
	in the sense of \S \ref{sec: Sites and algebraic groups: setup and first properties}
	between perfect $k$-algebras%
	\footnote{
		This means that $S \to T$ is flat of ind-finite presentation
		in the sense of \S \ref{sec: Sites and algebraic groups: setup and first properties}
		and faithfully flat.
	}
	such that $R \to T$ is (faithfully) flat of ind-finite presentation
	in the sense of \S \ref{sec: Sites and algebraic groups: setup and first properties}.
\end{Prop}

\begin{proof}
	Let $X = \Spec R$ and $Y = \Spec S$.
	Write $Y \to X$ as a filtered inverse limit of morphisms $Y_{\lambda} \to X$
	(faithfully) flat of finite presentation in the sense of this appendix
	with $Y_{\lambda}$ quasi-compact quasi-separated perfect $k$-schemes.
	We may assume that the index set $\Lambda$ is a cofinite directed poset.
	Take a system of affine Zariski coverings $\{Z_{\lambda} = \Spec T_{\lambda} \onto Y_{\lambda}\}$
	as in the previous proposition.
	The scheme $Z_{\lambda}$ for any $\lambda$ is perfect since $Y_{\lambda}$ is so.
	For each $\lambda$, the composite $Z_{\lambda} \to Y_{\lambda} \to X$ is
	(faithfully) flat of finite presentation in the sense of this appendix
	by \eqref{prop: composite of flat of finite presentation}.
	Therefore $Z_{\lambda} \to X$ can be written as the perfection of a morphism $Z_{\lambda 0} \to X$
	(faithfully) flat of finite presentation in the usual sense,
	where $Z_{\lambda 0}$ is a priori a quasi-compact quasi-separated $k$-scheme.
	Since $Z_{\lambda}$ is affine,
	we know that $Z_{\lambda 0}$ is actually affine
	by \cite[Prop.\ C.6]{TT90}.
	Hence $R \to T_{\lambda}$ is (faithfully) flat of finite presentation
	in the sense of \S \ref{sec: Sites and algebraic groups: setup and first properties}.
	Thus $T = \dirlim T_{\lambda}$ is (faithfully) flat of ind-finite presentation over $R$
	in the sense of \S \ref{sec: Sites and algebraic groups: setup and first properties}.
	Let $Z = \Spec T = \invlim Z_{\lambda}$.
	Then $Z = \invlim Z_{\lambda} \times_{Y_{\lambda}} Y$
	and $S \to T_{\lambda} \tensor_{S_{\lambda}} S$ is faithfully flat of finite presentation
	in the sense of \S \ref{sec: Sites and algebraic groups: setup and first properties}.
	Therefore $S \to T$ is faithfully flat of ind-finite presentation
	in the sense of \S \ref{sec: Sites and algebraic groups: setup and first properties}.
	(Of course this proof shows that $Z \to Y$ is actually a pro-Zariski covering.)
\end{proof}

\begin{proof}[Proof of \eqref{prop: topos equivalence for proalgebraic proschemes and perfect affine schemes}]
	Let $S = \Spec k^{\perf}_{\pro\fppf}$
	and $S' = \Spec k^{\pro\alge}_{\pro\fppf}$.
	Let $S''$ be the site on $k^{\perf}$ with topology induced from $\Spec k^{\pro\alge}_{\pro\fppf}$
	\cite[Exp.\ II, \S 3]{AGV72a}.
	We have morphisms $S' \to S'' \to S$ of sites defined by the identity.
	A sieve on an object of $k^{\perf}$ is a covering for $S$
	if and only if it is for $S''$
	by the previous proposition.
	Hence $S'' \isomto S$ as sites.
	Any object of $S'$ is covered by an object of $S''$
	by \eqref{prop: proalgebraic proschemes are pro-Zariski locally affine}.
	Hence $S' \to S''$ induces an equivalence on the topoi
	by \cite[Exp.\ II, Thm.\ 4.1]{AGV72a}.
\end{proof}

We can define the \'etale site $\Spec k^{\pro\alge}_{\et}$ and
the pro-\'etale site $\Spec k^{\pro\alge}_{\pro\et}$ on the category $k^{\pro\alge}$
in the same way.
See \cite[Prop.\ 17.7.8]{Gro67} for the necessary permanence property of \'etaleness
under passage to limits.
We can show that the morphism of sites
$\Spec k^{\pro\alge}_{\pro\et} \to \Spec k^{\perf}_{\pro\et}$
defined by the identity induces an equivalence on the topoi
by the same proof as above.

We generalize some more definitions in \cite[\S 3]{Suz13}.
A proalgebraic $k$-proscheme can be viewed as a locally ringed space
since the category of locally ringed spaces has all inverse limits
(\cite[Rmk.\ 2.1.1]{Tem11}, \cite[Cor.\ 5]{Gil11}).
If $\{X_{\lambda}\}$ is a filtered inverse system of quasi-compact quasi-separated perfect $k$-schemes,
then the underlying topological space of $X = \invlim X_{\lambda}$
is the inverse limit of the underlying topological spaces of $X_{\lambda}$.
The structure sheaf of $X$ is the direct limit of the pullbacks by $X \to X_{\lambda}$
of the structure sheaves of $X_{\lambda}$.
A morphism $Y \to X$ in $k^{\pro\alge}$ is faithfully flat of (pro)finite presentation
if and only if it is flat of (pro)finite presentation and surjective on the underlying topological spaces.
We say that $Y \to X$ is dominant if its image on the underlying topological spaces is dense in $X$.
(This definition will be needed to generalize \cite[\S 3.6]{Suz13}.)
Taking $X_{\lambda}$ to be quasi-algebraic,
we can define profinite sets of points of $X = \invlim X_{\lambda}$
and, if the transition morphisms $X_{\mu} \to X_{\lambda}$ are flat
(or more generally, send irreducible components of $X_{\mu}$
onto dense subsets of irreducible components of $X_{\lambda}$),
the generic point of $X$,
in the same way as \cite[Def.\ 3.2.1]{Suz13}.
They are $\Spec$'s of ind-rational $k$-algebras.

Then all of the discussions and results in \cite[\S 3]{Suz13} work without modifications
with perfect affine $k$-schemes generalized to proalgebraic $k$-proschemes and
with affine proalgebraic groups generalized to arbitrary proalgebraic groups.
This proves \eqref{thm: Ext of proalgebraic groups as sheaves} in full generality.


\end{document}